%% file: algebraic_presentation_of_4-dimensional_2-handlebodies_and_3-dimensional_cobordisms.tex
\input{preamble}

\begin{document}

\include{S0-title-abstract-toc}
\include{S1-introduction}
\include{S2-preliminaries}
\include{S3-algebra}
\include{S4-topology}
\include{S5-equivalence}

\appendix

\include{S6-appendix}

\include{references}
\end{document}

%% file: preamble.tex
\documentclass[a4paper,reqno]{amsart}
\usepackage[margin=25mm,headsep=7.5mm,footskip=7.5mm]{geometry}

\usepackage{amscd,amssymb,amsmath,amsthm,amsfonts}
\usepackage{caption}


\makeatletter
\def\section{\@startsection{section}{1}%
   \z@{1.2\linespacing\@plus\linespacing}{0.8\linespacing}%
      {\normalfont\bf\boldmath\centering}}
\def\subsection{\@startsection{subsection}{1}%
   \z@{0.8\linespacing\@plus\linespacing}{0.5\linespacing}%
      {\normalfont\bf\boldmath}}
\def\@seccntformat#1{%
  \protect\textup{\protect\@secnumfont
     \ifnum\pdfstrcmp{section}{#1}=0 \bfseries\fi
     \ifnum\pdfstrcmp{subsection}{#1}=0 \bfseries\fi
     \csname the#1\endcsname\protect\@secnumpunct}}
\makeatother

\usepackage[utf8x]{inputenc}
\usepackage[english]{babel}
\usepackage{yfonts}
\usepackage[sans]{dsfont}
\usepackage{bbm}
\usepackage{mathtools}
\usepackage{accents}
\usepackage{tensor}
\usepackage[pdftex]{graphicx}
\usepackage[rgb]{xcolor}

\usepackage{tikz}
\usepackage{tikz-cd}
\usetikzlibrary{calc,matrix,arrows,decorations.pathmorphing}

\usepackage{calc}
\usepackage[cal=boondox,scr=boondoxo]{mathalfa}
\usepackage{enumitem}
\usepackage{marginnote}
\usepackage{booktabs}
\usepackage{scalerel}[2016/12/29]
\usepackage{url}
\usepackage{contour}
\usepackage{ulem}
\usepackage{comment}
\usepackage{placeins}

\usepackage{hyperref}
\hypersetup{colorlinks=true,pageanchor=false,linkcolor=blue,citecolor=blue,urlcolor=blue}
\usepackage[all]{hypcap}

\allowdisplaybreaks
\setlist{itemsep=2.5pt,leftmargin=30pt,labelsep=0.5em,nosep}

\newtheoremstyle{slanted}
  {2ex}{2ex}{\slshape}{\parindent}{\scshape}{.}{0.5em}{}

\theoremstyle{slanted}
\newtheorem{theorem}{Theorem}[subsection]
\newtheorem{corollary}[theorem]{Corollary}
\newtheorem{proposition}[theorem]{Proposition}
\newtheorem{lemma}[theorem]{Lemma}

\newtheorem{universal_property}[theorem]{Universal Property}

\newtheorem{maintheorem}{Theorem}[section]
\newtheorem{maincorollary}[maintheorem]{Corollary}

\newtheoremstyle{upright}
  {2ex}{2ex}{\upshape}{\parindent}{\scshape}{.}{0.5em}{}

\theoremstyle{upright}
\newtheorem{definition}[theorem]{Definition}

\newtheorem{remark}[theorem]{Remark}

\numberwithin{equation}{subsection}
\numberwithin{figure}{subsection}
\numberwithin{table}{subsection}
\makeatletter \let\c@table\c@figure \makeatother

\makeatletter
\renewenvironment{proof}[1][\proofname]{\par \pushQED{\qed}%
\normalfont \topsep6\p@\@plus6\p@\relax \trivlist
\item\relax{\slshape#1\@addpunct{.}}\hspace\labelsep\ignorespaces}
{\popQED\endtrivlist\@endpefalse}
\makeatother

\let\textit\textsl

\newcommand{\rmB}{\mathrm{B}}
\newcommand{\rmL}{\mathrm{L}}
\newcommand{\rmR}{\mathrm{R}}
\newcommand{\rmU}{\mathrm{U}}

\newcommand{\bbS}{\mathbb{S}}
\newcommand{\calC}{\mathcal{C}}

\newcommand{\calS}{\mathcal{S}}

\newcommand{\leqs}{\leqslant}
\newcommand{\geqs}{\geqslant}
\newcommand{\TTH}{\mathrm{TKT}}

\newcommand\arxiv[2]{\href{https://arXiv.org/abs/#1}
  {\texttt{arXiv:\allowbreak #1} #2}}
\newcommand\doi[2]{\href{https://doi.org/#1}{#2}}

\newcommand{\dfrmtn}[1]{$#1$-de\-for\-ma\-tion}
\newcommand{\dfrmtns}[1]{$#1$-de\-for\-ma\-tions}
\newcommand{\dmnsnl}[1]{$#1$-di\-men\-sion\-al}
\newcommand{\hndl}[1]{$#1$-han\-dle}
\newcommand{\hndls}[1]{$#1$-han\-dles}
\newcommand{\hndlbd}[1]{$#1$-han\-dle\-bod\-y}
\newcommand{\hndlbds}[1]{$#1$-han\-dle\-bod\-ies}
\newcommand{\mnfld}[1]{$#1$-man\-i\-fold}
\newcommand{\mnflds}[1]{$#1$-man\-i\-folds}
\newcommand{\qvlnc}[1]{$#1$-e\-quiv\-a\-lence}
\newcommand{\qvlnt}[1]{$#1$-e\-quiv\-a\-lent}

\makeatletter

\def\(#1){{\upshape(\hbox to 1.2ex{\hss#1\hss})}}

\def\rel#1{\ifmmode\text{\upshape(\rel@{#1}\upshape)}%
           \else{\upshape(\rel@{#1}\upshape)}\fi}

\def\hrel#1{\ifx#1[\hrel@[\else\hrel@@{#1}\fi}
\def\hrel@[#1]#2{{\upshape(\hyperref[E:#1]{\rel@{#2}}\upshape)}}
\def\hrel@@#1{{\upshape(\hyperref[E:#1]{\rel@{#1}}\upshape)}}

\def\rel@#1{{\gdef\n@xtsp{}{\textsl{\rel@@#1@@}\n@xtsp}}}

\def\rel@@#1{\gdef\n@xt{\rel@@}%
   \ifx#1@\gdef\n@xt{\@gobble}%
   \else\ifx#1f\n@xtsp f\gdef\n@xtsp{\kern0.2ex}%
   \else\ifx#1i\n@xtsp\kern0.1ex i\gdef\n@xtsp{\kern0.1ex}%
   \else\ifx#1j\n@xtsp\kern0.2ex j\gdef\n@xtsp{\kern0.1ex}%
   \else\ifx#1l\n@xtsp\kern0.1ex l\gdef\n@xtsp{\kern0.2ex}%
   \else\ifx#1I\n@xtsp I\gdef\n@xtsp{\kern0.2ex}%
   \else\ifx#1J\n@xtsp\kern-0.1ex J\gdef\n@xtsp{\kern0.1ex}%
   \else\ifx#1'\/$\mkern0.5mu'$\gdef\n@xtsp{}%
   \else\ifx#1-{\/\upshape-}\gdef\n@xtsp{}%
   \else\n@xtsp#1\gdef\n@xtsp{}%
   \fi\fi\fi\fi\fi\fi\fi\fi\fi\n@xt}

\def\myput{\@ifnextchar[{\@put}{\@@rput[\z@,\z@][r]}}
\def\@put[#1]{\@ifnextchar[{\@@put[#1]}{\@@@@@put[#1]}}
\def\@@put[#1][{\@ifnextchar{l}{\@@lput[#1][}{\@@@put[#1][}}
\def\@@@put[#1][{\@ifnextchar{c}{\@@cput[#1][}{\@@@@put[#1][}}
\def\@@@@put[#1][{\@ifnextchar{r}{\@@rput[#1][}{\relax}}
\def\@@@@@put[{\@ifnextchar{l}{\@@lput[\z@,\z@][}{\@@@@@@put[}}
\def\@@@@@@put[{\@ifnextchar{c}{\@@cput[\z@,\z@][}{\@@@@@@@put[}}
\def\@@@@@@@put[{\@ifnextchar{r}{\@@rput[\z@,\z@][}{\@@@@@@@@put[}}
\def\@@@@@@@@put[#1]{\@@rput[#1][r]}

\let\hm@d@\leavevmode

\long\def\@@lput[#1,#2][l]#3{\setbox0\hbox{#3}\hm@d@\raise#2\hbox to\z@{\dimen0 #1%
  \advance\dimen0-\wd0\kern\dimen0\dp0\z@\ht0\z@\wd0\z@\box0\hss}\ignorespaces}
\long\def\@@cput[#1,#2][c]#3{\setbox0\hbox{#3}\hm@d@\raise#2\hbox to\z@{\dimen0 #1%
  \advance\dimen0-.5\wd0\kern\dimen0\dp0\z@\ht0\z@\wd0\z@\box0\hss}\ignorespaces}
\long\def\@@rput[#1,#2][r]#3{\setbox0\hbox{\kern#1\raise#2\hbox{#3}}%
  \dp0\z@\ht0\z@\wd0\z@\hm@d@\box0\ignorespaces}

\def\@@bold{bold}

\def\widebar#1{\text{\setbox15\hbox{$#1$}%
  \dimen15 0.45\wd15\advance\dimen15 0.15\ht15%
  \dimen16\ht15\advance\dimen16 0.00em\advance\dimen16 0.3ex%
  \dimen17 0.65\wd15\advance\dimen17 0.05\ht15\advance\dimen17 0.1ex%
  \dimen18 0.035em\advance\dimen18 0.00ex
  \@@cput[\dimen15,\dimen16][c]{\vrule depth 0pt height \dimen18 width \dimen17}}#1}

 \def\boldbar#1{\text{\setbox15\hbox{$#1$}%
   \dimen15 0.45\wd15\advance\dimen15 0.15\ht15%
   \dimen16\ht15\advance\dimen16 0.00em\advance\dimen16 0.26ex%
   \dimen17 0.65\wd15\advance\dimen17 0.05\ht15\advance\dimen17 0.1ex%
   \dimen18 0.05em\advance\dimen18 0.00ex
   \@@cput[\dimen15,\dimen16][c]{\vrule depth 0pt height \dimen18 width \dimen17}}#1}

\makeatother

\let\bar\widebar
\let\hat\widehat
\let\tilde\widetilde

\newcommand{\barGamma}{\mkern-1.5mu\bar{\mkern1.5mu\Gamma\mkern-1.5mu}\mkern1.5mu}
\newcommand{\barPhi}{\mkern-1mu\bar{\mkern1mu\Phi\mkern-1mu}\mkern1mu}

\def\varemptyset{{\text{\raise.21ex\hbox{$\not$}}\mkern.15mu\mathrm{O}\mkern.15mu}}

\let\emptyset\varemptyset

\def\cal#1{\mathcal{#1}}

\newcommand{\G}{{\cal G}}

\definecolor{Acolor}{cmyk}{0, 1, 0, 0.15}
\definecolor{Icolor}{cmyk}{0, 0.35, 1, 0.15}
\definecolor{Mcolor}{cmyk}{1, 0, 0, 0.25}
\definecolor{Rcolor}{cmyk}{0.5, 0, 1, 0.25}

\makeatletter

\definecolor{@@color}{cmyk}{0, 1, 1, 0.25}
\definecolor{a@color}{cmyk}{0, 1, 0, 0.15}
\definecolor{i@color}{cmyk}{0, 0.35, 1, 0.15}
\definecolor{m@color}{cmyk}{1, 0, 0, 0.25}
\definecolor{r@color}{cmyk}{0.5, 0, 1, 0.25}

\def\note{\@ifnextchar[{\@note}{\@note[]}}

\long\def\@@note[#1]#2{\def\temp{#1}%
  \ifx\temp\@empty%
    {\color{@@color}{\boldmath$\langle$}#2{\boldmath$\rangle$}}%
  \else\if\temp A%
    {\color{a@color}{\boldmath$\langle$}\textbf{A:} #2{\boldmath$\rangle$}}%
  \else\if\temp I%
    {\color{i@color}{\boldmath$\langle$}\textbf{I:} #2{\boldmath$\rangle$}}%
  \else\if\temp M%
    {\color{m@color}{\boldmath$\langle$}\textbf{M:} #2{\boldmath$\rangle$}}%
  \else\if\temp R%
    {\color{r@color}{\boldmath$\langle$}\textbf{R:} #2{\boldmath$\rangle$}}%
  \else%
    {\color{@@color}{\boldmath$\langle$}\textbf{#1:} #2{\boldmath$\rangle$}}%
  \fi\fi\fi\fi\fi}

\def\hidenotes{\long\def\@note[##1]##2{\ignorespaces}}
\def\shownotes{\let\@note\@@note}

\shownotes

\def\post{\@ifstar{\post@l}{\post@r}}

\def\post@r{\@ifnextchar[{\post@@r}{\post@@r[]}}
\def\post@l{\@ifnextchar[{\post@@l}{\post@@l[]}}

\long\def\post@@@r[#1]#2{\def\temp{#1}%
  \ifx\temp\@empty{\postitcolor{@@color}\postit{#2}}%
  \else\if\temp A{\postitcolor{a@color}\postit{\textbf{A:} #2}}%
  \else\if\temp I{\postitcolor{i@color}\postit{\textbf{I:} #2}}%
  \else\if\temp M{\postitcolor{m@color}\postit{\textbf{M:} #2}}%
  \else\if\temp R{\postitcolor{r@color}\postit{\textbf{R:} #2}}%
  \else{\postitcolor{@@color}\postit{\textbf{#1:} #2}}%
  \fi\fi\fi\fi\fi\ignorespaces}

\long\def\post@@@l[#1]#2{\def\temp{#1}%
  \ifx\temp\@empty{\postitcolor{@@color}\postit*{#2}}%
  \else\if\temp A{\postitcolor{a@color}\postit*{\textbf{A:} #2}}%
  \else\if\temp I{\postitcolor{i@color}\postit*{\textbf{I:} #2}}%
  \else\if\temp M{\postitcolor{m@color}\postit*{\textbf{M:} #2}}%
  \else\if\temp R{\postitcolor{r@color}\postit*{\textbf{R:} #2}}%
  \else{\postitcolor{@@color}\postit*{\textbf{#1:} #2}}%
  \fi\fi\fi\fi\fi\ignorespaces}

\def\hideposts{\long\def\post@@r[##1]##2{\ignorespaces}
               \long\def\post@@l[##1]##2{\ignorespaces}}
\def\showposts{\let\post@@r\post@@@r\let\post@@l\post@@@l}

\hideposts

\def\postitfontsize#1{\let\postitf@ntsize#1}
\def\postitcolor#1{\def\postitc@lor{#1}}

\postitfontsize{\tiny}
\postitcolor{red}

\newdimen\postitm@xwidth
\def\postitmaxwidth#1{\postitm@xwidth=#1}

\postitmaxwidth{8cm}

\def\postit{\@ifstar{\postit@l}{\postit@r}}

\makeatother

\newcommand{\N}{\mathbb{N}}

\newcommand{\R}{\mathbb{R}}

\newcommand{\Bd}{\mathop{\mathrm{Bd}}\nolimits}

\renewcommand{\wr}{\mathop{\mathrm{wr}}\nolimits}
\newcommand{\hrssh}{{\sqsubset}}

\newcommand{\csum}{\mathbin{\#}}
\newcommand{\bcsum}{\mathbin{\natural}}

\newcommand{\one}{{\boldsymbol{\mathbbm 1}}}
\newcommand{\idone}{\id_\one}

\newcommand{\sym}{\mathop{\mathrm{sym}}\nolimits}
\newcommand{\Forget}{\mathcal{F}}

\newcommand{\RHB}{4\mathrm{HB}}
\newcommand{\RCob}{3\mathrm{Cob}}
\newcommand{\Cob}{\mathrm{Cob}}

\newcommand{\Alg}{\mathrm{Alg}}
\newcommand{\Algf}{4\mathrm{Alg}}
\newcommand{\AlgD}{\mathrm{TAlg}}
\newcommand{\AlgL}{\mathrm{MAlg}}
\newcommand{\AlgLfree}{\mathrm{MAlg}^{\mathrm{F}}}
\newcommand{\Algt}{3\mathrm{Alg}}
\newcommand{\AlgH}{3\mathrm{Alg}^{\mathrm{H}}}
\newcommand{\AlgK}{3\mathrm{Alg}^{\mathrm{K}}}

\newcommand{\KT}{\mathrm{KT}}
\newcommand{\KTf}{\mathrm{4KT}}
\newcommand{\KTt}{\mathrm{3KT}}

\newcommand{\id}{\mathrm{id}}
\newcommand{\ad}{\mathrm{ad}}
\newcommand{\ev}{\mathrm{ev}}
\newcommand{\coev}{\mathrm{coev}}
\newcommand{\lev}{\smash{\stackrel{\leftarrow}{\mathrm{ev}}}}
\newcommand{\lcoev}{\smash{\stackrel{\longleftarrow}{\mathrm{coev}}}}
\newcommand{\rev}{\smash{\stackrel{\rightarrow}{\mathrm{ev}}}}
\newcommand{\rcoev}{\smash{\stackrel{\longrightarrow}{\mathrm{coev}}}}

\newcommand{\braid}{c}
\newcommand{\prodH}{\mu}
\newcommand{\unitH}{\eta}
\newcommand{\coprH}{\Delta}
\newcommand{\counH}{\varepsilon}
\newcommand{\antipH}{S}
\newcommand{\intfH}{\lambda}
\newcommand{\inteH}{\Lambda}
\newcommand{\ribmorH}{\tau}
\newcommand{\copairH}{w}
\newcommand{\pairH}{\bar w}
\newcommand{\ribelH}{v}
\newcommand{\adjH}{\ad}
\newcommand{\cadjH}{\mathcal{A}\kern-0.3em\mathcal{d}}
\newcommand{\idfunct}{\mathcal{I}\kern-0.45em\mathcal{d}}

\newcommand{\monH}{\Omega}
\newcommand{\barmonH}{\mkern-1.2mu\bar{\mkern1.2mu\Omega\mkern-1.2mu}\mkern1.2mu}

\newcommand{\Sou}{W}
\newcommand{\Tar}{\bar W}

\newcommand{\checkTheta}{\myput[3.8pt,15.5pt]{\scalebox{-1}{$\widehat{}$}}\Theta}

\usepackage{accents}
\newlength{\dhatheight}
\newcommand{\doublehat}[1]{%
  \settoheight{\dhatheight}{\ensuremath{#1}}%
  \addtolength{\dhatheight}{0.35ex}%
  \mathchoice%
  {\hat{\phantom{#1\rule{-0.055em}{\dhatheight}}}\mathllap{\hat{#1}}}%
  {\hat{\phantom{#1\rule{-0.055em}{\dhatheight}}}\mathllap{\hat{#1}}}%
  {\hat{\phantom{#1\rule{-0.055em}{0.7\dhatheight}}}\mathllap{\hat{#1}}}%
  {\hat{\phantom{#1\rule{-0.055em}{0.5\dhatheight}}}\mathllap{\hat{#1}}}}

\newcommand{\balpha}{\hat{\alpha}}
\newcommand{\bbalpha}{\doublehat{\alpha}}

\newcommand{\bn}{$\bar{\text{\sl n}}$}

\newcommand{\hatK}{$\mkern2mu\hat{\mkern-2mu\text{\sl K}}\mkern2mu$}


%% file: S0-title-abstract-toc.tex
\title[Algebraic Presentation of $4$-Dimensional $2$-Handlebodies and $3$-Dimensional Cobordisms]{Algebraic Presentation of $4$-Dimensional $2$-Handlebodies\\ and $3$-Dimensional Cobordisms}

\author[A. Beliakova]{Anna Beliakova} 
\address{Institute of Mathematics, University of Zurich, Winterthurerstrasse 190, CH-8057 Zurich, Switzerland} 
\email{anna@math.uzh.ch}

\author[I. Bobtcheva]{Ivelina Bobtcheva} 
\address{Institute of Mathematics, University of Zurich, Winterthurerstrasse 190, CH-8057 Zurich, Switzerland}
\email{ivelina.bobtcheva@math.uzh.ch}

\author[M. De Renzi]{Marco De Renzi} 
\address{Institut Monpelliéran Alexander Grothendieck, Université de Montpellier, Place Eugène Bataillon, 34090 Montpellier, France}
\email{marco.de-renzi@umontpellier.fr}

\author[R. Piergallini]{Riccardo Piergallini} 
\address{Scuola di Scienze e Tecnologie, Università di Camerino, Italy} 
\email{riccardo.piergallini@unicam.it}

\begin{abstract}
In this paper, we give a new direct proof of a result by Bobtcheva and Piergallini that provides finite algebraic presentations of two categories, denoted $\RCob$ and $\RHB$, whose morphisms are manifolds of dimension $3$ and $4$, respectively. More precisely, $\RCob$ is the category of connected {oriented} \dmnsnl{3} cobordisms between connected surfaces with connected boundary, while $\RHB$ is the category of connected {oriented} \dmnsnl{4} \hndlbds{2} up to \dfrmtns{2}. For this purpose, we explicitly construct the inverse of the functor $\Phi: \Algf \to \RHB$, where $\Algf$ denotes the free monoidal category generated by a Bobtcheva--Piergallini Hopf algebra. As an application, we deduce an algebraic presentation of $\RCob$ and show that it is equivalent to the one conjectured by Habiro.
\end{abstract}


\maketitle

\kern-6pt\vfill

\tableofcontents

%% file: S1-introduction.tex
\section{Introduction}
\label{introduction/sec}

Categories of \dmnsnl{n} cobordisms play a central role in low-dimensional topology, and have been the subject of extensive study. The category $2\Cob$ of \dmnsnl{2} cobordisms is known to be freely generated, as a symmetric monoidal category, by a commutative \textit{Frobenius} algebra: the circle. This algebraic presentation yields the classification of all \textit{Topological Quantum Field Theories} (\textit{TQFTs}) in dimension~$2$. This paper focuses on an extension of this result to dimensions~$3$ and $4$. More precisely, we discuss complete algebraic presentations ({with finitely many} generators and relations) of certain topological categories generated, as braided monoidal categories, by a single object: the punctured torus, in dimension~$3$, and the solid torus, in dimension~$4$\footnote{The categories we focus on are actually skeleta of larger categories whose objects are required to carry some very mild extra structures. More precisely, in these larger categories, the punctured torus is equipped with an identification between $S^1$ and its boundary, and the solid torus is equipped with an embedding of $D^2$ into its boundary, see Remarks~\ref{3Cob-skeleton/rmk} and \ref{4HB-skeleton/rmk}. We simply do not need to specify these extra structures here, because there always exist canonical ones in the specific skeleta we are considering.}. In both cases, these objects admit structures of braided \textit{Hopf} algebras that can be further enriched, thus leading to the notion of Bobtcheva--Piergallini Hopf algebras, or simply BP Hopf algebras, see Subsections~\ref{BP Hopf algebra/sec} and \ref{HabiroHalgebra/sec} for a definition. 

A nice and simple algebraic presentation, such as the one for $2\Cob$, cannot be expected for the standard categories of cobordisms in dimension~$3$ and $4$, since both admit infinitely many non-isomorphic connected objects. Indeed, a complete algebraic presentation of the standard category of \dmnsnl{n} cobordisms was given, for every $n \geqs 3$, by Juhász in terms of surgery operations \cite{Ju14}, but his lists of generating objects, generating morphisms, and relations between morphisms are all infinite. There is, however, a natural category of \dmnsnl{3} cobordisms that admits a single generating object: it is the category $\RCob$ of connected oriented (relative) $3$-dimensional cobordisms between connected surfaces with connected boundary, whose tensor product 
is given by boundary connected sum. This category is a PROB, meaning that it is a braided monoidal category whose set of objects can be identified with $\N$, and whose tensor product adds up natural numbers. Hence, $\RCob$ is monoidally generated by a single object, the once-punctured torus. The fact that the punctured torus admits the structure of a braided Hopf algebra in $\RCob$ was first discovered by Crane and Yetter \cite{CY94}. 

Building on this observation, Kerler provided a finite set of generating morphisms for $\RCob$, and exhibited a finite list of beautiful and conceptual relations between them \cite{Ke01}, although he was not able to prove that his list was complete, and that he had an algebraic presentation. Since finding one would also yield a classification of all TQFTs with source $\RCob$, this was recognized as one of the central problems in quantum topology, and included in Ohtsuki's list \cite[Problem~8.16.(1)]{Oh02}. A few years later, Habiro announced a solution to the problem, and his presentation appeared in \cite{As11}. Unfortunately, a proof of his claim was never written down.

Kerler's question was answered by two of the authors of the present paper, who first gave a complete algebraic presentation of $\RCob$ in \cite{BP11}. Surprisingly, the solution follows from an algebraic presentation of a category whose morphisms are manifolds one dimension higher.

In order to explain this, we need to turn our attention to \dmnsnl{4} \hndlbds{2}, which are smooth manifolds obtained from the $4$-ball by attaching finitely many \hndls{1} and \hndls{2}. Up to considering a natural equivalence relation on them, discussed here below, connected oriented \dmnsnl{4} \hndlbds{2} can be organized as the morphisms of a category $\RHB$ whose objects are connected oriented \dmnsnl{3} \hndlbds{1}\footnote{For the sake of simplicity, in the rest of the paper we will write \dmnsnl{4} \hndlbds{2} to mean connected oriented ones, and \dmnsnl{3} handlebodies to mean connected oriented \dmnsnl{3} \hndlbds{1}.}. As for $\RCob$, this is a close relative of the standard category of (smooth) connected oriented \dmnsnl{4} cobordisms, whose objects have boundary, and whose tensor product 
is induced by boundary connected sum. By contrast with $\RCob$, however, or with any other category of cobordisms, the vertical boundary of morphisms in $\RHB$ is not required to be trivial, in the sense that it is not necessarily the cylinder over a surface. 

The natural equivalence relation appearing in the definition of morphisms in $\RHB$ is called \qvlnc{2}, and it is induced by \dfrmtns{2}, which are diffeomorphisms that can be implemented by finite sequences of handle moves that never step outside of the class of \dmnsnl{4} \hndlbds{2}. In other words, when considering \dmnsnl{4} \hndlbds{2} up to \dfrmtns{2}, creation and removal of canceling pairs of handles of index $2/3$ and $3/4$ is forbidden. Whether \dfrmtns{2} form a proper subclass of diffeomorphisms is still an open question, which is closely related to a fundamental open problem in combinatorial group theory: the Andrews--Curtis conjecture.

A standard way of representing \dmnsnl{4} \hndlbds{2} is through Kirby tangles, which are obtained by drawing the attaching maps of \hndls{2} on the boundary of a single \hndl{0} with \hndls{1} glued to it, and then considering a generic planar projection. It is convenient to represent \hndls{1} as dotted unknots bounding Seifert disks in the plane. Under this convention, a \hndl{2} running over a \hndl{1} will appear as a knot that pierces the corresponding Seifert disk. Such tangles, modulo isotopy, \hndl{2} slides, and \hndl{1/2} cancellations, form a category $4\KT$ which is equivalent to $\RHB$ \cite{Ki89, GS99, Ke98, BP11}.

The algebraic counterpart of $\RHB$ is the category $\Algf$, which is a PROB that is freely generated by a Bobtcheva--Piergallini (or BP) Hopf algebra. The approach of \cite{BP11} consists in defining a functor $\Phi: \Algf \to 4\KT$ and showing that it is an equivalence by factoring it through an equivalence functor from the category of labeled ribbon surfaces to $\Algf$. A labeled ribbon surface serves as a branching set in the description of a \dmnsnl{4} \hndlbd{2} as a branched cover of the $4$-ball.

In the present paper we provide a simpler direct proof of the same result.

\begin{maintheorem}\label{thm:main}
The functor $\Phi : \Algf \to \RHB$ sending the generating BP Hopf algebra of $\Algf$ to the solid torus\footnote{As mentioned before, the generating object of $\RHB$ is a standard solid torus that naturally carries a canonical embedding of $D^2$ into its boundary.}
is an equivalence of braided monoidal categories. 
\end{maintheorem} 

The idea of our new proof is to construct the inverse functor $\barPhi: 4\KT \to \Algf$ directly and explicitly, without any reference to branched coverings. The assignment of a morphism in $\Algf$ to a Kirby tangle depends, in our approach, on many auxiliary choices. The main body of the proof deals with the independence on these choices.

An immediate application of Theorem~\ref{thm:main} is the following detection result. If $T$ and $T'$ are Kirby tangles such that $\barPhi(T) = \barPhi(T')$, then $T$ is isomorphic to $T'$ in $4\KT$ and the corresponding \dmnsnl{4} \hndlbds{2} in $\RHB$ can be 2-deformed into each other.

A further important consequence of Theorem~\ref{thm:main} is an algebraic presentation of $\RCob$. Indeed, there exists a natural boundary functor $\partial_+ : \RHB \to \RCob$ making the diagram
\[
\begin{tikzpicture}
 \node (P0) at (0,0) {$\Algf$};
 \node (P1) at (2.25,0) {$\Algt$};
 \node (P2) at (0,-1.5) {$\RHB$};
 \node (P3) at (2.25,-1.5) {$\RCob$};
 \draw
 (P0) edge[->] node[above] {$\textstyle \pi$} (P1)
 (P0) edge[->] node[left, xshift=-0.25ex] {$\textstyle \Phi$} (P2)
 (P1) edge[->] node[right, xshift=0.25ex] {$\textstyle \partial_+ \Phi$} (P3)
 (P2) edge[->] node[below] {$\textstyle \partial_+$} (P3);
\end{tikzpicture}
\]
into a commutative one. Here, $\Algt$ is a certain quotient of $\Algf$ obtained by adding two additional relations (which make the generating object into a factorizable and anomaly-free BP Hopf algebra). In order to represent morphisms in $\RCob$, we use top tangles in handlebodies, which are an adaptation to our conventions of Habiro's bottom tangles in handlebodies (since Habiro reads diagrams from top to bottom, while we do the opposite). Thus, we can deduce the following.

\begin{maincorollary} \label{cor:main}
The functor $\partial_+ \Phi : \Algt \to \RCob$ sending the generating factorizable and anomaly-free BP Hopf algebra of $\Algt$ to the punctured torus\footnote{As mentioned before, the generating object of $\RCob$ is a standard punctured torus that naturally carries a canonical identification between $S^1$ and its boundary.} is an equivalence of braided monoidal categories.
\end{maincorollary}

\begin{proof}[Proof $($assuming Theorem~\ref{thm:main}$)$]
We will show in Section~\ref{topology/sec} that $\RCob\cong \KTt$ is the quotient of $\RHB\cong\KTf$ by the two relations depicted in Table~\ref{table-Kirby-moves/fig}. Written algebraically, these relations correspond exactly to relations~\hrel{f} and \hrel{n} introduced in Subsection~\ref{HabiroHalgebra/sec}. Moreover, $\Algt$ is defined precisely as the quotient of $\Algf$ by these relations. The claim follows now from Theorem~\ref{thm:main} and Proposition~\ref{3KT/thm}.
\end{proof}

This algebraic presentation, which first appeared in \cite{BP11}, does not coincide with the one announced by Habiro (see \cite{As11}). Indeed, the latter identifies $\RCob$ with the free monoidal category $\Algt^{\rm H}$ generated by a Habiro Hopf algebra, which features a different set of generating morphisms, and a different list of relations (see Subsection~\ref{HabiroHalgebra/sec} for a definition). However, we prove that $\Algt$ and $\Algt^{\rm H}$ are equivalent as braided monoidal categories, thus establishing the Kerler--Habiro conjecture.

\begin{samepage}
\begin{maintheorem}[Kerler--Habiro Conjecture]\label{thm:KHC}
The functor $\Gamma: \AlgH \to \Algt$ sending the generating Habiro Hopf algebra of $\AlgH$ to the generating factorizable anomaly-free BP Hopf algebra of $\Algt$ is an equivalence of braided monoidal categories. Hence, the functor $\partial_+ \Phi \circ \Gamma : \AlgH \to \RCob$ sending the generating Habiro Hopf algebra of $\AlgH$ to the punctured torus is an equivalence of braided monoidal categories.
\end{maintheorem}
\end{samepage}

The braided monoidal functor $\Gamma: \AlgH \to \Algt$ was first constructed by the second author in \cite{Bo20}. In this paper, we define its inverse, thus proving that the algebraic presentations of $\RCob$ given in \cite{BP11} and \cite{As11} are equivalent. In addition, we provide a third algebraic presentation $\AlgK$ by adding to Kerler's original list of axioms the braided cocommutativity relation for the adjoint action (a crucial relation appearing in Habiro's presentation), and show that $\AlgK$ is equivalent to $\AlgH$. Clearly, also in dimension~$3$ the equality $\pi (\barPhi(T)) = \pi (\barPhi(T'))$ implies an isomorphism between $T$ and $T'$ in $3\KT$, and an equivalence of the corresponding cobordisms in $\RCob$.

Besides giving a complete algebraic presentation of $\RHB$ and $\RCob$, Theorems~\ref{thm:main} and \ref{thm:KHC} also classify braided monoidal functors on them. For what concerns existence of examples, in \cite{BD21} it is shown that every unimodular ribbon Hopf algebra, and more generally every unimodular ribbon category,\footnote{A ribbon category is unimodular if it is finite and if 
the projective cover of its tensor unit is self-dual.} gives rise to such a functor (a TQFT) on the category of \dmnsnl{4} \hndlbds{2} up to \dfrmtns{2}. We point out that the notion of \qvlnc{2} between \dmnsnl{4} \hndlbds{2} is conjectured {by Gompf in \cite{Go91}} to be different from the one of diffeomorpism, which in this context is equivalent to \qvlnc{3}. In order to prove Gompf's conjecture, we can look for a unimodular ribbon Hopf algebra whose corresponding quantum invariant distinguishes diffeomorphic handlebodies that are not \qvlnt{2}. The search for such Hopf algebras is a non-trivial challenge, since they have to combine several properties: at the very least, they should be unimodular, non-factorizable, and non-semisimple (see \cite[Subsection~1.1]{BD21} and \cite[Section~2]{BM02}). 
Quantum groups satisfying all these properties do not seem to lead to interesting invariants of \mnflds{4}, but rather to homological refinements of known quantum invariants of their \dmnsnl{3} boundaries \cite{BD22}.
On the other hand, if the conjecture is false, every unimodular ribbon Hopf algebra, and more generally every unimodular ribbon category, gives rise to a quantum invariant of \dmnsnl{4} \hndlbds{2} up to diffeomorphisms, and may be useful for detecting exotic structures on \mnflds{4}.

Apart from $\Algt$, there is another interesting quotient of $\Algf$, defined in \cite{Bo23} as the symmetric monoidal category freely generated by a BP Hopf algebra with trivial ribbon element (in particular, such a Hopf algebra is cocommutative, and its braiding is symmetric). Topologically, this quotient describes the category of \dmnsnl{2} CW-complexes up to \qvlnc{2}, and hence it is designed to study the Andrews--Curtis conjecture. Let us recall that the Andrews--Curtis conjecture states that every balanced\footnote{A presentation of a group is balanced if it has the same number of generators and relators.} presentation of the trivial group can be reduced to the trivial presentation trough balanced presentations (that is, by a sequence of Nielsen transformations on relators and conjugations of relators by generators). This conjecture is open since 1965, and expected to be false. To test potential counterexamples, new cocommutative BP Hopf algebras with symmetric braiding and trivial ribbon element need to be constructed.

The one-to-one correspondence between algebraic and topological structures established in this paper might also be useful for understanding quantum groups or ribbon Hopf algebras, since it provides new graphical methods for establishing identities or constructing central elements. Indeed, every time we happen to know that a complicated tangle can be trivialized, then it follows that the associated morphism in $\Algf$ is the identity. 

\subsection{Strategy of the proof of Theorem \ref{thm:main}}

Let us explain the main ideas behind the proof of Theorem~\ref{thm:main}. A Hopf algebra $H$ in a braided monoidal category $\calC$ comes equipped with the following structure morphisms:
\begin{itemize}
\item a product $\prodH: H \otimes H \to H$ and a unit $\unitH: \one \to H$;
\item a coproduct $\coprH: H \to H \otimes H$ and a counit $\counH : H \to \one$;
\item an invertible antipode $\antipH: H \to H$.
\end{itemize}
These structure morphisms are required to satisfy the standard axioms depicted in Table~\ref{table-Hopf/fig}. A Bobtcheva--Piergallini Hopf algebra (or BP Hopf algebra for short) is a Hopf algebra in $\mathcal C$ equipped with the following additional morphisms:
\begin{samepage}
\begin{itemize}
\item an integral form $\intfH : H \to \one$ and an integral element $\inteH: \one \to H$;
\item an invertible ribbon morphism $\ribmorH : H \to H$;
\item a copairing $w: \one\to H\otimes H$.
\end{itemize}
\end{samepage}
\begin{figure}[htb]
 \centering
 \includegraphics{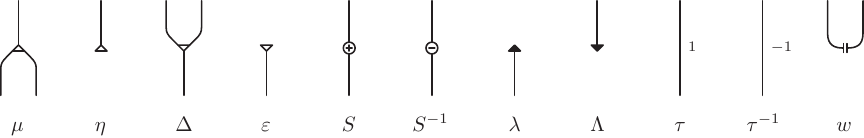}
 \caption{}
 \label{intro-hopf/fig}
\end{figure}
These morphisms are required to satisfy a set of axioms, which can be found in Subsection~\ref{BP Hopf algebra/sec}. To present the generating morphisms and relations between them we will use the graphical notation shown in Figure~\ref{intro-hopf/fig}. We define $\Algf$ as the PROB freely generated by a BP Hopf algebra.

In order to construct the functor $\Phi: \Algf \to 4\KT$, we need to assign Kirby tangles to generating morphisms,
and to check all relations. The images of the structure morphisms under $\Phi$ are given in Figure~\ref{Phi-intro/fig} and the relations are checked in Subsection~\ref{Phi/sec}.

\begin{figure}[htb]
 \centering
 \includegraphics{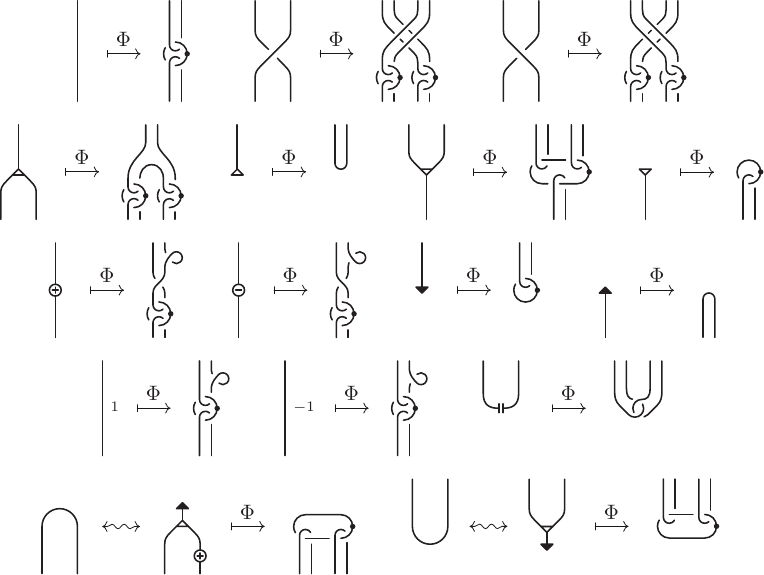}
 \caption{Definition of the functor $\Phi$ for the generating morphisms and the evaluation and coevaluation in $\Algf$.}
 \label{Phi-intro/fig}
\end{figure}

Notice that the assignment defined in Figure~\ref{Phi-intro/fig} replaces each strand representing a copy of the BP Hopf algebra $H$ with two undotted parallel strands representing (a portion of) a \hndl{2} in $4\KT$. The claim that the functor $\Phi$ is full might then be surprising, since a generic tangle in $4\KT$ does not have this property. However, for any diagram $D$ of a Kirby tangle $T$, we can choose a so-called \textsl{bi-ascending} state for all undotted components. This reduces to the choice of a collection of crossings that need to be reversed in order to trivialize the {undotted link representing the} 2-handles. Then, we can build a connected sum of each undotted component with its trivialization along chosen bands. The resulting diagram still represents $T$, and has the property that each undotted component is doubled {by a trivial copy which lies below it}. An example is given by the first and the last diagrams in Figure~\ref{MainExample/fig}, where the doubling is drawn in gray for convenience.

An algebra morphism $\barPhi(T)$ with the property that $\Phi(\barPhi(T)) = T$ is constructed as follows. Given a diagram of a Kirby tangle $T$, we specify a bi-ascending state by marking (with gray disks) those crossings that should be changed in order to trivialize the undotted link. Then, we pick a family of bands $\alpha$ connecting the undotted link to the bottom base of the projection plane, and we call the resulting diagram $T_\alpha$. Next, we decompose $T_\alpha$ into elementary pieces and assign algebra morphisms to each piece as prescribed in Figures~\ref{kirby-hopf03-1/fig}, \ref{kirby-hopf03-2/fig}, and \ref{kirby-hopf03-3/fig}. Finally, we tensor and compose all these morphisms together. This process is illustrated in Figure~\ref{MainExample/fig}. Notice that the algebra morphism we assign to a crossing depends on whether this crossing is affected by the trivialization or not. By applying the functor $\Phi$\break to the resulting algebra morphism $\barPhi(T)$, we can verify that $\Phi(\barPhi(T))$ is isotopic to the original tangle $T$.

\begin{figure}[htb]
 \centering
 \includegraphics{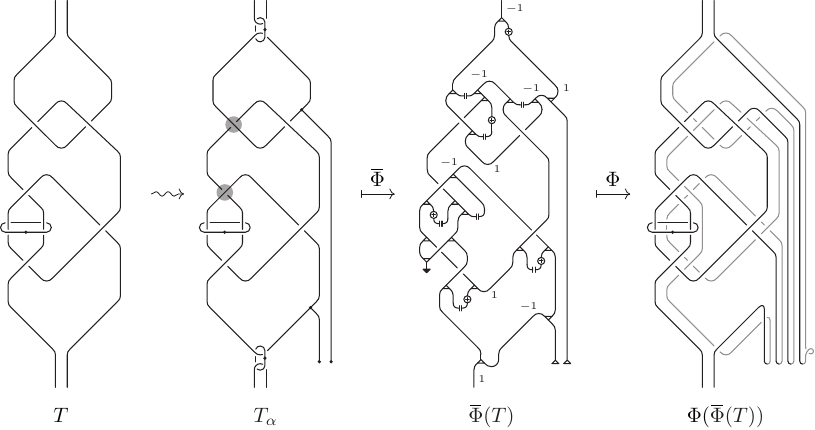}
 \caption{Example of assignment of the algebraic morphism $\barPhi(T)$ to a Kirby tangle $T$ satisfying $\Phi(\barPhi(T))=T$.}
 \label{MainExample/fig}
\end{figure}

The main body of the proof consists in checking that our assignment actually extends to a well-defined functor $\barPhi: 4\KT \to \Algf$ that is inverse to $\Phi$. For this purpose, we need to prove that $\barPhi(T)$ does not depend on the various choices we made, meaning that it is independent of the bi-ascending state, of the set of bands, and of the diagram we picked.

Moreover, we need to check that our assignment is invariant under isotopies and \dfrmtns{2} of $T$, and that it is compatible with identities, compositions, tensor products, and braidings. 
The main tool in the proof of these properties will be provided by some recursively constructed collection of morphisms $\Theta = \{ \Theta_k : H^{\otimes k+1} \to H^{\otimes k} \mid k \in \N \}$ that intertwines all morphisms in a natural subcategory $\AlgD$ of $\Algf$ generated (under tensor products and compositions)
by some morphisms in the image of $\barPhi$ (shown in Figure~\ref{kirby-hopf03-1/fig}). More precisely, if $\iota : \AlgD \hookrightarrow \Algf$ denotes the inclusion functor, then $\Theta: \iota \otimes H \Rightarrow \iota$ defines a natural transformation, meaning that, for a morphism $F : H^{\otimes s} \to H^{\otimes t}$ in $\AlgD$, we have
\[
\Theta_{t} \circ (F \otimes \id) = F \circ \Theta_{s}.
\]
Geometrically, $\Theta$ implements a $1$-handle embracing all the strands of $\Phi(\barPhi(T))$ corresponding to the trivialized copy of $T$ in gray. To check independence of the bi-ascending state, we will also need to implement algebraically a $1$-handle embracing the trivialized copy of a single component of $T$, which will require the construction of a family of labeled versions of $\Theta$.

\subsection{Organization}

We start our paper with some algebraic background, in Section~\ref{algebra/sec}. After recalling the notion of braided monoidal and ribbon categories, we introduce BP Hopf algebras, and define $\Algf$ as the braided monoidal category freely generated by a BP Hopf algebra. For each of these algebraic structures we give a diagrammatic presentation of the defining set of axioms. We prove that $\Algf$ admits the structure of a ribbon category, and that its generating object also admits the structure of a Frobenius algebra.

We introduce the notions of factorizable and anomaly-free BP Hopf algebras, which lead to the definition of the quotient category $\Algt$ of $\Algf$. Then, after recalling the definition of $\AlgH$, we construct a functor $\Gamma : \AlgH \to \Algt$, and prove that it is an equivalence. Furthermore, we deduce another presentation $\AlgK$ of $\Algt$, which is obtained from the list of axioms found by Kerler in \cite{Ke01} by adding the braided cocommutativity relation for the adjoint action. 

In Section~\ref{topology/sec}, we collect some topological background. First, we recall the definition of the categories $\RHB$ and $\RCob$, which are equivalent to the categories of Kirby tangles $\KTf$ and $\KTt$, respectively. They are naturally related by a functor $\partial_+ : \RHB \to \RCob$ that maps each \dmnsnl{4} \hndlbd{2} to its front boundary. Finally, we recall (an upside-down version of) Habiro's graphical notation for morphisms in $\RCob$ as top tangles in handlebodies.

Section~\ref{equivalence/sec} is devoted to the proof of Theorem~\ref{thm:main}. After defining the functor $\Phi : \Algf \to \KTf$, we proceed with the construction of its inverse. In order to do this, we start by introducing a certain subcategory $\AlgD$ of $\Algf$ whose image under $\Phi$ consist of Kirby tangles whose $2$-handles are separated by the projection plane in two levels. We describe generators of $\AlgD$ explicitly in terms of decorated crossings, and show that $\AlgD$ admits two different ribbon structures. Next, we define two natural transformations $\Theta$ and $\Theta'$ that will be extensively used in the proof of our main result.

In Subsection~\ref{bias/sec}, we introduce bi-ascending states of link diagrams, and we describe a complete set of moves relating any two bi-ascending states of the same link diagram. Subsection~\ref{FK/sec} is devoted to the construction of the inverse functor $\barPhi : \KTf \to \Algf$. In the following one, we define yet another pair of natural transformations $\Theta^{\rm L}_j$ and $\hat \Theta^{\rm L}_j$ on a labeled version of $\Algf$. In the last subsection, we prove independence of $\barPhi$ on the choice of bands, of the bi-ascending state, and of the representative of $T$ within its \qvlnc{2} class. Finally, we show that $\barPhi$ preserves compositions, identities, tensor 
products, and braidings, and that it is the inverse of $\Phi$.

For convenience of the reader, we collect all relations and their consequences in Appendix~\ref{tables/app}, and we recall (and sometimes establish) their proof in Appendix~\ref{proofs/app}. 

\subsection{Acknowledgments}
 
The authors would like to thank Kazuo Habiro for explaining them how to define integral form and element in $\AlgH$. AB, IB, and MD were supported by the NCCR SwissMAP and Grant $200020\_207374$ of the Swiss National Science Foundation. IB and RP thank the Institut für Mathematik at UZH for its hospitality during the initial conception of this article. The authors also gratefully acknowledge the support of the Simons Collaboration “New Structures in Low-Dimensional Topology”.

%% file: S2-preliminaries.tex
\section{Preliminaries}
\label{preliminaries/sec}


In this paper, we will study the relation between several monoidal categories (whose objects and morphisms have either a topological or an algebraic flavor) that have (in some cases more than one) braided and even ribbon structures. All these categories are going to be \textit{PROBs}, that is, braided strict monoidal categories whose objects are all monoidal powers of a single generating object. 

In this section, we will introduce the necessary definitions and recall Penrose graphical notation, which will be adapted to the specific type of categories we are interested in. The reader is referred to \cite{Ma71, EGNO15, Tu94, TV17} for more general and detailed discussions.

\subsection{Braided monoidal categories}
\label{categories/sec}

We list here some basic definitions from the general theory of monoidal categories that are used
repeatedly in the paper, following \cite[Sections 2.1, 2.8, 8.1]{EGNO15}.

\begin{definition}
\label{monoidal-cat/def}
A \textit{strict monoidal category} is a category $\calC$ equipped with a functor of the form $\otimes : \calC \times \calC \to \calC$, called the \textit{tensor product}, and an object $\one \in \calC$, called the \textit{tensor unit}, satisfying:
\begin{gather*}
 (X \otimes Y) \otimes Z = X \otimes (Y \otimes Z) \text{ for all } X,Y,Z \in \calC; \\
 \one \otimes X = X = X \otimes \one \text{ for every } X \in \calC.
\end{gather*}
Notice that, thanks to the associativity axiom, bracketing can be ignored in tensor products.
\end{definition}

In what follows, we will tacitly assume that all the monoidal structures we consider are strict. Notice that this will simplify some definitions.

Morphisms in a monoidal category $\calC$ can be efficiently represented using Penrose graphical notation, which is based on diagrams of planar graphs. Edges are labeled by objects of $\calC$ and are required to be nowhere-horizontal, while vertices are labeled by morphisms of $\calC$ and are represented as boxes (called \textit{coupons}) with distinguished opposite bases (an incoming one, on the bottom, and an outgoing one, on the top). Composition of diagrams is given by vertical stacking (and is read from bottom to top), while tensor product is given by horizontal juxtaposition (and is read from left to right). 

In the following, we will restrict our attention to a special kind of monoidal categories known as \textit{PROs} (short for \textit{product categories}). By definition, a PRO is a strict monoidal category $\calC$ whose objects are all tensor powers of a single one, the \textit{generating object}, which we will typically denote by $H \in \calC$. For a PRO $\calC$, we will adopt the following notations:
\begin{gather*}
 H^0 = \one \text{ and } H^1 = H; \\*
 H^n = H^{\otimes n} \text{ for every } n \geqs 2; \\*
 \id_n = \id_{H^n} \text{ for every } n \geqs 0; \\*
 \id = \id_1.
\end{gather*}
When working inside a PRO, up to replacing each edge labeled by $H^n$ with $n$ parallel edges labeled by $H$, we will always assume that all the edges share the same label $H$, so we will drop labels for edges altogether. Furthermore, we will usually replace vertices by special symbols encoding their label.

We point out that diagrams are considered up to planar isotopies through diagrams with nowhere-horizontal edges. In particular, the diagrams depicted in Figure~\ref{commutation-move/fig} represent the same morphism.

\begin{figure}[hbt]
 \includegraphics{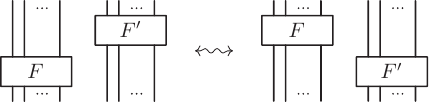}
 \caption{Example of planar isotopy, with $F$ and $F'$ arbitrary morphisms.}
 \label{commutation-move/fig}
\end{figure}




\begin{definition}
\label{braided-cat/def} 
A \textit{braided monoidal category} is a monoidal category $\calC$ equipped with a natural isomorphism of components $c_{X,Y} : X \otimes Y \to Y \otimes X$ for all $X,Y \in \calC$, called the \textit{braiding}, satisfying:
\begin{gather*}
 c_{X \otimes Y,Z} = (c_{X,Z} \otimes \id_Y) \circ (\id_X \otimes c_{Y,Z}) \text{ for all } X,Y,Z \in \calC;\\
 c_{X,Y \otimes Z} = (\id_Y \otimes c_{X,Z}) \circ (c_{X,Y} \otimes \id_Z) \text{ for all } X,Y,Z \in \calC.
\end{gather*}
\end{definition}

\begin{table}[ht]
 \includegraphics{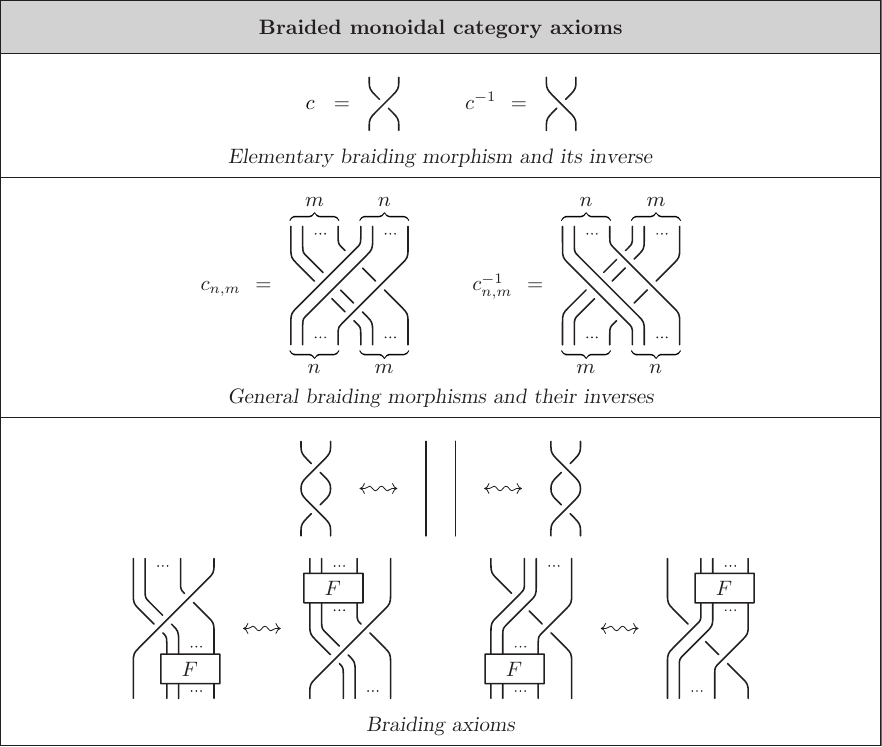}
 \caption{}
 \label{table-braided/fig}
\end{table}


In Penrose graphical notation, braidings are represented as crossings, so that diagrams can be understood as representing graphs embedded in \dmnsnl{3} space. Naturality of braidings translates to moves of planar diagrams  which correspond to the isotopies of these graphs in \dmnsnl{3} space shown in Table~\ref{table-braided/fig}, where $F$ denotes any morphism, including braidings themselves.

As we mentioned above, we will restrict our attention to \textit{PROBs} (short for \textit{braided PROs}). For a PROB $\calC$, whose objects are all tensor powers of a generating object $H \in \calC$, we will adopt the following short notations:
\begin{gather*}
 \braid_{n,m} = \braid_{H^n,H^m} \text{ for all } n,m \geqs 0; \\*
 \braid = \braid_{1,1}.
\end{gather*}

\FloatBarrier

\subsection{Ribbon categories}
\label{ribbon/sec}

We list here some basic definitions from the general theory of ribbon categories that are used repeatedly in the paper, following \cite[Sections 2.10, 4.7, 8.10]{EGNO15}.

\begin{definition}
\label{rigid-cat/def}
A monoidal category $\calC$ is \textit{left rigid} if every object $X \in \calC$ admits a \textit{left dual} $X^* \in \calC$ and two morphisms $\lev_X : X^* \otimes X \to \one$ and $\lcoev_X : \one \to X \otimes X^*$, called the \textit{left evaluation} and \textit{coevaluation}, satisfying
\[
(\id_X \otimes \lev_X) \circ (\lcoev_X \otimes \id_X) = \id_X
\quad \text{and} \quad
(\lev_X \otimes \id_{X^*}) \circ (\id_{X^*} \otimes \lcoev_X) = \id_{X^*}.
\]
Given any morphism $F \in \calC(X,Y)$, its \textit{left dual} $F^* \in \calC(Y^*,X^*)$ is defined as
\[
F^* = (\lev_Y \otimes \id_{X^*}) \circ (\id_{Y^*} \otimes F \otimes \id_{X^*}) \circ (\id_{Y^*} \otimes \lcoev_X).
\]
\end{definition}



\begin{definition}
\label{ribbon-cat/def}
A \textit{ribbon category} is a left rigid braided monoidal category $\calC$ equipped with a natural isomorphism of components $\theta_X : X \to X$ for every $X \in \calC$, called the \textit{twist}, satisfying
\begin{gather*}
 \theta_{X \otimes Y} = c_{Y,X} \circ c_{X,Y} \circ (\theta_X \otimes \theta_Y) \text{ for all } X,Y \in \calC; \\
 (\theta_X)^* = \theta_{X^*} \text{ for every } X \in \calC.
\end{gather*}
\end{definition}

Notice that, in a ribbon category, all duals are two-sided. Indeed, for every object $X \in \calC$, there exist morphisms $\rev_X: X \otimes X^* \to \one$ and $\rcoev_X: \one \to X^* \otimes X$, called the \textit{right evaluation} and \textit{coevaluation}, defined as
\[
\rev_X = \lev_X \circ c_{X,X^*} \circ (\theta_X \otimes \id_{X^*})
\quad \text{and} \quad
\rcoev_X = (\id_{X^*} \otimes \theta_X^{-1}) \circ c_{X^*,X}^{-1} \circ \lcoev_X,
\]
and satisfying
\[
(\rev_X \otimes \id_X) \circ (\id_X \otimes \rcoev_X) = \id_X
\quad \text{and} \quad
(\id_{X^*} \otimes \rev_X) \circ (\rcoev_X \otimes \id_{X^*}) = \id_{X^*}.
\]
In particular, the twist can be written as
\begin{gather*}
 \theta_X = (\id_X \otimes \rev_X) \circ (c_{X,X} \otimes \id_{X^*}) \circ (\id_X \otimes \lcoev_X) \text{ for every } X \in \calC.
\end{gather*}
Furthermore, for every morphism $F \in \calC(X,Y)$, the left dual $F^* \in \calC(Y^*,X^*)$ satisfies
\[
F^* = (\id_{X^*} \otimes \rev_Y) \circ (\id_{X^*} \otimes F \otimes \id_{Y^*}) \circ (\rcoev_X \otimes \id_{Y^*}).
\]

In Penrose graphical notation, by dropping the requirement on nowhere-horizontal edges at the level of diagrams, duality morphisms (evaluations and coevaluations) can be represented as maxima and minima (caps and cups), and twists can be represented by kinks. Their properties translate to moves of diagrams which correspond to all framing-preserving isotopies of ribbon graphs in \dmnsnl{3} space. 

In general, duals are encoded by orientations on edges, which allow for the distinction between left and right duality morphisms. However, we never actually need to orient edges in our setting. Indeed, we always work in the context of a PROB $\calC$ whose generating object $H \in \calC$ is by construction self-dual, in the sense that it coincides with its own dual. Since all edges are understood as being labeled by $H$, no orientation is needed. Therefore, we adopt the following short notations:
\begin{gather*}
 \ev_n = \lev_{H^n} = \rev_{H^n} \text{ for every } n \geqs 0; \\
 \coev_n = \lcoev_{H^n} = \rcoev_{H^n} \text{ for every } n \geqs 0; \\
 \theta_n = \theta_{H^n} \text{ for every } n \geqs 0; \\
 \ev = \ev_1, \quad
 \coev = \coev_1, \quad
 \theta = \theta_1.
\end{gather*}

Observe that, when they exist, left duals are unique up to unique isomorphisms, see \cite[Proposition~2.10.5.]{EGNO15}.


\begin{proposition}\label{Prob-ribbon/thm}
Let $\calC$ be a PROB whose generating object $H$ is self-dual through two-sided evaluation and coevaluation morphisms $\ev : H \otimes H \to \one$ and $\coev : \one \to H \otimes H$ satisfying
\begin{gather*}
 (\id \otimes \ev) \circ (\coev \otimes \id) = \id = (\ev \otimes \id) \circ (\id \otimes \coev), \\
 (\ev \otimes \id) \circ (\id \otimes \braid) \circ (\coev \otimes \id) = (\id \otimes \ev) \circ (\braid \otimes \id) \circ (\id \otimes \coev).
\end{gather*}
Then $\calC$ admits the structure of a ribbon category, with dual $(H^n)^* = H^n$ for every $n \geqs 0$, with two-sided evaluation $\ev_n : H^n \otimes H^n \to \one$ and coevaluation $\coev_n : \one \to H^n \otimes H^n$ defined inductively by $\ev_0 = \idone = \coev_0$ and by
\begin{gather*}
 \ev_n = \ev \circ (\id \otimes \ev_{n-1} \otimes \id), \\
 \coev_n = (\id \otimes \coev_{n-1} \otimes \id) \circ \coev
\end{gather*}
for every $n \geqs 1$, and with twist $\theta_n: H^n \to H^n$ defined by
\[
 \theta_n = (\ev_n \otimes \id_n) \circ (\id_n \otimes \braid_{n,n}) \circ (\coev_n \otimes \id_n)
\]
for every $n \geqs 0$. Moreover, $( \_ )^* : \calC \to \calC$ defines a contravariant braided antimonoidal functor called the \textit{duality functor}.
\end{proposition}

In Penrose graphical notation, the rigid and ribbon structures defined in Proposition~\ref{Prob-ribbon/thm} are represented in Tables~\ref{table-rigid/fig} and \ref{table-ribbon/fig}. Notice that, using the moves in Table~\ref{table-braided/fig} corresponding to the braiding axioms, the ribbon axiom in Table~\ref{table-ribbon/fig} is equivalent to the condition $(\theta_H)^* = \theta_{H^*}$.

\begin{table}[htb]
 \includegraphics{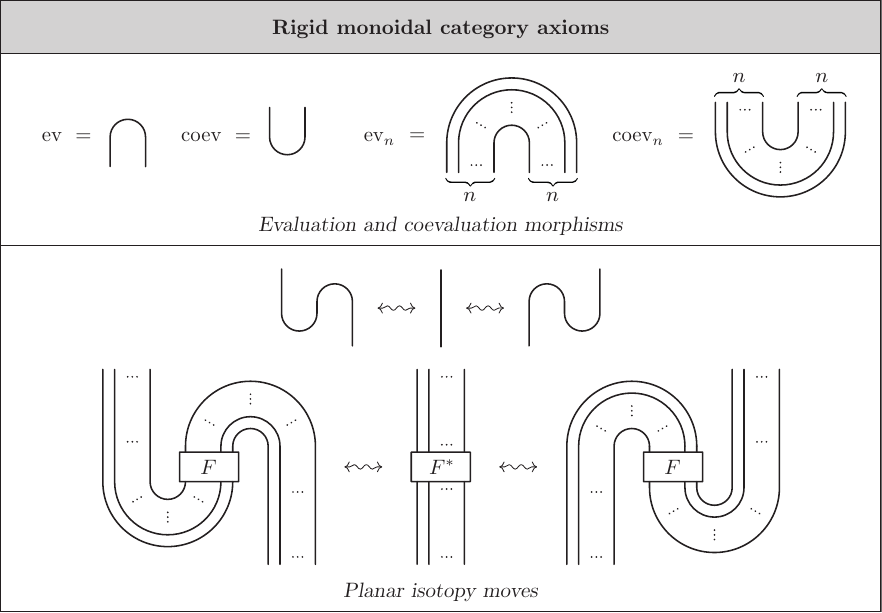}
 \caption{}
 \label{table-rigid/fig}
\end{table}

\begin{table}[htb]
 \includegraphics{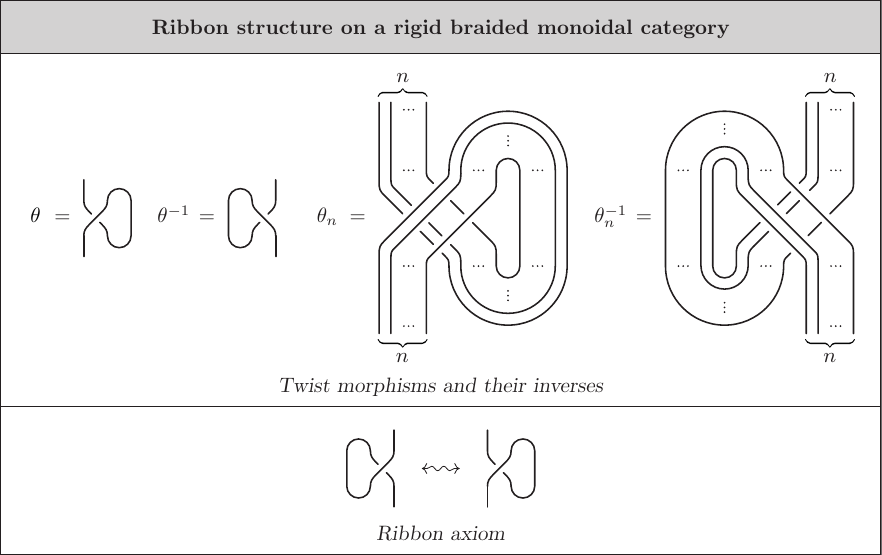}
 \caption{}
 \label{table-ribbon/fig}
\end{table}

\begin{proof}
By an inductive argument starting with $n=1$, it is easy to see that $\ev_n$ and $\coev_n$ satisfy the zig-zag relation for every $n \geqs 1$. Analogously, the proof that the family of morphisms $\theta_n : H^n \to H^n$ for $n \geqs 0$ define a ribbon structure on $\calC$, in the sense that they satisfy the relations in Definition~\ref{ribbon-cat/def}, proceeds by induction starting with $n = 1$, and uses the moves in Table~\ref{table-braided/fig} corresponding to the braiding axioms. The fact that $(\_)^* : \calC \to \calC$ defines a contravariant antimonoidal functor is a straightforward consequence of the definition, while, using the braided axioms, it can be easily seen that it preserves the braided structure.
\end{proof}




%% file: S3-algebra.tex
\section{Algebraic categories}
\label{algebra/sec}

\subsection{Braided Hopf algebras and the category \texorpdfstring{$\Alg$}{Alg}}\label{Halgebra/sec}

Let $\calC$ be a braided monoidal category with tensor product~$\otimes$, tensor unit~$\one$, and braiding~$\braid$. A \textit{braided Hopf algebra} in $\calC$, or simply a \textit{Hopf algebra} in $\calC$, is an object $H \in \calC$ equipped with the following structure morphisms:
\begin{itemize}
 \item a \textit{product} $\prodH : H \otimes H \to H$ and a \textit{unit} $\unitH : \one \to H$;
 \item a \textit{coproduct} $\coprH : H \to H \otimes H$ and a \textit{counit} $\counH : H \to \one$;
 \item an \textit{antipode} $\antipH : H \to H$ and its inverse $\antipH^{-1}: H \to H$.
\end{itemize}
These structure morphisms are subject to the following axioms:
\begin{gather*}
 \prodH \circ (\prodH \otimes \id) = \prodH \circ (\id \otimes \prodH),
 \tag*{\hrel{a1}}
 \\
 \prodH \circ (\unitH \otimes \id) = \id = \prodH \circ (\id \otimes \unitH),
 \tag*{\hrel{a2-2'}}
 \\
 (\coprH \otimes \id) \circ \coprH = (\id \otimes \coprH) \circ \coprH,
 \tag*{\hrel{a3}}
 \\
 (\counH \otimes \id) \circ \coprH = \id = (\id \otimes \counH) \circ \coprH, 
 \tag*{\hrel{a4-4'}}
 \\
 (\prodH \otimes \prodH) \circ (\id \otimes \braid \otimes \id) \circ (\coprH \otimes \coprH) = \coprH \circ \prodH, 
 \tag*{\hrel{a5}} 
 \\
 \counH \circ \prodH = \counH \otimes \counH, 
 \tag*{\hrel{a6}}
 \\
 \coprH \circ \unitH = \unitH \otimes \unitH, 
 \tag*{\hrel{a7}}
 \\
 \counH \circ \unitH = \idone, 
 \tag*{\hrel{a8}} 
 \\
 \prodH \circ (\antipH \otimes \id) \circ \coprH = \unitH \circ \counH = \prodH \circ(\id \otimes \antipH) \circ \coprH, 
 \tag*{\hrel{s1-1'}}
 \\
 \antipH \circ \antipH^{-1} = \id = \antipH^{-1} \circ \antipH.
 \tag*{\hrel{s2-3}}
\end{gather*}

\begin{table}[b]
 \centering
 \includegraphics{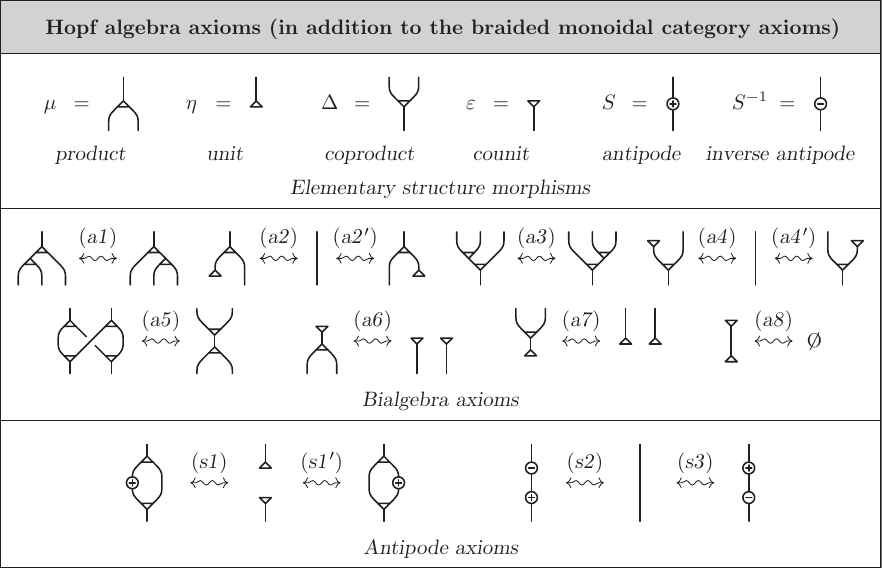}
 \caption{}
 \label{table-Hopf/fig}
 \label{E:a1} \label{E:a2-2'} \label{E:a2} \label{E:a2'} \label{E:a3} \label{E:a4-4'} \label{E:a4} \label{E:a4'} \label{E:a5} \label{E:a6} \label{E:a7} \label{E:a8} \label{E:s1-1'} \label{E:s1} \label{E:s1'} \label{E:s2-3} \label{E:s2} \label{E:s3}
\end{table}

A graphical representation of the generators and the defining axioms of a Hopf algebra can be found in Table~\ref{table-Hopf/fig}, where all edges are assumed to carry the label $H$ (compare with \cite[Tables~4.7.12 \& 4.7.13]{BP11}).

Notice that all these structure morphisms, except for the antipode, feature triangles that point either up or down. This choice is not arbitrary. Indeed, as we will see in Subsection~\ref{Phi/sec}, triangles pointing up correspond to $2$-handles, while those pointing down correspond to $1$-handles in the category of Kirby tangles $\KTf$ introduced in Subsection~\ref{4KT/sec}.

\begin{definition}\label{Alg/def} 
We denote by $\Alg$ the braided monoidal category freely generated by a Hopf algebra object $H$. In other words, objects of $\Alg$ are tensor powers of $H$, while morphisms of $\Alg$ are compositions of tensor products of identities, braidings, and structure morphisms $\prodH$, $\unitH$, $\coprH$, $\counH$, $\antipH$, $\antipH^{-1}$, modulo the defining axioms in Table~\ref{table-Hopf/fig}. 
\end{definition}

By definition, the category $\Alg$ satisfies the following universal property.

\begin{universal_property}\label{univ_prop/thm}
If $\calC'$ is a braided monoidal category and $H' \in \calC'$ is a Hopf algebra, then there exists a unique 
braided monoidal functor $\Xi_{H'} : \Alg \to \calC'$ sending $H$ to $H'$.
\end{universal_property}

The category $\Alg$ is the first of a series of braided monoidal categories that will be studied in the present work, so let us introduce some terminology and notations that apply to all of them. Every category $\calC$ we will consider is generated, as a braided monoidal category, by a single object $H \in \calC$ (in other words, $\calC$ is a PROB) together with a set of structure morphisms between tensor powers of $H$ (in the case of $\Alg$, these are $\prodH$, $\coprH$, $\unitH$, $\counH$, $\antipH$, and $\antipH^{-1}$). We will refer to the braiding morphism $c$, to its inverse $c^{-1}$, and to the structure morphisms as the \textit{elementary} or \textit{generating morphisms} of $\calC$, since all the other morphisms of $\calC$ are compositions of tensor products of identities and structure morphisms. Arbitrary morphisms in $\calC$ are represented by planar diagrams, and two diagrams $F_1$ and $F_2$ are said to be \textit{equivalent} if they represent the same morphism in $\calC$. In this case, we will write $F_1=F_2$, while in our pictures such equivalence will be denoted by a curly arrow.

In general, proving that two diagrams are equivalent consists in showing that one of them can be obtained from the other by applying a sequence of relations (or moves) which correspond either to the axioms or to the properties that have already been established. We outline the main steps in this procedure by drawing a sequence of intermediate diagrams, and for each step we denote (in the appropriate order, starting from the top) the main moves needed to transform the diagram on the left into the one on the right. Notice that some steps can actually be understood more easily by starting from the diagram on the right and by reading the moves in the reverse order. 
Obviously, the equivalence of two diagrams can be proven in many different ways, and our indications merely exhibit one possible proof. Moreover, the lists of moves we provide are not always complete, as sometimes only those moves that are considered necessary to guide the reader in reconstructing the equivalence are displayed. In particular, a single label may refer to multiple applications of the same type of move. 
As an example, see Figure~\ref{S/fig} below.
Observe that a label of type \hrel[a2]{a2-4} stands for relations~\hrel{a2} and \hrel{a4}, not for all the relations in the interval from \hrel{a2} to \hrel{a4}.     


In order to reduce the excessive proliferation of labelings, and to facilitate the reading of diagrammatic proofs, we will group certain elementary moves (or sequences thereof) into families. We will refer to each family as a single \textit{generalized move} denoted by a capital letter. The first such move is introduced in the lemma below and it is a direct consequence of the algebra axioms.


\begin{lemma}\label{s0/thm} We say that two diagrams are related by an \hrel{S} move if they can be obtained from each other by inserting or deleting a pair of parallel edges, only one of which carries a single antipode morphism, and which are attached to another pair of (possibly coinciding) edges by pairs of product and coproduct vertices, as represented in Figure~\ref{S/fig}.
Two diagrams related by an \hrel{S} move are equivalent. 
\end{lemma}

\begin{figure}[hbt]
 \centering
 \includegraphics{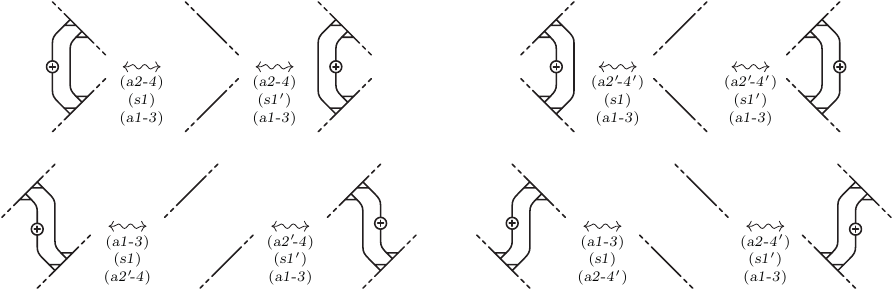}
 \caption{Generalized move \rel{S}.}
 \label{S/fig} \label{E:S}
\end{figure}

\begin{proof}
As indicated in Figure \ref{S/fig}, the statement is a direct consequence of axioms \hrel{s1-1'}, \hrel{a1}, \hrel{a2-2'}, \hrel{a3} and \hrel{a4-4'}.
\end{proof}

The category $\Alg$ has some important and well-known properties, which we collect here below. 

\begin{table}[htb]
 \centering
 \includegraphics{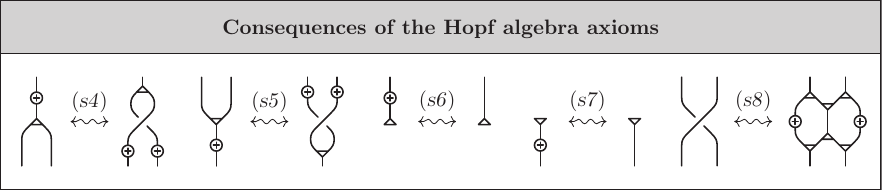}
 \caption{}
 \label{table-Hopf-prop/fig}
 \label{E:s4} \label{E:s5} \label{E:s6} \label{E:s7} \label{E:s8}
\end{table}

\begin{proposition}\label{hopf-prop/thm}
The antipode satisfies relations~\hrel{s4} to \hrel{s8} presented in Table~\ref{table-Hopf-prop/fig}.
\end{proposition}

For convenience of the reader, the proof of this proposition, together with the (rather technical) proofs of other known properties of the algebras we work with, are collected in Appendix~\ref{proofs/app}.

Next, we are going to introduce an anti-monoidal functor that is naturally isomorphic to the identity functor.
Variants of such functor can be introduced in more general contexts, for any braided monoidal category, but we restrict our definition to a specific form that is designed for the type of categories we are interested in.

\begin{definition}\label{sym/def} 
If $\calC$ is a PROB with generating object $H \in \calC$, and if $\antipH : H \to H$ is an invertible morphism, then we denote by $\calS_n : H^n \to H^n$ the invertible morphism given by
\[
 \calS_n = T_n^{\frac{1}{2}} \circ \antipH^{\otimes n},
\]
where $T_n^{\frac{1}{2}} : H^n \to H^n$ is the positive half-twist on $H^n$, defined inductively by $T_0^{\frac{1}{2}} = \id_H$ and by
\[
 T^{\frac{1}{2}}_{n} = (T^{\frac{1}{2}}_{n-1} \otimes \id_H) \circ c_{n-1,1}
\]
for every $n \geqs 1$. Then, the \textit{symmetry functor} $\sym :\calC \to \calC$ is the functor that sends every object of $\calC$ to itself and every morphism $F : H^s \to H^t$ of $\cal C$ to
\begin{gather*}
 \sym(F) = \calS_t \circ F \circ \calS_s^{-1}. 
\end{gather*}
\end{definition}


\begin{proposition}\label{funt-sym/thm}
If $\calC$ is a PROB with generating object $H \in \calC$, and if $\antipH : H \to H$ is an invertible morphism, then the functor $\sym : \calC \to \calC$ is a braided anti-monoidal equivalence functor, where anti-monoidality means
\[
 \sym(F \otimes F') = \sym(F') \otimes \sym(F)
\]
for all morphisms $F : H^s \to H^t$ and $F' : H^{s'} \to H^{t'}$, and the collection $\calS = \{ \calS_n \mid n \in \N \}$ of morphisms of $\calC$ defines a natural isomorphism between the identity functor of $\calC$ and $\sym$. In particular, for every morphism $F : H^s \to H^t$ of $\calC$, we have
\begin{gather*}
 \sym(F) \circ \calS_s = \calS_t \circ F,
\tag*{\hrel{s9}}
\end{gather*}
see Figure \ref{s9/fig}.
\end{proposition}


\begin{figure}[htb]
 \centering
 \includegraphics{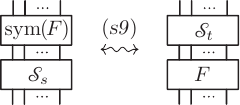}
 \caption{The natural isomorphism $\calS$.}
 \label{s9/fig} \label{E:s9}
\end{figure}

Notice that, in Definition~\ref{sym/def} and in Proposition~\ref{funt-sym/thm}, $H \in \calC$ is simply required to be the generating object of a PROB, and $\antipH : H \to H$ is simply required to be an invertible morphism. However, as our notation suggests, we will always apply this result to PROBs whose generating object $H$ is a Hopf algebra with antipode $\antipH : H \to H$.

\begin{proof}  
The fact that $\calS$ provides a natural isomorphisms between $\sym$ and the identity functor follows directly from the definition.
To establish anti-monoidality of $\sym$, notice that, by the naturality of the braiding,
for all $m,n \geqs 0$ we have
\[
 \calS_{m+n} = c_{m,n} \circ (\calS_m \otimes \calS_n).
\]
Therefore, for all morphisms $F : H^s \to H^t$ and $F' : H^{s'} \to H^{t'}$ of $\calC$, we have
\begin{align*}
 \sym(F \otimes F')
 &= c_{t,t'} \circ (\calS_t \otimes \calS_{t'}) \circ (F \otimes F') \circ (\calS_s^{-1} \otimes \calS_{s'}^{-1}) \circ c_{s,s'}^{-1}\\
 &= c_{t,t'} \circ (\sym(F) \otimes \sym(F')) \circ c_{s,s'}^{-1} = \sym(F') \otimes \sym(F). \qedhere
\end{align*}
\end{proof}

Proposition~\ref{funt-sym/thm} provides a powerful tool for deducing relations from other identities. Indeed, it implies that a relation $F_1=F_2$, denoted \rel{x}, holds in $\calC$ if and only if the relation $\sym(F_1) = \sym(F_2)$, denoted \rel{x'}, holds as well. Together with the properties of the antipode established in Proposition~\ref{hopf-prop/thm} it has the following important consequence. 



\begin{proposition}\label{symmetry-alg/thm}
There is an involutive braided anti-monoidal equivalence functor
\[
 \sym : \Alg \to \Alg,
\]
called the \textit{symmetry functor}, that sends every object and every elementary morphism to itself. 
\end{proposition}

\begin{proof}  
The statement follows from Proposition~\ref{funt-sym/thm}, provided we check that $\sym$ sends all elementary morphisms of $\Alg$ to themselves. For the elementary structure morphisms, this reduces to relations~\hrel{s4} to \hrel{s7} in Table~\ref{table-Hopf-prop/fig}, while for the elementary braiding morphisms it is a consequence of \hrel{s2-3} and of the braiding axioms in Table~\ref{table-braided/fig}.
\end{proof}

According to Proposition~\ref{symmetry-alg/thm}, if we think of the diagrammatic presentation of morphisms in $\Alg$ as representing graphs embedded in $\R^3$, then applying the symmetry functor corresponds to rotating a graph of an angle $\pi$ around a vertical axis lying in the diagram plane. In particular, $\sym$ reverses the order of factors in tensor products, but preserves all crossings.
In Figure~\ref{sym-example/fig} we give an example of a morphism in $\Alg$ and its image under the functor $\sym$. Here, the curly equivalence arrow corresponds to a sequence of applications of the braiding axioms and relations~\hrel{s4} to \hrel{s7} in Table~\ref{table-Hopf-prop/fig} as prescribed in the proofs of Proposition~\ref{funt-sym/thm} and Proposition~\ref{symmetry-alg/thm}.


\begin{table}[htb]
 \centering
 \includegraphics{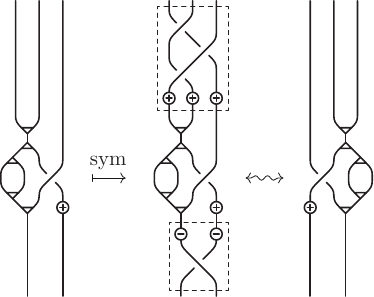}
 \caption{An example of application of the functor $\sym$.}
 \label{sym-example/fig}
\end{table}

\subsection{Adjoint action}
\label{adjoint/sec}

Let $\calC$ be a braided monoidal category, let $(H, \prodH_H, \unitH_H, \coprH_H, \counH_H, \antipH_H)$ be a braided Hopf algebra in $\calC$, and let $(A,\prodH_A,\unitH_A)$ be an algebra in $\calC$. We recall that a morphism $\alpha : H \otimes A \to A$ defines a left action of $H$ on $A$ if the following holds:
\begin{gather*}
 \alpha \circ (\unitH_H \otimes \id_A) = \id_A,\\ 
 \alpha \circ (\prodH_H \otimes \id_A) = \alpha \circ (\id_H \otimes \alpha),\\ 
 \alpha \circ (\id_H \otimes \unitH_A) = \unitH_A \circ \counH_H,\\ 
 \alpha \circ (\id_H \otimes \prodH_A) = \prodH_A \circ (\alpha \otimes \alpha) \circ (\id_H \otimes \braid_{H,A} \otimes \id_A) \circ (\coprH_H \otimes \id_{A \otimes A}).
\end{gather*}

The first two conditions express the fact that $A$ is a left $H$-module, while the last two conditions express the fact that the action intertwines the product and the unit of $A$. The notion of right action is symmetric, and corresponds to a right $H$-algebra structure on $A$.

\begin{definition}\label{alpha/def}
The left adjoint action $\adjH_n : H \otimes H^n \to H^n$ is inductively defined by $\adjH_0 =\counH$ and by
\begin{gather*}
 \adjH_1 = \adjH = \prodH \circ (\prodH \otimes \antipH) \circ (\id \otimes \braid) \circ (\coprH \otimes \id), 
 \tag*{\hrel{d1}}
 \\*
 \adjH_n = (\adjH \otimes \adjH_{n-1}) \circ (\id \otimes \braid \otimes 
 \id_{n-1}) \circ (\coprH \otimes \id_n)
 \tag*{\hrel{d2}}
\end{gather*}
for every $n \geqs 1$, see Table~\ref{table-adjoint/fig}. The symmetric right adjoint action $\adjH'_n : H^n \otimes H \to H^n$ is defined by
\[
\adjH'_n = \sym(\adjH_n).
\tag*{\hrel{d1'-2'}}
\]
for every $n \geqs 0$, see Table~\ref{table-adjoint/fig}.
\end{definition}

\begin{table}[htb]
 \centering
 \includegraphics{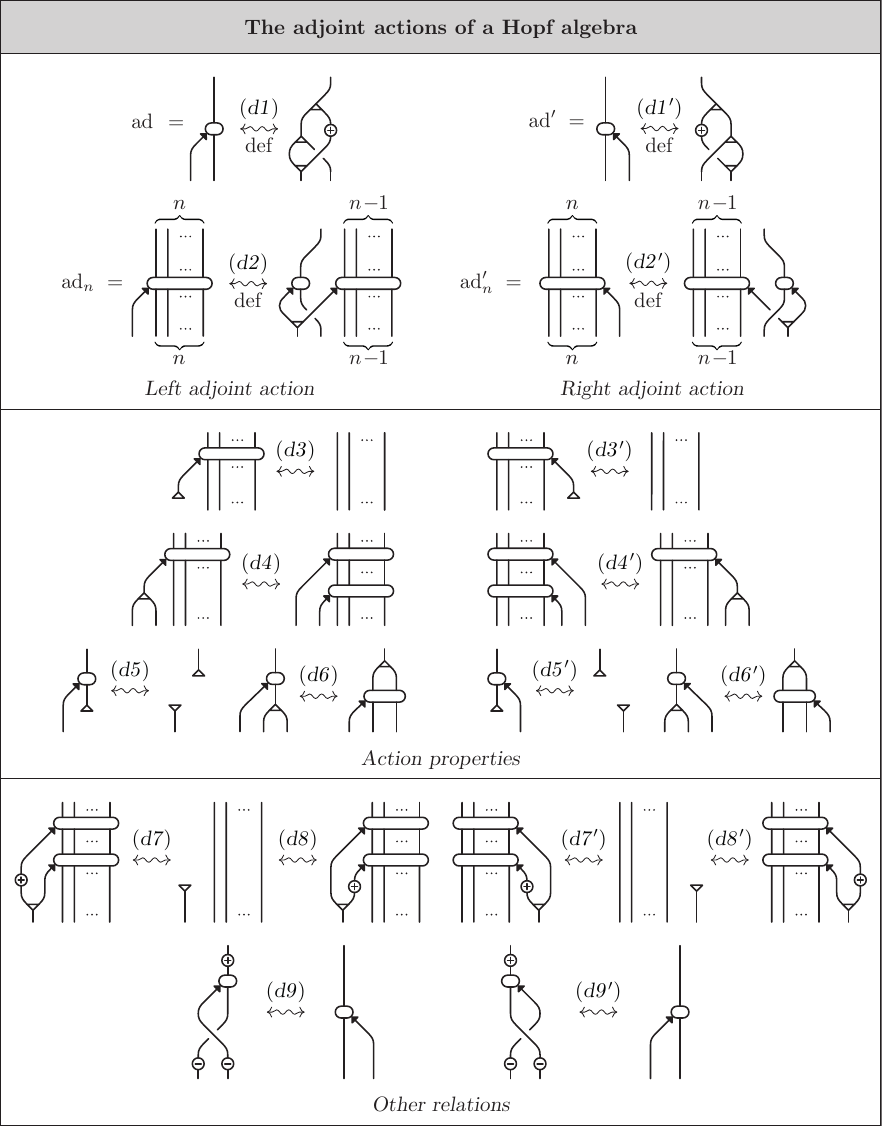}
 \caption{}
 \label{table-adjoint/fig}
 \label{E:d1-2} \label{E:d1} \label{E:d2} \label{E:d1'-2'} \label{E:d1'} \label{E:d2'} \label{E:d3-3'} \label{E:d3} \label{E:d3'} \label{E:d4-4'} \label{E:d4} \label{E:d4'} \label{E:d5-5'} \label{E:d5} \label{E:d5'} \label{E:d6-6'} \label{E:d6} \label{E:d6'} \label{E:d7-7'} \label{E:d7} \label{E:d7'} \label{E:d8-8'} \label{E:d8} \label{E:d8'} \label{E:d9-9'} \label{E:d9} \label{E:d9'}
\end{table}

We denote by $\cadjH$ the collection of morphisms $\{ \adjH_n \mid n \in \N \}$, and similarly by $\cadjH'$ the collection $\{ \adjH'_n \mid n \in \N \}$. The fact that these are indeed left and right actions is a classical result in the theory of Hopf algebras, and the reader can find the proof in Proposition~\ref{alpha/thm} below. In particular, the adjoint action intertwines the product and the unit. 

\begin{proposition}\label{alpha/thm}
If $H$ is a Hopf algebra in $\calC$, then its structure morphisms satisfy the identities appearing in Table~\ref{table-adjoint/fig}. In particular, for every integer $n \geqs 0$, the adjoint morphisms $\adjH_n$ and $\adjH'_n$ define a left and a right action of $H$ on $H^n$ respectively.
\end{proposition}


\begin{proof}
Notice that it is enough to prove the statements for $\adjH_n$, since the ones for $\adjH'_n$ follow by applying the functor $\sym$, thanks to Proposition~\ref{symmetry-alg/thm}.

Identity~\hrel{d3} is an immediate consequence of relations~\hrel{a2-2'}, \hrel{a7}, and \hrel{s6} in Tables~\ref{table-Hopf/fig} and \ref{table-Hopf-prop/fig}. In order to show \hrel{d4}, we first prove the special case $n=1$ in Figure~\ref{adjoint01/fig},
and then the general case follows by the inductive argument shown in Figure~\ref{adjoint02/fig}. Identity~\hrel{d5} follows from axioms~\hrel{a2'} and \hrel{s1'} in Table~\ref{table-Hopf/fig}. Identity~\hrel{d6} is proved in Figure~\ref{adjoint03/fig}, identity~\hrel{d7} is verified in Figure~\ref{adjoint04/fig}, while \hrel{d8} can be proved in a similar way, by using \hrel{s1'} instead of \hrel{s1}. 
%
%
Finally, identity~\hrel{d9} and its symmetric \hrel{d9'} state that $\sym(\adjH)=\adjH'$ and $\sym(\adjH')=\adjH$, and therefore they follow from Proposition~\ref{symmetry-alg/thm}.
\end{proof}

\begin{figure}[htb]
 \centering
 \includegraphics{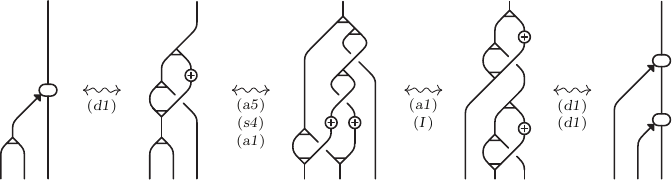}
 \caption{Proof of \hrel{d4}: case $n=1$.}
 \label{adjoint01/fig}
\end{figure}

 \begin{figure}[htb]
 \centering
 \includegraphics{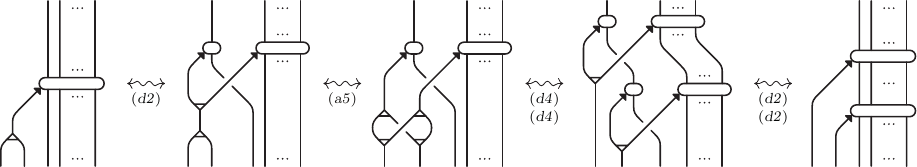}
 \caption{Proof of \hrel{d4}: inductive step.}
 \label{adjoint02/fig}
\end{figure}

\begin{figure}[htb]
 \centering
 \includegraphics{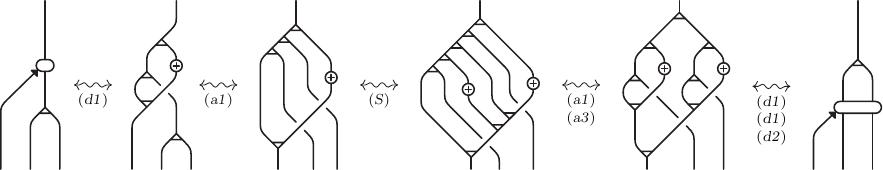}
 \caption{Proof of \hrel{d6}.}
 \label{adjoint03/fig}
\end{figure}

\begin{figure}[htb]
 \centering
 \includegraphics{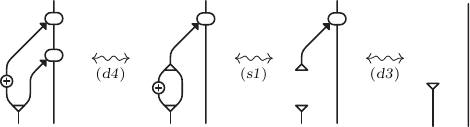}
 \caption{Proof of \hrel{d7}.}
 \label{adjoint04/fig}
\end{figure}


\FloatBarrier

If $\calC$ is the category of left modules over a ring $R$, equipped with its standard symmetric braiding, then the adjoint action is only known to intertwine the coproduct and the antipode when the Hopf algebra is \textit{cocommutative}, that is, when $\braid \circ \coprH = \coprH$, see \cite[Lemma~5.7.2]{Mo93}. The following definition provides a weaker condition on the Hopf algebra that ensures this intertwining property in the case of an arbitrary braided category. This condition was first introduced by Majid under the name \textit{$\calC$-cocommutative} action, see \cite[Definition~2.3]{Ma93}, or \textit{braided cocommutative} action, see \cite[Definition~2.9]{Ma94}. 

\begin{definition}\label{cocommutative/def}
The left adjoint action $\adjH: H \otimes H \to H$ of a Hopf algebra $H$ in a braided monoidal category $\calC$ is \textit{braided cocommutative} if the following holds:
\begin{equation*}
 (\adjH \otimes \id) \circ (\id \otimes \braid) \circ (\coprH \otimes \id) = \braid^{-1} \circ (\id \otimes \adjH) \circ (\coprH \otimes \id). \tag*{\hrel{h0}}
\end{equation*}
Analogously, the right adjoint action $\adjH : H \otimes H \to H$ of a Hopf algebra $H$ in a braided monoidal category $\calC$ is \textit{braided cocommutative} if the following holds:
\begin{equation*}
 (\id \otimes \adjH') \circ (\braid \otimes \id) \circ (\id \otimes \coprH) = \braid^{-1} \circ (\adjH' \otimes \id) \circ (\id \otimes \coprH). \tag*{\hrel{h0'}}
\end{equation*}
\end{definition}

A graphical representation of the braided cocommutativity axiom for adjoint actions is given in Table~\ref{table-cocommutative/fig}. Notice that, when the braiding of $\calC$ is symmetric, relations~\hrel{h0} and \hrel{h0'} are implied by the cocommutativity condition $\braid \circ \coprH = \coprH$, although they are not equivalent to it. 

\begin{table}[htb]
 \centering
 \includegraphics{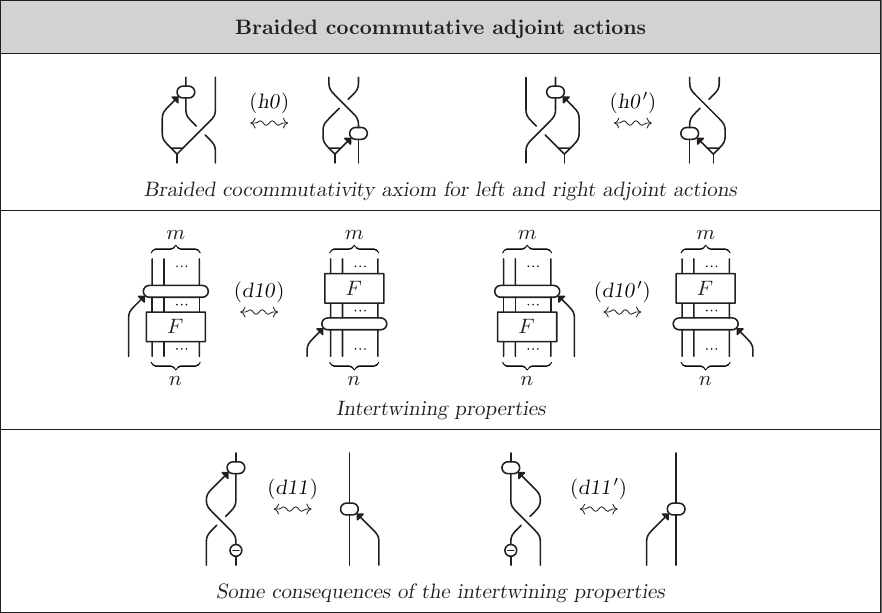}
 \caption{}
 \label{table-cocommutative/fig}
 \label{E:h0-0'} \label{E:h0} \label{E:h0'} \label{E:d10-10'} \label{E:d10} \label{E:d10'} \label{E:d11-11'} \label{E:d11} \label{E:d11'}
\end{table}

\begin{lemma}\label{cocommutative/thm}
If $\Alg^\rmB$ denotes the braided monoidal category freely generated by a Hopf algebra $H$ with braided cocomutative left  adjoint action, then the latter defines natural transformations $\cadjH : H \otimes \idfunct \Rightarrow \idfunct$ and $\cadjH' : \idfunct \otimes H \Rightarrow \idfunct$, where $\idfunct$ denotes the identity functor, meaning that, for every morphism $F : H^n \to H^m$ in $\Alg^\rmB$, we have 
\begin{gather*}
 \adjH_m \circ (\id \otimes F) = F \circ \adjH_n,
 \tag*{\hrel{d10}}
\\
\adjH'_m \circ (F \otimes \id) = F \circ \adjH'_n.
 \tag*{\hrel{d10'}}
\end{gather*}
Moreover, identities \hrel{h0'} and \hrel{d11-11'} in Table~\ref{table-cocommutative/fig} hold in $\Alg^\rmB$. 
\end{lemma}

\begin{proof}
Observe that the category $\Alg^\rmB$  is the quotient of $\Alg$ by the braided cocommutativity axiom~\hrel{h0}  and that the functor $\sym: \Alg \to \Alg$ induces an equivalence of categories $\sym: \Alg^\rmB \to \Alg^\rmB$. Therefore \hrel{h0'}  follows from \hrel{h0} by applying $\sym$. Analogously, \hrel{d10'} and \hrel{d11'} follow respectively from \hrel{d10} and \hrel{d11} which are proved below.

In order to show \hrel{d10}, it is enough to consider the case when $F$ is a structure morphism of $H$. For $F = \prodH$ and for $F = \unitH$ it was already established in Proposition~\ref{alpha/thm} (see relations~\hrel{d5} and \hrel{d6}), while for $F = \counH$ the statement follows from \hrel{a6} in Table~\ref{table-Hopf/fig}. Moreover, since relation~\hrel{s8} in Table~\ref{table-Hopf-prop/fig} allows us to express $\braid$ in terms of the rest of the generating morphisms, and since, whenever $F$ is invertible, the identity~\hrel{d10} for $F^{-1}$ is implied by the one for $F$, we only need to prove \hrel{d10} for $F = \coprH, \antipH$. This is done in Figures~\ref{ad-inv-Delta-Alg/fig} and \ref{ad-inv-S-Alg/fig}. 

Now \hrel{d11} and \hrel{d11'} follow directly from \hrel{d9} and \hrel{d9'}, respectively, by intertwining the adjoint action and the antipode.
\end{proof}

\begin{figure}[htb]
\centering
 \includegraphics{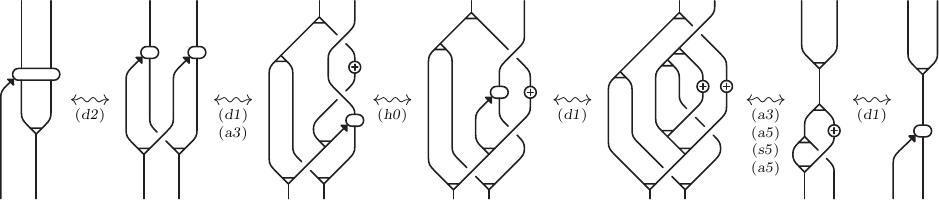}
 \caption{Proof of \hrel{d10} for $F = \coprH$.}
 \label{ad-inv-Delta-Alg/fig}
\end{figure}

\begin{figure}[htb]
 \centering
 \includegraphics{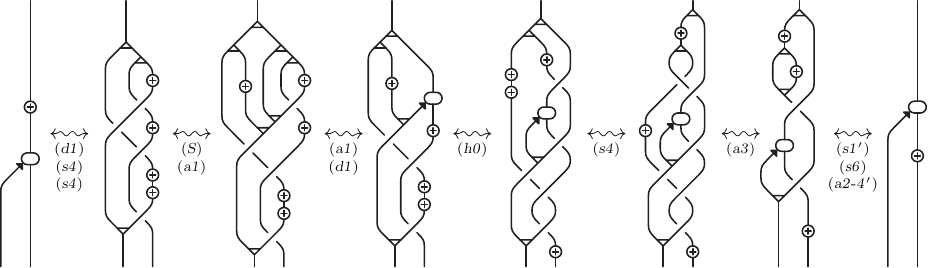}
 \caption{Proof of \hrel{d10} for $F = \antipH$.}
 \label{ad-inv-S-Alg/fig}
\end{figure}

\FloatBarrier

\subsection{BP Hopf algebras and the category \texorpdfstring{$\Algf$}{4Alg}} 
\label{BP Hopf algebra/sec}

In this subsection, we recall the definition and the properties of BP Hopf algebras. These algebraic structures were first defined and studied in \cite{BP11} in the general context of \textit{groupoid Hopf algebras}, where all the edges of the diagrams representing the structure morphisms of the algebra are labeled by elements of a groupoid $\G$. The notion of a BP Hopf algebra was introduced in \cite{BD21} and corresponds to the special case of the trivial groupoid $\G = \{1\}$. 

\begin{definition}\label{Hr/def}
If $\calC$ is a braided monoidal category with tensor product $\otimes$, tensor unit $\one$, and braiding $\braid$, a \textit{Bobtcheva--Piergallini Hopf algebra}, or \textit{BP Hopf algebra}, is a Hopf algebra $H$ in $\calC$ equipped with the following structure morphisms:
\begin{itemize}
 \item an \textit{integral form} $\intfH : H \to \one$ and an \textit{integral element} $\inteH : \one \to H$;
 \item a \textit{ribbon morphism} $\ribmorH : H \to H$ and its inverse $\ribmorH^{-1} : H \to H$;
 \item a \textit{copairing} $\copairH : \one \to H \otimes H $.
\end{itemize}
These structure morphisms are subject to the following axioms:
\begin{gather*}
 (\id \otimes \intfH) \circ \coprH = \unitH \circ \intfH,
 \tag*{\hrel{i1}}
 \\
 \prodH \circ (\inteH \otimes \id) = \inteH \circ \counH,
 \tag*{\hrel{i2}}
 \\
 \intfH \circ \inteH = \idone,
 \tag*{\hrel{i3}}
 \\
 \antipH \circ \inteH = \inteH,
 \tag*{\hrel{i4}}
 \\
 \intfH \circ \antipH = \intfH,
 \tag*{\hrel{i5}}
 \\
 \antipH \circ \ribmorH = \ribmorH \circ \antipH, 
 \tag*{\hrel{r3}}
 \\
 \counH \circ \ribmorH = \counH, 
 \tag*{\hrel{r4}}
 \\
 \prodH \circ (\ribmorH \otimes \id) = \ribmorH \circ \prodH, 
 \tag*{\hrel{r5}} 
 \\
 \copairH = (\ribmorH \otimes \ribmorH) \circ \coprH \circ \ribmorH^{-1} \circ \unitH
 \tag*{\hrel{r6}}
 \\
 (\id \otimes \coprH) \circ \copairH = (\prodH \otimes \id_2) \circ (\id \otimes \copairH \otimes \id) \circ \copairH,
 \tag*{\hrel{r7}}
 \\
 \coprH \circ \ribmorH^{-1} =(\ribmorH^{-1} \otimes \ribmorH^{-1}) \circ \monH \circ c^{-1} \circ \coprH,
 \tag*{\hrel{r8}}
 \\
 (\prodH \otimes \prodH) \circ (\antipH \otimes (\monH \circ c^{-1} \circ \monH) \otimes \antipH) \circ (\rho_L \otimes \rho_R) = \braid, 
 \tag*{\hrel{r9}}
\end{gather*}
where
\[
 \monH = (\prodH \otimes \prodH) \circ (\id \otimes \copairH \otimes \id) : H \otimes H \to H \otimes H
\]
is called the \textit{monodromy}\label{monodromy/def}, while the morphisms 
\[
 \rho_L = (\id \otimes \prodH) \circ (\copairH \otimes \id) : H \to H \otimes H 
 \quad \text{and} \quad
 \rho_R = (\prodH \otimes \id) \circ (\id \otimes \copairH) : H \to H \otimes H
 \]
define a left and a right $H$-comodule structure on $H$, respectively. 
\end{definition}

A graphical representation of the additional generators and defining axioms of a BP Hopf algebra can be found in Table~\ref{table-BPHopf/fig}, to be added to the list of generators and defining axioms of Hopf algebras given in Table~\ref{table-Hopf/fig} (compare with \cite[Tables~4.7.12 \& 4.7.13]{BP11}).

\begin{table}[htb]
 \centering
 \includegraphics{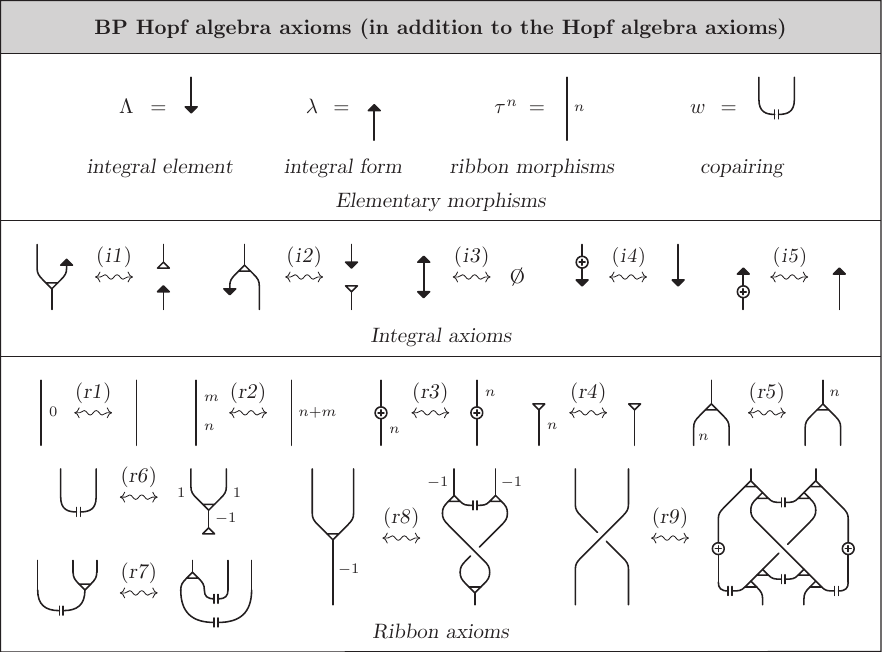}
 \caption{}
 \label{table-BPHopf/fig}
 \label{E:i1} \label{E:i2} \label{E:i3} \label{E:i4} \label{E:i5} \label{E:r1} \label{E:r2} \label{E:r3} \label{E:r4} \label{E:r5} \label{E:r6} \label{E:r7} \label{E:r8} \label{E:r9}
\end{table}

\begin{definition}\label{Algf/def} 
We denote by $\Algf$ the braided monoidal category freely generated by a BP Hopf algebra $H$. In other words, objects of $\Algf$ are tensor powers of $H$, while morphisms of $\Algf$ are compositions of tensor products of identities, braidings, and structure morphisms $\prodH$, $\unitH$, $\coprH$, $\counH$, $\antipH$, $\antipH^{-1}$, $\intfH$, $\inteH$, $\ribmorH$, $\ribmorH^{-1}$, $\copairH$, modulo the defining axioms listed in Definition~\ref{Hr/def}. 
\end{definition}

Observe that, since $H$ is a Hopf algebra in $\Algf$, then, according to the Universal Property \ref{univ_prop/thm}, there exists a unique functor $4\Xi : \Alg\to \Algf$ that sends $H$ to itself. Moreover, by definition, we have the following universal property.

\begin{universal_property}\label{univ_prop_BP/thm}
If $\calC'$ is a braided monoidal category and $H' \in \calC'$ is a BP Hopf algebra, then the braided monoidal functor $\Xi_{H'} : \Alg \to \calC'$ given by the universal property of $\Alg$ factors through $4\Xi : \Alg \to \Algf$.
\end{universal_property}

\begin{remark} \label{BPaxioms/rmk} 
As it is shown in Figure~\ref{proof-r10/fig}, relation~\hrel{r6} is not an independent axiom, but it is a consequence of \hrel{r8} and the Hopf algebra axioms. We present it as an axiom, first of all, because it gives an explicit expression for the copairing in terms of the ribbon morphism and the coproduct, and in second place since, as it will be shown in Proposition~\ref{adjoint/thm}, in the presence of \hrel{r6} axiom \hrel{r8} can be expressed in terms of the adjoint action by its equivalent forms \hrel{d12} or \hrel{d12'}.
 
Moreover, the original definition of BP Hopf algebra in \cite[Table~4.7.13]{BP11} uses as an axiom relation \hrel{p4} in Table~\ref{table-BPHopf-prop21/fig} in place of \hrel{r6}. As it is shown by Kerler in \cite[Lemma 4]{Ke01} (see also Figure~\ref{proof-eq-r10-r6/fig} in Appendix~\ref{proofs/app}), those two relations are equivalent modulo the axioms of braided Hopf algebra and the ribbon axioms~\hrel{r1} to \hrel{r5} and \hrel{r7}. Therefore Definition~\ref{Hr/def} is equivalent to the one in \cite{BP11}. 

\begin{figure}[hbt]
 \centering
 \includegraphics{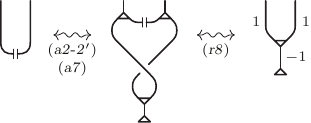}
 \caption{Proof of \hrel{r6} using the rest of the axioms of a BP Hopf algebra.}
 \label{proof-r10/fig}
\end{figure}

\end{remark}

\begin{proposition}\label{symmetry/thm}
 The functor $\sym : \Algf \to \Algf$ is an involutive anti-monoidal equivalence functor that sends every object and every elementary morphism to itself. Moreover, $\sym$ fits into the commutative diagram of functors
\[
\begin{tikzpicture}
 \node (P0) at (0,0) {$\Alg$};
 \node (P1) at (2.25,0) {$\Alg$};
 \node (P2) at (0,-1.5) {$\Algf$};
 \node (P3) at (2.25,-1.5) {$\Algf$};
 \draw
 (P0) edge[->] node[above] {$\textstyle \sym$} (P1)
 (P0) edge[->] node[left, xshift=-0.25ex] {$\textstyle 4\Xi$} (P2)
 (P1) edge[->] node[right, xshift=0.25ex] {$\textstyle 4\Xi$} (P3)
 (P2) edge[->] node[below, yshift=-0.5ex] {$\textstyle \sym$} (P3);
\end{tikzpicture}
\]
\end{proposition}

\begin{proof}
The statement follows from Proposition \ref{funt-sym/thm}, provided we check that $\sym$ sends all elementary morphisms of $\Algf$ to themselves. For the Hopf algebra structure morphisms this was already proved in Proposition~\ref{symmetry-alg/thm}. For the integral element, the integral form, and the ribbon morphism this fact reduces to axioms~\hrel{i4}, \hrel{i5} and \hrel{r3}, while for the copairing it follows from its definition in terms of a diagram which is invariant under the functor $\sym$, which is given by \hrel{r6} (see Figure~\ref{proof-inv-copairing/fig}).
\end{proof}

\begin{figure}[hbt]
 \centering
 \includegraphics{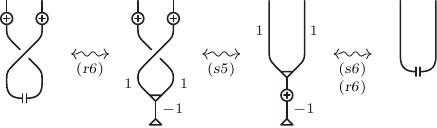}
 \caption{Proof that $\sym(\copairH) = \copairH$.}
 \label{proof-inv-copairing/fig}
\end{figure}

The reader can find the diagrammatic proofs of the following propositions in Appendix~\ref{proofs/app} (see also \cite[Propositions~4.1.4, 4.1.5, 4.1.6, 4.1.9, 4.1.10, Lemmas~4.2.5, 4.2.6, Propositions~4.2.7, 4.2.11, 4.2.13]{BP11} and \cite[Lemmas~1--8]{Ke01}). 

\clearpage

\begin{proposition}\label{BP-prop1/thm}
The identities in Table~\ref{table-BPHopf-prop1/fig} hold for every Hopf algebra $H$ in any braided monoidal category $\calC$ equipped with an integral form $\intfH : H \to \one$ and an integral element $\inteH : \one \to H$ satisfying axioms~\hrel{i1}--\hrel{i5} in Table~\ref{table-BPHopf/fig}. In particular, they hold in $\Algf$. 
\end{proposition}


\begin{table}[htb]
 \centering
 \includegraphics{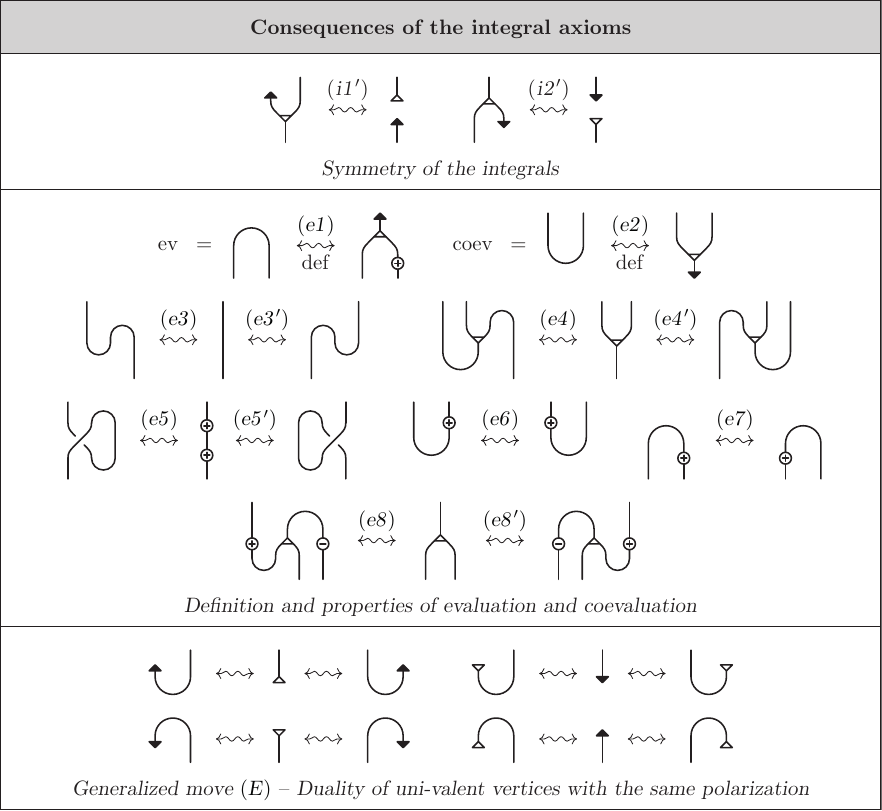}
 \caption{}
 \label{table-BPHopf-prop1/fig}
 \label{E:i1'} \label{E:i2'} \label{E:e1-2} \label{E:e1} \label{E:e2} \label{E:e3-3'} \label{E:e3} \label{E:e3'} \label{E:e4-4'} \label{E:e4} \label{E:e4'} \label{E:e5-5'} \label{E:e5} \label{E:e5'} \label{E:e6} \label{E:e7} \label{E:e8-8'} \label{E:e8} \label{E:e8'} \label{E:E}
\end{table}

\FloatBarrier

\clearpage

\begin{proposition}\label{BP-prop21/thm}
The identities in Table~\ref{table-BPHopf-prop21/fig} hold in any braided monoidal category with Hopf algebra $H$ and a family of ribbon morphisms $\ribmorH^n: H\to H$ satisfying axioms \hrel{r1} to \hrel{r7} in Table \ref{table-BPHopf/fig}. In particular, they hold in $\Algf$.
\end{proposition}

\begin{table}[htb]
 \centering
 \includegraphics{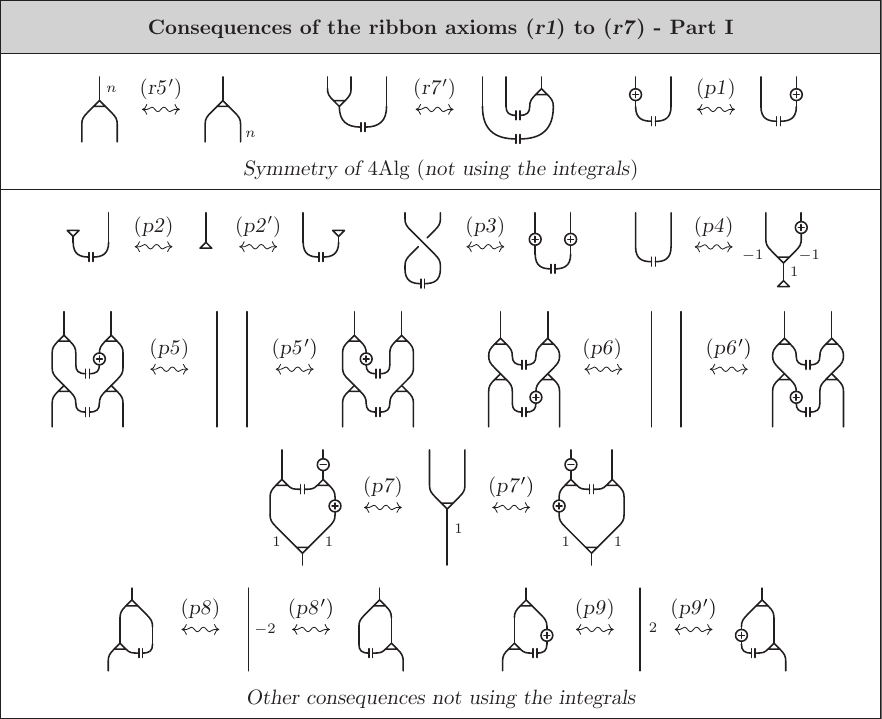}
 \caption{}
 \label{table-BPHopf-prop21/fig}
 \label{E:r5'} \label{E:r7'} \label{E:p1} \label{E:p2-2'} \label{E:p2} \label{E:p2'} \label{E:p3} \label{E:p4} \label{E:p5-5'} \label{E:p5} \label{E:p5'} \label{E:p6-6'} \label{E:p6} \label{E:p6'} \label{E:p7-7'} \label{E:p7} \label{E:p7'} \label{E:p8-8'} \label{E:p8} \label{E:p8'} \label{E:p9-9'} \label{E:p9} \label{E:p9'}
\end{table}

\FloatBarrier

\begin{proposition}\label{BP-prop22/thm}
 The identities in Table~\ref{table-BPHopf-prop22/fig} hold in any braided monoidal category with Hopf algebra $H$, a family of ribbon morphisms $\ribmorH^n: H\to H$ satisfying axioms \hrel{r1} to \hrel{r7} in Table \ref{table-BPHopf/fig}, and an integral form $\intfH: H\to \one$ and an integral element $\inteH: \one\to H$ satisfying axioms \hrel{i1} to \hrel{i5} in the same table. In particular, the identities in Table~ \ref{table-BPHopf-prop22/fig} hold in $\Algf$.
\end{proposition}

\begin{table}[htb]
 \centering
 \includegraphics{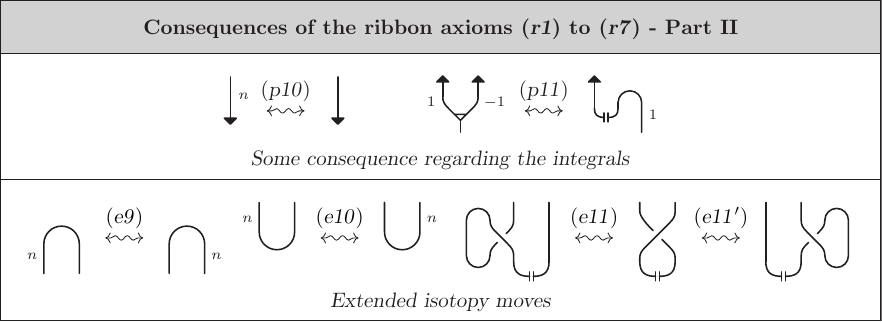}
 \caption{}
 \label{table-BPHopf-prop22/fig}
 \label{E:p10} \label{E:p11} \label{E:e9} \label{E:e10} \label{E:e11-11'} \label{E:e11} \label{E:e11'}
\end{table}


\begin{proposition}\label{BP-prop3/thm}
The identities in Table~\ref{table-BPHopf-prop3/fig} are satisfied in $\Algf$. Moreover, modulo the rest of the defining axioms, relation~\hrel{p12} is an equivalent reformulation of axiom~\hrel{r8}, while relation~\hrel{p13} is an equivalent reformulation of axiom~\hrel{r9}. 
\end{proposition}

\begin{table}[htb]
 \centering
 \includegraphics{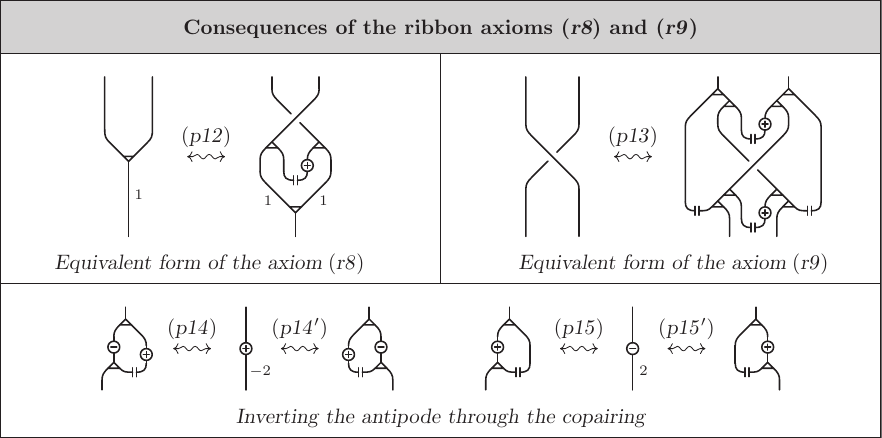}
 \caption{}
 \label{table-BPHopf-prop3/fig}
 \label{E:p12} \label{E:p13} \label{E:p14-14'} \label{E:p14} \label{E:p14'} \label{E:p15-15'} \label{E:p15} \label{E:p15'}
\end{table} 

\FloatBarrier

The propositions above have the following implications.

\begin{proposition}\label{autonomous/thm}
$\Algf$ is a ribbon category (see Definition \ref{rigid-cat/def}) with dual $(H^n)^* = H^n$ for every $n \geqs 0$, with two-sided evaluation $\ev_{n} : H^n \otimes H^n \to \one$ and coevaluation $\coev_{n} : \one \to H^n \otimes H^n$ defined inductively by $\ev_0 = \idone = \coev_0$, by
\begin{gather*}
 \ev_1 = \ev = \intfH \circ \prodH \circ (\id \otimes \antipH),
 \tag*{\hrel{e1}}
 \\
 \coev_1 = \coev = \coprH \circ \inteH,
 \tag*{\hrel{e2}}
\end{gather*}
and by
\begin{gather*}
 \ev_n  = \ev \circ (\id \otimes \ev_{n-1} \otimes \id), \\
 \coev_n  = (\id \otimes \coev_{n-1} \otimes \id) \circ \coev,
\end{gather*}
for every $n>1$, and with twist $\theta_n: H^n \to H^n$ defined by
\[
\theta_n  = (\ev_n \otimes \id_n) \circ (\id_n \otimes \braid_{n,n}) \circ (\coev_n \otimes \id_n)
\]
for every $n \geqs 0$.
\end{proposition}

\begin{proof}
The statement follows from Proposition~\ref{Prob-ribbon/thm} once we observe that $\ev$ and $\coev$ defined in Table~\ref{table-BPHopf-prop1/fig} satisfy the zig-zag relation~\hrel{e3-3'} and the ribbon axiom~\hrel{e5-5'}.
\end{proof}

Notice that the invertibility of the antipode, together with \hrel[e6]{e6-7}, implies that we can move also its inverse from one side of the evaluation and coevaluation morphisms to the other. Analogously, relations~\hrel{e5-5'} imply that a negative right or left kink is equivalent to $S^{-2}$. 

\begin{definition}\label{gen-istop/def}\label{E:I} 
We will say that two diagrams representing morphisms in $\Algf$ are related by a \textit{general isotopy} move~\rel{I} if they can be obtained from each other by a sequence of the following moves: 
\begin{itemize}
 \item braiding axioms in Table~\ref{table-braided/fig};
 \item relation~\hrel{e3-3'}, which implements planar isotopies, relation~\hrel{e5} paired with \hrel{e5'}, which implements the ribbon axiom, and relation~\hrel[e6]{e6-7} in Table~\ref{table-BPHopf-prop1/fig};
 \item relation~\hrel[e9]{e9-10} in Table~\ref{table-BPHopf-prop22/fig}.
\end{itemize}
 In particular, if we think of such graphs as ribbon graphs with coupons, these moves generate any arbitrary isotopy, see \cite[Lemma~3.4]{Tu94}. 
\end{definition}




The following lemma introduces an auxiliary morphism $J : H \to H$ that can be understood as a categorical analog of the Drinfeld map, whose properties will allow us to extend the notion of isotopy moves to strands that contain copairing morphisms. 


\begin{figure}[htb]
 \centering
 \includegraphics{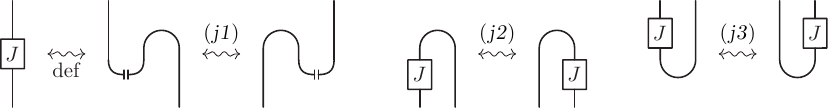}
 \caption{}
 \label{J-moves/fig}
 \label{E:j1} \label{E:j2} \label{E:j3}
\end{figure}

\begin{lemma}\label{copair-slide/thm}\label{E:J}
The morphism $J : H \to H$ of $\Algf$ defined as (see Figure~\ref{J-moves/fig})
\begin{gather*}
 J = (\id \otimes \ev) \circ (\copairH\otimes\id ) = (\ev \otimes \id) \circ ( \id\otimes \copairH)
 \tag*{\hrel{j1}}
\end{gather*}
satisfies the following identities:
\begin{gather*}
  (J \otimes \id)\circ\coev = (\id \otimes J)\circ\coev,
 \tag*{\hrel{j2}}
 \\
 \ev\circ(J \otimes \id)= \ev\circ(\id \otimes J).
 \tag*{\hrel{j3}}
\end{gather*}
As a consequence, copairing and coevaluation morphisms lying on the same strand can always be exchanged and we will refer to such modification as generalized move \rel{J}. 
\end{lemma}

\begin{figure}[htb]
 \centering
 \includegraphics{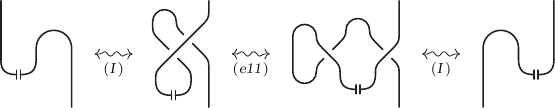}
 \caption{Proof of identity \hrel{j1}.}
 \label{J-def/fig}
\end{figure}

\begin{proof}
Relation \hrel{j1} is proved in Figure~\ref{J-def/fig}. 
Then, relations \hrel{j2} and \hrel{j3} reduce to \hrel{j1} modulo \hrel{e3-3'}.

Relations~\hrel[j2]{j2-3}, together with the naturality of the braiding, imply that we can slide the morphisms $J$ along any strand. Suppose now that a strand contains both a copairing and a coevaluation morphism (an example is shown in the left-hand side of Figure~\ref{isot-copairing/fig}). Then, we can incorporate the copairing in a copy of the morphism $J$, by inserting evaluation and coevaluation morphisms (using moves~\hrel{e3-3'}). We can then slide $J$ along the strand until it reaches the coevaluation morphism, and cancel the latter against the evaluation morphism in $J$ (using moves~\hrel{e3-3'}) after possibly applying \hrel{j1} (see Figure~\ref{isot-copairing/fig} for an example). In the resulting configuration, copairing and coevaluation morphisms have been exchanged.
\end{proof}

\begin{figure}[htb]
 \centering
 \includegraphics{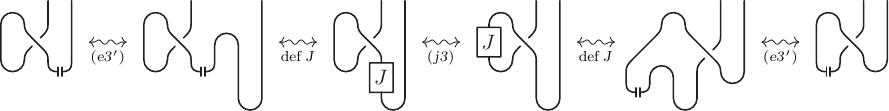}
 \caption{Exchanging copairing and coevaluation morphisms that lie on the same strand.}
 \label{isot-copairing/fig}
\end{figure}



 
\FloatBarrier

\subsection{Frobenius structure and braided cocommutativity in \texorpdfstring{$\Algf$}{4Alg}}
\label{4Alg-adj/sec}

Let us sidetrack for a moment, and introduce a modified product $\tilde{\prodH}$ which, together with the modified unit $\tilde{\unitH} = \inteH$, and with the standard coproduct $\coprH$ and counit $\counH$, provides every BP Hopf algebra $H$ with a Frobenius algebra structure.

\begin{table}[htb]
 \centering
 \includegraphics{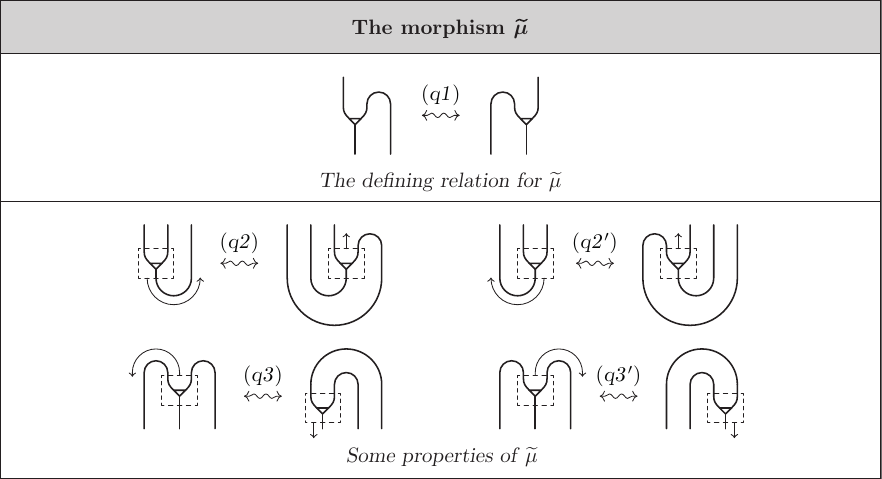}
 \caption{}
 \label{table-mu/fig}
 \label{E:q1} \label{E:q2-2'} \label{E:q2} \label{E:q2'} \label{E:q3-3'} \label{E:q3} \label{E:q3'} \label{E:q4-4'} 
\end{table}

\begin{proposition}\label{mu/thm}
If we set $\tilde{\prodH} = (\id \otimes \ev) \circ (\coprH \otimes \id)$ and $\tilde{\prodH}' = (\ev \otimes \id) \circ (\id \otimes \coprH)$, then the following identities hold in $\Algf$:
\begin{gather*}
 \tilde{\prodH} = \tilde{\prodH}', 
 \tag*{\hrel{q1}}
 \\
 (\coprH \otimes \id) \circ \coev = (\id_2 \otimes \tilde{\prodH}) \circ \coev_2,
 \tag*{\hrel{q2}}
 \\
 (\tilde{\prodH} \otimes \id_2) \circ \coev_2 = (\id \otimes \coprH) \circ \coev,
 \tag*{\hrel{q2'}}
 \\
 \ev_2 \circ (\coprH \otimes \id_2) = \ev \circ (\id \otimes \tilde{\prodH}), 
 \tag*{\hrel{q3}}
 \\
 \ev_2 \circ (\id_2 \otimes \coprH) = \ev \circ(\tilde{\prodH} \otimes \id), 
 \tag*{\hrel{q3'}}
\end{gather*}
\end{proposition}

A graphical representation of relations~\hrel{q1} to \hrel{q3'} can be found in Table~\ref{table-mu/fig} (where the reader should ignore for now the dashed boxes and arrows). Notice that relations~\hrel{q2}, \hrel{q2'}, \hrel{q3}, and \hrel{q3'} imply that $\tilde{\prodH}$ and $\coprH$ are dual to each other with respect to the coevaluation. Furthermore, as mentioned above, $H$ admits the structure of a Frobenius algebra in $\Algf$, determined by the product $\tilde{\prodH}$, the unit $\tilde{\unitH} = \inteH$, the coproduct $\coprH$, and the counit $\counH$ (see \cite[Appendix~A.2]{FS10}).

\begin{proof}
Relation~\hrel{q1} follows directly from \hrel{e3'} and \hrel{e4}. Relations~\hrel{q2} and \hrel{q3} are proved in Figure~\ref{sliding-thm1A/fig}, while relations~\hrel{q2'} and \hrel{q3'} follow by applying the symmetry functor, thanks to Proposition~\ref{symmetry/thm}.
\end{proof}

\begin{figure}[htb]
 \centering
 \includegraphics{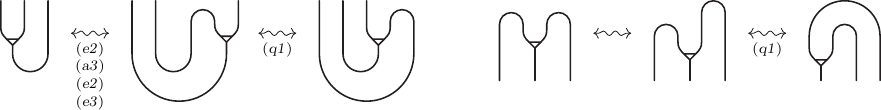}
 \caption{Proof of \hrel{q2}, \hrel{q3}.}
 \label{sliding-thm1A/fig}
\end{figure}



Next, let us establish some properties of the adjoint actions of a BP Hopf algebra. Such properties have already been proved in \cite[Subsection 4.4]{BP11}, but we present here an alternative argument, based on the fact that the left and right adjoint actions of a BP Hopf algebra are braided cocommutative.

\clearpage

\begin{proposition}\label{adjoint/thm}
In a BP Hopf algebra, modulo the other axioms, \hrel{r8} and \hrel{r9} admit the equivalent forms presented in Table~\ref{table-adjoint-prop/fig}. Namely, \hrel{d12-12'} are equivalent to \hrel{r8}, while \hrel{d13-13'} and \hrel{d14-14'} are equivalent to \hrel{r9}.
\end{proposition}
 
\begin{table}[htb]
 \centering
 \includegraphics{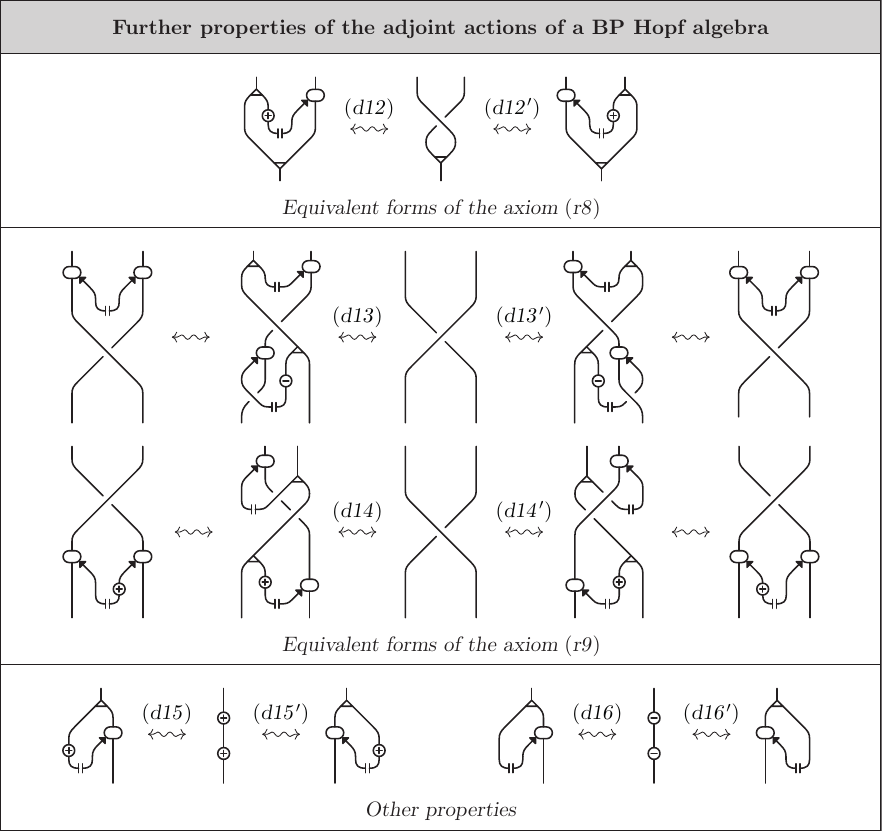}
 \caption{}
 \label{table-adjoint-prop/fig}
 \label{E:d12-12'} \label{E:d12} \label{E:d12'} \label{E:d13-13'} \label{E:d13} \label{E:d13'} \label{E:d14-14'} \label{E:d14} \label{E:d14'} \label{E:d15-15'} \label{E:d15} \label{E:d15'} \label{E:d16-16'} \label{E:d16} \label{E:d16'}
\end{table}

\FloatBarrier

\begin{proof}
In Figure~\ref{adjoint05/fig} we prove that, modulo the rest of the BP Hopf algebra axioms, excluded \hrel{r9}, axiom \hrel{r8} implies \hrel{d12} and the other way around. Therefore, \hrel{d12} is an equivalent reformulation of \hrel{r8}. Analogously, we show in Figure~\ref{adjoint08/fig} that \hrel{r9} is equivalent to \hrel{d13} modulo the rest of the BP Hopf algebra axioms, except \hrel{r8}. Then a straightforward application of \hrel{d7-7'} and \hrel{d8-8'} shows that the diagrams in \hrel{d14} represent the inverse morphisms of those represented by the diagrams in \hrel{d13}, which gives the equivalence between \hrel{r8} and \hrel{d14}. Then the statements for \hrel{d12'}, \hrel{d13'} and \hrel{d14'} are obtained by applying the functor $\sym$.
\end{proof}

\begin{figure}[htb]
 \centering
 \includegraphics{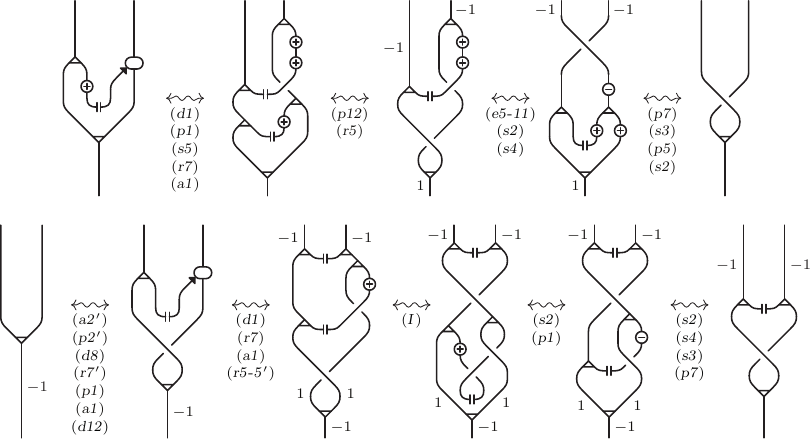}
 \caption{Equivalence between \hrel{r8} and \hrel{d12}.}
 \label{adjoint05/fig}
\end{figure}

\begin{figure}[htb]
 \centering
 \includegraphics{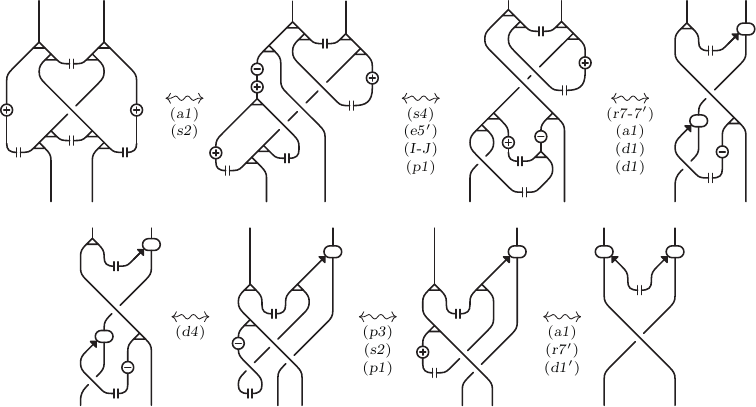}
 \caption{Equivalence between \hrel{r9} and \hrel{d13}.}
 \label{adjoint08/fig}
\end{figure}

\FloatBarrier

\begin{proposition}\label{intertwining/thm}
The left and right adjoint actions of a BP Hopf algebra satisfy the braided co\-commutativity axiom~\hrel{h0-0'}, which implies the intertwining properties~\hrel{d10-10'} and relations~\hrel{d11-11'} in Table~\ref{table-cocommutative/fig}. In particular, left and right adjoint actions intertwine all morphisms in $\Algf$. Moreover, they satisfy relations~\hrel{d15-15'} and \hrel{d16-16'} in Table~\ref{table-adjoint-prop/fig}.
 \end{proposition}

\begin{proof} 
Concerning the left adjoint action, relation~\hrel{h0} is proven in Figure~\ref{proof-h0/fig}. Then, Lemma~\ref{cocommutative/thm} implies that \hrel{d10} holds for the product, the coproduct, the unit, the counit, the antipode and its inverse, and also that \hrel{d11-11'} are satisfied. We have to show that \hrel{d10} holds for the integrals, for the ribbon morphism, and for the copairing. For the integral element it follows from \hrel{i2}, \hrel{i2'} and \hrel{s7}, while for the integral form it is shown in Figure~\ref{natural-xi01/fig}. For the ribbon morphism it follows from \hrel{r5}, and \hrel{p4} implies that it holds for the copairing as well. Then, by applying the functor $\sym$, we get the analogous properties for the right adjoint action. Finally, the proofs of \hrel{d15} and \hrel{d16} are shown in Figure~\ref{adjoint07/fig}, while \hrel{d15'} and \hrel{d16'} are obtained by applying $\sym$ once again.
\end{proof}

\begin{figure}[htb]
 \centering
 \includegraphics{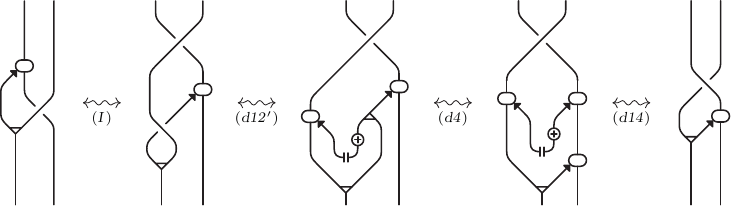}
 \caption{Proof of \hrel{h0}.}
 \label{proof-h0/fig}
\end{figure}

\begin{figure}[htb]
 \centering
 \includegraphics{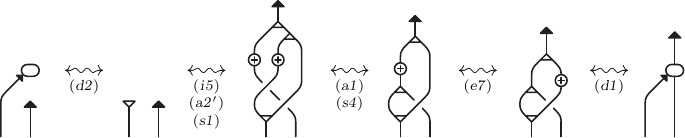}
 \caption{Proof of \hrel{d10} for $F = \intfH$.}
 \label{natural-xi01/fig}
\end{figure}

\begin{figure}[htb]
 \centering
 \includegraphics{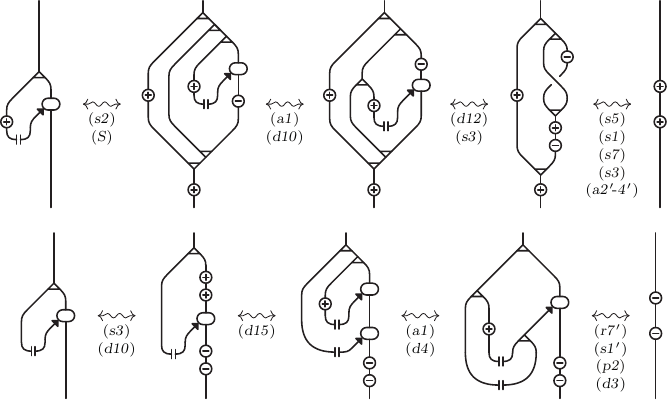}
 \caption{Proof of \hrel{d15} and \hrel{d16}.}
 \label{adjoint07/fig}
\end{figure}

\begin{remark}\label{adjointprop/rem}
Notice that the proofs of \hrel{d15} and \hrel{d16} presented above only use identities \hrel{d10}, \hrel{d12}, the Hopf algebra axioms, and their consequences for the adjoint action presented in Table~\ref{table-adjoint/fig}. This fact is going to be important later, when we will prove that $\Algt$ is equivalent to the category $\AlgH$ introduced in Subsection~\ref{HabiroHalgebra/sec}. 
\end{remark}

\FloatBarrier

\subsection{Factorizable BP Hopf algebras and the categories \texorpdfstring{$\Algt$}{3Alg} and \texorpdfstring{$\AlgH$}{AlgH}}
\label{HabiroHalgebra/sec}

In this subsection, we will introduce an important non-degeneracy condition for BP Hopf algebras, called \textit{factorizability}. We will prove that factorizable anomaly-free BP Hopf algebras\footnote{Factorizable anomaly-free BP Hopf algebras were introduced \cite{BP11} under the name \textit{boundary ribbon Hopf algebras.}} are equivalent to \textit{Habiro Hopf algebras}, a notion due to Habiro that was first defined in \cite{As11}.

\begin{definition}\label{factorizable/def}
A BP Hopf algebra $H$ in $\calC$ is \textit{factorizable} if it satisfies
\begin{equation*}
 (\intfH \otimes \id) \circ \copairH = \inteH, 
 \tag*{\hrel{f}}
\end{equation*}
and it is \textit{anomaly-free} if it satisfies
\begin{equation*}
 \intfH \circ \ribmorH \circ \unitH = \idone. 
 \tag*{\hrel{n}}
\end{equation*}
\end{definition}

The axioms of anomaly-free factorizable BP Hopf algebras are presented in Table~\ref{table-BPHopf3/fig}. These axioms imply the relations in Table~\ref{table-BPHopf3-prop/fig}, as it is shown in Section~\ref{BPHopf3-proofs/sec} of Appendix~\ref{proofs/app} (see also \cite[Propositions~5.4.2 \& 5.4.3]{BP11}). In particular, axiom~\hrel{f} implies the existence of a Hopf pairing $\pairH : H \otimes H \to \one$ which, together with the copairing $\copairH$, satisfies the zigzag identities \hrel{f2-2'} in Table~\ref{table-BPHopf3-prop/fig}. Therefore, both $\pairH$ and $\copairH$ are non-degenerate. By analogy with the standard theory of ribbon Hopf algebras, we use the term factorizability to denote this property.

\begin{table}[htb]
 \centering
 \includegraphics{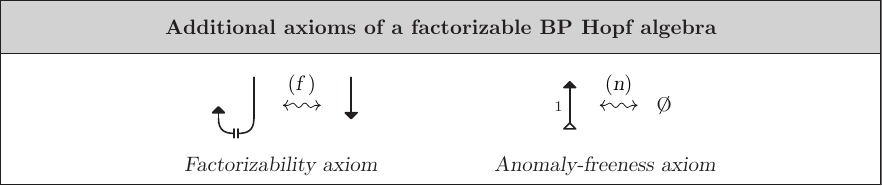}
 \caption{}
 \label{table-BPHopf3/fig}
 \label{E:f} \label{E:n}
\end{table}

\begin{table}[htb]
 \centering
 \includegraphics{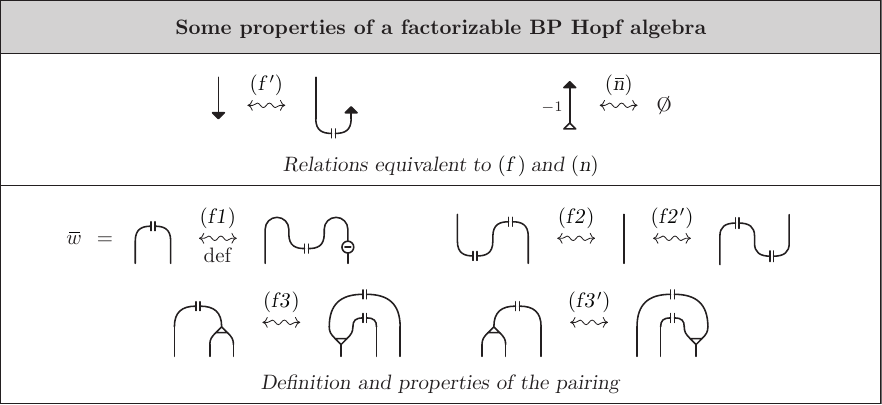}
 \caption{}
 \label{table-BPHopf3-prop/fig}
 \label{E:f'} \label{E:n'} \label{E:f1} \label{E:f2-2'} \label{E:f2} \label{E:f2'} \label{E:f3-3'} \label{E:f3} \label{E:f3'}
\end{table}

\begin{definition}\label{Algt/def} 
We denote by $\Algt$ the braided monoidal category freely generated by an anomaly-free factorizable BP Hopf algebra $H$. In other words, $\Algt$ is the quotient of $\Algf$ by the relations~\hrel{f} and \hrel{n}.
\end{definition}

\begin{definition}\label{AlgH/defn}
Let $\calC$ be a braided monoidal category with tensor product $\otimes$, tensor unit $\one$, and braiding $\braid$. A \textit{Habiro Hopf algebra} is a Hopf algebra $H$ in $\calC$ with braided cocommutative left adjoint action, equipped with the following structure morphisms:
\begin{itemize}
 \item a \textit{copairing} $\copairH : \one \to H \otimes H$ and a pairing $\pairH : H \otimes H \to \one$;
 \item a \textit{ribbon element} $\ribelH_+ : \one \to H$ and its multiplicative inverse $\ribelH_- : \one \to H$.
\end{itemize}
These structure morphisms are subject to the following axioms:
\begin{gather*}
 \prodH \circ (\ribelH_+ \otimes \id) = \prodH \circ (\id \otimes \ribelH_+),
 \tag*{\hrel{h1}}
 \\
 \prodH \circ (\ribelH_+ \otimes \ribelH_-) = \unitH, 
 \tag*{\hrel{h2}}
 \\
 \counH \circ \ribelH_+ = \idone, 
 \tag*{\hrel{h3}}
 \\
 \antipH \circ \ribelH_+ = \ribelH_+, 
 \tag*{\hrel{h4}}
 \\
\copairH=(\prodH \otimes \prodH) \circ (\ribelH_- \otimes \id_2 \otimes \ribelH_-)\circ \coprH \circ \ribelH_+, 
 \tag*{\hrel{h5}}
 \\
 (\id \otimes \coprH) \circ \copairH = (\prodH \otimes \id_2) \circ (\id \otimes \copairH \otimes \id) \circ \copairH, 
 \tag*{\hrel{h6}}
 \\
 (\id \otimes \pairH) \circ (\copairH \otimes \id) = \id , 
 \tag*{\hrel{h7}}
 \\
 \pairH \circ (\prodH \otimes \ribelH_+) \circ (\ribelH_+ \otimes \ribelH_+) = \idone, 
 \tag*{\hrel{h8}}
\end{gather*}
We denote by $\AlgH$ the braided monoidal category freely generated by a~Habiro Hopf algebra $H$. 
\end{definition}

\begin{table}[htb]
 \centering
 \includegraphics{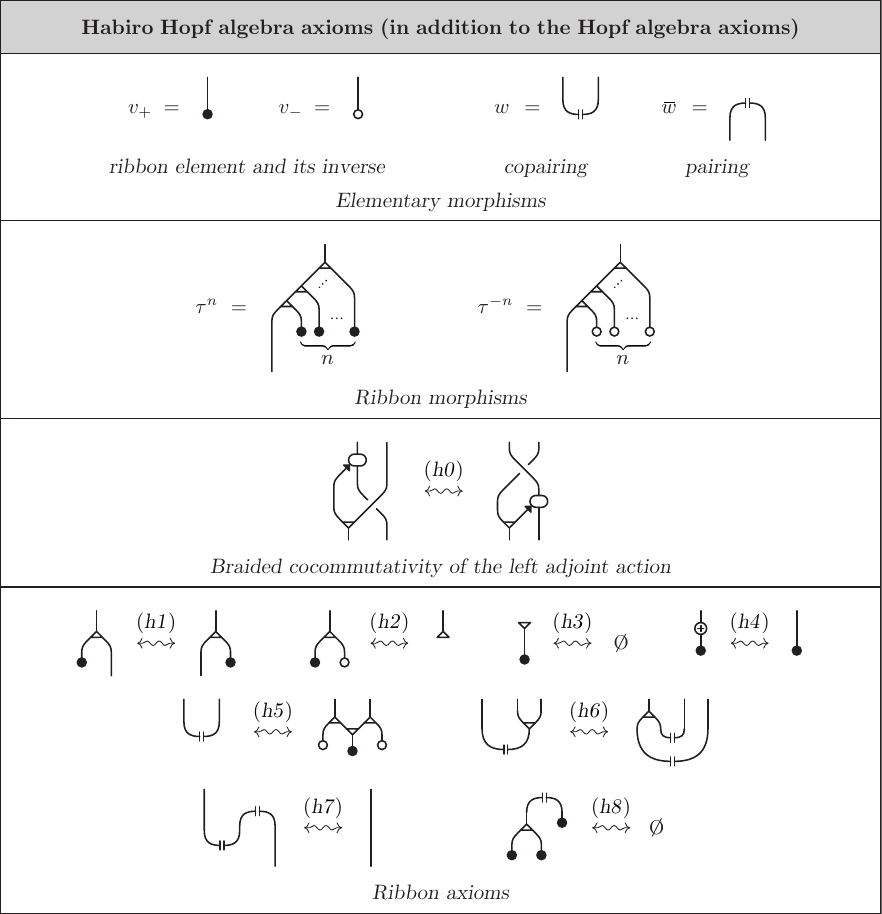}
 \caption{}
 \label{table-Habiro/fig}
 \label{E:h0H} \label{E:h1} \label{E:h2} \label{E:h3} \label{E:h4} \label{E:h5} \label{E:h6} \label{E:h7} \label{E:h8}
\end{table}

\begin{table}[htb]
 \centering
 \includegraphics{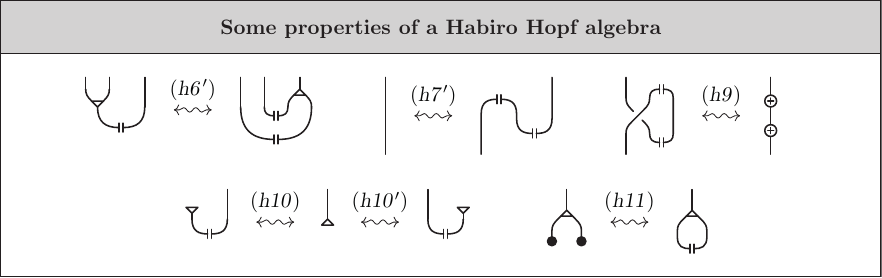}
 \caption{}
 \label{table-Habiro-prop/fig}
 \label{E:h6'} \label{E:h7'} \label{E:h9} \label{E:h10-10'} \label{E:h10} \label{E:h10'} \label{E:h11} 
\end{table}

A diagrammatic representation of the generators and the axioms of a Habiro Hopf algebra can be found in Table~\ref{table-Habiro/fig}. Notice that the notation adopted for all structure morphisms, with the exception of the ribbon elements, is the same as the one used for the analogous structure morphisms of BP Hopf algebras in Tables~\ref{table-BPHopf/fig} and \ref{table-BPHopf3/fig}. This should not cause any confusion since, as we will see below, the equivalence functor from $\AlgH$ to $\Algt$ matches the corresponding structure morphisms. 

\begin{lemma}\label{Halg-prop/thm}
The ribbon morphisms and the copairing of $\AlgH$, defined in Table \ref{table-Habiro/fig}, satisfy the ribbon axioms \hrel{r1} to \hrel{r7} of a BP Hopf algebra in Table \ref{table-BPHopf/fig}. Therefore the relations in Table \ref{table-BPHopf-prop21/fig} are satisfied in $\AlgH$.
\end{lemma} 

\begin{proof}
The fact that the ribbon morphisms of $\AlgH$ satisfy axioms \hrel{r1} to \hrel{r5} in Table \ref{table-BPHopf/fig} is a straightforward consequence of axioms \hrel{h1} to \hrel{h4} and the associativity of the product. Moreover, axioms \hrel{r6} and \hrel{r7} are equal correspondingly to \hrel{h5} and \hrel{h6}. Therefore, according to Proposition \ref{BP-prop1/thm}, the relations in Table \ref{table-BPHopf-prop21/fig} hold in $\AlgH$. 
\end{proof}

\begin{proposition}\label{symmetryH/thm}
 The functor $\sym : \AlgH \to \AlgH$ is an involutive anti-monoidal equivalence functor that sends every object and every elementary morphism to itself. 
\end{proposition}

\begin{proof}
The statement follows from Proposition \ref{funt-sym/thm}, provided we check that $\sym$ sends all elementary morphisms of $\AlgH$ to themselves. For the Hopf algebra structure morphisms this was already proved in Proposition~\ref{symmetry-alg/thm}. For the ribbon element and its inverse this fact reduces to axiom~\hrel{h4}, while for the copairing it follows from its definition in terms of a diagram which is invariant under the functor $\sym$, which is given by \hrel{h5}. The proof for the pairing morphism is shown in Figure~\ref{proof-inv-pairingH/fig}, where in the second step we have used the invariance of the copairing under $\sym$.
\end{proof}

\begin{figure}[hbt]
 \centering
 \includegraphics{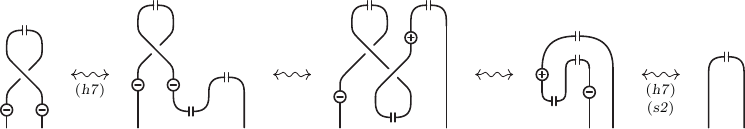}
 \caption{Proof that $\sym(\pairH) = \pairH$, where, in the second step, $\sym(\copairH)=\copairH$ is used.}
 \label{proof-inv-pairingH/fig}
\end{figure}

The set of axioms of a Habiro Hopf algebra, as originally presented in \cite{As11}, also contains the relations in Table~\ref{table-Habiro-prop/fig}, but, as it is shown in the proposition below, these relations are actually a consequence of the axioms in Table~\ref{table-Habiro/fig}.

\begin{proposition}\label{Hprop/thm}
The relations in Table~\ref{table-Habiro-prop/fig} are satisfied in $\AlgH$.
\end{proposition}

\begin{proof}
Relations~\hrel{h6'} and \hrel{h7'} follow from \hrel{h6} and \hrel{h7} by applying the functor $\sym$ (Proposition~\ref{symmetryH/thm}). On the other hand, \hrel{h10-10'} and \hrel{h11} follow from Proposition~\ref{Halg-prop/thm}, since \hrel{h10-10'} are equal to \hrel{p2-2'} in Table~\ref{table-BPHopf-prop21/fig},  while \hrel{h11} is obtained by composing \hrel{p8} with the unit morphism. Finally, \hrel{h9} is proved in Figure~\ref{proof-h9H/fig} using \hrel{p1} in Table~\ref{table-BPHopf-prop21/fig}.
\end{proof}

\begin{figure}[hbt]
 \centering
 \includegraphics{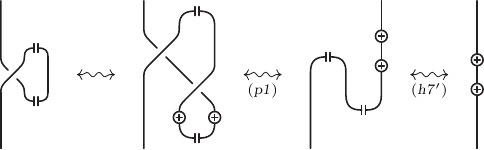}
 \caption{Proof of \hrel{h9}, where, in the second step, $\sym(\pairH)=\pairH$ is used.}
 \label{proof-h9H/fig}
\end{figure}

\begin{proposition} \label{functor-Hab/thm}
There exists a braided monoidal functor $\Gamma: \AlgH \to \Algt$ which preserves the Hopf algebra structure morphisms, sends the pairing and the copairing in $\AlgH$ to the corresponding ones in $\Algt$ (see Table~\ref{table-BPHopf3/fig}), and sends the ribbon elements to the morphisms represented in Figure~\ref{defn-gamma/fig}, meaning
\[
 \Gamma(\ribelH_+) = \ribmorH^{-1} \circ \unitH 
 \quad \text{and} \quad 
 \Gamma(\ribelH_-) = \ribmorH \circ \unitH.
\]
\end{proposition}

\begin{figure}[htb]
 \centering
 \includegraphics{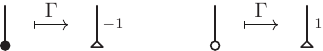}
 \caption{Images under $\Gamma: \AlgH \to \Algt$ of the ribbon element and its inverse.}
 \label{defn-gamma/fig}
\end{figure}

\begin{proof}
Since $\Algt$ and $\AlgH$ are both generated by braided Hopf algebras with braided cocommutative left actions (see Proposition~\ref{intertwining/thm}), it is enough to show that the defining ribbon axioms of $\AlgH$ in Table~\ref{table-Habiro/fig} are satisfied in $\Algt$, once each elementary morphism has been replaced by its image under $\Gamma$. Indeed, ribbon axioms~\hrel{h1} to \hrel{h6} follow directly from axioms~\hrel{r1} to \hrel{r7} in Table~\ref{table-BPHopf/fig}, while \hrel{h7} coincides with \hrel{f2} in Table~\ref{table-BPHopf3-prop/fig}.
The proof of \hrel{h8} is presented in Figure~\ref{proof-h10/fig}. 
\end{proof}

\begin{figure}[htb]
 \centering
 \includegraphics{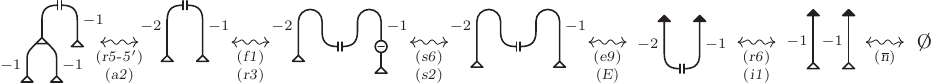}
 \caption{Proof of \hrel{h8}.}
 \label{proof-h10/fig}
\end{figure}


In order to prove that $\Gamma$ is an equivalence of categories, we need some preliminary results. 

\begin{lemma}\label{Hadj2/thm} 
Identities~\hrel{h0'}, \hrel{d10-10'} and \hrel{d11-11'} in Table~\ref{table-cocommutative/fig} are satisfied in $\AlgH$. In particular, the left and right adjoint actions intertwine all morphisms in $\AlgH$.
\end{lemma}

\begin{figure}[htb]
 \centering
 \includegraphics{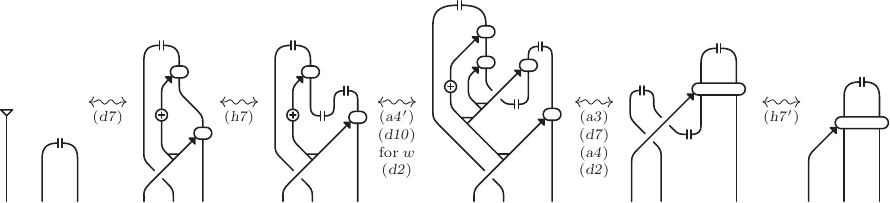}
 \caption{Proof of \hrel{d10} for $\pairH$.}
 \label{ad-inv-pairing/fig}
\end{figure}

\begin{proof} 
Since the left adjoint action in $\AlgH$ is braided cocommutative, Lemma~\ref{cocommutative/thm} implies that relations~\hrel{h0'}, \hrel{d10-10'} for $F=\braid, \prodH,\unitH,\coprH,\counH,\antipH$ and \hrel{d11-11'} hold in $\AlgH$. On the other hand, \hrel{d10} for $\ribelH_\pm$ follows directly from axioms~\hrel{h1}, \hrel{h2}, \hrel{a1} and \hrel{s1'}. Moreover, axiom~\hrel{h5} implies that \hrel{d10} for $F=\copairH$ follows from \hrel{d10} for $\ribelH_\pm$, $\coprH$, and $\prodH$. Finally, as it is shown in Figure~\ref{ad-inv-pairing/fig}, \hrel{d10} for $F = \pairH$ follows from \hrel{d10} for $\copairH$ and from \hrel[h7]{h7-7'}.
\end{proof}

\begin{lemma}\label{Hadj3/thm}
Identities~\hrel{d12-12'}, \hrel{d13-13'}, \hrel{d15-15'} and \hrel{d16-16'} in Table~\ref{table-adjoint-prop/fig} hold in $\AlgH$.
\end{lemma}

\begin{proof}
The proofs of \hrel{d12} and \hrel{d13} are presented in Figures~\ref{proof-r8-alg/fig} and \ref{proof-r9-alg/fig}, while \hrel{d15} and \hrel{d16} follow from \hrel{d10} and \hrel{d12}, as shown in Figure~\ref{adjoint07/fig} (see Remark~\ref{adjointprop/rem}). Then, the symmetric relations follow from the ones above by applying the functor $\sym$.
\end{proof}

\begin{figure}[ht]
 \centering
 \includegraphics{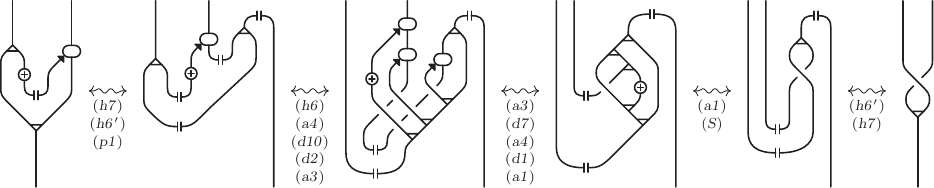}
 \caption{Proof of \hrel{d12} in $\AlgH$.}
 \label{proof-r8-alg/fig}
\end{figure}

\begin{figure}[ht]
 \centering
 \includegraphics{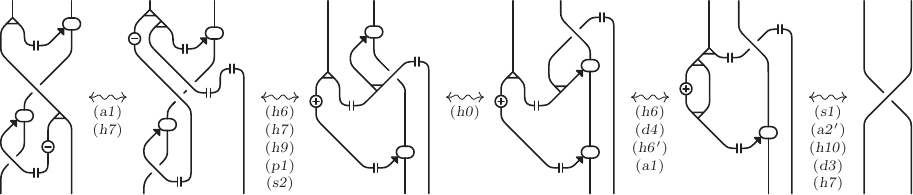}
 \caption{Proof of \hrel{d13} in $\AlgH$.}
 \label{proof-r9-alg/fig}
\end{figure}

Recall that $\Algf$, and hence $\Algt$, are ribbon categories whose evaluation and coevaluation are constructed using the integral form and element. On the other hand, due to relations~\hrel[h7]{h7-7'}, $\AlgH$ admits another rigid structure, with evaluation and coevaluation given by the Hopf pairing and copairing. In the next proposition we will show that such rigid structure can be extended to a ribbon structure, and we will use the notation $(\_)^\vee$ to denote the corresponding duality functor defined in Proposition \ref{Prob-ribbon/thm}.

\begin{proposition}\label{autonomous_Habiro/thm}
$\AlgH$ is a ribbon category (see Definition~\ref{rigid-cat/def}), with dual $(H^n)^\vee = H^n$ for every $n \geqs 0$, with two-sided evaluation $\pairH_n : H^n \otimes H^n \to \one$ and coevaluation $\copairH_n : \one \to H^n \otimes H^n$ inductively defined by $\pairH_0 = \idone = \copairH_0$, by
\begin{gather*}
 \pairH_1 = \pairH, \\
 \copairH_1 = \copairH,
\end{gather*}
and by
\begin{gather*}
 \pairH_n = \pairH \circ (\id \otimes \pairH_{n-1} \otimes \id), \\
 \copairH_n = (\id \otimes \copairH_{n-1} \otimes \id) \circ \copairH,
\end{gather*}
for every $n>1$, and with twist $\vartheta_n : H^n \to H^n$ is defined by
\[
\vartheta_n = (\pairH_n \otimes \id_n) \circ (\id_n \otimes \braid_{n,n}) \circ (\copairH_n \otimes \id_n)
\]
for every $n \geqs 0$. Moreover, in $\AlgH$ we have
\[
\prodH^\vee = \coprH, \quad 
\unitH^\vee = \counH, \quad
\antipH^\vee = \antipH, \quad
\copairH^\vee = \pairH.
\]
\end{proposition}

\begin{proof}
Observe that $\antipH^\vee = (\id \otimes \pairH) \circ (\id \otimes \antipH \otimes \id) \circ (\copairH \otimes \id)$ is equal to $\antipH$, thanks to \hrel{p1} in Table~\ref{table-BPHopf-prop21/fig} and \hrel{h7} which, together with \hrel{h9} in Table~\ref{table-Habiro-prop/fig}, imply that $(\theta_H)^\vee=\theta_{H^\vee}$. Then, the statement concerning the ribbon structure follows from relations~\hrel[h7]{h7-7'} and Proposition \ref{Prob-ribbon/thm}. The identities concerning duals of $\prodH$, $\unitH$, and $\copairH$ follow directly from relations~\hrel[h6]{h6-6'} and \hrel[h7]{h7-7'} in Table~\ref{table-Habiro/fig} and \hrel{h10-10'}\break in Table~\ref{table-Habiro-prop/fig}.
\end{proof}


Proposition \ref{autonomous_Habiro/thm} implies that, if a morphism $F$ is a composition of tensor products of structure morphisms other than ribbon elements, then the diagram representing $F^\vee$ is obtained from the diagram representing $F$ by rotation of an angle $\pi$ with respect to an axis that is perpendicular to the plane. Moreover, by dualizing each side of a given relation between morphisms of $\AlgH$, we obtain another relation between the corresponding dual morphisms, to which we will refer as the \textit{dual} relation, or property. For example, the dual of relation~\hrel{p1} states that $\pairH \circ (\antipH \otimes \id) = \pairH \circ (\id \otimes \antipH)$, and we will refer to it as \rel{p1{$\kern1pt{}^\vee$}}.

The following result is due to Habiro.

\begin{proposition}[Habiro] \label{H/thm} 
In $\AlgH$, the morphisms $\intfH = \pairH \circ (\prodH \otimes \id) \circ (\id \otimes \ribelH_+ \otimes \ribelH_+)$ and $\inteH = \intfH^\vee$ are an $\antipH$-invariant integral form and integral element satisfying relations~\hrel{i1} to \hrel{i5} in Table~\ref{table-BPHopf/fig}.
\end{proposition}

\begin{table}[htb]
 \centering
 \includegraphics{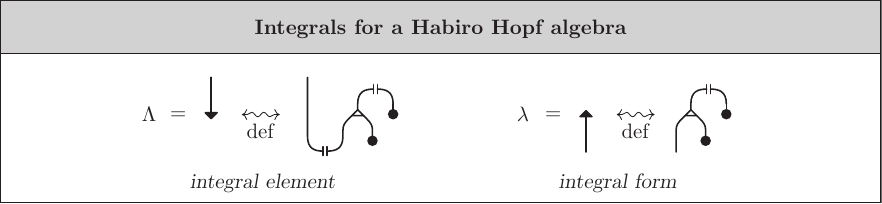}
 \caption{}
 \label{table-Habiro-int/fig}
\end{table}

\begin{proof}
Observe that axioms~\hrel{h1} and \hrel{h4} and relations~\hrel{s4} and \hrel[p1]{p1{$\kern1pt{}^\vee$}} imply that the integral form is $\antipH$-invariant, meaning that $\intfH \circ \antipH = \intfH$. Therefore, if we show that $\intfH$ is a right integral form, meaning that it satisfies relations~\hrel{i1} and \hrel{i5}, then, by considering the dual relation, we will get that $\inteH = \intfH^\vee$ is an $\antipH$-invariant integral element. 

The proof that $\intfH$ is a right integral form is shown in Figure~\ref{cointegral-alg/fig}. 

Finally, the relation $\intfH \circ \inteH = \idone$ is proved in Figure~\ref{normalization/fig}.
\end{proof}

\begin{figure}[htb]
 \centering
 \includegraphics{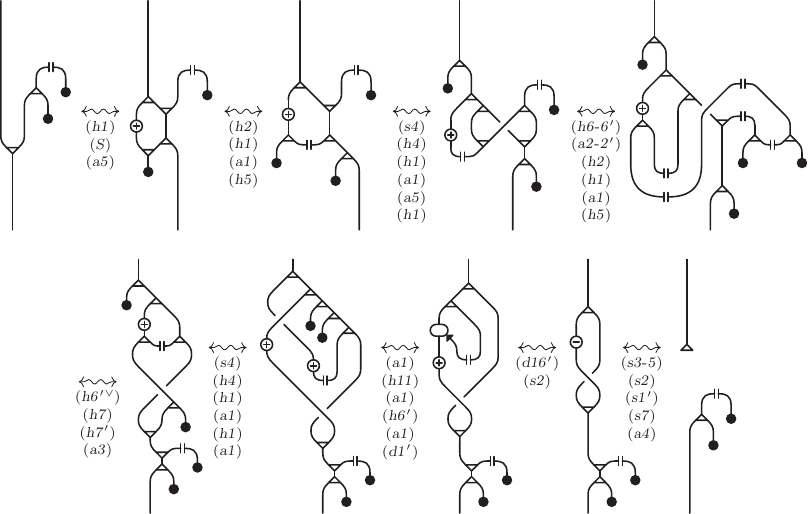}
 \caption{Proof of \hrel{i1} in $\AlgH$.}
 \label{cointegral-alg/fig}
\end{figure}



\begin{figure}[htb]
 \centering
 \includegraphics{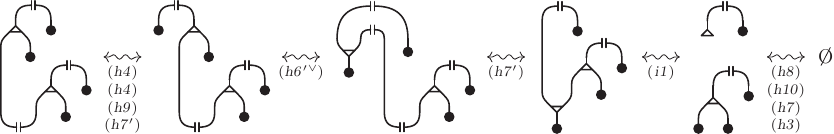}
 \caption{Proof of \hrel{i3} in $\AlgH$, that is, $\intfH \circ \inteH = \idone$.}
 \label{normalization/fig}
\end{figure}

\begin{theorem}\label{3equivalence/thm}
The braided monoidal functor $\Gamma: \AlgH \to \Algt$ is an equivalence of categories.
\end{theorem}

\begin{proof}
We recall that, modulo the other axioms of $\Algf$, \hrel{r8} and \hrel{r9} are equivalent to \hrel{d12} and \hrel{d13} in Table~\ref{table-adjoint-prop/fig}, respectively. Therefore, the quotient $\Algt$ is equivalent to the category freely generated by the elementary morphisms and relations presented in Tables~\ref{table-Hopf/fig}, \ref{table-BPHopf/fig} and \ref{table-BPHopf3/fig}, where axioms~\hrel{r8} and \hrel{r9} have been replaced by \hrel{d12} and \hrel{d13} in Table~\ref{table-adjoint-prop/fig}.

We define now a braided monoidal functor $\barGamma: \Algt\to\AlgH$ that sends all elementary morphisms of $\Algt$ to the corresponding morphisms of $\AlgH$. In order to see that the functor is well-defined, we have to check that all axioms of $\Algt$ are satisfied in the image of $\barGamma$. For the integral axioms and for relations~\hrel{d12} and \hrel{d13}, this follows from Proposition~\ref{H/thm} and from Lemma~\ref{Hadj3/thm}, respectively. The ribbon axioms~\hrel{r1} to \hrel{r7} are equivalent to \hrel{h1} to \hrel{h6}, while axiom~\hrel{n} in Table~\ref{table-BPHopf3/fig} follows from \hrel{h2}, \hrel{h3}, and \hrel{h10}. Now, we only need to observe that $\barGamma \circ \Gamma = \id_{\AlgH}$ and $\Gamma \circ \barGamma = \id_{\Algt}$.
\end{proof}

Let us finish this subsection by introducing yet another equivalent presentation of $\Algt$. More precisely, we will show that, by adding the braided cocommutativity relation for the adjoint action to Kerler's original list of axioms, we obtain a category that is equivalent to $\Algt$.

\begin{definition}\label{AlgK/defn}
Let $\calC$ be a braided monoidal category with tensor product $\otimes$, tensor unit $\one$, and braiding $\braid$. A \textit{Kerler Hopf algebra} is a Hopf algebra $H$ in $\calC$ with braided cocommutative left adjoint action, equipped with the following structure morphisms:
\begin{itemize}
\item an \textit{integral form} $\intfH : H \to \one$ and an \textit{integral element} $\inteH : \one \to H$;
\item a \textit{copairing} $\copairH : \one \to H \otimes H$;
\item a \textit{ribbon element} $\ribelH_+ : \one \to H$ and its multiplicative inverse $\ribelH_- : \one \to H$.
\end{itemize}
These structure morphisms are subject to the following axioms:
\begin{gather*}
 (\id \otimes \intfH) \circ \coprH = \unitH \circ \intfH,
 \tag*{\hrel[i1K]{i1}}
 \\
 \prodH \circ (\inteH \otimes \id) = \inteH \circ \counH,
 \tag*{\hrel[i2K]{i2}}
 \\
 \intfH \circ \inteH = \idone,
 \tag*{\hrel[i3K]{i3}}
 \\
 \antipH \circ \inteH = \inteH,
 \tag*{\hrel[i4K]{i4}}
 \\
 \intfH \circ \antipH = \intfH,
 \tag*{\hrel[i5K]{i5}}
 \\
 \prodH \circ (\ribelH_+ \otimes \id) = \prodH \circ (\id \otimes \ribelH_+),
 \tag*{\hrel[h1K]{h1}}
 \\
 \prodH \circ (\ribelH_+ \otimes \ribelH_-) = \unitH, 
 \tag*{\hrel[h2K]{h2}}
 \\
 \counH \circ \ribelH_+ = \idone, 
 \tag*{\hrel[h3K]{h3}}
 \\
 \antipH \circ \ribelH_+ = \ribelH_+, 
 \tag*{\hrel[h4K]{h4}}
 \\
 \copairH=(\prodH \otimes \prodH) \circ (\ribelH_- \otimes \id_2 \otimes \ribelH_-)\circ \coprH \circ \ribelH_+, 
 \tag*{\hrel[h5K]{h5}}
 \\
 (\id \otimes \coprH) \circ \copairH = (\prodH \otimes \id_2) \circ (\id \otimes \copairH \otimes \id) \circ \copairH, 
 \tag*{\hrel[h6K]{h6}}
 \\
 (\intfH \otimes \id) \circ \copairH = \inteH, 
 \tag*{\hrel[fK]{f}}
 \\
 \intfH \circ \ribelH_+ = \idone. 
 \tag*{\hrel[nK]{n}}
\end{gather*}
We denote by $\AlgK$ the braided monoidal category freely generated by a Kerler Hopf algebra $H$. 
\end{definition}

A diagrammatic representation of the generators and the axioms of a Kerler Hopf algebra are presented in Table~\ref{table-Kerler/fig}.

\begin{table}[htb]
 \centering
 \includegraphics{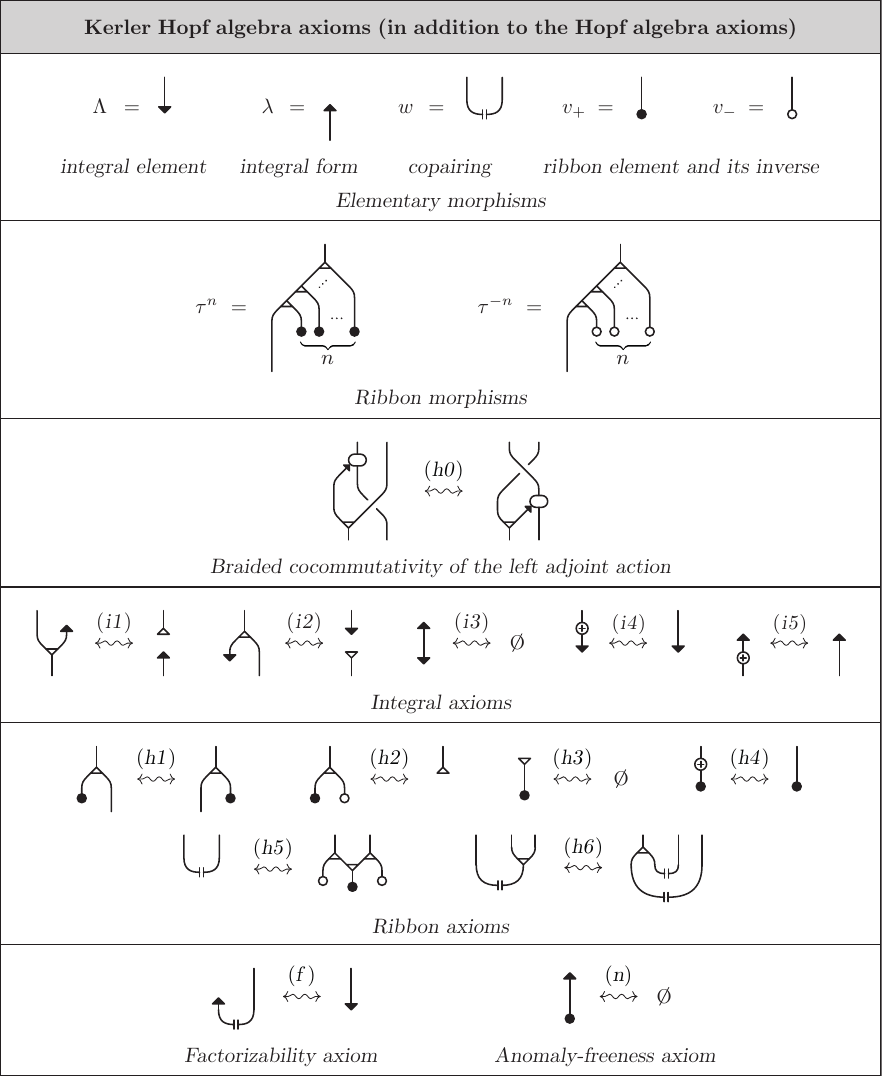}
 \caption{}
 \label{table-Kerler/fig}
 \label{E:h0K} \label{E:i1K} \label{E:i2K} \label{E:i3K} \label{E:i4K} \label{E:i5K} \label{E:h1K} \label{E:h2K} \label{E:h3K} \label{E:h4K} \label{E:h5K} \label{E:h6K} \label{E:fK} \label{E:nK} 
\end{table}

\begin{proposition}\label{symmetryK/thm}
 The functor $\sym : \AlgK \to \AlgK$ is an involutive anti-monoidal equivalence functor that sends every object and every elementary morphism to itself. Moreover, the properties in Tables~\ref{table-BPHopf-prop1/fig}, \ref{table-BPHopf-prop21/fig}, and \ref{table-BPHopf-prop22/fig} hold in $\AlgK$. 
\end{proposition}

\begin{proof}
Since $\AlgK$ has the same elementary morphisms as $\Algf$, the statement regarding the symmetry functor follows from Propositions~\ref{funt-sym/thm} and \ref{symmetry/thm}. 

Furthemore, according to Propositions~\ref{BP-prop1/thm}, \ref{BP-prop21/thm}, and \ref{BP-prop22/thm}, the properties in Tables~\ref{table-BPHopf-prop1/fig}, \ref{table-BPHopf-prop21/fig}, and \ref{table-BPHopf-prop22/fig} hold in any category in which the integral axioms \hrel{i1}--\hrel{i5} and ribbon axioms~\hrel{r1}--\hrel{r7} hold. Now, it is enough to observe that the copairing and the ribbon morphisms of $\AlgK$, defined in Table~\ref{table-Kerler/fig}, satisfy axioms~\hrel{r1}--\hrel{r7} in Table~\ref{table-BPHopf/fig}.
\end{proof}

\begin{corollary}\label{3equivalence/cor}
 There exists a braided monoidal equivalence between the categories $\AlgK$ and $\AlgH$ that preserves the corresponding Hopf algebra structures. 
\end{corollary}

\begin{proof} 
According to Proposition~\ref{symmetryK/thm}, we can define evaluation and coevaluation morphisms by relations \hrel{e1-2} in Table~\ref{table-BPHopf-prop1/fig} and a pairing by relation \hrel{f1} in Table \ref{table-BPHopf3-prop/fig}. Then, it is easy to see that relations~\hrel{h7} and \hrel{h8}  in Table~\ref{table-Habiro/fig} and \hrel[n']{\bn} in Table~\ref{table-BPHopf3-prop/fig} hold in $\AlgK$ as well. Indeed, the proofs of \hrel{h7}, \hrel{h8} and \hrel[n']{\bn}, which are presented in Figures~\ref{proof-d2-app/fig}, \ref{proof-h10/fig} and \ref{proof-nbar-app/fig}, respectively, only use axioms and relations that are satisfied in $\AlgK$. Therefore, there exists a well-defined braided monoidal functor from $\AlgH$ to $\AlgK$ sending the elementary morphisms of $\AlgH$ to the corresponding morphisms of $\AlgK$.

The inverse functor from $\AlgK $ to $\AlgH$ is defined by sending the integral form and element in $\AlgK$ to the ones shown in Table~\ref{table-Habiro-int/fig} in $\AlgH$. Then, the only non-trivial relations to be checked are the integral relations, which are satisfied by Proposition~\ref{H/thm}.
\end{proof}

Notice that, in his original definition \cite{Ke01}, Kerler used the non-degeneracy of the copairing instead of the integral axioms, but he also showed that these axioms are interchangeable.

\FloatBarrier

%% file: S4-topology.tex
\section{Topological categories}
\label{topology/sec}

\subsection{The category \texorpdfstring{$\KT$}{KT} of Kirby tangles}
\label{KT/sec}

Our main object of interest will be the category of oriented \dmnsnl{4} relative \hndlbds{2} (see Subsection~\ref{4HB/sec}) modulo \qvlnc{2}, which is an equivalence relation generated by slides and cancellations of $1$-handles and $2$-handles. Following \cite{Ki89,GS99,BP11}, morphisms in this category will be described in terms of a particular class of tangles, called admissible Kirby tangles (compare with Definition~\ref{kirby-admtangle/def} below), considered up to \dfrmtns{2}, which implement the above handle moves. In this section, we will discuss the general notion of Kirby tangle, which will be further restricted in Subsection~\ref{4KT/sec} to the notion of admissible Kirby tangle, in order to represent \dmnsnl{4} \hndlbds{2}. 

We start by fixing the following notation. For any integer $k \geqs 0$, we set
\[
E_k := \{ e_{k,1}, e_{k,2}, \dots, e_{k,k} \} \subset [0,1]^2,
\]
with the $k$ points $e_{k,i}$ uniformly distributed along 
$\left] 0,1 \right[ \times \{ 1/2 \}$. 
In particular, $E_0 = \emptyset$.

\begin{definition}
\label{kirby-tangle/def}
Given two integers $k,\ell \geqs 0$ such that $k + \ell$ is even, a \textsl{Kirby tangle} from $E_k$ to $E_\ell$ consists of the following data:
\begin{itemize}
\item[\(a)]
a collection of $m \geqs 0$ dotted unknots $U_1, U_2, \dots, U_m$, together with disjoint flat spanning disks $D_1, D_2, \dots, D_m$ embedded into $\left] 0,1 \right[^3$;
\item[\(b)]
an undotted tangle properly and smoothly embedded into $[0,1]^3$, which is transverse to the spanning disks, and which consists of a link $L = L_1 \cup L_2 \cup \dots \cup L_n$ formed by $n \geqs 0$ closed components, and of $(k+\ell)/2$ arcs whose endpoints belong to $(E_k \times \{0\}) \cup (E_\ell \times \{1\})$, all endowed with the blackboard framing with respect to the projection $[0,1]^3 \to [0,1]^2$ that forgets the second coordinate.
\end{itemize}
\end{definition}

\begin{definition}
\label{kt-equivalence/def}
Two Kirby tangles are said to be \textsl{\qvlnt{2}} if they are related by a finite sequence of the following operations, called \textsl{\dfrmtns{2}}:
\begin{itemize}
\item[\(a)]
performing an ambient isotopy of the Kirby tangle in $[0,1]^3$ that fixes the boundary\footnote{Notice that the isotopy can move both undotted and dotted components together with their spanning disks while preserving all their intersections.};
\item[\(b)]
pushing an arc of any undotted (possibly open) component $C$ through the disk $D$ spanned by any dotted unknot $U$ in such a way that two opposite transverse intersection points between $C$ and $D$ appear/disappear;
\item[\(c)]
adding/deleting a dotted unknot $U$ and an undotted closed component $C$ such that the disk $D$ spanned by $U$ is pierced only once by $C$ and by no other undotted component;
\item[\(d)]
sliding any (possibly open) undotted component $C$ over any different closed one $C'$, that is, replacing $C$ by a (blackboard parallel) band connected sum of itself with a parallel copy of $C'$.
\end{itemize}
\end{definition}

Next, \qvlnc{2} classes of Kirby tangles can be organized as the morphisms of a monoidal category, as specified by the following definition.

\begin{definition}
\label{KT/def}
We denote by $\KT$ the monoidal category whose objects are the sets $E_k$ for $k \geqs 0$, and whose morphisms from $E_k$ to $E_\ell$ are \qvlnc{2} classes of Kirby tangles from $E_k$ to $E_\ell$. 

The composition $T' \circ T$ of two morphisms $T : E_k \to E_\ell$ and $T' : E_{k'} \to E_{\ell'}$ with $\ell = k'$ is given by vertical juxtaposition, with $T'$ on top of $T$, and by rescaling the third coordinate of a factor $1/2$.

The tensor product, denoted $\sqcup$, is given by horizontal juxtaposition, followed by a suitable reparameterization of the first coordinate, in such a way that
\[
 E_k \sqcup E_{k'} = E_{k + k'}
\] 
on the level of objects. For the tensor product of two morphisms $T : E_k \to E_\ell$ and $T' : E_{k'} \to E_{\ell'}$, the reparameterization of the first coordinate depends on the third one, in order to simultaneously realize the above equality at both the source and the target level, and to get in this way a Kirby tangle from $E_{k + k'}$ to $E_{\ell + \ell'}$ representing $T \sqcup T'$.

For each $k \geqs 0$, the identity $ \id_{E_k}$ is represented by the product $E_k \times [0,1]$, interpreted as a Kirby tangle consisting of $k$ undotted arcs. In particular, the empty Kirby tangle represents $\idone$, since $\one = E_0 = \emptyset$.
\end{definition}

\begin{table}[b]
 \includegraphics{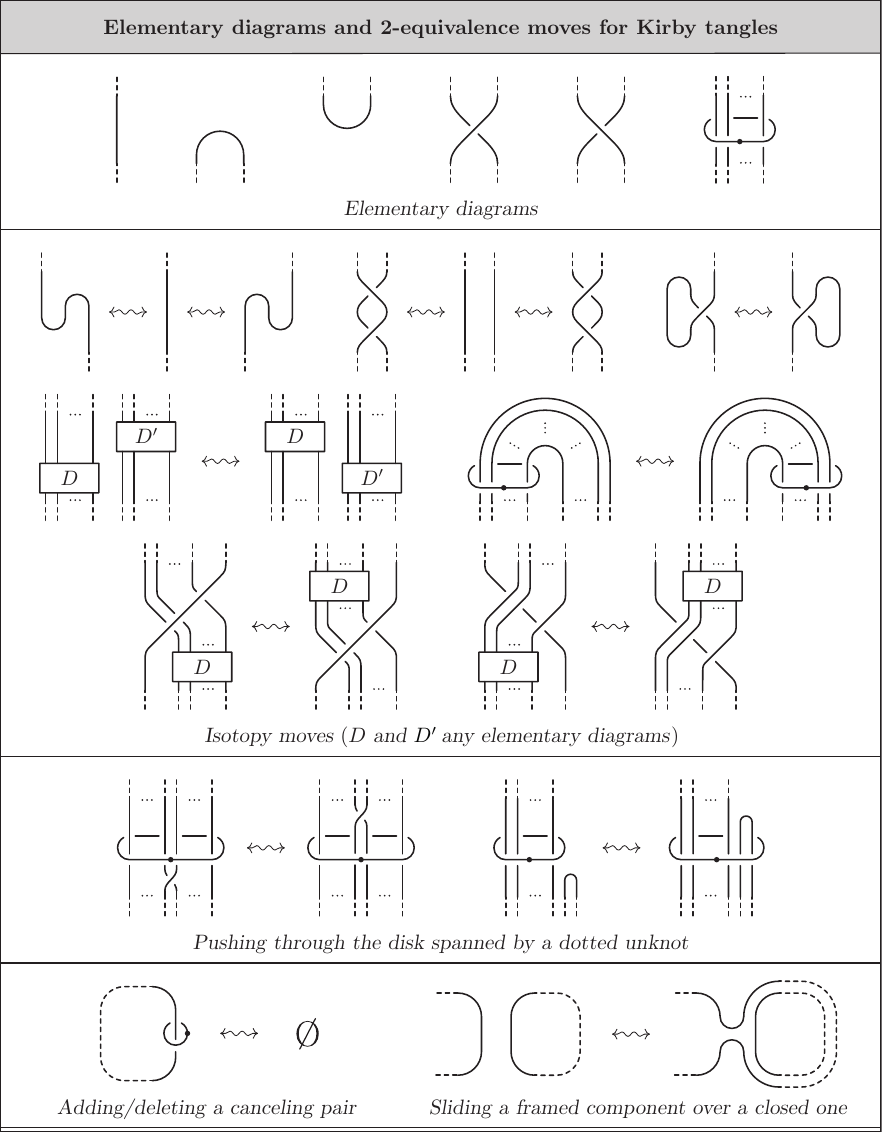}
 \caption{}
 \label{table-Ktangles/fig}
\end{table}

Kirby tangles live in 
$\left] 0,1 \right[^2 \times [0,1]$ 
and will be always represented through their planar diagrams by the projection to the square $\left] 0,1 \right[ \times [0,1]$ that forgets the second coordinate, in such a way that the factor 
$\left] 0,1 \right[^2$ 
projects to $\left] 0,1 \right[$. As usual, we require that the restriction of the projection to the tangle, including both dotted and undotted components, is regular, and that it is injective except for a finite number of transverse double points, which give rise to crossings. We will use the same letter to denote both a Kirby tangle and its plane projection.

In the following, we will need to consider particular planar diagrams of Kirby tangles whose projection satisfies an extra regularity property, as specified by the next definition.

\begin{definition}
\label{strictly-regular/def}
Given a Kirby tangle $T$ as in Definition~\ref{kirby-tangle/def}, we say a planar diagram of $T$ is \textsl{strictly regular} if the disks $D_1, \dots, D_m$ spanned by the dotted unknots project bijectively onto disjoint planar disks, and if the projection of the undotted tangle intersects each of such disks as presented on the top right figure in Table~\ref{table-Ktangles/fig}. 
\end{definition}

All the planar diagrams we have drawn until now are strictly regular, but using strictly regular diagrams to represent admissible Kirby tangles sometimes makes pictures quite heavy. In the following, when this will not cause confusion, we will often draw planar diagrams that are not strictly regular. However, we will always keep the condition that the disks $D_1, \dots, D_m$ project bijectively onto disjoint planar disks.

The next proposition provides a presentation of the monoidal category $\KT$ in terms of the generators and relations represented in Table~\ref{table-Ktangles/fig}. Here, the isotopy moves correspond to those ambient isotopies of Kirby tangles in $[0,1]^3$ that preserve the intersections between the undotted components and the disks spanned by the dotted unknots in the standard form shown as the rightmost elementary diagram, while the pushing-through moves are needed to relax this last condition. On the other hand, the diagram operations on the bottom correspond to operations \(c) and \(d) in Definition~\ref{kt-equivalence/def}.

\begin{proposition}
\label{K-pres/thm} 
Up to ambient isotopy in $[0,1]^3$, any Kirby tangle $T \in \KT$ can be expressed as a composition of tensor products of the elementary diagrams in Table~\ref{table-Ktangles/fig} that yields a strictly regular planar diagram of $T$. Moreover, any two strictly regular planar diagrams expressed in this way represent\break \qvlnt{2} Kirby tangles if and only if, up to planar isotopy preserving the expression as composition of tensor products, they are related by a finite sequence of the isotopy moves and the diagram operations in the same Table~\ref{table-Ktangles/fig}. 
\end{proposition}

\begin{proof}
The first part of the statement concerning generators immediately follows from a standard trans\-versality argument. On the other hand, all the moves and operations in Table~\ref{table-Ktangles/fig} clearly represent \dfrmtns{2} of Kirby tangles, so we only need to prove that they are sufficient to realize any \qvlnc{2}. Since operations \(c) and \(d) in Definition~\ref{kt-equivalence/def} correspond to the last two moves in Table~\ref{table-Ktangles/fig}, we are left to prove that the remaining moves in that table can generate any isotopy of Kirby tangles. 

Modulo the pushing-through moves in Table~\ref{table-Ktangles/fig}, we can assume that, during the isotopy, the disks spanned by the dotted unknots are rigidly moved in space, and that the intersections between the disks spanned by the dotted unknots and the undotted components are preserved, as in point \(a) of Definition~\ref{kt-equivalence/def}. Actually, this would require also the move where a cup is pushed through the disk spanned by a dotted unknot from above, but up to the isotopy moves this is equivalent to the second pushing-through move in Table~\ref{table-Ktangles/fig}, where a cap is pushed through that disk from below. 

Furthermore, modulo the isotopy move where a dotted component passes from one side of a multiple cap to the other, we can also assume that, at the end of the isotopy, each disk spanned by a dotted unknot is sent into its image in such a way that the orientation induced by the plane projection of the diagram is preserved.

These assumptions allow us to consider the last elementary diagram in Table~\ref{table-Ktangles/fig} as a coupon with the same number of incoming and outgoing edges, and hence to apply \cite[Chapter~I, Lemma~3.4]{Tu94}. Then, it is enough to observe that the relations in that lemma can be generated by the moves in Table~\ref{table-Ktangles/fig}.
\end{proof}

We conclude this subsection with a simple proposition, which reduces the slide operation in Table~\ref{table-Ktangles/fig} to a special case. This will be useful to prove our main theorem.

\begin{proposition}
\label{2-equiv-special/thm}
In a Kirby tangle, any slide of a (possibly open) undotted component over a closed one can be realized, up to isotopy and addition/deletion of canceling pairs, by a sequence of slides over undotted components that form at most one self-crossing, and hence are unknots with framing $0$ or $\pm 1$.
\end{proposition}

\begin{figure}[b]
 \includegraphics{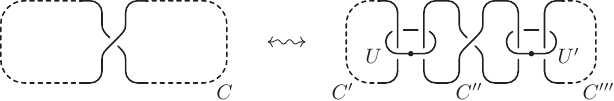}
 \caption{Proof of Proposition~\ref{2-equiv-special/thm}.}
 \label{cutting-proof/fig}
\end{figure}

\begin{proof}
Consider a slide over a closed undotted component $C$, and proceed by induction on the number $c \geqs 0$ of self-crossings of $C$. If $c \leqs 1$, there is nothing to prove. So assume $c > 1$, and look at any self-crossing of $C$. Here, we modify the diagram as indicated in Figure~\ref{cutting-proof/fig}. The original diagram on the left-hand side can be obtained from the one on the right-hand side by sliding $C'$, first over $C''$ and then over $C'''$, and by deleting in sequence two canceling $1/2$-pairs. Since both $C''$ and $C'''$ have less than $c$ self-crossings, we can realize the slides over them by using the inductive hypothesis. To complete the proof, it is enough to observe that, up to the modification in the figure, a slide over $C$ is the same as a sequence of slides over $C'$, $C''$, and $C'''$, and hence we can use the inductive hypothesis once again, since also $C'$ has less than $c$ self-crossing.
\end{proof}

\subsection{The category \texorpdfstring{$\KTf$}{KT} of admissible Kirby tangles}
\label{4KT/sec}

We introduce admissible Kirby tangles, which, as it will be shown in Subsection~\ref{4HB/sec}, are exactly the Kirby tangles that actually represent relative 4-dimensional \hndlbds{2}.

\begin{definition}
\label{kirby-admtangle/def}
A Kirby tangle from $E_k$ to $E_\ell$ as in Definition~\ref{kirby-tangle/def} is said to be \textsl{admissible} if the following properties hold:
\begin{itemize}
\item[\(a)] both $k$ and $\ell$ are even, say $k = 2s$ and $\ell = 2t$;
\item[\(b)] the open components of the undotted tangle consist of $s$ arcs $A_{1,0}, A_{2,0}, \dots, A_{s,0}$ such that the endpoints of $A_{i,0}$ are $(e_{2s,2i-1},0)$ and $(e_{2s,2i},0)$ in $E_{2s} \times \{0\}$ for each $i=1, 2, \dots, s$, and $t$ arcs $A_{1,1}, A_{2,1}, \dots, A_{t,1}$ such that the endpoints of $A_{j,1}$ are $(e_{2t,2j-1},1)$ and $(e_{2t,2j},1)$ in $E_{2t} \times \{1\}$ for each $j=1, 2, \dots, t$.
\end{itemize}
In particular, in an admissible Kirby tangle, no undotted arc connects a point at level 0 to one at level 1 (compare with \cite{MP92, KL01}). Two admissible Kirby tangles are 2-equivalent if they are 2-equivalent as Kirby tangles.
\end{definition}

\begin{figure}[htb]
 \includegraphics{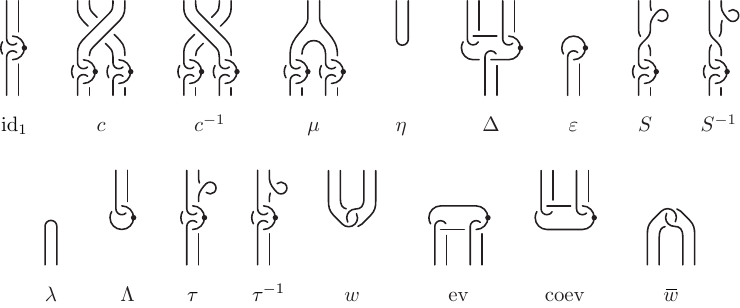}
 \caption{Some relevant morphisms in $\KTf$.}
 \label{KT-morph/fig}
\end{figure}

\begin{proposition}
\label{4KT/thm}
\qvlnc{2} classes of admissible Kirby tangles form a category $\KTf$ whose objects are the sets $E_{2s}$ with $s \geqs 0$, and whose morphisms from $E_{2s}$ to $E_{2t}$ are \qvlnc{2} classes of admissible Kirby tangles. The composition in $\KTf$ is induced by the one in $\KT$ (see Definition~\ref{KT/def}), while the identity morphism $\id_s$ of $E_{2s}$ is defined inductively as follows: $\id_0$ is the empty diagram, $\id_1$ is the first diagram in Figure~\ref{KT-morph/fig}, and $\id_s = \id_1 \sqcup \id_{s-1}$ for any $s > 1$. 

$\KTf$ has a braided monoidal structure whose tensor product is induced by the one of $\KT$ (see Definition~\ref{KT/def}), and whose tensor unit is $\one = E_0$; the elementary braidings $c_{1,1} = c : E_4\to E_4$ and $c_{1,1}^{-1} = c^{-1} : E_4\to E_4$ are presented in the second and third diagrams in Figure~\ref{KT-morph/fig}, while $c_{s,s'} : E_{2(s+s')} \to E_{2(s+s')}$ for $s+s'>2$ is obtained inductively using the relations in Definition~\ref{braided-cat/def} (see Table~\ref{table-braided/fig}).
\end{proposition}


\begin{proof}
We only need to show that $\id_s$ are indeed identity morphisms. In other words, for any admissible Kirby tangle $T$ from $E_{2s}$ to $E_{2t}$, both $T \circ \id_s$ and $\id_t \circ T $ are 2-equivalent to $T$. To see this, it is enough to observe that in $T \circ \id_s$ the upper undotted components of $\id_s$ get closed and we can slide the lower open components over the closed ones and then cancel them with the dotted components; a symmetric argument works for the top part of ${\id_t} \circ T$. 
\end{proof}

\begin{remark}\label{KTbis/rmk}
Since 2-equivalence preserves admissibility, morphisms from $E_{2s}$ to $E_{2t}$ in $\KTf$ form a subset of the set of morphisms with same source and target in $\KT$. Therefore, we have a set-theoretic inclusion of $\KTf$ in $\KT$ at the level of both objects and morphisms, and this inclusion respects compositions and products. However, the identity of $E_{2s}$ in $\KT$ is not represented by an admissible Kirby tangle for $s > 0$, and hence it is not a morphism of $\KTf$, so $\KTf$ is not a subcategory of $\KT$. Indeed, $\KTf$ is rather the full monoidal subcategory of the Karoubi envelope of $\KT$ generated by the idempotent morphism $\id_1$.
\end{remark}


\subsection{The category \texorpdfstring{$\RHB$}{4HB} of relative 4-dimensional 2-handlebodies}
\label{4HB/sec}

We review now the notion of an oriented \dmnsnl{4} relative \hndlbd{2} built over a connected \mnfld{3} with (possibly empty) boundary. For the basic definitions about handle decompositions we refer to \cite{GS99}, while a detailed discussion of the specific topic mentioned above can be found in \cite[Subsections~2.1 \& 2.2]{BP11}, where the notion is actually considered in the more general context of multiple $0$-handles.

\begin{definition}
\label{handlebody/def}
Given a compact connected oriented \mnfld{3} $M$ with (possibly empty) boundary, an oriented \dmnsnl{4} {\sl relative \hndlbd{2}} built on $M$ is an oriented smooth \mnfld{4} with a given handle decomposition
\[
W = W_0 \cup_{i=1}^m H^1_i \cup_{j=1}^n H^2_j,
\]
where $W_0 = M \times [0,1] \subset W$ is a smooth collar of $M \times \{0\}$ with product orientation, $W_1 = W_0 \cup_{i=1}^m H^1_i \subset W$ is a smooth submanifold obtained by attaching the 1-handles $H^1_i = B^1 \times B^3$ to the interior of the \textit{front boundary} $\partial_+ W_0 = M \times \{1\}$, and finally $W = W_1 \cup_{j=1}^n H^2_j$ is obtained by attaching the 2-handles $H^2_j = B^2 \times B^2$ to the interior of the \textit{front boundary} $\partial_+ W_1 = \partial W_1 \smallsetminus \partial \big( M \times \left[ 0,1 \right[ \big)$.

By identifying $M$ with $M \times \{ 0 \} \subset W$, we think of it as a smooth submanifold of $\partial W$, and we call the family of handles forming $W$ starting from $M \times [0,1]$ a \textsl{relative \hndlbd{2} decomposition} of the pair $(W,M)$.
\end{definition}

We remark that this definition reduces to the standard one when $M$ is a closed $3$-manifold (compare with \cite[Definition~4.2.1]{GS99}). In particular, for $M \cong S^3$, we can fill $S^3 \cong S^3 \times \{ 0 \}$ with $B^4$ and get in this way the notion of an (absolute) connected oriented \dmnsnl{4} \hndlbd{2}, by thinking of $B^4$ as the starting 0-handle.

For a handlebody decomposition of $(W,M)$ as above, the connectedness of $M$ and the orientability of $W$ imply that there is a unique way to attach the 1-handles, up to ambient isotopy of their attaching balls in $\partial_+ W_0$, which does not change the diffeomophism type of the pair $(W,M)$. On the other hand, the 2-handles can be specified by a framed link in $\partial_+ W_1$, whose $j$th component uniquely determines up to isotopy an embedding $S^1 \times B^2 \to \partial_+ W_1$ giving the attaching map of a single 2-handle $H^2_j$, once it is identified with $B^2 \times B^2$. In this case too, an ambient isotopy in $\partial_+ W_1$ of the framed link representing the 2-handles does not affect the diffeomorphism type of $(W,M)$.

\begin{definition}
\label{2-equivalent/def}
Two oriented \dmnsnl{4} relative \hndlbds{2} $W$ and $W'$ built on the same compact connected oriented 3-manifold $M$ are said to be \textsl{\qvlnt{2}} if the relative \hndlbd{2} decompositions of $(W,M)$ and $(W',M)$ are related by a \textsl{$2$-deformation}, meaning a finite sequence of the following operations:
\begin{itemize}
\item[\(a)]
isotoping the attaching maps of the handles;
\item[\(b)]
adding/deleting a canceling pair consisting of a 1-handle and a 2-handle;
\item[\(c)]
sliding a 2-handle over another one.
\end{itemize}
\end{definition}

It is worth noticing that also the operation of sliding a 1-handle over another one is admitted, as it can be obtained from \(b) and \(c), see for instance \cite[Figure~2.2.11]{BP11}.

We already observed that the operations of type \(a) preserve the diffeomorphism type of the handlebody, and it is easy to see that the same holds for the those of type \(b) and \(c). Hence, if two oriented \dmnsnl{4} relative \hndlbds{2} are \qvlnt{2}, then they are diffeomorphic.

Viceversa, whether diffeomorphic oriented \dmnsnl{4} relative \hndlbds{2} are always \qvlnt{2} is an open question, which is expected to have negative answer (see \cite[Section~I.6]{Ki89} and \cite[Section~5.1]{GS99}). A list of \dmnsnl{4} \hndlbds{2} which are diffeomorphic but conjecturally not \qvlnt{2} can be found in \cite{Go91}.

On the other hand, it is known that homeomophic oriented \dmnsnl{4} relative \hndlbds{2} are not necessarily diffeomorphic. See \cite[Section~9.1]{Ak16} for examples of such exotic handlebodies.

In the following, we will focus on the special case when $M = M_{s,t} \cong M_s \bcsum M_t$ is the boundary connected sum of two (absolute) connected oriented \dmnsnl{3} \hndlbds{1}
\[
M_s \cong H^0 \cup_{i=1}^s H^1_i 
\text{ \ and \ } 
M_t \cong H^0 \cup_{i=1}^t H^1_i,
\]
with $s,t \geqs 0$. We assume that $M_s$ and $M_t$ are canonically realized inside $\R^3$, by identifying $H^0$ with $[0,1]^3$ and attaching each 1-handle $H^1_i$ to $\left] 0,1 \right[^2 \times \{1\}$. So, we can set 
\[
M_{s,t} = (M_s \times \{0\}) \cup
([0,1]^2 \times \{0\} \times [0,1]) \cup (M_t \times \{1\}) \subset \R^4,
\] 
as depicted on the left-hand side of Figure~\ref{cobordism01/fig}. Then, we consider a canonical identification
\[
M_{s,t} \times [0,1] \cong (M_s \times [0,0.1]) \cup ([0,1]^3 \times [0.1,0.9]) \cup (M_t \times [0.9,1]) \subset \R^4
\]
such that $M_{s,t}$ corresponds to $M_{s,t} \times \{0\}$ (see the right-hand side of Figure~\ref{cobordism01/fig}). Notice that, since $M_{s,t}$ is a subset of $\R^4$, the cylinder $M_{s,t} \times [0,1]$ is defined as a subset of $\R^5$. Under the canonical identification represented on the right-hand side of Figure~\ref{cobordism01/fig}, the last coordinate of $M_{s,t} \times [0,1]$ can no longer be interpreted as the height in the picture, but rather as a parametrization of the thickness of the cylinder. In particular, $M_{s,t} \times \{0\}$ corresponds to the union of top, back, and bottom face, while $M_{s,t} \times \{1\}$ corresponds to the intersection between the front face and the strip $\R^3 \times [0.1,0.9]$.

\begin{figure}[htb]
 \centering
 \includegraphics{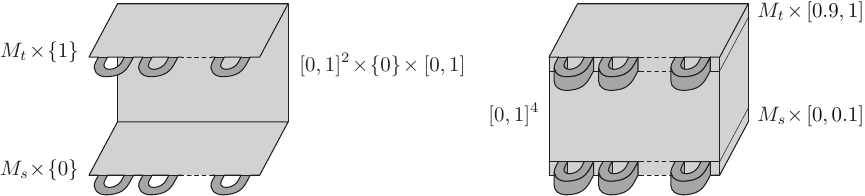}
 \caption{Canonical realization of $M_{s,t}$ and $M_{s,t} \times [0,1]$ in $\R^4$.}
 \label{cobordism01/fig}
\end{figure}

The $2$-equivalence classes of oriented \dmnsnl{4} relative \hndlbds{2} built over the $3$-manifolds $M_{s,t}$ with $s,t \geqs 0$ form a monoidal category $\RHB$.

\begin{definition}
\label{4HB/def}
We denote by $\RHB$ the monoidal category whose objects are connected oriented \dmnsnl{3} \hndlbds{1} $M_s$ for $s \geqs 0$, and whose morphisms from $M_s$ to $M_t$ are $2$-equivalence classes of oriented \dmnsnl{4} relative \hndlbds{2} built on $M_{s,t}$. 

The composition of two morphisms $W = (W,M_{s,t})$ and $W' = (W',M_{s',t'})$ in $\RHB$ with $t = s'$ is obtained by a taking their vertical juxtaposition, with $W'$ on top of $W$, by gluing the two morphisms (identifying canonically the target of the first with the source of the second), and then by rescaling by a factor $1/2$, that is,
\[
W' \circ W \cong (W \cup_{M_t \times \{ 1 \} = M_{s'} \times \{0\}} W', M_{s,t'}),
\]
with $M_{s,t'}$ canonically contained in $M_{s,t} \cup_{M_t \times \{1\} = M_{s'} \times \{0\}} M_{s',t'}$, and with handlebody decomposition consisting of all the handles of $W$ and $W'$ plus the 1-handles deriving from the thickening of $M_t = M_{s'}$.

The tensor product, denoted by $\bcsum$, is given by horizontal juxtaposition, from left to right. For two objects $M_s$ and $M_{s'}$ it corresponds to the boundary connected sum $M_s \bcsum M_{s'}$, which is canonically identified with $M_{s + s'}$, while for two morphisms $W = (W,M_{s,t})$ and $W' = (W',M_{s',t'})$ it corresponds to the boundary connected sum of pairs, that is,
\[
W \bcsum W' \cong (W \bcsum W',M_{s,t} \bcsum M_{s',t'} \cong M_{s + s',t + t'}),
\]
with $M_{s + s',t + t'}$ canonically identified to $M_{s,t} \bcsum M_{s',t'}$, and handlebody decomposition consisting of all the handles of $W$ and $W'$.

For each $s \geqs 0$, the identity $\id_{M_s}$ is represented by the product $M_s \times [0,1]$ with the natural handlebody decomposition. In particular, $\idone = M_0 \times [0,1]$, since $\one = M_0$.
\end{definition}

\begin{remark}\label{4HB-skeleton/rmk}
$\RHB$ is a skeleton of a category whose objects are arbitrary connected oriented \dmnsnl{3} \hndlbds{1} $M$ equipped with an orientation-preserving embedding $f : D^2 \hookrightarrow \partial M$, and whose morphisms are 2-equivalence classes of \dmnsnl{4} \hndlbds{2} $W$ built on $\overline{M} \bcsum M'$, with boundary connected sum performed along the images of $f : D^2 \hookrightarrow \partial M$ and $f' : D^2 \hookrightarrow \partial M'$. It is convenient however to restrict our attention to the standard models for objects we are considering here.
\end{remark}

Finally, we state and sketch the proof of the equivalence of the categories $\RHB$ and $\KTf$. The reader can find all details, illustrated with adequate figures, in \cite[Sections~2.2. \& 2.3]{BP11}.

\begin{proposition}[{\cite[Propositions~2.2.8 \& 2.3.1]{BP11}}]
\label{4KT-equiv/thm}
There is an equivalence  monoidal functor $\cal T:\RHB\to \KTf $ that sends the handlebody $M_{s,t} \times [0,1]$ to the admissible Kirby tangle presented in Figure \ref{cobordism4/fig}.
\end{proposition}

\begin{figure}[htb]
 \centering
 \includegraphics{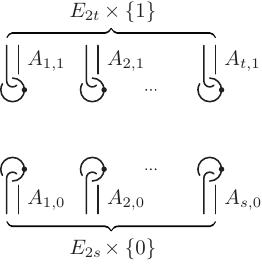}
 \caption{ $\cal T(M_{s,t} \times [0,1])$.}
 \label{cobordism4/fig}
\end{figure}

\begin{proof}
First of all, we observe that $M_{s,t} \times [0,1]$ consists of $B^4=[0,1]^4$ with $s + t$ \dmnsnl{4} 1-handles attached to it, $s$ on the bottom part of the front face $[0,1]^2 \times \{ 1 \} \times [0,1]$ and $t$ on the top part of the same front face (Figure \ref{cobordism01/fig}). Then, for any \dmnsnl{4} \hndlbd{2} 
\[
 W = (W,M_{s,t})= W_0 \cup_{i=1}^m H^1_i \cup_{j=1}^n H^2_j=B^4\cup_{i=1}^{m+s+t} H^1_i \cup_{j=1}^n H^2_j,
\]
the Kirby tangle $\cal T(W)$ is obtained in the following way. Up to isotopy, we can assume that both of the attaching balls of each 1-handle $H^1_i$ are contained in a local chart $A_i \cong \R^3$ of the front face of $B^4$, and we can think of $H^1_i$ as the result of removing from $B^4$ a complementary 2-handle living inside a collar of $A_i$ in $B^4$, whose attaching map into $A_i$ is determined by a trivially framed unknot $U_i \subset A_i$. A dotted unframed version of the unknot $U_i$, together with a spanning disk $D_i \subset A_i$ of it, is taken to represent $H^1_i$ in the so called dot notation. Observe that the unknots and the spanning disks corresponding to the 1-handles of $M_{s,t} \times [0,1]$ can (and should) be chosen in a standard way, so that they project to the spanning disks of the dotted components of the tangle in Figure~\ref{cobordism4/fig}. Then, by considering the standard open arcs $A_{i,j}$ appearing in the same picture, we obtain a morphism from $E_{2s}$ to $E_{2t}$.

Once all the 1-handles are represented in the dot notation, with the disks $D_i$ taken to be pairwise disjoint, the framed link representing the attaching maps of the 2-handles can be drawn on the front face of $B^4$, with each transverse intersection between a framed component and a disk $D_i$ corresponding to a passage of that component through the 1-handle $H^1_i$. 
Moreover,  we can assume that all the undotted components, including the open ones, are endowed with the blackboard framing, since any framing can be reduced to the blackboard one by adding some positive or negative kinks. 

Now, by projecting to $[0,1]^2$, we get a Kirby tangle representation of $W$ as in Figure~\ref{cobordism02/fig}.

\begin{figure}[htb]
 \centering
 \includegraphics{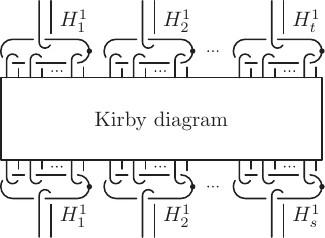}
 \caption{The Kirby tangle of a relative handlebody built on $M_{s,t}$.}
 \label{cobordism02/fig}
\end{figure}


\begin{figure}[htb]
 \centering
 \includegraphics{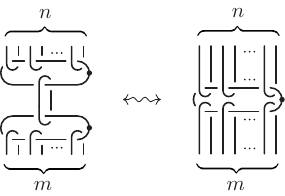}
 \caption{$\cal T:\RHB\to \KTf $ preserves compositions.}
 \label{cobordism05/fig}
\end{figure}

We observe that the map $\cal T$ is well-defined on 2-equivalence classes. First of all, an isotopy of the attaching maps of the handles translates o an isotopy of the corresponding Kirby tangle. Moreover, any operation of type \(b) in Definition~\ref{2-equivalent/def} can be realized  in  $\KTf $ by adding/deleting the dotted and undotted components corresponding to the canceling 1-handle and 2-handle in question, respectively (as in Definition~\ref{kt-equivalence/def} \(c)). On the other hand, any operation of type \(c) in Definition~\ref{2-equivalent/def} can be realized by replacing the undotted component representing the 2-handle to be slided by its band connected sum with a parallel copy of the undotted component representing the 2-handle over which the slide is performed (as in Definition~\ref{kt-equivalence/def} \(d)). 

Since the handlebody decomposition of the composition of $(W,M_{s,t})$ and $(W,M_{s',t'})$, with $s'=t$, consists of all the handles of $W$ and $W'$ plus the 1-handles resulting from the gluing of the thickenings of $M_t$ and $M_{s'}$, the proof that $\cal T$ preserves compositions reduces to verifying the identity in Figure~\ref{cobordism05/fig}. This follows by sliding the lower arcs over the closed undotted component, and then deleting the resulting canceling pair. 

It is straightforward to check that every admissible Kirby tangle is in the image of $\cal T$, and that the functor is invertible and preserves identities and monoidal structures.
\end{proof}


\subsection{3-dimensional relative cobordisms}
\label{cobordisms/sec}

For any $s \geqs 0$, let $F_s$ denote the connected oriented surface of genus $s$ with connected non-empty boundary, canonically realized in $\R^3$ as the \textit{front boundary} $\partial_+ M_s$ of the \dmnsnl{3} handlebody $M_s \subset \R^3$ considered in Subsection~\ref{4HB/sec}, given by
\[
\partial_+ M_s = 
\partial {M_s} \smallsetminus \partial \big( [0,1]^2 \times \left[ 0,1 \right[ \big).
\]
We remark that $\partial F_s = (\partial [0,1]^2) \times \{ 1 \} \cong S^1$ does not depend on $s$, hence it is the same for every $s \geqs 0$.
Then, for any $s,t \geqs 0$, we can consider the connected closed surface of genus $s+t$ given by
\[
F_{s,t} = \partial M_{s,t} = (F_s \times \{0\}) \cup ((\partial [0,1]^2) \times \hrssh) \cup (F_t \times \{1\}) \subset \R^4,
\]
oriented according to the identifications $F_t \times \{1\} \cong F_t$ and $- F_s \times \{0\} \cong - F_s$, where 
\[
 \hrssh = (\partial[0,1]^2) \smallsetminus \left] 0,1 \right[ \times \{1\} = ([0,1] \times \{0\}) \cup (\{0\} \times [0,1]) \cup ([0,1] \times \{1\})
\]
is a piece-wise linear arc embedded into $\R^2$ (notice that $(\partial [0,1]^2) \times \hrssh$ is represented as a pair of horseshoe-shaped arcs yielding the side boundary of $M_{s,t}$ in left-hand part of Figure~\ref{cobordism01/fig}).

By an oriented \dmnsnl{3} relative cobordism, we mean an oriented cobordism between the compact connected oriented surfaces $F_s$ and $F_t$ which is relative to the common boundary $\partial F_s = \partial F_t$ in the sense of the following definition.

\begin{definition}
\label{cobordism/def}
An oriented \dmnsnl{3} \textsl{relative cobordism} from $F_s$ to $F_t$, with $s,t \geqs 0$, is a compact connected oriented 3-manifold $M$ whose boundary coincides with $F_{s,t}$, that is, $\partial M = F_{s,t}$.

Two relative cobordisms $M$ and $M'$ from $F_s$ to $F_t$ are said to be \textsl{equivalent} if there exists a homeomorphism $h : M \to M'$ that coincides with the identity on the common boundary $\partial M = \partial M' = F_{s,t}$.
\end{definition}

According to this definition, if $W$ is an oriented relative \dmnsnl{4} handlebody built on $M_{s,t}$, then its \textsl{front boundary}
\[
\partial_+ W = \partial W \smallsetminus \partial \big( M_{s,t} \times \left[ 0,1 \right[ \big)
\]
is a relative cobordism from $F_s$ to $F_t$, since $\partial (\partial_+ W) = F_{s,t} \times \{1\}$ can be canonically identified with $F_{s,t}$.

In particular, the front boundary 
\[
\partial_+ (M_{s,t} \times [0,1]) = M_{s,t} \times \{1\}
\]
of the trivial handlebody $M_{s,t} \times [0,1]$ (with no handles) is a relative cobordism from $F_s$ to $F_t$ that can be canonically identified with $M_{s,t}$. See the left-hand part of Figure~\ref{cobordism31/fig} for an ``ironed-out'' picture of $\partial_+ (M_{s,t} \times [0,1])$, to be compared with the ``horseshoe'' version of $M_{s,t}$ represented in the left-hand part of Figure~\ref{cobordism01/fig}.
This can be considered as a basic relative cobordism from which any other relative cobordism between $F_s$ and $F_t$ can be obtained by surgery. Since attaching handles to a \dmnsnl{4} relative handlebody induces surgery on its front boundary, this is a immediate consequence of the following extension of the Lickorish-Rokhlin-Wallace's theorem about the surgery presentation of closed 3-manifolds (see \cite{KL01}).

\pagebreak

\begin{figure}[htb]
 \centering
 \includegraphics{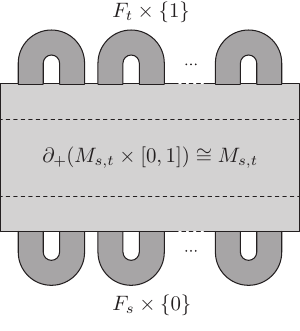}
 \caption{The relative cobordism $\partial_+ (M_{s,t} \times [0,1])$  corresponding to the Kirby tangle in Figure \ref{cobordism4/fig}.}
 \label{cobordism31/fig}
\end{figure}

\begin{proposition}
\label{surgery-pres/thm}
Any \dmnsnl{3} relative cobordism $M$ from $F_s$ to $F_t$ is homeomorphic to the front boundary $\partial_+ W$ of a \dmnsnl{4} relative \hndlbd{2} $W$ built over $M_{s,t}$, hence, up to homeomorphism, it can be obtained by surgery on $M_{s,t}$. Moreover, $W$ can be assumed to have only 2-handles, so only 2-surgery is needed to realize $M$ starting from $M_{s,t}$. 
\end{proposition}

Similarly, Kirby calculus relating surgery presentations of homeomorphic closed 3-manifolds can be extended to \dmnsnl{3} relative cobordisms, as stated by the following proposition (see \cite{KL01}).

\begin{proposition}
\label{Kirby-calculus/thm}
Two oriented \dmnsnl{4} relative \hndlbds{2} $W$ and $W'$ have equivalent front boundaries $\partial_+ W$ and $\partial_+ W'$ (as relative cobordisms) if and only if they are related by a finite sequence of the operations \(a), \(b), and \(c) in Definition~\ref{2-equivalent/def}, and of the following further two operations:
\begin{itemize}
\item[\(d)]
replacing a $1$-handle by a trivially attached $2$-handle and vice-versa (handle trading);
\item[\(e)]
adding/deleting a $2$-handle attached along a separate unknot with framing $\pm 1$ (blow-up/down).
\end{itemize}
Moreover, operations \(a)--\(d) suffice to relate $W$ and $W'$ if these have equivalent front boundaries and the same signature $\sigma(W) = \sigma(W')$, since operation \(e) is the unique one that changes the signature of the handlebody by $\pm1$.
\end{proposition}

\begin{figure}[b]
 \centering
 \includegraphics{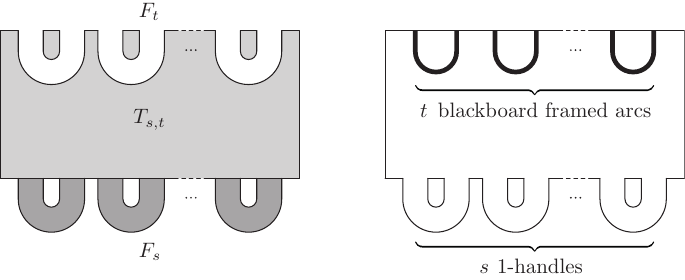}
 \caption{The relative cobordism $T_{s,t}$ and its top-tangle diagram.}
 \label{cobordism32/fig}
\end{figure}

In light of the above proposition, operation \(d) allows us to replace all the 1-handles in any Kirby tangle presentation of a \dmnsnl{4} relative \hndlbds{2} $W$ while preserving both the front boundary $\partial W$ up to homeomorphism and the signature $\sigma(W)$. In this way, any surgery presentation of a \dmnsnl{3} relative cobordism can be changed into one consisting of 2-surgeries only, simply by erasing all the dots from the corresponding Kirby tangle.

We observe that, according to Proposition~\ref{surgery-pres/thm}, any ``consistent'' family of \dmnsnl{3} relative cobordisms from $F_s$ to $F_t$ for all $s,t \geqs 0$ could be chosen, instead of $M_{s,t}$, as the base for the surgery presentation of any such cobordism. For example, this is the case for the relative cobordisms $T_{s,t}$ schematically depicted in Figure~\ref{cobordism32/fig}, which are obtained by attaching $s$ $1$-handles to the bottom face of $[0,1]^3$, and removing open tubular neighborhoods of $t$ arcs whose endpoints lie on the top face. Here, the bottom part of the boundary is canonically identified with $F_s$, while the blackboard framing of the tangle is used to determine an identification of the top part of the boundary with $F_t$.

This alternative choice leads to the top-tangle surgery presentation of \dmnsnl{3} relative cobordisms considered in \cite{BD21}. This is an upside-down version of Habiro's bottom-tangles in handlebodies (see \cite{Ha05,As11}), to which surgery is applied. The reason for the vertical inversion is that we read cobordisms from bottom to top, like in \cite{BP11,BD21}, while in \cite{Ha05,As11} they are read from top to bottom.

For the reader's convenience, in Figure~\ref{TT-morph/fig}, we show the top-tangle presentation of the structure morphisms of $\Algt$.
Here, the thick blackboard framed arcs stand for removed open tubular neighborhoods, as in Figure~\ref{cobordism32/fig}, while the thin blackboard framed closed curves stand for $2$-surgery. They are obtained by performing $2$-surgery on the top-tangle in Figure~\ref{cobordism32/fig}, followed by suitable slidings and cancellations. Observe that most of the cobordisms in the figure can be realized without any $2$-surgery, which means that they embed directly into $\mathbb R^3$.

\begin{figure}[htb]
 \includegraphics{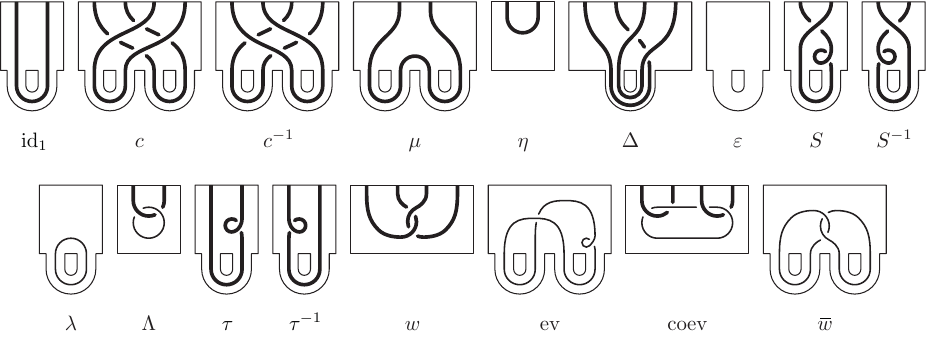}
 \caption{Top-tangle diagrams corresponding to the Kirby diagrams in Figure~\ref{KT-morph/fig}.}
 \label{TT-morph/fig}
\end{figure}

\subsection{The quotient categories \texorpdfstring{$\RCob$}{3Cob} and \texorpdfstring{$\KTt$}{3KT}}
\label{3KT/sec}

We will show that equivalence classes of oriented \dmnsnl{3} relative cobordisms form a monoidal category $\RCob$, which admits a quotient front boundary functor $\partial_+ : \RHB \to \RCob$. This will give rise to a corresponding quotient functor $\partial_+ : \KTf \to \KTt$, once $\RCob$ is shown to be equivalent to the category $\KTt$ of admissible Kirby tangles up to a suitable front boundary equivalence.

\begin{definition}
\label{3Cob/def}
We denote by $\RCob$ the monoidal category whose objects are connected oriented surfaces (with boundary) $F_s$ for $s \geqs 0$, and whose morphisms from $F_s$ to $F_t$ are equivalence classes of \dmnsnl{3} relative cobordisms from $F_s$ to $F_t$, as defined in Subsection~\ref{cobordisms/sec}. 

The composition of two morphisms $M$ from $F_s$ to $F_t$ and $M'$ from $F_{s'}$ to $F_{t'}$ with $t = s'$ is given by vertical juxtaposition, with $M'$ on top of $M$, and by rescaling by a factor $1/2$, which corresponds to gluing the two morphisms by canonically identifying the target of the first with the source of the second, that is,
\[
M' \circ M \cong M \cup_{F_t \times \{1\} = F_{s'} \times \{0\}} M',
\]
with $\partial(M' \circ M) \cong F_{s,t'}$ canonically contained in $F_{s,t} \cup_{F_t \times \{1\} = F_{s'} \times \{0\}} F_{s',t'}$.

The tensor product, denoted by $\bcsum$, is given by horizontal juxtaposition, from left to right, and it corresponds to the boundary connected sum for both the objects and the morphisms, with canonical identifications $F_s \bcsum F_{s'} \cong F_{s + s'}$ for the product of objects, and $\Bd(M \bcsum M') = \Bd M \csum \Bd M' = F_{s,t} \csum F_{s',t'} \cong F_{s+s',t+t'}$ for the product $M \bcsum M'$ of morphisms $M$ from $F_s$ to $F_t$ and $M'$ from $F_{s'}$ to $F_{t'}$.

For each $s \geqs 0$, the identity $\id_{F_s}$ is represented by the product cobordism $F_s \times [0,1]$. In particular, $\idone = F_0 \times [0,1]$, since $\one = F_0$.
\end{definition}

\begin{remark}\label{3Cob-skeleton/rmk}
$\RCob$ is a skeleton of a category whose objects are arbitrary connected oriented surfaces $\varSigma$ with connected non-empty boundary equipped with an orientation-preserving identification $f : S^1 \to \partial \varSigma$, and whose morphisms are equivalence classes of \dmnsnl{3} relative cobordisms $M$ from $\varSigma$ to $\varSigma'$, with boundary identifications agreeing with $f : S^1 \to \partial \varSigma$ and $f' : S^1 \to \partial \varSigma'$. Once again, it is convenient to work with the standard models for objects we are considering here.
\end{remark}

In light Definitions~\ref{4HB/def} and \ref{3Cob/def}, the front boundary operator $\partial_+$ introduced in Subsection~\ref{cobordisms/sec} induces a monoidal functor from $\RHB$ to $\RCob$. In fact, we have the following proposition.

\begin{proposition}
\label{4HBto3Cob/thm}
There is a quotient monoidal functor $\partial: \RHB \to \RCob$ such that $\partial M_s = F_s$ for all $s \geqs 0$, which sends any morphism of $\RHB$ given by the 2-equivalence class of the relative \hndlbd{2} $W$ to the morphism of $\RCob$ given by the equivalence class of its front boundary $\partial_+ W$.
\end{proposition}

\begin{proof}
The claim that $\partial_+$ is well-defined as a monoidal functor immediately follows from the definition of the front boundary of a \dmnsnl{4} relative \hndlbd{2}, the ``if'' part of Proposition~\ref{Kirby-calculus/thm}, and the fact that vertical and horizontal juxtaposition of \dmnsnl{4} relative \hndlbds{2} restrict to analogous operations on their front boundaries. Notice that, for all $W = (W,M_{s,t})$ and $W' = (W',M_{s',t'})$ with $t=s'$, the front boundary $\partial_+ (W' \circ W)$ and the composition $\partial_+ (W') \circ \partial_+(W)$ differ only by a canonical collar of the middle surface $F_{t} = F_{s'}$, and so they are equivalent. On the other hand, the functor $\partial_+$ is trivially surjective on objects, while Proposition~\ref{surgery-pres/thm} implies its surjectivity on morphisms.
\end{proof}

Now, based on the front boundary equivalence moves shown Table~\ref{table-Kirby-moves/fig}, which provide a Kirby tangle interpretation of operations \(d) and \(e) in Proposition~\ref{Kirby-calculus/thm}, we can define the category $\KTt$, which is the diagrammatic counterpart of $\RCob$.

\begin{table}[htb]
 \includegraphics{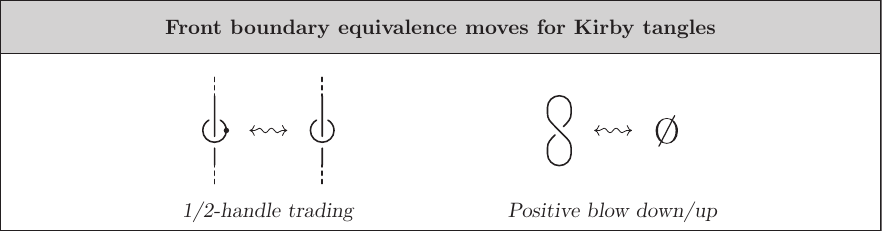}
 \caption{}
 \label{table-Kirby-moves/fig}
\end{table}

\begin{definition}
\label{fb-equivalence/def}
Two admissible Kirby tangles are said to be \textsl{front boundary equivalent} if they are related by a finite sequence of the moves and the operations in Tables~\ref{table-Ktangles/fig} and \ref{table-Kirby-moves/fig}.
\end{definition}

Actually, the relations in Table~\ref{table-Kirby-moves/fig} imply any $1/2$-handle trading and negative blow-up/down, as specified by the next proposition (compare with \cite{Ki89} and \cite{GS99}).

\begin{proposition}[{\cite[Lemma~5.2.1]{BP11}}]
\label{fb-equivalence/thm}
Modulo $1/2$-handle cancellation and $2$-handle sliding, any $1/2$-handle trading can be reduced to one presented on the left-hand side of Table~\ref{table-Kirby-moves/fig}. Moreover, modulo $2$-handle sliding and $1/2$-handle trading, positive and negative blow-up/down are inverse to one another.
\end{proposition}

\begin{definition}
\label{3KT/def}
We denote by $\KTt$ the quotient category of $\KTf$ with respect to the front boundary equivalence relations presented in Table~\ref{table-Kirby-moves/fig}. Then $\KTt$ inherits the structure of a braided monoidal category making the quotient functor $\partial: \KTf \to \KTt$ into a braided monoidal functor.
\end{definition}

As an immediate consequence of the above definitions and of Proposition~\ref{fb-equivalence/thm}, we have the following proposition.

\begin{proposition}[{\cite[Proposition~5.2.2]{BP11}}]
\label{3KT/thm}
The maps sending any morphism of $\KTt$ given by the front boundary equivalence class of a Kirby tangle $T$ to the morphism of $\RCob$ given by the equivalence class of the relative cobordism represented by $T$ defines an equivalence of braided monoidal categories $\KTt \cong \RCob$. Furthermore, the following diagram commutes:
\[
\begin{tikzpicture}
 \node (P0) at (0,0) {$\KTf$};
 \node (P1) at (2.25,0) {$\KTt$};
 \node (P2) at (0,-1.5) {$\RHB$};
 \node (P3) at (2.25,-1.5) {$\RCob$};
 \draw
 (P0) edge[->] node[above] {$\textstyle \partial_+$} (P1)
 (P0) edge[->] node[left, xshift=-0.25ex] {$\textstyle \cong$} (P2)
 (P1) edge[->] node[right, xshift=0.25ex] {$\textstyle \cong$} (P3)
 (P2) edge[->] node[below] {$\textstyle \partial_+$} (P3);
\end{tikzpicture}
\]
\end{proposition}

\FloatBarrier

%% file: S5-equivalence.tex
\section{Algebraic presentation of \texorpdfstring{$\RHB$}{4HB} and \texorpdfstring{$\RCob$}{3Cob}}\label{equivalence/sec}

\label{maintheorems/sec}

This section is devoted to the proof of Theorem~\ref{thm:main}, which provides an algebraic presentation of $\RHB \cong \KTf$. We start by recalling the definition of the functor $\Phi:\Algf\to \KTf$. 
The inverse functor $\barPhi:\KTf\to\Algf$ is constructed in Subsection~\ref{FK/sec}, and its independence of many auxiliary choices is shown in subsequent subsections. A crucial role in the proof is played by a subcategory $\AlgD$ and its labeled version $\AlgL$ as well as a natural transformation $\Theta$ (Subsection~\ref{AlgD/sec}) and its labeled version $\Theta^{\rm L}_j$ (Subsection~\ref{AlgL/sec}). The proof of Theorem~\ref{thm:main} is given in Subsection~\ref{Psi/sec}.

\subsection{The functor \texorpdfstring{$\Phi : \Algf \to \KTf$}{from 4Alg to 4KT}}
\label{Phi/sec}

It was first shown in \cite{CY94} that the category $\RCob$ of \dmnsnl{3} cobordisms contains a braided Hopf algebra, the punctured torus. It was later shown in \cite{Ke01} that this braided Hopf algebra, together with a ribbon element and an integral, generates $\RCob$, or equivalently the category $\KTt$ of admissible framed tangles (see also \cite{Ha05} for a similar statement concerning the category of bottom tangles in handlebodies). We recall here a generalization of these results established in \cite{BP11}, where it is proved that the solid torus satisfies axioms~\hrel{r8} and \hrel{r9} in $\RHB$, and is thus a BP Hopf algebra.

\begin{theorem}[{\cite[Theorem~4.3.1]{BP11}}]\label{phi/thm}
There exists a braided monoidal functor $\Phi: \Algf \to \KTf$ that sends $H$ to $E_2$ (see Proposition~\ref{4KT/thm}) and each structure morphism of $H$ to the corresponding Kirby tangle represented in Figure~\ref{phi01/fig}.
\end{theorem}

\begin{figure}[htb]
 \centering
 \includegraphics{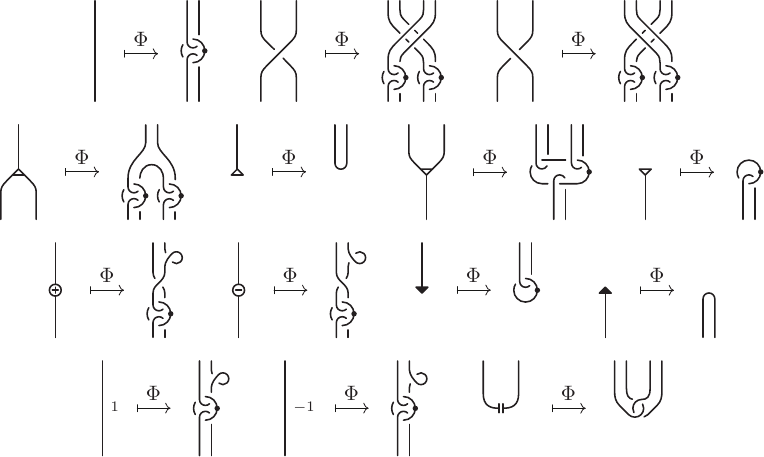}
 \caption{The functor $\Phi : \Algf \to \KTf$.}
 \label{phi01/fig}
\end{figure}

Notice that the image of the copairing $\copairH$ is the rotation along a horizontal axis in $\R^3$ of the bottom tangle defining Lyubashenko's pairing in \cite{L94}. 

For completeness, we present the proof.

\begin{proof}[Proof of Theorem~\ref{phi/thm}]
We have to check that the images under $\Phi$ of the structure morphisms of $H$ satisfy the axioms of a BP Hopf algebra in Tables~\ref{table-Hopf/fig} and \ref{table-BPHopf/fig}.
For most of them this is quite straightforward. In particular, the proof reduces to an isotopy for the braid axioms, and to the removal of canceling 1/2-pairs for axioms~\hrel{a4-4'}, \hrel{a6}, \hrel{a8}, \hrel{i3}, and \hrel{i4}, while some handle slides are also required for axioms~\hrel{a1}, \hrel{a2-2'}, \hrel{a3}, \hrel{a7}, \hrel{s2-3}, \hrel{i1}, \hrel{i2},  \hrel{i5}, \hrel[r3]{r3-4-5}.

Axioms~\hrel{a5} and \hrel{s1} are proved in Figures~\ref{phi04/fig} and \ref{phi05/fig} respectively, while the proof of axiom~\hrel{s1'} is obtained by rotating the diagrams in Figure~\ref{phi05/fig} of an angle $\pi$ along the vertical axis. 
Finally, the ribbon axioms \hrel[r6]{r6-7-8-9} are shown respectively in Figures~\ref{phi03/fig}, \ref{phi06/fig}, \ref{phi07/fig}, and \ref{phi08/fig}.
\end{proof}

\begin{figure}[htb]
 \centering
 \includegraphics{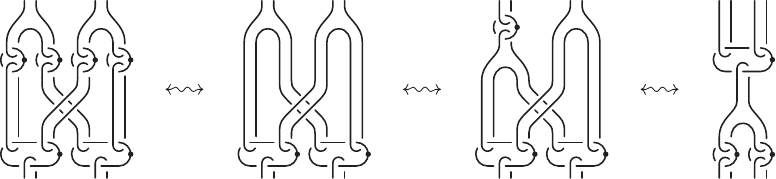}
 \caption{The definition of $\Phi$ is compatible with axiom~\hrel{a5}.}
 \label{phi04/fig}
\end{figure}

\begin{figure}[htb]
 \centering
 \includegraphics{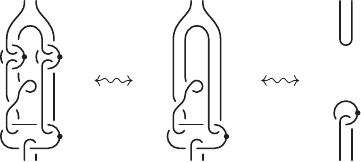}
 \caption{The definition of $\Phi$ is compatible with axiom~\hrel{s1-1'}.}
 \label{phi05/fig}
\end{figure}

\begin{figure}[htb]
 \centering
 \includegraphics{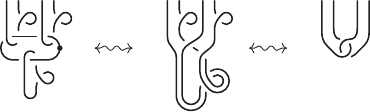}
 \caption{The definition of $\Phi$ is compatible with axiom~\hrel{r6}.}
 \label{phi03/fig}
\end{figure}

\begin{figure}[htb]
 \centering
 \includegraphics{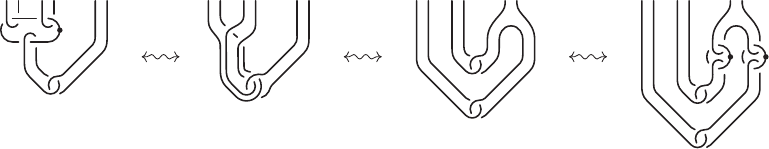}
 \caption{The definition of $\Phi$ is compatible with axiom~\hrel{r7}.}
 \label{phi06/fig}
\end{figure}

\begin{figure}[htb]
 \centering
 \includegraphics{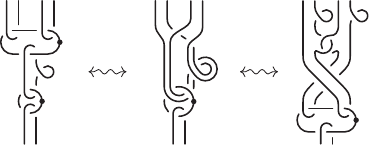}
 \caption{The definition of $\Phi$ is compatible with axiom~\hrel{r8}.}
 \label{phi07/fig}
\end{figure}

\begin{figure}[htb]
 \centering
 \includegraphics{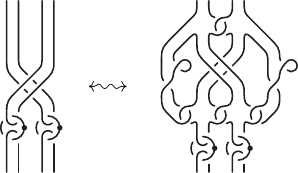}
 \caption{The definition of $\Phi$ is compatible with axiom~\hrel{r9}.}
 \label{phi08/fig}
\end{figure}

We observe that the images under $\Phi$ of the evaluation, the coevaluation, and the adjoint action are equivalent, in $\KTf$, to the tangles represented in Figure~\ref{phi02/fig}. The proof is straightforward and left to the reader. 

\begin{figure}[htb]
 \centering
 \includegraphics{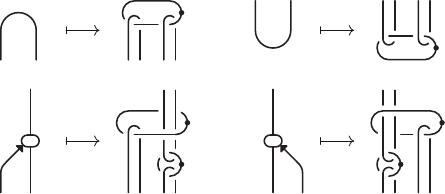}
 \caption{Images under $\Phi$ of $\ev$, $\coev$, $\ad$, and $\ad'$.}
 \label{phi02/fig}
\end{figure}

\subsection{The subcategory \texorpdfstring{$\AlgD$}{TAlg} of \texorpdfstring{$\Algf$}{4Alg}}
\label{AlgD/sec}

In this subsection, we define a monoidal subcategory $\AlgD$ of $\Algf$ whose morphisms are sent by $\Phi$ to a family of special two-level Kirby tangles. $\AlgD$ will play an essential role in the definition of the inverse functor $\barPhi$ of $\Phi$ in Subsection~\ref{FK/sec}, where the image of any Kirby tangle in $\KTf$ under $\barPhi$ will be defined as some sort of closure of morphisms in $\AlgD$. We will show below that $\AlgD$ has some interesting algebraic properties. In particular, it admits two ribbon structures. 
Moreover, there exist two families of morphisms in $\Algf$ that intertwine all morphisms in $\AlgD$, and whose images under $\Phi$ are given by $1$-handles which embrace the upper/lower level of the tangle.

\begin{table}[b]
 \centering
 \includegraphics{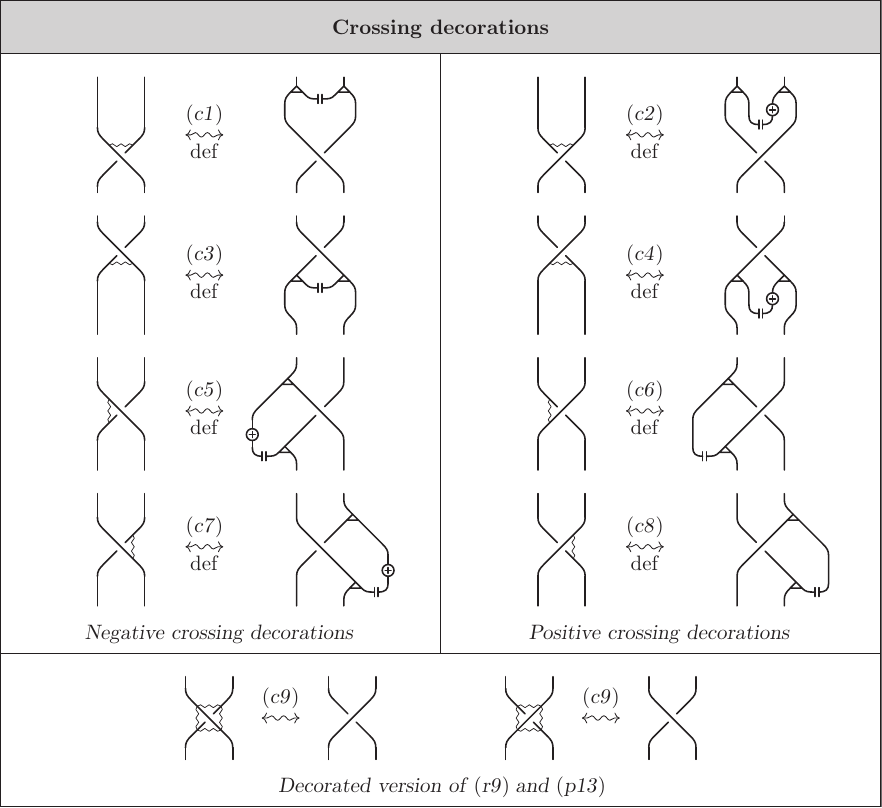}
 \caption{}
 \label{table-decorated/fig}
 \label{E:c1} \label{E:c2} \label{E:c3} \label{E:c4} 
 \label{E:c5} \label{E:c6} \label{E:c7} \label{E:c8}
 \label{E:c9}
\end{table}

We start by introducing in Table~\ref{table-decorated/fig} a compact notation for certain decorations (featuring copairings) of the braiding morphisms $c^{\pm 1}$ of $\Algf$. A decoration of a crossing is a wavy line attached to the two edges which form the crossing, and it is entirely contained in one of the four regions that make up the complement of the crossing inside a circular neighborhood in the projection plane. In particular, we have four possible decorations for both positive and negative crossings, and the relations~\hrel{c1}--\hrel{c8} in Table~\ref{table-decorated/fig} define these decorations as morphisms in $\Algf$.

Observe that a decoration, which is a wavy line attached to arbitrary edges, doesn't have a meaning on its own; it acquires one only in a neighborhood of a crossing, and it has to appear in one of the forms shown in Table~\ref{table-decorated/fig}. In particular, to a crossing we can attach at most four decorations.
To understand the meaning of decorations, we observe that the image of $c$ (respectively, of $c^{-1}$) under $\Phi$ consists of four positive (respectively, negative) crossings between two double strands (see Figure~\ref{phi01/fig}), and adding a decoration corresponds to inverting one of these four crossings (see Figure~\ref{phi08/fig}). Notice that inverting all four of them transforms $\Phi(c)$ into $\Phi(c^{-1})$, or the other way round. The algebraic versions of these moves are the relations shown in the bottom section of Table~\ref{table-decorated/fig}, which are the representation of axiom~\hrel{r9} in Table~\ref{E:r9} and relation~\hrel{p13} in Table~\ref{E:p13} in terms of decorated crossings. Moreover, these two relations can be generalized as stated in the following proposition.

\begin{proposition}\label{basic-decorated-move/thm}
Any decorated crossing is equivalent to the opposite crossing with complementary decorations. We will denote by \hrel{c9} this general class of relations.
\end{proposition}

\begin{proof}
The statement follows from axiom~\hrel{r9} in Table~\ref{E:r9} and relation~\hrel{p13} in Table~\ref{E:p13}, by applying relations~\hrel{p5-5'} and \hrel{p6-6'} in Table~\ref{E:p4}.
\end{proof}

We will focus now on studying the properties of the decorated crossings $X$, $\hat{X}$, $Y$, and $\hat{Y}$ defined in the top two lines of Table~\ref{defnXY/fig}, which will play an important role in the definition of $\barPhi$ in the next subsection. 

\begin{table}[htb]
 \centering
 \includegraphics{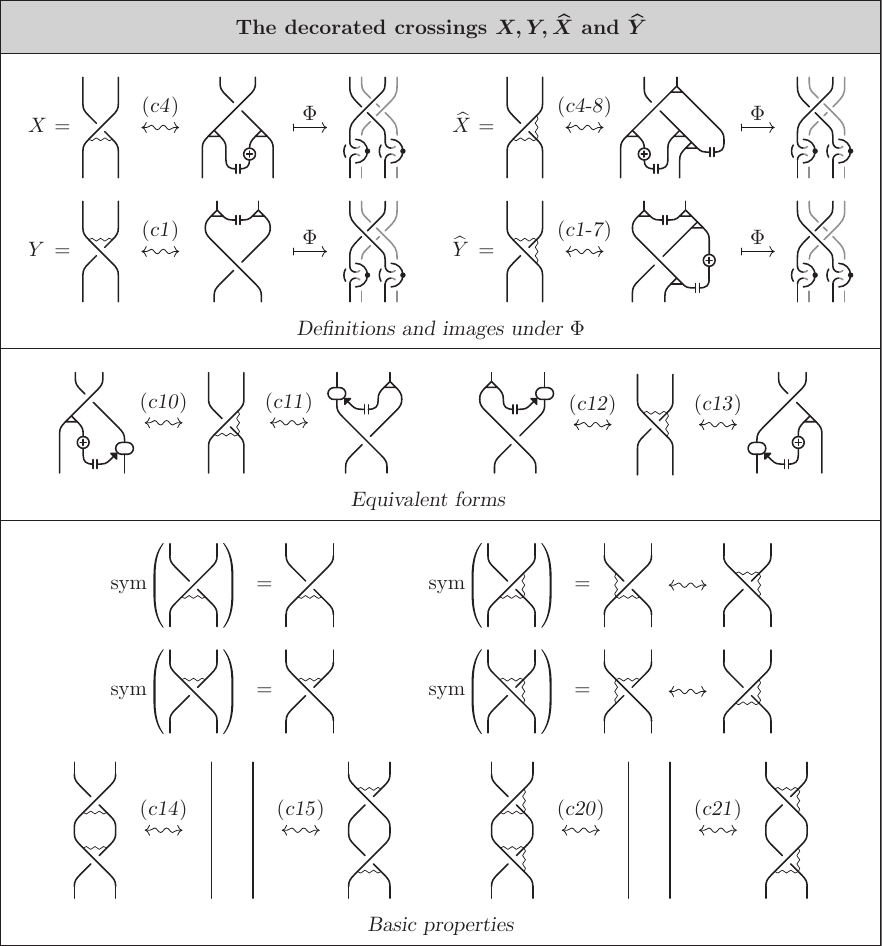}
 \caption{}
 \label{defnXY/fig}
 \label{c4-8} \label{c1-7}
 \label{E:c10} \label{E:c11} \label{E:c12} \label{E:c13} 
\end{table}

\begin{lemma}\label{decorated-sym-inv/thm}
$\hat{X}$ (respectively, $\hat{Y}$) can be presented in the equivalent forms \hrel[c10]{c10-11} (respectively, \hrel[c12]{c12-13}) represented in Table~\ref{defnXY/fig}. Moreover, we have the following identities (see Table~\ref{defnXY/fig})\footnote{Identities~\hrel{c14-15} and \hrel{c20-21} are part of the defining properties of the two braided structures of the category $\AlgD$ in Table~\ref{table-decorated-moves/fig}, which is where their natural enumeration comes from.}: 
\begin{gather*}
 \sym(X) = X, \quad \sym(Y) = Y, 
 \\
 \sym(\hat{X}) = \hat{Y}, \quad \sym(\hat{Y}) = \hat{X},
\\
 X\circ Y=Y\circ X=\id_2,
 \tag*{\hrel{c14-15}}
 \\
 \hat X\circ \hat Y=\hat Y\circ \hat X=\id_2.
 \tag*{\hrel{c20-21}}
\end{gather*}
In particular $Y = X^{-1}$ and $\hat{Y} = \hat{X}^{-1}$.
\end{lemma}
 
\begin{proof}
 The identities \hrel{c14-15} and \hrel{c20-21} follow from the properties of the copairing, the adjoint action, and the antipode, as indicated in Figure~\ref{isotopy1/fig}, while \hrel[c10]{c10-11}, \hrel[c12]{c12-13} and the identities involving the symmetry functor are proved in Figure~\ref{kirby-hopf04/fig}.
\end{proof}

\begin{figure}[hbt]
 \centering
 \includegraphics{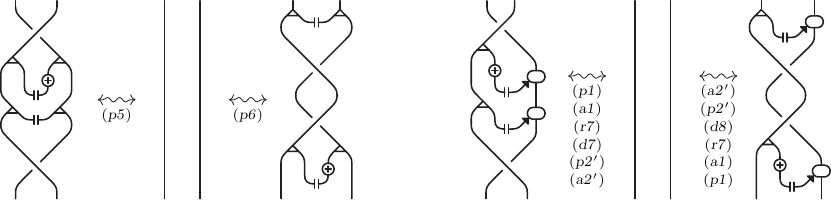}
\caption{Proof of \hrel{c14-15} and \hrel{c20-21}. 
}
 \label{isotopy1/fig}
\end{figure}

\begin{figure}[htb]
 \centering
 \includegraphics{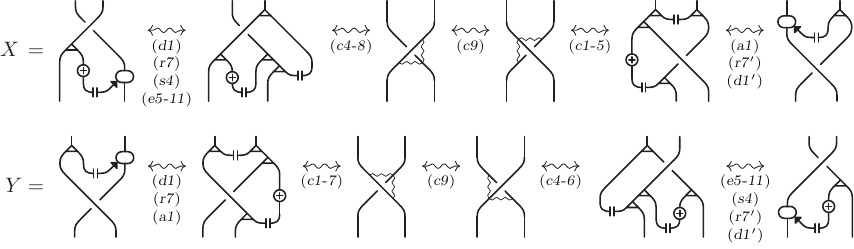}
 \caption{Equivalent forms of $\hat{X}$ and $\hat{Y}$.}
 \label{kirby-hopf04/fig}
\end{figure}


\begin{remark}\label{TAlg/rmk} Before going further, we would like to explain the intuition behind the definitions and constructions in the present subsection. Notice that the images of $X$, $\hat X$, $Y$, and $\hat Y$ under $\Phi$, represented in Table~\ref{defnXY/fig}, share the following properties.
\begin{itemize}
\item[\(a)] The intersection of the projection plane with the Kirby tangle coincides with some of the disks spanned by the dotted components, and the intersections with such disks divide the undotted components, and in particular each open component, into two parts: one which stays above and one which stays below the projection plane, represented respectively in black and gray.
\item[\(b)] In the projection plane, the lower arc of each undotted component projects onto the right of the upper one.
\end{itemize}
A Kirby tangle that is equivalent to one satisfying properties (a) and (b) above for some choice of dotted components in (a) will be called a \textit{two-level} Kirby tangle. For example, $\id$ (see Figure~\ref{KT-morph/fig}) is a two-level Kirby tangle, since the spanning disk of the single dotted component can be chosen to lie in the projection plane in such a way that it divides each open component into two parts satisfying conditions (a) and (b) above. Moreover, the composition and the tensor product of two two-level tangles is again a two-level tangle. Therefore, two-level Kirby tangles form a monoidal subcategory $\TTH$ of $\KTf$.

Notice that the images under $\Phi$ of $\coprH$, $\counH$, $\inteH$, $\ribmorH$, $\ev$, and $\coev$ (see Figures~\ref{phi01/fig} and \ref{phi02/fig}) belong to this subcategory, while one can see that the images of $\prodH$, $\unitH$, $\antipH$, $c^{\pm 1}$, and $\copairH$ do not. On the other hand, 
any closed Kirby tangle can be thought of as a two-level tangle by
assigning the whole tangle to the upper or to the lower level.
In the remainder of this subsection, our goal is to define and study a subcategory of $\Algf$ that is an algebraic analog of $\TTH$.
\end{remark}

We first define two families of morphisms $\Theta_k$ and $\Theta'_k$, with $k \geqs 0$, designed to provide an algebraic analogue of a dotted component that embraces either the lower (gray) or the upper (black) strands of a two-level tangle. 

\begin{definition}\label{theta/defn}
For every $k \geqs 0$, the morphism $\Theta_k : H^k \otimes H \to H^k$ in $\Algf$ is recursively defined by the following identities (compare with Figure~\ref{theta/fig} ):
\begin{gather*}
 \Theta_0 = \counH, \quad \Theta_1 = \prodH,
 \\
 \Theta_k = (\Theta_1 \otimes \Theta_{k-1}) \circ (\id \otimes \braid_{k-1,1} \circ \id) \circ (\id_{k-1} \otimes \coprH). 
\end{gather*}
Define also $\Theta'_k = \sym(\Theta_k) : H \otimes H^k \to H^k$ to be the symmetric morphism (see Proposition~\ref{symmetry/thm}).

We denote by $\Theta$ the collection of morphisms $\{ \Theta_k \mid k \in \N \}$, and similarly by $\Theta'$ the collection $\{\Theta'_k \mid k \in \N \}$.
\end{definition}

\begin{figure}[ht]
 \centering
 \includegraphics{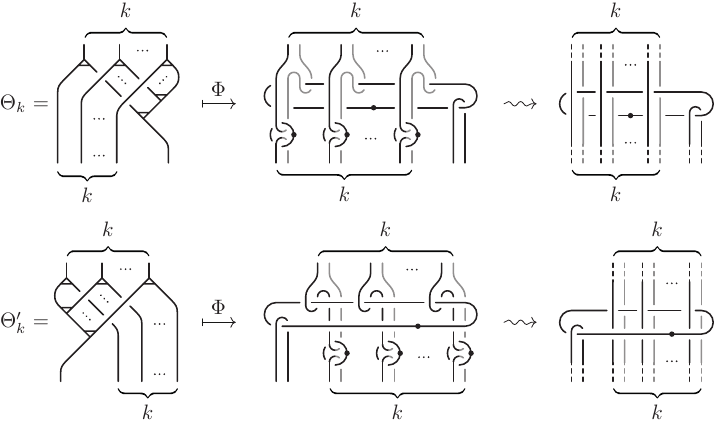}
 \caption{The morphisms $\Theta_k$ and $\Theta_k'$ and their images under $\Phi$, $k \geqs 0$.}
 \label{theta/fig}
\end{figure}

For every $k \geqs 0$, we also introduce the morphisms\footnote{$\Theta'$ and $U'$ are just auxiliary morphisms which will be used to simplify some proofs, while $\Theta$ and $U$ will play an essential role in the following. This is why we swap the prime in the notation.}
\[
U_k' = \Theta_k \circ (\id_k \otimes \inteH) \mbox{ and }
U_k = \Theta_k' \circ (\inteH \otimes \id_k),
\]
see Figure~\ref{defnU/fig}. Notice that the image under $\Phi$ of $U_k$ (respectively, $U_k'$) is a two-level Kirby tangle that consists of a single dotted component embracing the strand in the upper (respectively, lower) level of an identity tangle.

We are now ready to define the algebraic analog of $\TTH$.

\begin{figure}[htb]
 \centering
 \includegraphics{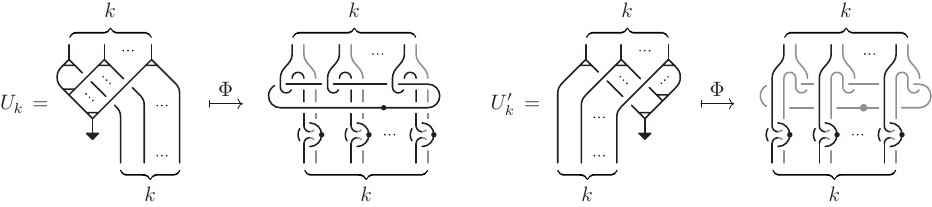}
 \caption{The morphisms $U_k$ and $U_k'$ and their images under $\Phi$.}
 \label{defnU/fig}
\end{figure}

\begin{definition}\label{AlgD/def} 
We denote by $\AlgD$ the monoidal subcategory of $\Algf$ generated by the morphisms $\coprH$, $\counH$, $\inteH$, $\ev $, $\ribmorH$, $X$, $\hat X$, $Y$, $\hat Y$, $U_k$, and $U_k'$ defined in Tables~\ref{table-Hopf/fig}, \ref{table-BPHopf/fig}, and \ref{defnXY/fig}, and Figure~\ref{defnU/fig}. 
\end{definition}

 Notice that $\coev$ and $\tilde{\prodH}$ belong to $\AlgD$, since they are compositions of morphisms in $\AlgD$.

 Observe also that $\Phi$ sends every generator of $\AlgD$ to a morphism of $\TTH$. Nevertheless, we do not claim that $\Phi$ restricts to a full or faithful functor from $\AlgD$ to $\TTH$, even if it may as well be true. Indeed, in the present context, the subcategory $\TTH$ was introduced only to provide motivation for the definition of $\AlgD$.

As an immediate consequence of Lemma~\ref{decorated-sym-inv/thm} and Definition~\ref{AlgD/def}, we have the following corollary. 

\begin{corollary}\label{symTAlg/thm}
The subcategory $\AlgD$ of $\Algf$ is invariant under the action of the symmetry functor $\sym : \Algf \to \Algf$ defined in Proposition~\ref{symmetry/thm}.
\end{corollary}


We introduce below a generalized move that consists in sliding the product morphism through a decorated crossing. The move is illustrated in Figure~\ref{a1-decor/fig}, where we show examples with a decorated crossing featuring two decorations. The case of a single decoration can be obtained from these by simply deleting the same decoration on both sides.

\begin{figure}[hbt]
 \centering
 \includegraphics{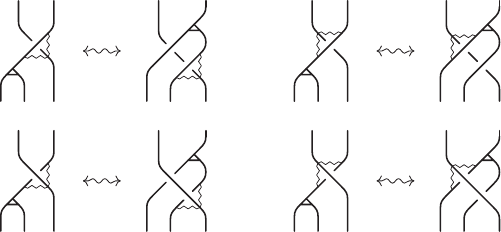}
 \caption{Examples of move~\hrel{A} over a decorated crossing with two decorations.}
 \label{a1-decor/fig}
\end{figure}


\begin{lemma}\label{a1-TAlg/thm} \label{E:A}
A product morphism can be slid through a decorated crossing at the price of creating a second ordinary crossing of the same sign, provided that the decorated crossing has at most two decorations attached to the same side of the outgoing strand of the product morphism. Such generalized move will be denoted by \rel{A}.
\end{lemma}

\begin{proof}
Notice that, as defined by relations~\hrel{c1} to \hrel{c8} in Table~\ref{table-decorated/fig}, every crossing decoration is attached through a pair of product morphisms to the pair of strands involved in the crossing. So, the statement follows directly from the associativity relation~\hrel{a1} and the naturality of the braiding. 
\end{proof}




\begin{lemma}\label{theta/thm}
If $\iota : \AlgD \hookrightarrow \Algf$ denotes the inclusion functor, then $\Theta : \iota \otimes H \Rightarrow \iota$ and $\Theta' : H \otimes \iota \Rightarrow \iota$ define natural transformations, meaning that, for every morphism $F : H^s \to H^t$ in $\AlgD$, we have (see Figure \ref{theta-moves/fig})
\begin{align*}
 \Theta_t \circ (F \otimes \id) &= F \circ \Theta_s, 
 \tag*{\hrel{t1}}
 \\
 \Theta'_t \circ (\id \otimes F) &= F \circ \Theta'_s.
 \tag*{\hrel{t1'}}
\end{align*}
 
\begin{figure}[hbt]
 \centering
 \includegraphics{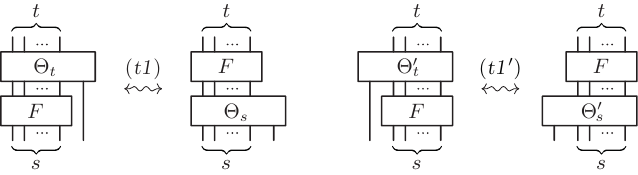}
 \caption{Naturality of $\Theta$ and $\Theta'$.}
 \label{theta-moves/fig} 
 \label{E:t1-1'} \label{E:t1} \label{E:t1'}
\end{figure}
\end{lemma}

\begin{proof}
According to  Corollary~\ref{symTAlg/thm} and Proposition~\ref{symmetry/thm}, the statement for $\Theta'$ can be derived from the one for $\Theta$ by applying the functor $\sym$. For what concerns $\Theta$, the coassociativity relation \hrel{a3} implies that, if the statement is true for two morphisms, then it is true for their product as well. Therefore, it is enough to show that \hrel{t1} holds whenever $F$ is one of the generating morphisms of $\AlgD$. For $F = \coprH, \counH, \inteH, \ribmorH$, the statement follows directly from \hrel{a5}, \hrel{a6}, \hrel{i2}, and \hrel{r5}, while for $F = \ev, U_k, U_k'$ it is shown in Figure~\ref{theta1/fig}. Then, in the first two lines of Figure~\ref{theta2/fig}, we prove \hrel{t1} for $F = X$, which in turn implies \hrel{t1} for $F = Y$, since, thanks to Lemma~\ref{decorated-sym-inv/thm}, $Y = X^{-1}$. Finally, in the last line of Figure~\ref{theta2/fig}, by using \hrel{t1} for $Y$, we prove the statement for $F = \hat X$, which implies \hrel{t1} for $F = \hat Y = \hat X^{-1}$.
\end{proof}

\begin{figure}[htb]
 \centering
 \includegraphics{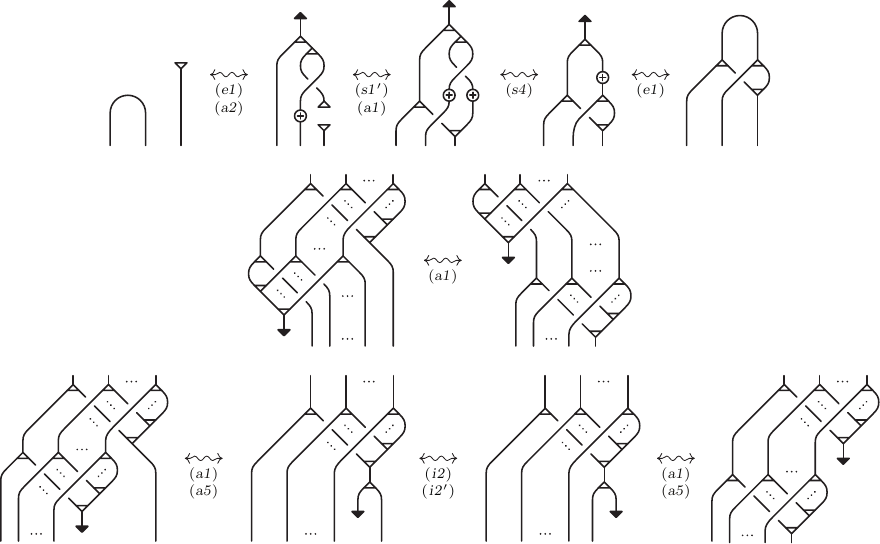}
 \caption{Proof of relation~\hrel{t1} for $F = \ev, U_k, U_k'$.}
 \label{theta1/fig}
\end{figure}

\begin{figure}[htb]
 \centering
 \includegraphics{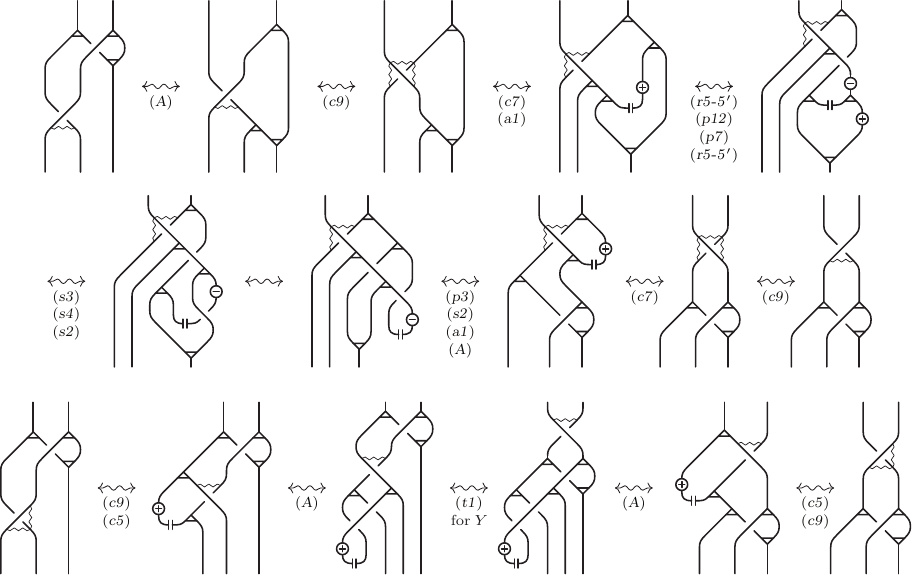}
 \caption{Proof of relation~\hrel{t1} for $F = X, \hat X$.}
 \label{theta2/fig}
\end{figure}

\FloatBarrier
\clearpage

\begin{theorem}\label{decorated-moves/thm}
The relations in Tables~\ref{table-decorated-moves/fig} and \ref{table-decorated-moves-3/fig} are satisfied in $\AlgD$. In particular, $\AlgD$ admits two distinct ribbon structures, in which braiding morphisms (and their inverses) are given by $X, X^{-1} = Y : H \otimes H \to H \otimes H$ and by $\hat X,\hat X^{-1} = \hat Y : H \otimes H \to H \otimes H$, respectively.
\end{theorem}

\begin{table}[htb]
 \centering
 \includegraphics{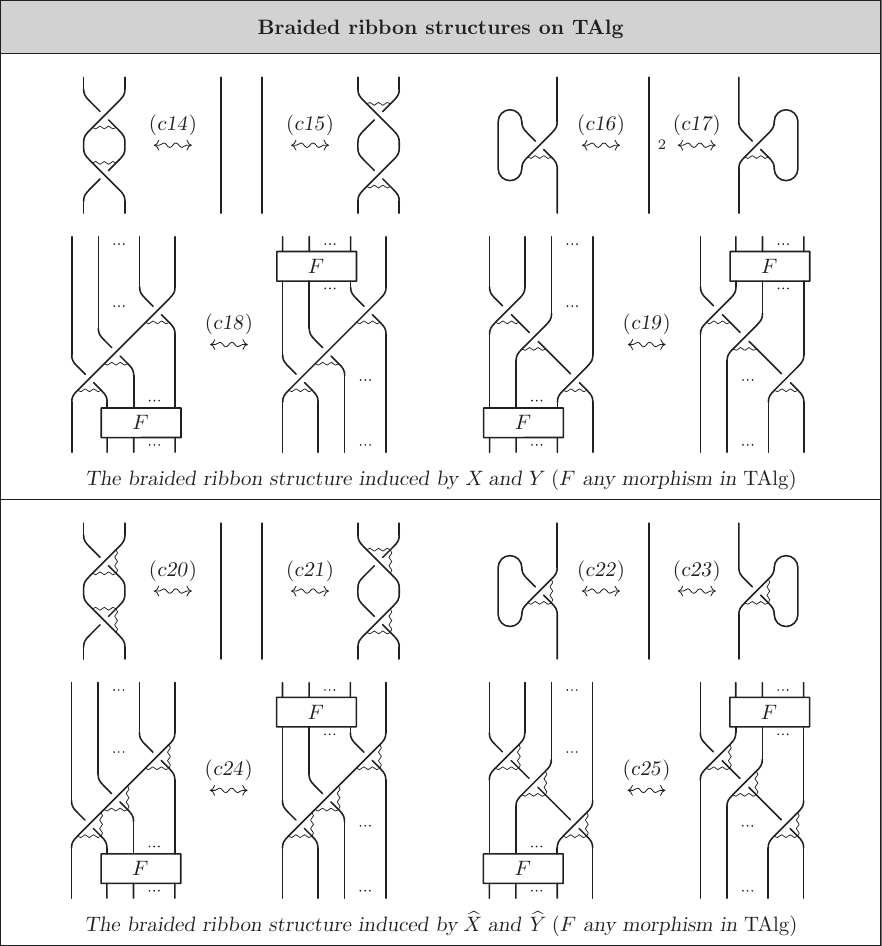}
 \caption{}
 \label{table-decorated-moves/fig}
 \label{E:c14-15} \label{E:c14} \label{E:c15} \label{E:c16} \label{E:c17}
 \label{E:c18} \label{E:c19} \label{E:c20-21} \label{E:c20} \label{E:c21}
 \label{E:c22} \label{E:c23} \label{E:c24} \label{E:c25}
\end{table}

\begin{table}[htb]
 \centering
 \includegraphics{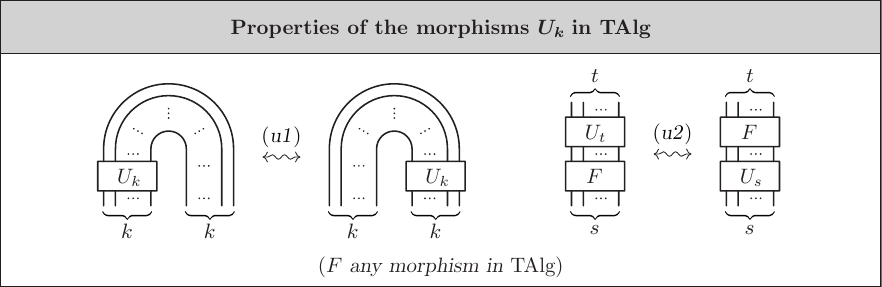}
 \caption{}
 \label{table-decorated-moves-3/fig}
 \label{E:u1} \label{E:u2}
\end{table}

\FloatBarrier

\begin{proof}
The evaluation and coevaluation morphisms of $\AlgD$ are induced by the ones of $\Algf$. Relations~\hrel{c14-15} and \hrel{c20-21} have been proved in Lemma~\ref{decorated-sym-inv/thm}. Relations~\hrel{c16}, \hrel{c22}, and \hrel{c23} are proved (in this order) in Figures~\ref{isotopy2a/fig} and \ref{isotopy2b/fig}, while \hrel{c17} follows by symmetry from \hrel{c16}. Relation~\hrel{c24} follows from \hrel{d10} and Figure~\ref{isotopy3a/fig}, while relations~\hrel{c18} and \hrel{c25} follow from \hrel{t1'} and from Figure~\ref{isotopy3b/fig}. Then, \hrel{c19} follows by symmetry from \hrel{c18}. Relation~\hrel{u1} is proved in Figure~\ref{isotopy4/fig}, while \hrel{u2} is a direct consequence of \hrel{t1'}.
\end{proof}

\begin{figure}[htb]
 \centering
 \includegraphics{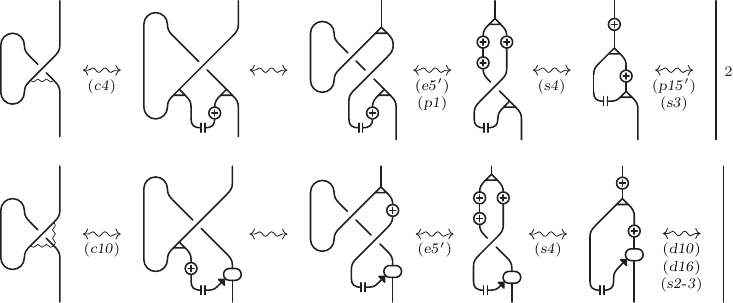}
 \caption{Proof of relations~\hrel{c16} and \hrel{c22}.}
 \label{isotopy2a/fig}
\end{figure}

\begin{figure}[htb]
 \centering
 \includegraphics{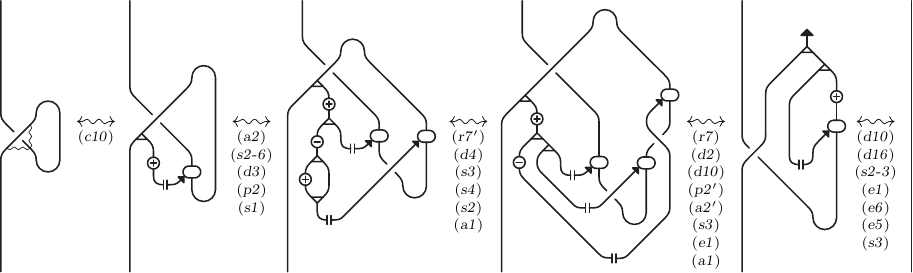}
 \caption{Proof of relation~\hrel{c23}.}
 \label{isotopy2b/fig}
\end{figure}

\begin{figure}[htb]
 \centering
 \includegraphics{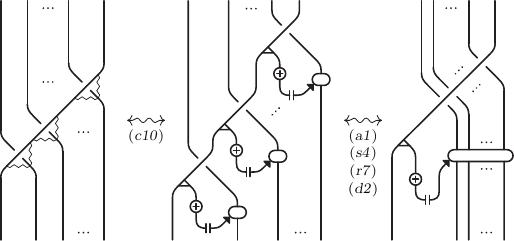}
 \caption{Proof of relation~\hrel{c24}.}
 \label{isotopy3a/fig}
\end{figure}

\begin{figure}[htb]
 \centering
 \includegraphics{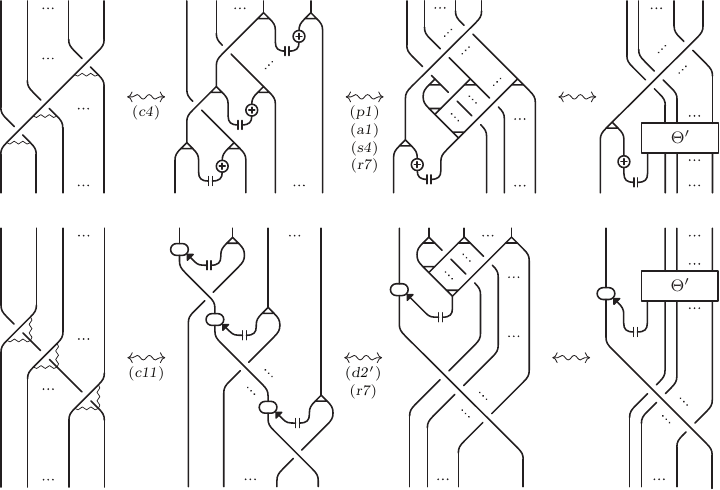}
 \caption{Proof of relations~\hrel{c18} and \hrel{c25}.}
 \label{isotopy3b/fig}
\end{figure}

\begin{figure}[htb]
 \centering
 \includegraphics{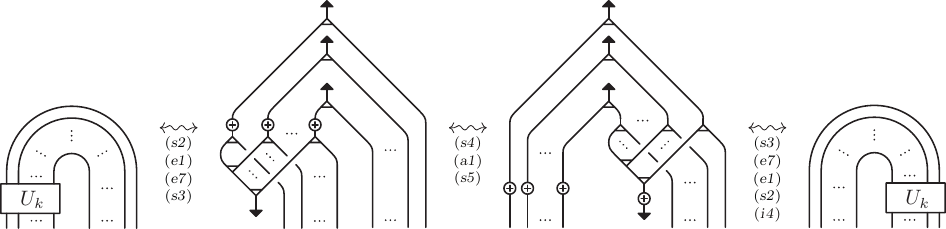}
 \caption{Proof of relation~\hrel{u1}.}
 \label{isotopy4/fig}
\end{figure}

\FloatBarrier

\begin{definition}\label{gen-istop-dec/def} \label{E:K} \label{E:Khat}
We say that two diagrams representing the same morphism in $\AlgD$ are related by a \textit{generalized isotopy move of type~\rel{K}} if they can be obtained from each other by a sequence of moves~\hrel{e3-3'} in Table~\ref{table-BPHopf-prop1/fig} and moves~\hrel{c14} to \hrel{c19} in Table~\ref{table-decorated-moves/fig}, and we say that they are related by a \textit{generalized isotopy move of type~\rel{\hatK}} if they can be obtained from each other by a sequence of moves~\hrel{e3-3'} in Table~\ref{table-BPHopf-prop1/fig} and moves~\hrel{c20} to \hrel{c25} in Table~\ref{table-decorated-moves/fig}. Notice that we also include in the family of generalized moves of type~\rel{K} (respectively, \rel{\hatK}) the substitution of a negative right or left kink by $\tau^{-2}$ (respectively, by $\id$) and the naturality of the inverse braiding morphism, that is, the analogue of moves~\hrel[c16]{c16-19} (respectively, \hrel[c22]{c22-25}) where $X$ (respectively, $\hat X$) has been replaced by $Y$ (respectively, $\hat Y$).
\end{definition}


\subsection{Bi-ascending states of link diagrams}
\label{bias/sec}

In \cite{BP11}, a key ingredient for inverting $\Phi$ was the notion of vertically trivial state of a link diagram. In the present more algebraic context, based on the presentation of the category $\KTf$ provided by Proposition~\ref{4KT/thm}, it seems convenient to replace that notion with the completely diagrammatic notion of bi-ascending state.

As usual, we represent a link $L \subset \R^3 \subset \R^3 \cup \{\infty\} \cong \bbS^3$ by a planar \textsl{diagram} $D \subset \R^2$ consisting of the orthogonal projection of the link onto $\R^2$, which can be assumed to be self-transverse after a suitable horizontal (that is, height-preserving) isotopy, together with a \textsl{crossing state} for each double point, encoding which arc passes over the other. Such a diagram $D$ uniquely determines the link $L$ up to vertical isotopy. On the other hand, link isotopy can be represented in terms of diagrams by crossing-preserving isotopy in $\R^2$ and Reidemeister moves.

It is well-known that any link diagram $D$ can be transformed into the diagram $D'$ of a trivial link by a suitable sequence of crossing changes, that is, by inverting the state of some of its crossings. We say that $D'$ is a \textsl{trivial state} of $D$.

The simplest trivial states of a link diagram $D$, are given by so-called ascending states (see \cite{Li97}). Bi-ascending states of $D$ form a larger family of trivial states of $D$ satisfying the following crucial property (which does not hold for ascending states): any two bi-ascending states of the same knot diagram can be related by a finite sequence of bi-ascending states, each obtained from the previous by inverting a single crossing (see Proposition~\ref{ba-states/thm} below). Before defining the notion of bi-ascending diagram, we need to introduce some terminology.

Given a diagram $D$ of a link $L = L_1 \cup \dots \cup L_n \subset \R^3$, where each $L_i \subset L$ is a component of $L$, we write $D = D_1 \cup \dots \cup D_n \subset \R^2$, with each $D_i \subset D$ being the subdiagram of $D$ corresponding to $L_i$, and we refer to each $D_i \subset D$ as a \textsl{component} of $D$. Similarly, by an \textsl{arc} $A \subset D$ we mean any part of $D$ corresponding to the projection to an arc in $L$ (not only the arcs ending at two consecutive under-crossings, as usual). Moreover, we say that $A$ is an \textsl{ascending} arc with respect to a given orientation if, at each of its self-crossings, the subarc that comes first passes under the other one.

\begin{definition}\label{ba-diagram/def}
A link diagram $D$ is said to be \textsl{bi-ascending} if it is possible to number its components $D_1, \dots, D_n$ and to choose on each $D_i$ an orientation and two distinct points $p_i$ and $q_i$ away from the crossings of $D$ in such a way that, if we denote by $A_i^\pm$ the two oriented arcs from $p_i$ to $q_i$ in $D_i$ (with the sign $+$ for the arc whose orientation coincides with the chosen one for $D_i$), the following properties hold:
\begin{enumerate}
\item[\(a)] $D_i$ crosses always over $D_j$, for every $1 \leqs i < j \leqs n$;
\item[\(b)] $A_i^+$ crosses always over $A_i^-$, for every $1 \leqs i \leqs n$;
\item[\(c)] $A_i^\pm$ are both ascending arcs, for every $1 \leqs i \leqs n$.
\end{enumerate}
\end{definition}

We note that the crossings of a bi-ascending diagram $D$, as specified in the definition, are compatible with a height function which vertically separates the components, and whose restriction to each component $D_i$ has a single local minimum at $p_i$ and a single local maximum at $q_i$. Therefore, any bi-ascending diagram represents a trivial link. In particular, bi-ascending diagrams whose arcs $A_i^-$ form no crossing coincide with ascending ones.

In the following, we simply refer to a bi-ascending trivial state of a link diagram $D$ as a \textsl{bi-ascending state} of $D$. Given a link diagram $D$, for any choice of the numbering and orientations of its components $D_i$ and of different non-crossing points $p_i$ and $q_i$ along each $D_i$, there is a unique bi-ascending state $D'$ of $D$ which satisfies the properties in the above definition, taking into account the canonical correspondence between the components $D_i$ of $D$ and the components $D_i'$ of $D'$. On the other hand, different choices can lead to the same bi-ascending state.

The next proposition is an analog of \cite[Proposition~1.1.3]{BP11} for bi-ascending states of a diagram.

\begin{proposition}\label{ba-states/thm}
Any two bi-ascending states $D'$ and $D''$ of a link diagram $D$ are related by a finite sequence $D^{(0)}, D^{(1)}, \dots, D^{(k)}$ of bi-ascending states of $D$ such that $D^{(0)} = D'$, $D^{(k)} = D''$ and, for every $1 \leqs i \leqs k$, the state $D^{(i)}$ is obtained from $D^{(i-1)}$ either by changing all the crossings between two vertically adjacent components, or by changing a single self-crossing of one component. Moreover, in the second case, the singular diagram between $D^{(i-1)}$ and $D^{(i)}$ (whose changing crossing has been replaced by a singular point) is a bi-ascending diagram of a trivial singular link. Namely, its components are vertically separated, meaning that they satisfy the property \(a) of Definition \ref{ba-diagram/def}, and are all bi-ascending diagrams of unknots but one, which is the 1-point union of two vertically separated bi-ascending diagrams of unknots.
\end{proposition}

\begin{proof} 
Changing all the crossings between two vertically adjacent components in a bi-ascending state of $D$ has the effect of transposing those components. Then, by iterating this kind of operation, we can permute components as we want. Therefore, we are left to address the case when $D$ is a knot diagram.

Let $D' \subset \R^2$ be a bi-ascending state of a knot diagram $D$. Then, the crossings of $D'$ are uniquely determined by the choice of an orientation on $D$ and two distinct non-crossing points $p$ and $q$ splitting $D$ into two ascending arcs $A^\pm$ from $p$ to $q$ such that $A^+$ is positively oriented and crosses always over $A^-$.

Let us fix for the moment the orientation, and see what happens to the induced bi-ascending state $D'$ when we move one of the points $p$ and $q$ along $D$ while keeping it distinct from the other. The crossings of $D'$ do not change until the moving point passes through a crossing of $D$, in which case we have one of the four situations depicted in Figure~\ref{ba-states/fig}, depending on which is the moving point ($p$ on the left-hand side of the figure, $q$ on the right-hand side) and what is the relative position of the other point along the diagram. As a simple inspection shows, in the two top cases only the crossing which is passed through by the moving point changes in $D''$, while no crossing change occurs in the two bottom cases.

\begin{figure}[htb]
 \centering
 \includegraphics{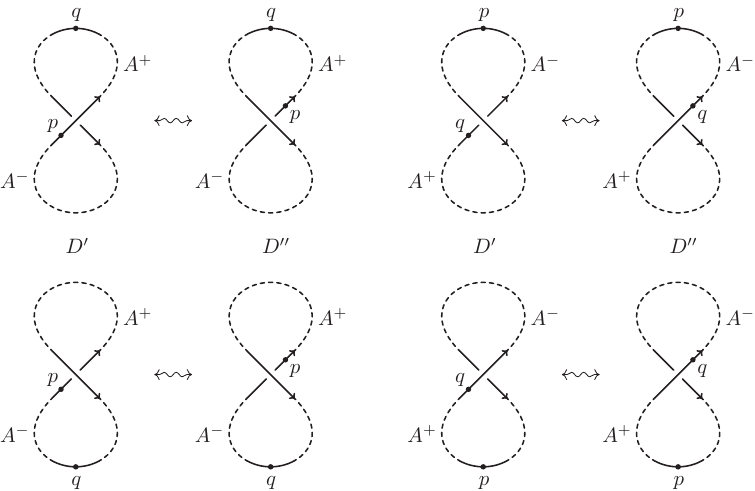}
 \caption{Letting $p$ or $q$ pass through a single crossing of $D$.}
 \label{ba-states/fig}
\end{figure}

This way, we can relate any two bi-ascending states of $D$ determined by the same orientation and by different choices of the points $p$ and $q$. In particular, we can relate any bi-ascending state to an ascending one.

Concerning the orientation of $D$, it is enough to observe that its inversion does not affect the induced state $D'$ when this is an ascending state. In fact, in this case
the interchange of the two arcs $A^+$ and $A^-$ is irrelevant, since there is no crossing between them, and hence property \(b) in Definition \ref{ba-diagram/def} is vacuous.

For the second part of the statement, let $D'$ be a bi-ascending state of $D$, and suppose that we pass from $D'$ to a bi-ascending state $D''$ of $D$ that differs from $D'$ by a single self-crossing change of a single component. We can focus on the changing component and forget the others, that is, we can assume that $D$ is a knot diagram. Moreover, according to Definition~\ref{ba-states/thm} and to the proof of the first part of the statement above, we can also assume that $D'$ and $D''$ are bi-ascending states of $D$ determined by the same orientation and by different choices for the points $p$ and $q$, and that they are related as in the top line of Figure~\ref{ba-states/fig}.

In both cases, once the changing crossing is replaced by a singular point $s$, the resulting loops are easily seen to be bi-ascending, with one always crossing over the other. Namely, if the moving point is $p$, then the upper loop is bi-ascending and determined by $s$ and $q$ with the inherited orientation, while the lower one is ascending and starting from $s$ with the opposite orientation. On the other hand, if the moving point is $q$, then the upper loop is ascending and starting from $s$ with the inherited orientation, while the lower one is bi-ascending and determined by $p$ and $s$ with the opposite orientation.
\end{proof}


\subsection{Definition of the inverse functor \texorpdfstring{$\barPhi: \KTf \to \Algf$}{from 4KT to 4Alg}}
\label{FK/sec}

Given a Kirby tangle $T: E_{2s} \to E_{2t}$ in $\KTf$, we will now explain how to construct a morphism $\barPhi(T): H^s \to H^t$ in $\Algf$ whose image $\Phi(\barPhi(T))$ under the functor $\Phi$ is the \qvlnc{2} class of $T$. The construction depends on some choices, but in the next subsections we will show that different choices lead to equivalent morphisms in $\Algf$, so that $\barPhi(T)$ depends only on the Kirby tangle $T$ up to \dfrmtns{2}. Moreover, the assignment respects compositions and identities, therefore it defines a functor $\barPhi: \KTf \to \Algf$.

\begin{figure}[b]
 \centering
 \includegraphics{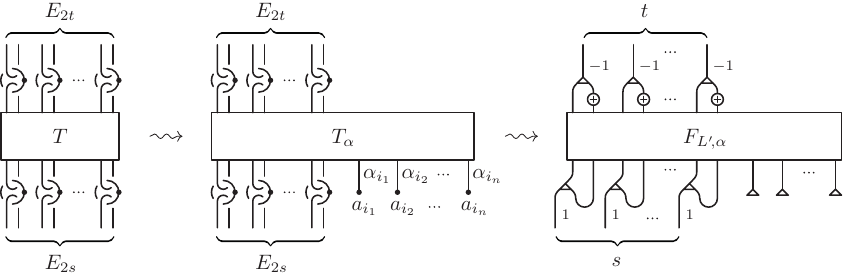}
 \caption{Outline of the construction of $\barPhi(T)$.}
 \label{kirby-hopf01/fig}
\end{figure}

We represent $T$ by a strictly regular planar diagram, which is a composition of tensor products of elementary diagrams in Table~\ref{table-Ktangles/fig}, and, up to composing it on the top and on the bottom with the diagrams in Figure \ref{KT-morph/fig} representing the identity morphisms, we will assume that $T$ is of the form represented in the leftmost diagram of Figure~\ref{kirby-hopf01/fig} (see Proposition~\ref{K-pres/thm} and Remark~\ref{KTbis/rmk}). Let $B_1, \dots, B_m \subset \left] 0,1 \right[^2$ be the planar projections of the disjoint disks spanned by the dotted unknots $U_1, \dots, U_m $ of $T$, and let $L$ be the strictly regular planar subdiagram which represents the blackboard framed link formed by the closed undotted components of $T$. Then, the construction of the morphism $\barPhi(T)$ is achieved by the following steps, illustrated in Figure~\ref{MainExample/fig}):

\begin{enumerate}

\item[\(1)] Choose a numbering $L = L_1 \cup \dots \cup L_n$ of the components of $L$ and, on each component $L_i$, choose both an orientation and a pair of points $p_i$ and $q_i$ inducing a bi-ascending state $L' = L'_1 \cup \dots \cup L'_n$ of $L$. In particular, we require that $L_i'$ crosses always over $L_j'$ for any $1 \leqs i < j \leqs n$, and that the positively oriented ascending arc of $L'_i$ determined by $p_i$ and $q_i$ crosses over the negatively oriented one for any $1 \leqs i \leqs n$ (see Definition~\ref{ba-diagram/def}). Mark with small gray disks $C_1, \dots, C_\ell$ the crossings of $L$ that have to be inverted in order to get $L'$ (see Figure~\ref{kirby-hopf02/fig}).

\begin{figure}[htb]
 \centering
 \includegraphics{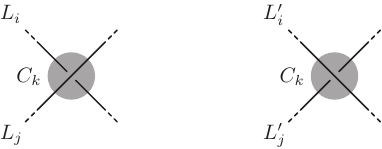}
 \caption{The disk $C_k$ and the diagrams $D$ and $D'$ at a changing crossing, 
 $i \leqs j$.}
 \label{kirby-hopf02/fig} 
\end{figure}

\item[\(2)] Fix $n$ horizontally aligned points $a_1,a_2, \dots, a_n$ on the bottom right of $T$ (their numbering is not required to respect the horizontal order), and choose $n$ embedded arcs $\alpha_i: [0,1] \to [0,1]^2$ such that $\alpha_i(0) = a_i$ and $\alpha_i(1) = b_i \in L_i$. The arcs $\alpha_i$ are required to form regular crossings both with $L$ and among themselves, with crossing states that can be arbitrarily chosen (see the middle diagram in Figure~\ref{kirby-hopf01/fig}). Assume also that each $\alpha_i$ avoids the crossings of $L$, the points $p_i$ and $q_i$, the disks $B_1, \dots, B_m$, and local maxima and minima of $L$ in the plane diagram. Since $L_i'$ is a bi-ascending state of $L_i$, the points $p_i$, $q_i$, and $b_i$ divide $L_i$ in three arcs $L_i = L_i^1 \cup L_i^2 \cup L_i^3$, numbered in such a way that either $b_i=L_i^1 \cap L_i^2$ or $b_i=L_i^2 \cap L_i^3$ and that, if we denote by $(L_i^j)'$ the corresponding arcs of  $L_i'$, then $(L_i^j)'$ crosses always over $(L_i^k)'$ if $j < k$ (see Figure~\ref{choice-alpha/fig}, where the arrows indicate the preferred orientation of the bi-ascending state of $L_i$). For every $1 \leqs i \leqs n$, set $L_{i,\alpha} = L_i \cup \alpha_i$ and $L'_{i,\alpha} = L'_i \cup \alpha_i$. Furthermore, set $\alpha = \{ \alpha_1, \ldots, \alpha_n \}$, $L_{\alpha} = L_{1,\alpha} \cup \dots \cup L_{n,\alpha}$, and $L'_{\alpha} = L'_{1,\alpha} \cup \dots \cup L'_{n,\alpha}$. Then, for every $1 \leqs i \leqs n$, mark with small gray disks as above the crossings formed by the arc $\alpha_i$ and by $L$ in which one of the following things happen:
\begin{enumerate}
\item[\(a)] $\alpha_i$ crosses either under $L_{j,\alpha}$ for $i < j$ or over $L_{j,\alpha}$ for $i > j$;
\item[\(b)] $\alpha_i$ crosses either under $L_i^2 \cup L_i^3$ or over $L_i^1$ with $b_i = L_i^1 \cap L_i^2$;
\item[\(c)] $\alpha_i$ crosses either under $L_i^3$ or over $L_i^1 \cup L_i^2$ with $b_i = L_i^2 \cap L_i^3$.
\end{enumerate}

\begin{figure}[htb]
 \centering
 \includegraphics{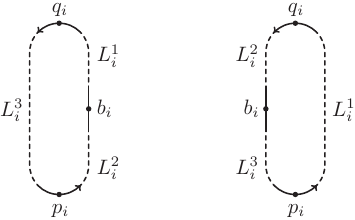}
 \caption{Possible subdivisions of $L_i=L_i^1\cup L_i^2\cup L_i^3$.}
 \label{choice-alpha/fig}
\end{figure}

Finally, without loss of generality, we assume that $\alpha_i$ reaches $L_i$ in a neighborhood of $b_i$ always coming from the bottom right, as shown in Figure~\ref{kirby-hopf03-1/fig} below.

\item[\(3)] Replace each elementary diagram of $L_\alpha$ as indicated by the arrows on the left-hand side of Figures~\ref{kirby-hopf03-1/fig} and \ref{kirby-hopf03-2/fig}, where $f_i = 1 - \wr(L_i')$ with $\wr(L_i')$ denoting the algebraic sum of the signs of all crossings in $L_i'$. In particular, the replacement for a crossing depends on whether it is marked as a changing crossing or not, the image being $X$ or $Y$ in the case of a unmarked crossing, and $\hat X$ or $\hat Y$ otherwise. Replace also the dotted components of $T$ and the identity morphisms lying outside of the $T_\alpha$-labeled box as prescribed by the arrows on the left-hand side of Figure~\ref{kirby-hopf03-3/fig}. Then $\barPhi(T) = \barPhi_{L',\alpha} (T)$ is defined as (see the right-hand side of Figure~\ref{kirby-hopf01/fig})
\[
\barPhi(T) = \barPhi_{L',\alpha}(T) = \Tar^{\otimes t} \circ F_{L',\alpha} \circ \left( \Sou^{\otimes s} \otimes \unitH^{\otimes n}  \right),
\]
where $\Tar = \prodH \circ (\ribmorH^{-1} \otimes \antipH)$, where $\Sou = (\prodH \otimes \ribmorH) \circ (\id \otimes \coev)$ (see Figure~\ref{kirby-hopf03-3/fig}), and where the morphism $F_{L',\alpha}$ belongs to the subcategory $\AlgD$ of $\Algf$.

The notation $\barPhi_{L',\alpha}(T)$ highlights the choice of the bi-ascending state $L'$ and of the arcs $\alpha_i$ for $1 \leqs i \leqs n$, as required by the construction. In Subsection~\ref{barPsi/sec} (see Propositions~\ref{indep-bands/thm}, \ref{vert-exchange/thm}, \ref{vert-st-comp/thm}, and \ref{sliding/thm}) we will show that the \qvlnc{2} class of $\barPhi_{L',\alpha}(T)$ is independent of such choices, and that it only depends on the \qvlnc{2} class of $T$. This will justify the notation $\barPhi(T)$.

\end{enumerate}

\begin{figure}[htbb]
 \centering
 \includegraphics{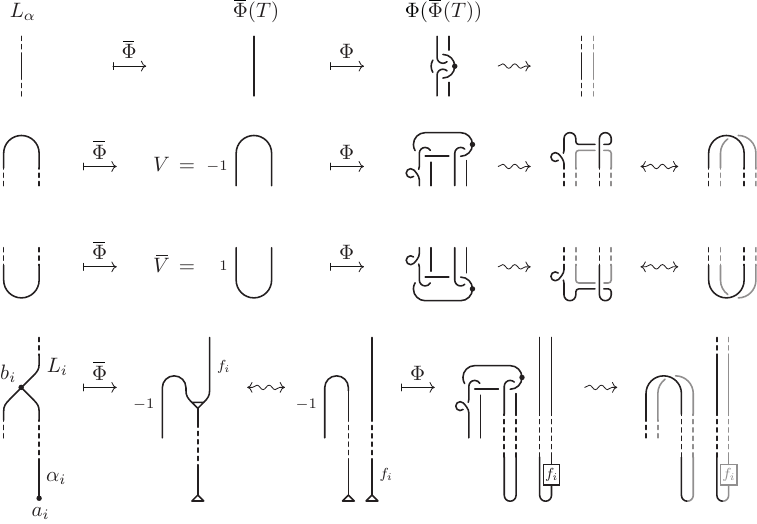}
 \caption{Definition of $\barPhi(T)$ and its image under $\Phi$ -- Part 1 (the box with label $f_i$ stands for $f_i$ kinks).}
 \label{kirby-hopf03-1/fig}
\end{figure}

\begin{figure}[htb]
 \centering
 \includegraphics{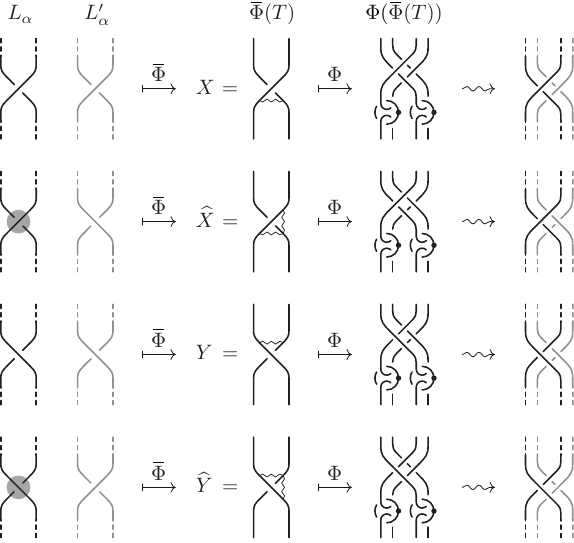}
 \caption{Definition of $\barPhi(T)$ and its image under $\Phi$ -- Part 2.}
 \label{kirby-hopf03-2/fig}
\end{figure}

\begin{figure}[htb]
 \centering
 \includegraphics{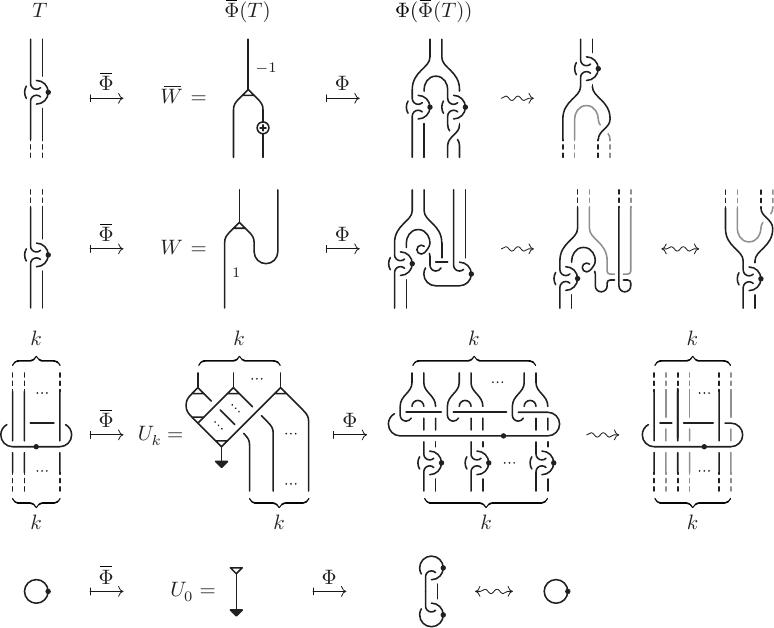}
 \caption{Definition of $\barPhi(T)$ and its image under $\Phi$ -- Part 3 ($k \geqs 1$).
 }
 \label{kirby-hopf03-3/fig}
\end{figure}

\FloatBarrier

Later, in Proposition~\ref{indep-bands/thm}, we will show that the conditions imposed in Step~\(2) on the arcs $\alpha_i$ for $1\leqs i\leqs n$ can be weakened to exclude \(b) and \(c). Nevertheless, considering for now only arcs $\alpha_i$ that satisfy those two conditions as well makes it much easier to see that $\barPhi$ is the inverse of $\Phi$ in the next proposition.
 
\begin{proposition}
\label{full-phi/thm}
$\Phi(\barPhi(T)) = T$ for every Kirby tangle $T$ in $\KTf$.
\end{proposition}

\begin{proof}
Before applying the functor $\Phi$, we modify $\barPhi(T)$ by sliding the coproduct $\coprH$ that appears in the image of the attaching point of each $\alpha_i$ (see the last line in Figure~\ref{kirby-hopf03-1/fig}) along $\barPhi(\alpha_i)$\, until it reaches the unit at its end, and then apply \hrel{a7} in Table~\ref{table-Hopf/fig} to split $\barPhi(\alpha_i)$ into two parallel arcs (see the first two steps in Figure~\ref{kirby-hopf01-1/fig}). We observe that the operation of sliding along $\coev$ and $\ev$ morphisms transforms $\coprH$ into $\tilde{\prodH}$ and vice-versa, as indicated by the arrows in relations~\hrel{q2} and \hrel{q3} in Table~\ref{table-mu/fig},  
while sliding $\coprH$ and $\tilde{\prodH}$ through decorated crossings uses the naturality of the two braided structures of $\AlgD$ given by relations~\hrel[c18]{c18-19} and \hrel[c24]{c24-25} in Table~\ref{table-decorated-moves/fig}. The resulting morphism, denoted by $F_{L',\beta}$, still lies in the subcategory $\AlgD$ of $\Algf$.

\begin{figure}[htb]
 \centering
 \includegraphics{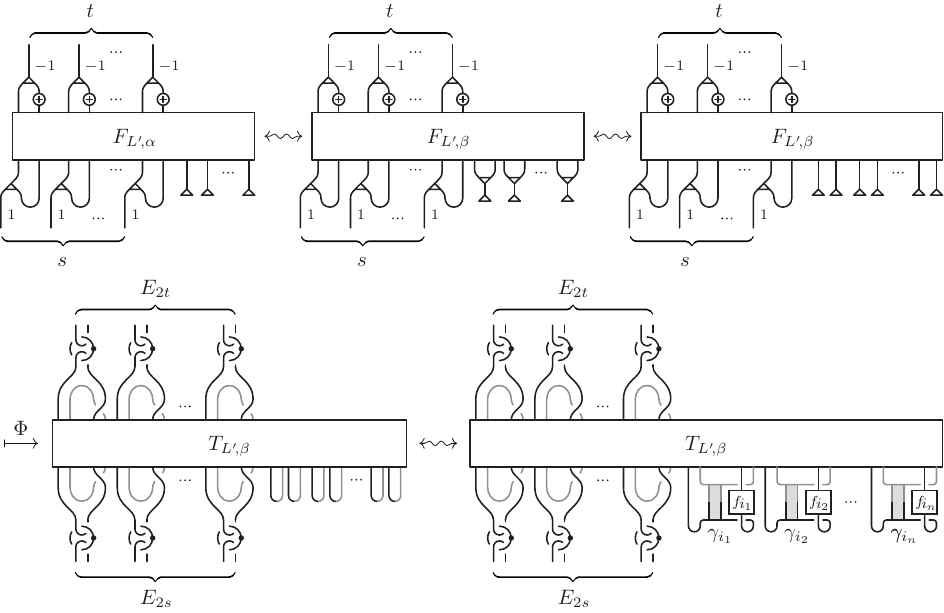}
 \caption{$\Phi \circ \barPhi(T)$.}
 \label{kirby-hopf01-1/fig}
\end{figure}

Since $\Phi$ is a monoidal functor, the morphism $\Phi(\barPhi(T))$ is given by the corresponding composition of tensor products of diagrams represented on the right-hand side of Figures~\ref{kirby-hopf03-1/fig}, \ref{kirby-hopf03-2/fig}, and \ref{kirby-hopf03-3/fig}, where the rightmost diagrams are obtained from the previous ones by \hndl{2} slides and \hndl{1/2} cancellations.

Comparing $T$ and $\Phi(\barPhi(T))$, we observe the following.
\begin{itemize}
\item Each component $L_i$ has been isotoped by pulling a small arc in a neighborhood of $b_i$ all the way down to the bottom-right part of the diagram through a narrow blackboard-parallel band $\beta_i$ obtained by doubling $\alpha_i$. Denote by $L_{\beta} = L_{\beta,1} \cup \dots \cup L_{\beta,n}$ the resulting link diagram. Observe that, in Step~\(2) above, the signed crossings between $\alpha_i$ and $L_{\alpha}$ have been chosen in such a way that, by inverting both them and the signed crossings identified in Step~\(1), we obtain a bi-ascending state $L_\beta'= L_{\beta,1}' \cup \dots \cup L_{\beta,n}'$ of $L_{\beta}$ with respect to the same choice of numbering, orientations, and points $p_i$ and $q_i$.
\item The link $L_\beta$, represented in black in the rightmost diagrams in Figures~\ref{kirby-hopf03-1/fig}, \ref{kirby-hopf03-2/fig}, and \ref{kirby-hopf03-3/fig}, has been ``doubled'' by a copy of the trivial link, represented by the bi-ascending diagram $L'_\beta$ in gray, which lies below the original Kirby tangle $T$. Moreover, we observe that, by definition, $f_i$ compensates the contribution to the (blackboard) framing of $L_{\beta,i} \#_{\gamma_i} L'_{\beta,i}$ due to $\wr(L'_i)$ and to the negative kink near $\gamma_i$, and hence we are left with $\wr(L_{\beta,i} \#_{\gamma_i} L'_{\beta,i}) = \wr(L_{\beta_i}) = \wr(L_i)$, which is the original (blackboard) framing of $L_i$. 
\item Each component $L'_{\beta,i}$ is connected to the corresponding component $L_{\beta,i}$ by a band $\gamma_i$, shown in gray in the bottom-right part of Figure~\ref{kirby-hopf01-1/fig}, that merge the ends of the two copies of $\beta_i$ in $L_{\beta,i}$ and $L'_{\beta,i}$.
\end{itemize}
A three-dimensional view of $\Phi(\barPhi(T))$ in the spacial case when the points $a_i$ are ordered from left to right is presented in Figure~\ref{kirby-hopf05/fig}. Therefore, the Kirby tangle $\Phi(\barPhi(T))$ can be isotoped to the original one $T$ by first contracting all the unknots of $L'_{\beta,i}$, and then retracting the corresponding bands $\beta_i$ one by one.
\end{proof}

\begin{figure}[htb]
 \centering
 \includegraphics{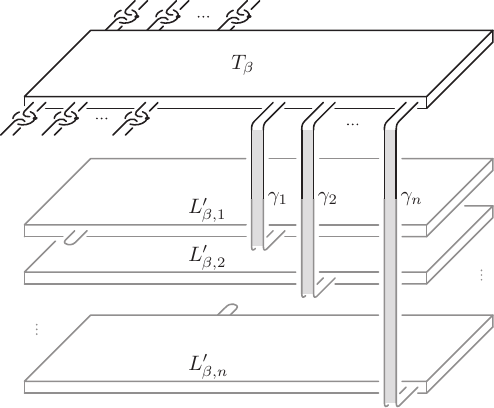}
 \caption{Three-dimensional view of $\Phi(\barPhi(T))$ in the spacial case when the points $a_i$ are ordered from left to right. Notice that, for levels from $1$ to $n$, there is either a single cap on the top (like for level $n$ in this example), or a single cup on the bottom (like for level $1$ in this example), or nothing at all (like for level $2$ in this example), depending on whether the corresponding undotted component of the original tangle $T$, before composing with identities, was either open on the top, or open on the bottom, or closed, respectively.}
 \label{kirby-hopf05/fig}
\end{figure}

\subsection{The category \texorpdfstring{$\AlgL$}{MAlg}}
\label{AlgL/sec}

As we have seen in the previous subsection, $\barPhi(T)$ encodes algebraically a multiple-level Kirby tangle, and in order to prove that $\barPhi$ is a well-defined functor, we need to develop suitable algebraic tools that allow us to work with such structures. The main idea is to consider a category $\AlgL$ that is similar to $\AlgD$, but whose objects and morphisms carry labels. In other words, objects are tensor products $H_{\underline{i}} = H_{i_1} \otimes H_{i_2} \otimes \ldots \otimes H_{i_k}$ with $i_\ell \geqs 0$ for $0 \leqs \ell \leqs k$, while morphisms are labeled versions of the corresponding morphisms of $\Algf$. The images of the morphisms of $\AlgL$ under $\Phi$ satisfy the same conditions~\(a) and \(b) in Remark~\ref{TAlg/rmk} as morphisms of $\AlgD$, but, in addition, the lower (gray) arc of each undotted component of label $i$ stays above the lower (gray) arc of each undotted component of label $j>i$; in other words, labels denote the ``depth'' of those arcs in the corresponding Kirby tangle.

Here is the formal definition.

\begin{definition}\label{AlgL/def}
We denote by $\AlgLfree$ the monoidal category freely generated by objects $H_i$ for $i \geqs 0$ and by morphisms
\begin{gather*}
 \ev_i : H_i \otimes H_i \to \one \mbox{ for } i \geqs 1 ,\\
 \coprH_i : \one \to H_i \otimes H_i \mbox{ for } i \geqs 1,\\
 \counH_i : H_i \to \one \mbox{ for } i \geqs 1,\\
 \inteH_i : \one \to H_i \mbox{ for } i \geqs 1,\\
 \ribmorH_i : H_i \to H_i \mbox{ for } i \geqs 1,\\
 X_{i,j}, \hat Y_{i,j} : H_i \otimes H_j \to H_j \otimes H_i \mbox{ for } 1 \leqs i \leqs j,\\
 \hat X_{i,j}, Y_{i,j} : H_i \otimes H_j \to H_j \otimes H_i \mbox{ for } i \geqs j \geqs 1,\\
 \Sou_i : H_0 \to H_i \otimes H_i \mbox{ for } i \geqs 1,\\
 \Tar_i : H_i \otimes H_i \to H_0 \mbox{ for } i \geqs 1,\\
 U_{\underline{i}} = U_{i_1,\ldots,i_k} : H_{i_1} \otimes \ldots \otimes H_{i_k} \to H_{i_1} \otimes \ldots \otimes H_{i_k} \mbox{ for } k \geqs 1 \mbox{ and } i_1,\ldots,i_k \geqs 1
\end{gather*}

Let $\Forget: \AlgLfree \to \Algf$ denote the forgetful functor that discards labels; in particular, $\Forget(H_i)=H$ for every $i \geqs 0$, $\Forget(U_{i_1,\ldots,i_k})=U_k$, and each of the remaining generating morphisms is sent by $\Forget$ to the morphism of $\Algf$ carrying the same name without indices. Then, we denote by $\AlgL$ the quotient category $\AlgLfree / \ker \Forget$\footnote{If $F : \calC \to \calC'$ is a functor, then, using the terminology of \cite[Section~II.8]{Ma71}, we denote by $\ker F$ the congruence obtained by setting $f \sim g \in \calC(x,y)$ whenever $F(f) = F(g) \in \calC'(F(x),F(y))$.}. 
\end{definition}

The diagrammatic notation for the morphisms in the image of $\Forget$ is introduced in Figure~\ref{image-F/fig}. In particular, we represent $\Forget(F_{\underline{i},\underline{j}})$ by a box that contains in its lower (respectively upper) part the labels of the string in the source (respectively target) of $F_{\underline{i},\underline{j}}$.

\begin{figure}[htb]
 \centering
 \includegraphics{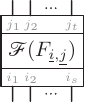}
 \caption{Diagrammatic notation for $\Forget(F_{\underline{i},\underline{j}})$, where $\underline{i}=(i_1,\ldots,i_s)$ and $\underline{j} = (j_1,\ldots,j_t)$.}
 \label{image-F/fig}
\end{figure}

We will now define a family of natural transformations $\Theta^\rmL_k$ for $k \geqs 1$ designed to provide the algebraic analogue of a dotted component that embraces the $k$th level, while passing below the $i$th level, for $0 \leqs i \leqs k-1$, and above the $j$th level, for $j \geqs k+1$.

\begin{definition}\label{theta1/defn}
For all $k \geqs 1$ and $\underline{i} = (i_1,i_2,\ldots,i_\ell)$, with $\ell \geqs 0$ and $i_h \geqs 0$ for every $1 \leqs h \leqs \ell$, the morphisms $\gamma_{\underline{i},k} : H^\ell \otimes H \to H \otimes H^\ell$ and $\Theta^\rmL_{\underline{i},k} : H^\ell \otimes H \to H^\ell$ of $\Algf$ are recursively defined by the following identities:
\begin{gather*}
 \gamma_{\varnothing,k} = \id, \quad \Theta^\rmL_{\varnothing,k} = \counH, \\
 \gamma_{i,k} = \gamma_{(i),k} = 
 \begin{cases}
  \braid & \mbox{ if  } i \leqs k, \\ 
  \hat{X} & \mbox{ if } i > k,
 \end{cases} \quad
 \Theta^\rmL_{i,k} = \Theta^\rmL_{(i),k} =
 \begin{cases}
  \prodH & \mbox{ if } i = k, \\ 
  \id \otimes \counH & \mbox{ if } i \neq k,
 \end{cases} \\
 \gamma_{\underline{i},k} = 
 (\gamma_{i_1,k} \otimes \id_{\ell-1}) \circ (\id \otimes \gamma_{(i_2,\ldots,i_\ell),k}), \\
 \Theta^\rmL_{\underline{i},k} = 
 (\Theta^\rmL_{i_1,k} \otimes \Theta^\rmL_{(i_2,\ldots,i_\ell),k}) 
 \circ (\id \otimes \gamma_{(i_2,\ldots,i_\ell),k} \otimes \id) 
 \circ (\id_\ell \otimes \coprH).
\end{gather*}
Notice that, in order to simplify notation, we use $i$ to denote non only the index $i$, but also the list $(i)$ of length 1.
\end{definition}

The definitions of $\Theta_{\underline{i},k}^\rmL$ and $\gamma_{\underline{i},k}$ are illustrated in Figure~\ref{def-thetaL/fig}. The gray strands drawn inside a box labeled by $\gamma_{\underline{i},k}$ should remind the reader that the corresponding morphism is very closely related to the braiding $c_{\ell,1}$, from which it is obtained by decorating with the pattern of $\hat X$ all crossings involving strands of label $i>k$.

\begin{figure}[htb]
 \centering
 \includegraphics{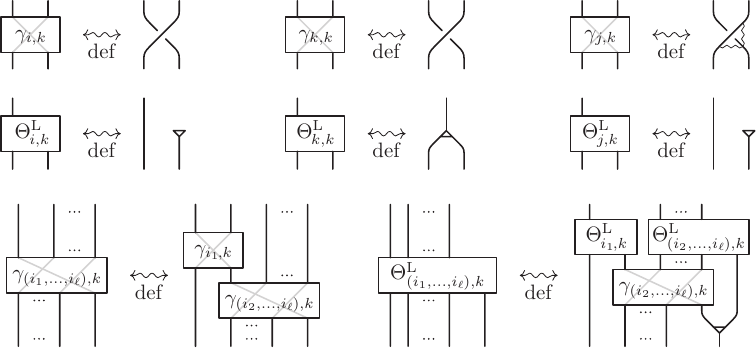}
 \caption{Inductive definition of $\Theta^\rmL_k$ ($i < k < j$).}
 \label{def-thetaL/fig} 
\end{figure}

We denote by $\Theta^\rmL_k$ the collection of morphisms $\{ \Theta^\rmL_{\underline{i},k} \mid \underline{i} \in \N^\ell, \ell \in \N\}$.

\begin{proposition}\label{theta1/thm}
Let $\Forget : \AlgL \to \Algf$ denote the forgetful functor that discards labels.
Then $\Theta^\rmL_k : \Forget \otimes H \Rightarrow \Forget$ defines a natural transformation, meaning that, for every morphism $F_{\underline{i},\underline{j}} : H_{\underline{i}} \to H_{\underline{j}}$ in $\AlgL$, we have (see Figure~\ref{naturality-theta/fig}):
\begin{align*}
 \Theta^\rmL_{\underline{j},k} \circ (\Forget(F_{\underline{i},\underline{j}}) \otimes \id) 
 &= \Forget(F_{\underline{i},\underline{j}}) \circ \Theta^\rmL_{\underline{i},k}.
 \tag*{\hrel{t2}}
\end{align*}
\end{proposition}

\begin{figure}[htb]
 \centering
 \includegraphics{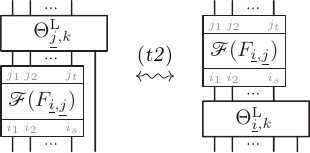}
 \caption{Naturality of $\Theta^\rmL_k$.}
 \label{naturality-theta/fig} \label{E:t2}
\end{figure}

Before giving a proof of Proposition~\ref{theta1/thm}, we present in Figure~\ref{theta-pi/fig} a specific example of the natural transformation $\Theta^\rmL_{\underline{i},k} : \Forget(H_{\underline{i}})\otimes H \to \Forget(H_{\underline{i}})$ and its image under the functor $\Phi$ (notice that, in the diagram appearing on the left-hand side, all counits have been retracted by applying \hrel{a4'}). We point out that, since $\Theta^\rmL_k$ is a natural transformation between functors with source $\AlgL$ and target $\Algf$, it is a collection of morphisms in the target category (which are unlabeled), one for every object in the source category (which is labeled). In other words, $\Theta^\rmL_{\underline{i},k}$ does not really carry labels, but its definition depends on the labeled object $H_{\underline{i}}$. Therefore, the labels attached to the morphisms represented in Figure~\ref{theta-pi/fig} and below indicate that these morphisms are in the image of $\Forget$, but keeping track of labels in pictures will allow us to understand which form of $\Theta^\rmL_k$ we need to use.

\begin{figure}[htb]
 \centering
 \includegraphics{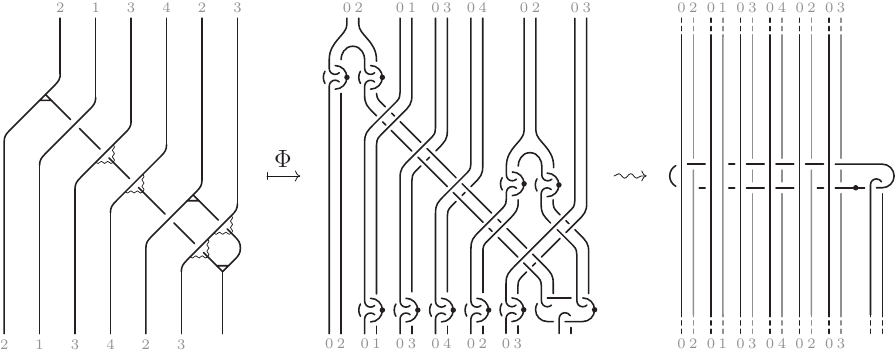}
 \caption{The morphism $\Theta^\rmL_{\underline{i},k}:\Forget(H_{\underline{i}})\otimes H \to \Forget(H_{\underline{i}})$ for $\underline{i} = (2,1,3,4,2,3)$ and $k = 2$, and its image under the functor $\Phi: \Algf \to \KTf$. }
 \label{theta-pi/fig}
\end{figure}

\FloatBarrier

\begin{lemma}\label{theta2/thm}
If $\underline{i} = (i_1, \dots, i_h, i_{h+1}, \dots, i_\ell)$ with $\ell \geqs 1$ and $1 \leqs h < \ell$, then
\[
 \Theta^\rmL_{\underline{i},k} = 
 (\Theta^\rmL_{(i_1, \dots, i_h),k} \otimes \Theta^\rmL_{(i_{h+1},\ldots,i_\ell),k}) 
 \circ (\id_h \otimes \gamma_{(i_{h+1}, \dots, i_\ell),k} \otimes \id) \circ (\id_\ell \otimes \coprH).
\]
\end{lemma}

\begin{proof} 
For $h = 1$ and for any $\ell \geqs 1$, the statement is true by definition of $\Theta^\rmL_{\underline{i},k}$. Then the claim follows by induction on $h$. The proof of the inductive step is presented in Figure~\ref{theta3/fig}, where the first step follows from the definition of $\Theta^\rmL_{\underline{i},k}$, while the second step follows from the inductive hypothesis and the decomposition of $\gamma_{(i_2,\dots,i_\ell),k}$ as $(\gamma_{(i_2,\dots,i_h),k} \otimes \id_{\ell-h}) \circ (\id_{h-1} \otimes \gamma_{(i_{h+1},\dots,i_\ell),k})$. Then, for the third step, we apply the coassociativity axiom~\hrel{a3} to collect together the rightmost strands in the sources of the two copies of $\gamma_{(i_{h+1},\dots,i_\ell),k}$, and use the naturality of the braiding of $\AlgD$, and in particular relation~\hrel{c24}, to push the resulting $\coprH$ past them. Finally, we apply once more the defining relation of $\Theta^\rmL_{(i_1,\dots,i_h),k}$.
\end{proof}

\begin{figure}[htb]
 \centering
 \includegraphics{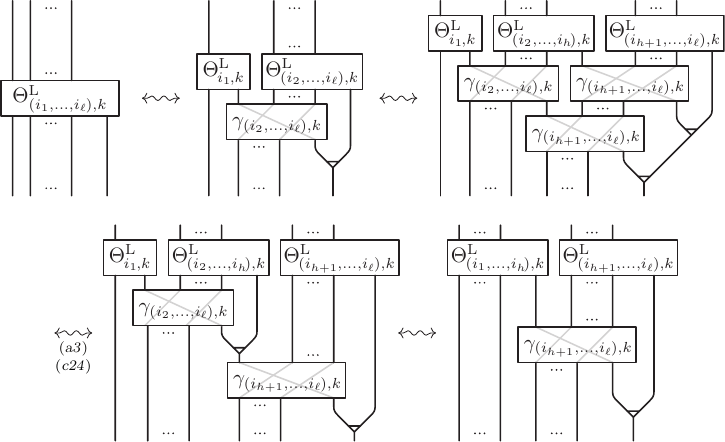}
 \caption{Proof of the inductive step of Lemma~\ref{theta2/thm}.}
 \label{theta3/fig}
\end{figure}

\begin{proof}[Proof of Proposition~\ref{theta1/thm}]
We will first prove the statement in the case where $F_{\underline{i},\underline{j}}$ is a generating morphism of $\AlgL$. We observe that \hrel{t2} holds trivially if none of the edges of $F_{\underline{i},\underline{j}}$ is labeled by $k$, while it follows from \hrel{t1} in Proposition~\ref{theta/thm} if all incoming and outgoing edges of $F_{\underline{i},\underline{j}}$ are labeled by $k$, since in this case both $\Theta^\rmL_{\underline{i},k}$ and $\Theta^\rmL_{\underline{j},k}$ coincide with $\Theta_k$.
Moreover, the proof for $\Tar_i$ is identical to the proof of \hrel{t1} for $V$ in Figure~\ref{theta2/fig}, while \hrel{t2} for $\Sou_i$ follows directly by applying the bialgebra axiom \hrel{a5} and the property of the integral element \hrel{i2'}. On the other hand, when $F_{\underline{i},\underline{i}} = U_{\underline{i}}$, then \hrel{t2} follows directly from the associativity axiom \hrel{a1}.

In order to complete the proof of \hrel{t2} for the generating morphisms of $\AlgL$, it remains to consider the case where $F_{\underline{i},\underline{j}}$ is a decorated crossing, meaning one of the morphisms $X_{i,j}, Y_{i,j}, \hat{X}_{i,j}, \hat{Y}_{i,j}$, with exactly one of the indices $i$ or $j$ equal to $k$. Since, according to Lemma~\ref{decorated-sym-inv/thm}, we have $Y_{i,j} = X_{j,i}^{-1}$ and $\hat Y_{i,j} = \hat X_{j,i}^{-1}$, it is enough to prove the statement for $X_{i,k}$ and $\hat X_{k,i}$ if $i < k$, and for $\hat{X}_{i,k}$ and $X_{k,i}$ if $i > k$. This is done in Figure~\ref{theta4/fig}.

Notice that, in this figure and in the next one, we use move~\hrel{c24} with $F = \prodH,\braid$, which are not morphisms of $\AlgD$. Indeed, even though Theorem~\ref{decorated-moves/thm} establishes relation~\hrel{c24} for every morphism $F$ in $\AlgD$, the same actually holds for every morphism $F$ in $\Algf$ since, according to the proof in Figure~\ref{isotopy3a/fig}, relation~\hrel{c24} reduces to applying relation~\hrel{d10} and the naturality of the braiding $\braid$ of $\Algf$, both of which hold in $\Algf$.

\begin{figure}[htb]
 \centering
 \includegraphics{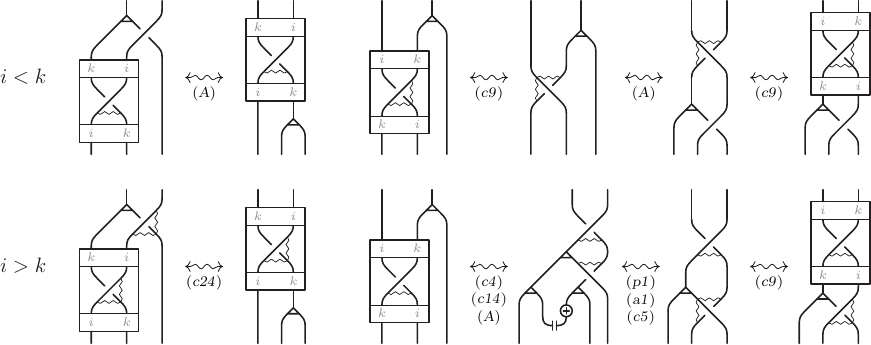}
 \caption{Naturality of $\Theta^\rmL_k$ with respect to decorated crossings with mixed labels (the extended version of move \hrel{c24} for $F = \prodH$ is used).}
 \label{theta4/fig}
\end{figure}

Now, by using an inductive argument on the number of tensor factors and Lemma~\ref{theta2/thm}, the claim will follow for every morphism 
$F_{\underline{i},\underline{j}} : H_{\underline{i}} \to H_{\underline{j}}$ in $\AlgL$ if we can show that
\[
\gamma_{\underline{j},k} \circ (\Forget(F_{\underline{i},\underline{j}}) \otimes \id) 
= (\id \otimes \Forget(F_{\underline{i},\underline{j}})) \circ \gamma_{\underline{i},k}. 
\eqno{\rel{g1}}
\label{E:g1}
\]
every time $F_{\underline{i},\underline{j}}$ is a generating morphism of $\AlgL$. Observe that, if no label of $F_{\underline{i},\underline{j}}$ is strictly greater than $k$, then \hrel{g1} follows from the naturality of the braiding. On the other hand, if all of its labels are strictly greater than $k$, then it follows from \hrel[c4]{c24-25}. 

Therefore, we are left to prove \hrel{g1} for morphisms $F_{\underline{i},\underline{j}}$ in which some of the labels are strictly greater than $k$, and some are not. In this case, $F_{\underline{i},\underline{j}}$ is either a decorated crossing, or $\Sou_i$ or $\Tar_i$ for $i > k$. For what concerns decorated crossings, we observe that, thanks to Lemma~\ref{decorated-sym-inv/thm} once again, it is enough to prove \hrel{g1} for $X_{i,j}$ and $\hat X_{j,i}$ with $i \leqs k < j$. This is done in Figure~\ref{theta5/fig}. The proofs for $\Tar_i$ and $\Sou_i$ with $i > k$ is shown in Figure~\ref{theta6/fig}. 
\end{proof}

\begin{figure}[htb]
 \centering
 \includegraphics{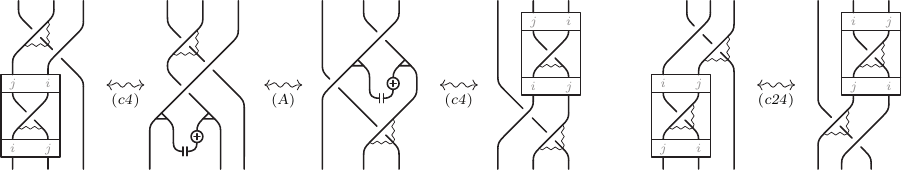}
 \caption{Naturality of $\gamma_{\underline{i},k}$ with respect to decorated crossings with $i \leqs k < j$ (the extended version of move \hrel{c24} for $F = c$ is used).}
 \label{theta5/fig}
\end{figure}

\begin{figure}[htb]
 \centering
 \includegraphics{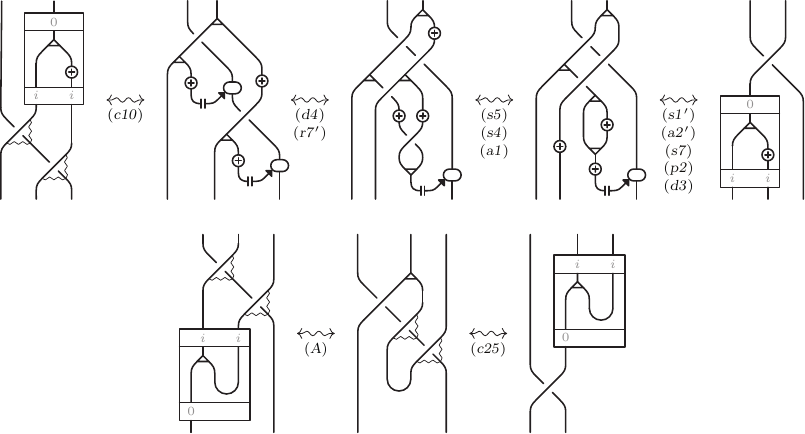}
 \caption{Naturality of $\gamma_{\underline{i},k}$ with respect to the morphisms $\Forget(\Tar_i)$ and $\Forget(\Sou_i)$ with $i > k$.}
 \label{theta6/fig}
\end{figure}

\FloatBarrier

Finally, we define a family of natural transformations $\hat{\Theta}^\rmL_k$ for $k \geqs 0$ designed to provide the algebraic analogue of a dotted component that embraces the $k$th level, while passing below the $i$th level, for $1 \leqs i \leqs k-1$, and above the $j$th level, for $k+1 \leqs j \leqs n$, composed with a positive double braiding between the strands of the $k$th and $(k+1)$st level.

\begin{definition}\label{ttheta1/defn}
For all $k \geqs 1$ and $\underline{i} = (i_1,i_2,\ldots,i_\ell)$, with $\ell \geqs 0$ and $i_h \geqs 0$ for every $1 \leqs h \leqs \ell$, the morphisms $\hat{\gamma}_{\underline{i},k} : H^\ell \otimes H \to H \otimes H^\ell$ and $\hat{\Theta}^\rmL_{\underline{i},k} : H^\ell \otimes H \to H^\ell$ in $\Algf$ are recursively defined by the following identities (see Figure \ref{def-thetahatL/fig}):
\begin{gather*}
 \hat{\gamma}_{\varnothing,k} = \id, \quad \hat{\Theta}^\rmL_{\varnothing,k} = \counH, \\
 \hat{\gamma}_{i,k} = \hat{\gamma}_{(i),k} =
 \begin{cases}
  \braid & \mbox{ if } i \leqs k, \\ 
  X & \mbox{ if } i = k+1, \\
  \hat{X} & \mbox{ if } i > k+1, 
 \end{cases} \quad
 \hat{\Theta}^\rmL_{i,k} = \hat{\Theta}^\rmL_{(i),k} = 
 \begin{cases}
  \prodH & \mbox{ if } i = k, \\
  \id \otimes \counH & \mbox{ if } i \neq k,
 \end{cases} \\
 \hat{\gamma}_{\underline{i},k} = (\hat{\gamma}_{i_1,k} \otimes \id_{\ell-1}) \circ (\id \otimes \hat{\gamma}_{(i_2,\ldots,i_\ell),k}), \\
 \hat{\Theta}^\rmL_{\underline{i},k} = (\hat{\Theta}^\rmL_{i_1,k} \otimes \hat{\Theta}^\rmL_{(i_2,\ldots,i_\ell),k}) \circ (\id \otimes \hat{\gamma}_{(i_2,\ldots,i_\ell),k} \otimes \id) \circ (\id_\ell \otimes \coprH).
\end{gather*}
\end{definition}

We denote by $\hat{\Theta}^\rmL_k$ the collection of morphisms $\{\hat{\Theta}^\rmL_{\underline{i},k} \mid \underline{i} \in \N^\ell, \ell \in \N\}$.
\begin{figure}[htb]
 \centering
 \includegraphics{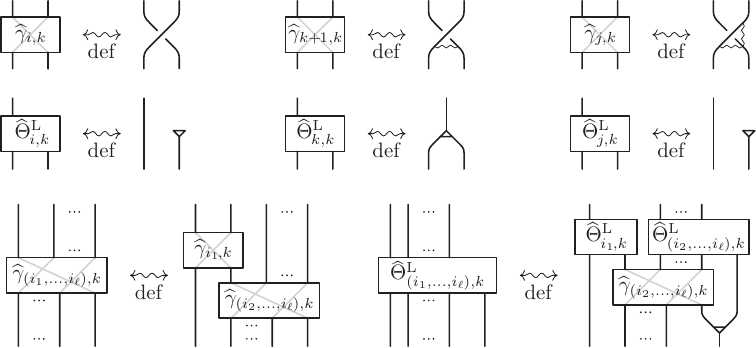}
 \caption{Inductive definition of $\hat{\Theta}^\rmL_k$ ($i < k+1 < j$).}
 \label{def-thetahatL/fig} 
\end{figure}

A specific example of the natural transformation $\hat{\Theta}^\rmL_{\underline{i},k}$ can be found in Figure~\ref{ttheta2/fig}. Observe that the only difference between $\Theta^\rmL_{\underline{i},k}$ and $\hat{\Theta}^\rmL_{\underline{i},k}$ is in the crossing with the $(k+1)$-labeled strand.

\begin{figure}[htb]
 \centering
 \includegraphics{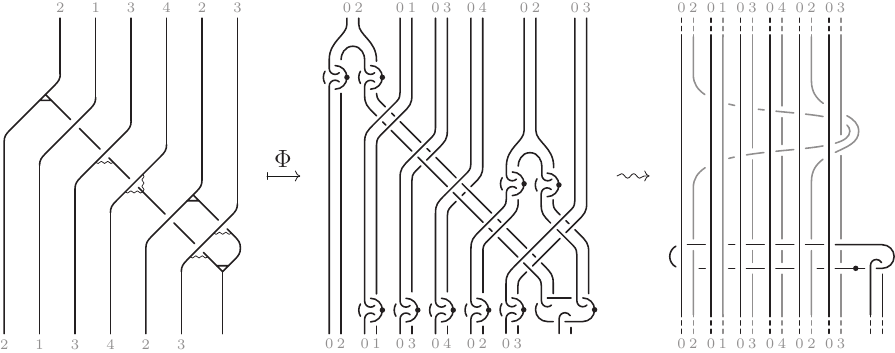}
 \caption{The morphism $\hat{\Theta}^\rmL_{\underline{i},k}:\Forget(H_{\underline{i}}) \otimes H \to \hat{\Forget}_k(H_{\underline{i}})$ for $\underline{i}=(2,1,3,4,2,3)$ and $k=2$.}
 \label{ttheta2/fig}
\end{figure}

\FloatBarrier

\begin{lemma}\label{ttheta1/thm}
Let $\underline{i} = (i_1,\dots i_h, i_{h+1},\dots,i_\ell)$ with $\ell \geqs 1$ and $1 \leqs h < \ell$, then
\[
\hat{\Theta}^\rmL_{\underline{i},k} = (\hat{\Theta}^\rmL_{(i_1,\dots i_h), k} \otimes \hat{\Theta}^\rmL_{(i_{h+1}, \ldots, i_\ell), k}) \circ (\id_h \otimes \hat{\gamma}_{(i_{h+1}, \ldots, i_\ell), k} \otimes \id) \circ (\id_\ell \otimes \coprH).
\]
\end{lemma}

\begin{proof} 
The proof proceeds by induction on $h$, and it is completely analogous to the proof of Lemma~\ref{theta2/thm} (see Figure~\ref{theta3/fig}).
\end{proof}

\begin{proposition}\label{ttheta2/thm}
For every $k \geqs 1$, there exists a unique functor $\hat{\Forget}_k : \AlgL \to \Algf$ that first exchanges the object $H_k$ with $H_{k+1}$ and the morphisms $X_{k,k+1}, \hat{Y}_{k,k+1} : H_k \otimes H_{k+1} \to H_{k+1} \otimes H_k$ with $\hat{X}_{k+1,k}, Y_{k+1,k} : H_{k+1} \otimes H_k \to H_k \otimes H_{k+1}$, respectively, and then discards labels. Furthermore, $\hat{\Theta}^\rmL_k : \Forget \otimes H \Rightarrow \hat{\Forget}_k$ defines a natural transformation, meaning that, for every morphism $F_{\underline{i},\underline{j}} : H_{\underline{i}} \to H_{\underline{j}}$ in $\AlgL$, we have (see Figure~\ref{naturality-theta-hat/fig})
\begin{align*}
 \hat{\Theta}^\rmL_{\underline{j},k} \circ (\Forget(F_{\underline{i},\underline{j}}) \otimes \id) &= \hat{\Forget}_k(F_{\underline{i},\underline{j}}) \circ \hat{\Theta}^\rmL_{\underline{i},k}. \tag*{\hrel{t3}}
\end{align*}
\end{proposition}

\begin{figure}[htb]
 \centering
 \includegraphics{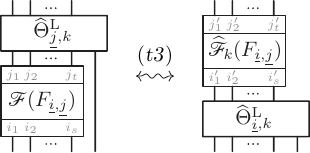}
 \caption{Naturality of $\hat\Theta^\rmL_k$ (the sequences $\underline{i}'$ and $\underline{j}'$ are obtained from $\underline{i}$ and $\underline{j}$, respectively, by exchanging $k$ and $k+1$ at any of their occurrences).}
 \label{naturality-theta-hat/fig} \label{E:t3}
\end{figure}

\begin{proof}
We start by proving that \hrel{t3} holds for any morphism $F_{\underline{i},\underline{j}}$ in $\AlgL$. Lemma~\ref{ttheta1/thm} implies that it is enough to show that \hrel{t3} and
\begin{align*}
 \hat{\gamma}_{\underline{j},k} \circ (\Forget(F_{\underline{i},\underline{j}}) \otimes \id) &= (\id \otimes \Forget(F_{\underline{i},\underline{j}})) \circ \hat{\gamma}_{\underline{i},k}
 \tag*{\rel{g2}}
 \label{E:g2}
\end{align*}
hold every time $F_{\underline{i},\underline{j}}$ is a generating morphism of $\AlgL$.

For what concerns \hrel{t3}, we observe that, if all labels of $F_{\underline{i},\underline{j}}$ are different from $k+1$, then \hrel{t3} reduces to \hrel{t2}, while if all labels are different from $k$, then it becomes trivial by applying \hrel{c18} to the counit $\counH$. Therefore, it is enough to show that \hrel{t3} holds whenever $F_{\underline{i},\underline{j}}$ is a generating morphism with mixed labels featuring at least one label equal to $k+1$ and another equal to $k$. In other words, it is enough to consider $F_{\underline{i},\underline{j}} = U_{\underline{i}}, X_{k,k+1}, \hat X_{k+1,k}, Y_{k+1,k}, \hat Y_{k,k+1}$. The statement for $U_{\underline{i}}$ follows directly from the associativity axiom \hrel{a1}. For what concerns decorated crossings, since $Y_{i,j} = X_{j,i}^{-1}$ and $\hat Y_{i,j} = \hat X_{j,i}^{-1}$ (see Lemma~\ref{decorated-sym-inv/thm}), it is enough to prove \hrel{t3} for $X_{k,k+1}$ and $\hat X_{k+1,k}$. This is done in Figure~\ref{ttheta3/fig}.

\begin{figure}[htb]
 \centering
 \includegraphics{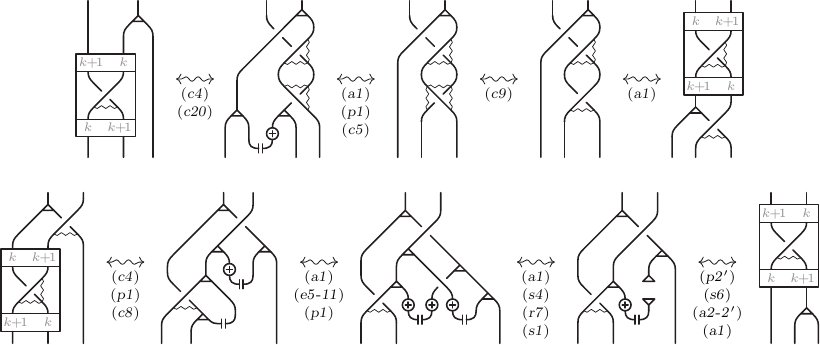}
 \caption{Naturality of $\hat{\Theta}^\rmL_k$ with respect to decorated crossings with mixed labels.}
 \label{ttheta3/fig}
\end{figure}

For what concerns \hrel{g2}, if all labels of $F_{\underline{i},\underline{j}}$ are different from $k+1$, then \hrel{g2} reduces to \hrel{g1}, while if all labels are equal to $k+1$, then it follows from \hrel{c19}. Therefore, it is enough to show that \hrel{g2} holds whenever $F_{\underline{i},\underline{j}}$ is a generating morphism of $\AlgL$ with mixed labels featuring at least one label equal to $k+1$. In other words, it is enough to consider $F_{\underline{i},\underline{j}} = U_{\underline{i}}, \Tar_{k+1}, \Tar_{k+1}$, or a decorated crossing. Once again, the statement for $U_{\underline{i}}$ follows directly from the associativity axiom \hrel{a1}, while the proofs of \hrel{g2} for $\Tar_{k+1}$ and $\Sou_{k+1}$ are analogous to the ones shown in Figure~\ref{theta6/fig}, where in the first line the adjoint action has to be replaced by the product, and in the second line $\hat X$ has to be replaced by $X$. For what concerns decorated crossings, we observe that, thanks to relations~\hrel{c14} and \hrel{c20}, it is enough to prove \hrel{g2} for $X_{i,k+1}$ and $\hat X_{k+1,i}$ with $i \leqs k$ and for $\hat X_{i,k+1}$ and $X_{k+1,i}$ with $i > k+1$. This is done in Figures~\ref{ttheta4-1/fig} and \ref{ttheta4-2/fig}, respectively.

\begin{figure}[htb]
 \centering
 \includegraphics{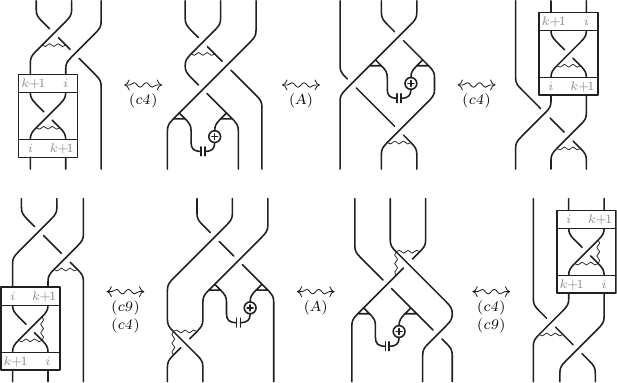}
 \caption{Naturality of $\hat{\gamma}_k$ with respect to decorated crossings with mixed labels, $i \leqs k$.}
 \label{ttheta4-1/fig}
\end{figure}

\begin{figure}[htb]
 \centering
 \includegraphics{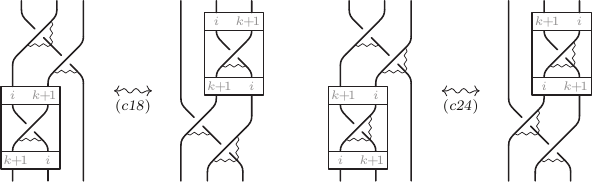}
 \caption{Naturality of $\hat{\gamma}_k$ with respect to decorated crossings with mixed labels, $i > k+1$.}
 \label{ttheta4-2/fig}
\end{figure}

We will show now that \hrel{t3} implies that $\hat{\Forget}_k : \AlgL \to \Algf$ is a functor. In order to see this, consider, for all $\ell \geqs 0$, $k \geqs 1$, and $\underline{i} = (i_1,i_2,\ldots,i_\ell)$ with $i_h \geqs 0$ for all $1 \leqs h \leqs \ell$, the morphisms $\Omega_{\underline{i},k}, \Omega_{\underline{i},k}^{-1}: H^\ell \to H^\ell$ in $\Algf$ defined as 
\begin{gather*}
 \Omega_{\underline{i},k} = \hat{\Theta}^\rmL_{\underline{i},k} \circ (\id_\ell \otimes \unitH),\\
 \Omega_{\underline{i},k}^{-1} = \checkTheta^\rmL_{\underline{i},k} \circ (\id_\ell \otimes \unitH),
\end{gather*}
where $\checkTheta^\rmL_{\underline{i},k}$ is defined recursively, for $\ell \geqs 0$, as follows:
\begin{gather*}
 \checkTheta^\rmL_{\varnothing,k} = \counH, \quad 
 \checkTheta^\rmL_{i,k} = \checkTheta^\rmL_{(i),k} = 
 \begin{cases}
  \prodH \circ (\id \otimes \antipH) & \mbox{ if } i = k, \\
  \id \otimes \counH & \mbox{ if } i \neq k,
 \end{cases} \\
 \checkTheta^\rmL_{\underline{i},k} = (\checkTheta^\rmL_{i_1,k} \otimes \id_{\ell-1}) \circ (\id \otimes \hat{\gamma}_{(i_2,\ldots,i_\ell),k}) \circ (\id \otimes \checkTheta^\rmL_{(i_2,\ldots,i_\ell),k} \otimes \id) \circ (\id_\ell \otimes \coprH).
\end{gather*}

\begin{figure}[b]
 \centering
 \includegraphics{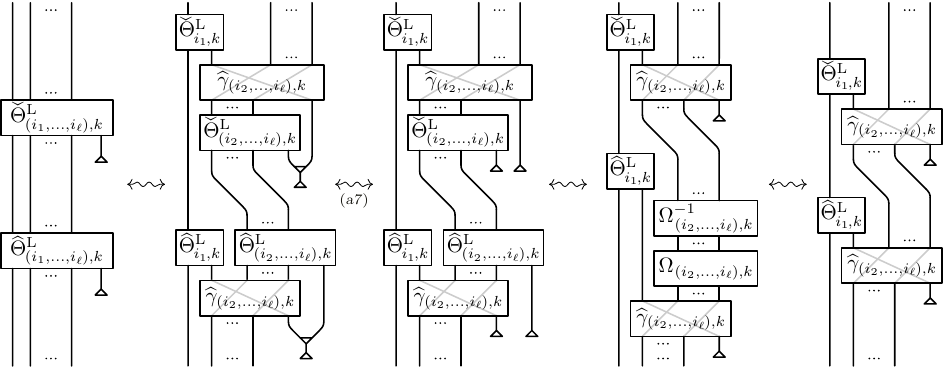}
 \caption{Invertibility of $\Omega_{\underline{i},k}$ -- Part 1, reducing $\Omega_{\underline{i},k}^{-1} \circ \Omega_{\underline{i},k} = \id_\ell$ to \hrel{g3}.}
 \label{Omega1-1/fig}
\end{figure}

\begin{figure}[htb]
 \centering
 \includegraphics{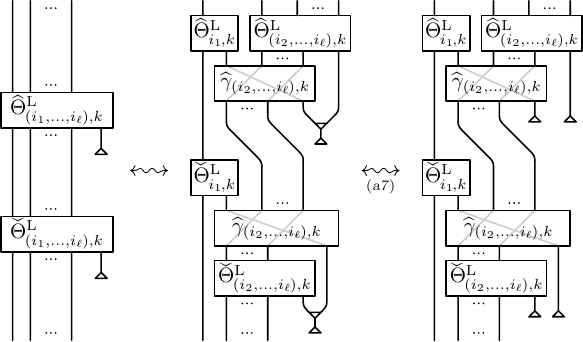}
 \caption{Invertibility of $\Omega_{\underline{i},k}$ -- Part 1, reducing $\Omega_{\underline{i},k}\circ \Omega_{\underline{i},k}^{-1} = \id_\ell$ to \hrel{g4}.}
 \label{Omega1-2/fig}
\end{figure}

\begin{figure}[htb]
 \centering
 \includegraphics{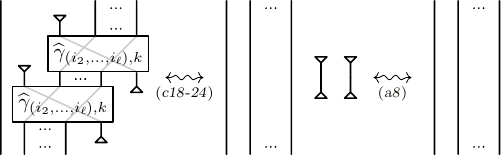}
 \caption{Invertibility of $\Omega_{\underline{i},k}$ -- Part 2, establishing \hrel{g3} and \hrel{g4} when $i_1 \neq k$.}
 \label{Omega2/fig}
\end{figure}

\begin{figure}[htb]
 \centering
 \includegraphics{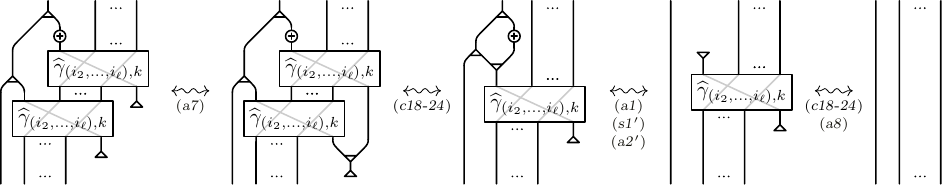}
 \caption{Invertibility of $\Omega_{\underline{i},k}$ -- Part 2, establishing \hrel{g3} when $i_1 = k$ (establishing \hrel{g4} requires using \hrel{s1} instead of \hrel{s1'}).}
 \label{Omega3/fig}
\end{figure}

By induction on $\ell \geqs 0$, we can see that $\Omega_{\underline{i},k}^{-1}$ is the inverse of $\Omega_{\underline{i},k}$. Indeed, for $\ell = 1$, the statement follows by definition. Then, the inductive step is proved in Figures~\ref{Omega1-1/fig}--\ref{Omega3/fig}. In particular, in Figures~\ref{Omega1-1/fig} and \ref{Omega1-2/fig}, it is shown that, up to the inductive hypotheses, the identities $\Omega_{\underline{i},k}^{-1}\circ \Omega_{\underline{i},k} = \id_\ell$ and $\Omega_{\underline{i},k}\circ \Omega_{\underline{i},k}^{-1} = \id_\ell$ reduce to
\begin{align*}
 (\checkTheta^\rmL_{i_1,k} \otimes \id_{\ell-1}) \circ (\id \otimes \hat{\gamma}_{(i_2,\dots,i_\ell),k}) \circ (\id_\ell \otimes \unitH) \circ (\hat{\Theta}^\rmL_{i_1,k} \otimes \id_{\ell-1}) \circ (\id \otimes \hat{\gamma}_{(i_2,\dots,i_\ell),k}) \circ (\id_\ell \otimes \unitH) &= \id_\ell,
 \tag*{\rel{g3}}
 \label{E:g3} \\
 (\hat{\Theta}^\rmL_{i_1,k} \otimes \id_{\ell-1}) \circ (\id \otimes \hat{\gamma}_{(i_2,\dots,i_\ell),k}) \circ (\id_\ell \otimes \unitH) \circ (\checkTheta^\rmL_{i_1,k} \otimes \id_{\ell-1}) \circ (\id \otimes \hat{\gamma}_{(i_2,\dots,i_\ell),k}) \circ (\id_\ell \otimes \unitH) &= \id_\ell,
 \tag*{\rel{g4}}
 \label{E:g4}
\end{align*}
respectively. Equations~\hrel{g3} and \hrel{g4} are proved in Figures~\ref{Omega2/fig} and \ref{Omega3/fig}. Now, \rel{t3} implies that, for every morphism $F_{\underline{i},\underline{j}} : H_{\underline{i}} \to H_{\underline{j}}$ in $\AlgL$, we have
$
\Omega_{\underline{j},k} \circ \Forget(F_{\underline{i},\underline{j}}) = \hat\Forget_k(F_{\underline{i},\underline{j}}) \circ \Omega_{\underline{i},k} 
$, and therefore
\[
 \hat\Forget_k(F_{\underline{i},\underline{j}}) = \Omega_{\underline{j},k} \circ \Forget(F_{\underline{i},\underline{j}}) \circ \Omega_{\underline{i},k}^{-1}.
\] 
 Since $\Forget$ is a functor, the last identity implies the functoriality of $\hat\Forget_k$, while $\Omega_{\underline{j},k}$ defines a natural equivalence between them.
\end{proof}

\FloatBarrier


\subsection{Invariance of \texorpdfstring{$\barPhi(T)$}{the inverse functor}}
\label{barPsi/sec}

Let $T : E_{2s} \to E_{2t}$ be a tangle in $\KTf$ presented by a strictly regular planar diagram of the form represented in the leftmost part of Figure~\ref{kirby-hopf01/fig}, and let $L$ be the subdiagram which represents the blackboard framed link formed by the closed undotted components of $T$.

The construction of the morphism $\barPhi(T) = \barPhi_{L',\alpha}(T)$ in $\Algf$ presented in Subsection~\ref{FK/sec} required the following choices:
\begin{enumerate}
\item a numbering of the components $L_i$ of $L = L_1 \cup \dots \cup L_n$, an orientation of each component $L_i$, and two points $p_i$ and $q_i$ in $L_i$, all inducing a bi-ascending state $L' = L'_1 \cup \dots \cup L'_n$ of $L$, as in Step~\(1) in Subsection~\ref{FK/sec};
\item a family $\alpha = \{ \alpha_1, \ldots, \alpha_n \}$ of embedded arcs $\alpha_i : [0,1] \to [0,1]^2$ that satisfy $\alpha_i(0) = a_i$ and $\alpha_i(1) = b_i\in L_i$, as well as the conditions listed in Step~\(2) in Subsection~\ref{FK/sec}.
\end{enumerate}

We are going to prove now that $\barPhi(T)$ is independent of such choices, and that it is invariant under \dfrmtns{2} of $T$. We will use the notations introduced in Subsection~\ref{FK/sec}. In particular, $L_{i,\alpha} = L_i \cup \alpha_i$ and $L'_{i,\alpha} = L'_i \cup \alpha_i$ for every $1 \leqs i \leqs n$, with $L_{\alpha} = L_{1,\alpha} \cup \dots \cup L_{n,\alpha}$ and $L'_{\alpha} = L'_{1,\alpha} \cup \dots \cup L'_{n,\alpha}$.

\begin{proposition}\label{labeled-kirby-hopf/thm} There exists a morphism $G_{L',\alpha}: H_0^{\otimes s} \otimes H_{(i_1,i_2,\dots,i_n)} \to H_0^{\otimes t}$ in $\AlgL$, with $\{i_1,i_2,\dots,i_n\}=\{1,2,\dots,n\}$, such that
\[
\barPhi_{L',\alpha}(T) = \Forget(G_{L',\alpha}) \circ (\id_s \otimes \unitH^{\otimes n}),
\]
where $\Forget : \AlgL \to \Algf$ is the forgetful functor which discards labels.
\end{proposition}

\begin{proof}
Recall that, by definition (see the right-hand side of Figure~\ref{kirby-hopf01/fig}),
\[
\barPhi_{L',\alpha}(T) = \Tar^{\otimes t} \circ F_{L',\alpha} \circ \left( \Sou^{\otimes s} \otimes \id_n \right) \circ (\id_s \otimes \unitH^{\otimes n}),
\]
where the morphism $\Tar^{\otimes t} \circ F_{L',\alpha} \circ \left( \Sou^{\otimes s} \otimes \id_n \right)$ is assembled using the images of the elementary tangles making up $T$, as presented in the second column of Figures~\ref{kirby-hopf03-1/fig}, \ref{kirby-hopf03-2/fig}, and \ref{kirby-hopf03-3/fig}, with the exception of the unit morphisms that are images of the ends $a_i$ for $1 \leqs i \leqs n$.
 
\begin{figure}[htb]
 \centering
 \includegraphics{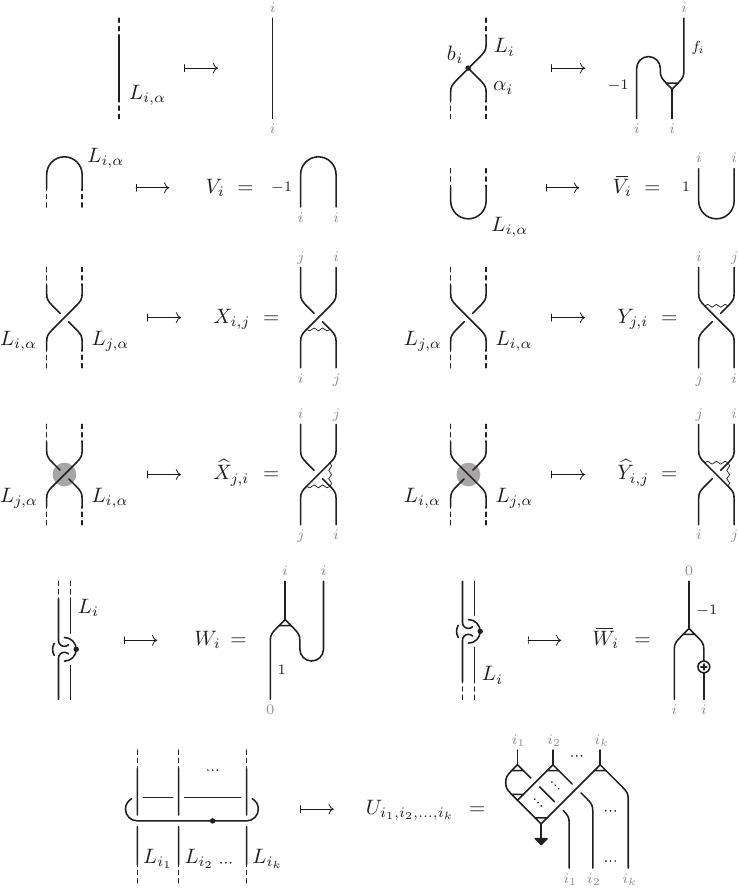}
 \caption{Construction of the morphism $G_{L',\alpha}: H_0^{\otimes s} \otimes H_{(1,2,\dots,n)} \to H_0^{\otimes t}$ in $\AlgL$, with $1 \leqs i \leqs j \leqs n$.}
 \label{kirby-MAlg/fig}
\end{figure}

Then, we only need to show that there exists a morphism $G_{L',\alpha}: H_0^{\otimes s} \otimes H_{(i_1,i_2,\dots,i_n)} \to H_0^{\otimes t}$ in $\AlgL$ such that $\Forget(G_{L',\alpha}) = \Tar^{\otimes t} \circ F_{L',\alpha} \circ \left( \Sou^{\otimes s} \otimes \id_{n} \right)$. We can obtain $G_{L',\alpha}$ simply by attaching labels to the morphisms appearing in the decomposition of $\Tar^{\otimes t} \circ F_{L',\alpha} \circ \left( \Sou^{\otimes s} \otimes \id_n \right)$, making sure that the assignment is compatible with the definition of the category $\AlgL$, see Definition~\ref{AlgL/def}. This is done in Figure~\ref{kirby-MAlg/fig}, where we label by $0$ the source of $\Sou$ and the target of $\Tar$, and where we label the remaining morphisms according to the numbering of the components of $L_\alpha$. Observe that, since the number attached to a component of the link corresponds to its depth in the bi-ascending state $L'$ of $L$, the decorated crossings which appear are exactly the ones in the definition of the category $\AlgL$, see Definition~\ref{AlgL/def}. Hence, $G_{L',\alpha}$ is a morphism in $\AlgL$, as required.
\end{proof}

\begin{remark}\label{rib-weight/rmk} 
Notice that, since all decorations at crossings are attached to the $i$-labeled strands by product morphisms, and since the same is true for the $0$-labeled strands at $W_i$ and $\bar W_i$, the ribbon morphisms in the image of $L_{i,\alpha}$ can be slid along the images of the $i$-labeled strands and collected together, by using moves~\hrel[r5]{r5-5'}. Let $|A|$ denote the number of morphisms of type $A$ in the decomposition of $\Forget(G_{L',\alpha})$. Then, for the total ribbon weight $r_i$ of the image of $L_{i,\alpha}$, in $\Forget(G_{L',\alpha})$ we obtain
\[
 r_i=(|\bar V_i|+|W_i|)-(|V_i|+|\bar W_i|)+f_i-1=- \wr(L_i'),
\]
where we have used that $|\bar V_i|+|W_i|$ (respectively, $|V_i|+|\bar W_i|$) equals the number of caps (respectively, cups) of $L_i$, and in a closed component, these two numbers coincide.
\end{remark}

\begin{proposition}\label{indep-bands/thm} 
For a fixed choice of the bi-ascending state $L'$, and hence of the numbering of the components of $L$, the morphism $\barPhi_{L',\alpha}(T)$ of $\Algf$ does not depend on the choice of the family $\alpha$ of embedded arcs $\alpha_i : [0,1] \to [0,1]^2$ for $1 \leqs i \leqs n$. Moreover, every arc $\alpha_i$ can be chosen to intersect the component $L_i$ to which it is attached in an arbitrary way, provided it still crosses below $L_{j,\alpha}$ for $j < i$ and above $L_{j,\alpha}$ for $j > i$. In other words, the conditions on $\alpha_i$ in Step~\(2) of Subsection~\ref{FK/sec} can be weakened to exclude \(b) and \(c). 
\end{proposition}

\begin{proof}
In order to see that $\barPhi_{L',\alpha}(T)$ is independent of the choice of the family of arcs $\alpha$, we have to show that $\barPhi_{L',\alpha}(T) = \barPhi_{L',\balpha}(T)$ for any other family of arcs $\balpha$ satisfying the same conditions required in Step~\(2), except for \(b) and \(c). We can do that by assuming the additional hypothesis that $\alpha$ intersects $\balpha$ regularly and, in particular, that $\balpha_i(1) = \hat b_i \neq b_i = \alpha_i(1)$ for every $1 \leqs i \leqs n$. In fact, if this were not the case, we could always consider a third family $\bbalpha$ of arcs satisfying such additional hypothesis with respect to both $\alpha$ and $\balpha$, and then show that $\barPhi_{L',\alpha}(T) = \barPhi_{L',\bbalpha}(T) = \barPhi_{L',\balpha}(T)$.


Therefore, we can proceed to replace the arcs of $\alpha$ with those of $\balpha$ one at a time. In other words, it is enough to show that $\barPhi_{L',\alpha}(T) = \barPhi_{L',\balpha}(T)$ whenever $\barPhi_{L',\balpha}(T)$ is the morphism obtained by replacing the arc $\alpha_k$ with an arc $\balpha_k : [0,1] \to [0,1]^2$ that satisfies the conditions in Step~\(2) of Subsection~\ref{FK/sec} except for \(b) and \(c), and by keeping every other arc $\alpha_i$ with $i \neq k$ fixed.

The main idea behind the proof is the following. We consider the diagram $T_{\alpha,\balpha_k} = T_\alpha \cup \balpha_k = T \cup_{j=1}^n \alpha_j \cup \balpha_k$ in which both arcs $\alpha_k$ and $\balpha_k$ are simultaneously attached to $L_k$, and we choose the crossing state for the crossings between $\alpha_k$ and $\balpha_k$ in an arbitrary way. We can assume for instance that $\balpha_k$ crosses always over $\alpha_k$, and we do not mark these crossings. Under the conditions listed above, $T_{\alpha,\balpha_k}$ has only regular intersections, and we can associate to it a morphism in $\Algf$ following the same rules used in the definition of $\barPhi$. Then, we will show (see equations~\hrel{v1} and \hrel{v2} below) that if, in this last morphism, the unit $\unitH$ in the image of $\hat a_k$ (respectively $a_k$) is replaced by the integral element $\inteH$, then the resulting morphism is equivalent to $\barPhi_{L',\alpha}(T)$ (respectively $\barPhi_{L',\balpha}(T)$). Finally, we will show that the two morphisms that are obtained by exchanging $\unitH$ and $\inteH$, which are associated to $a_k$ and $\hat a_k$, are \qvlnt{2} in $\Algf$ (see equation~\hrel{v3} below). This last step will require the use of the natural transformation $\Theta^\rmL_k$, which means that we will interpret the essential part of the above morphisms as the image under the forgetful functor of a labeled morphism in $\AlgD$. Here are the details.


By Proposition~\ref{labeled-kirby-hopf/thm}, there exist morphisms 
\[
G_{L',\alpha} : H_0^{\otimes s} \otimes H_{(i_1,\dots,i_{\ell-1}, i_\ell=k, i_{\ell+1},\dots, i_n)} \to H_0^{\otimes t}
\]
and
\[
G_{L',\balpha} : H_0^{\otimes s} \otimes H_{(i_1,\dots,i_{\ell-1},i_{\ell+1},\dots,i_h, k, i_{h+1},\dots, i_n)} \to H_0^{\otimes t}
\]
in $\AlgL$ such that
\[
\barPhi_{L',\alpha}(T) = \Forget(G_{L',\alpha}) \circ (\id_s \otimes \unitH^{\otimes n})
\text{ and }
\barPhi_{L',\balpha}(T) = \Forget(G_{L',\balpha}) \circ (\id_s \otimes \unitH^{\otimes n}).
\]
Then $\barPhi_{L',\alpha}(T) = \barPhi_{L',\balpha}(T)$ will follow if we can find a third morphism 
\[
G_{L',\alpha,\balpha_k} : H_0^{\otimes s} \otimes H_{(i_1,\dots, i_{\ell-1},k,i_{\ell+1},\dots,i_h,k,i_{h+1} ,\dots,i_n)} \to H_0^{\otimes t}
\]
in $\AlgL$ such that 
\begin{gather*}
 \Forget(G_{L',\alpha}) = \Forget(G_{L',\alpha,\balpha_k}) \circ (\id_{s+h} \otimes \inteH \otimes \id_{n-h-1}),
 \tag*{\rel{v1}} \label{E:v1}
 \\
 \Forget(G_{L',\balpha}) = \Forget(G_{L',\alpha,\balpha_k}) \circ (\id_{s+\ell-1} \otimes \inteH \otimes  \id_{n-\ell}), 
 \tag*{\rel{v2}} \label{E:v2}
 \\
 \begin{aligned}
  \Forget(G_{L',\alpha,\balpha_k}) \circ (\id_{s+\ell-1} \otimes \inteH \otimes \id_{h-\ell} \otimes \unitH \otimes \id_{n-h-1}) \kern 5cm \\
  \kern5cm = \Forget(G_{L',\alpha,\balpha_k}) \circ (\id_{s+\ell-1} \otimes \unitH \otimes \id_{h-\ell} \otimes \inteH \otimes \id_{n-h-1}).
 \end{aligned}
 \tag*{\rel{v3}}
 \label{E:v3}
\end{gather*}

In order to construct $G_{L',\alpha,\balpha_k}$, consider the diagram $T_{\alpha,\balpha_k}$.
The morphism $G_{L',\alpha,\balpha_k}$ is obtained by associating to the elementary morphisms making up $T_{\alpha,\balpha_k}$ the morphisms of $\AlgL$ listed in Figures~\ref{kirby-MAlg/fig} and \ref{second-arc/fig}-\(a). Notice that the edges in the image of $\hat b_k$ shown in Figure~\ref{second-arc/fig}-\(a) are not weighted by the ribbon morphism, as opposed to the ones corresponding to $b_k$. The global form of the morphism $\Forget(G_{L',\alpha,\balpha_k})$ is represented in Figure~\ref{second-arc/fig}-\(b).

\begin{figure}[htb]
 \centering
 \includegraphics{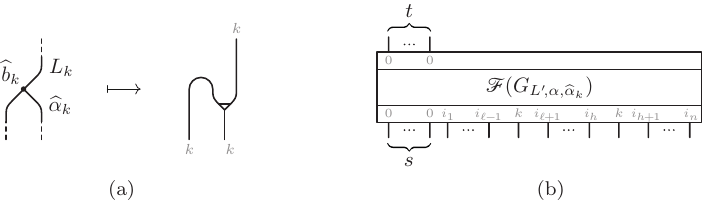}
 \caption{Image of $b_k'$ in $\AlgL$ and global form of $\Forget(G_{L',\alpha,\balpha_k})$.}
 \label{second-arc/fig}
\end{figure}

Consider now the morphism $\Forget(G_{L',\alpha,\balpha_k}) \circ (\id_{s+h} \otimes \inteH \otimes \id_{n-h-1})$ of $\Algf$ obtained by composing the image of $\balpha_k$ with the integral element $\inteH$ (see the second diagram in Figure~\ref{two-arcs/fig}). Since $\inteH$ belongs to $\AlgD$, and since the image of $\balpha_k$ is made up entirely of decorated crossings of type $X$, $\hat X$, $Y$, and $\hat Y$, which also belong to $\AlgD$, and of $\ev$ and $\coev$ morphisms, we can apply relations \hrel{c18}, \hrel{c19}, \hrel{c24}, and \hrel{c25} in Table~\ref{table-decorated-moves/fig} and the duality between $\counH$ and $\inteH$ in Table~\ref{table-BPHopf-prop1/fig} to pull up $\inteH$ towards the image of $\hat b_k$, thus obtaining $\Forget(G_{L',\alpha})$ (see the last two steps in Figure~\ref{two-arcs/fig}). This proves \hrel{v1}, and the proof of \hrel{v2} is completely analogous.

\begin{figure}[htb]
 \centering
 \includegraphics{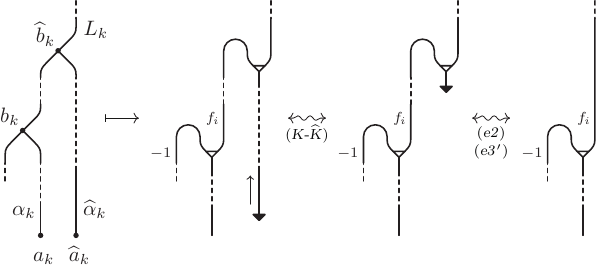}
 \caption{Independence of the choice of $\alpha_k$: proof of \hrel{v1}.}
 \label{two-arcs/fig}
\end{figure}

The proof of \hrel{v3} is shown in Figure~\ref{indep-bands/fig}. Here, in the second diagram, $\gamma^{-1}_{(i_{h+1},\dots, i_n),k}$ is the inverse of the morphism $\gamma_{(i_{h+1},\dots, i_n),k} : H^{n-h} \otimes H \to H \otimes H^{n-h}$ (see Definition~\ref{theta1/defn}), and can be represented as a composition of tensor products of identities, inverse braidings $c^{-1}$, and decorated crossings of type $\hat Y$. To implement this first step, we are using the fact that the counit $\counH$ belongs to $\AlgD$, and can thus be pulled up to the top-right using the naturality of the braided structures of $\Algf$ and $\AlgD$. Then, the top part of the second diagram can be interpreted as $\id_t \otimes \counH = \Theta^\rmL_{(0,\dots,0),k}$, which allows us to use, in the second step, the naturality property \hrel{t2} of $\Theta^\rmL_k$ to intertwine it with $\Forget(G_{L',\alpha,\balpha_k})$. 
\end{proof}

\begin{figure}[htb]
 \centering
 \includegraphics{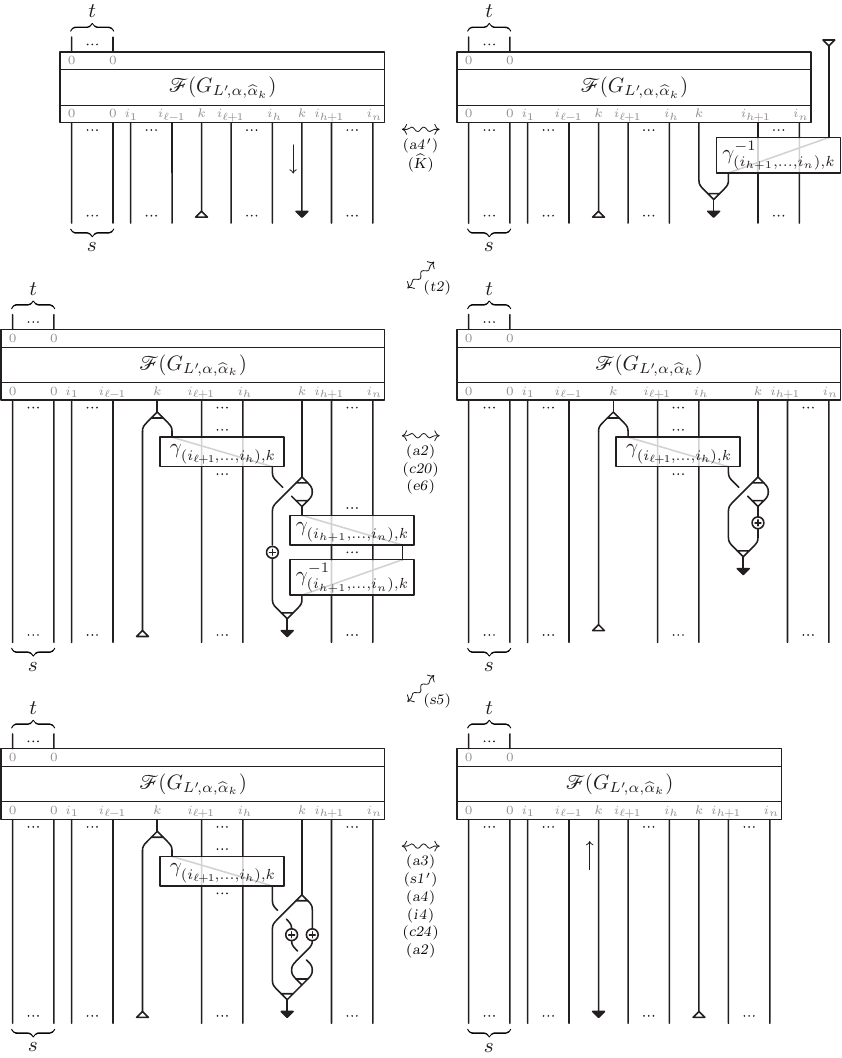}
 \caption{Independence of the choice of $\alpha_k$: proof of \hrel{v3}.}
 \label{indep-bands/fig}
\end{figure}

\FloatBarrier

Since $\barPhi_{L',\alpha}(T)$ is independent of the choice of the arcs $\alpha_i$ for all $1 \leqs i \leqs n$, from now we will denote this morphism simply as $\barPhi_{L'}(T)$.

\begin{proposition}\label{vert-exchange/thm} 
The morphism $\barPhi_{L'}(T)$ does not depend on the numbering of the components of $L$, that is, on the vertical order of the components of the bi-ascending state $L'$.
\end{proposition}

\begin{proof}
In order to show that $\barPhi_{L'}(T)$ is independent of the numbering of the components of $L$, it is enough to show that $\barPhi_{L'}(T) = \barPhi_{L''}(T)$ when $L''$ is obtained from $L'$ by exchanging the order of two consecutive components $L_k$ and $L_{k+1}$, for some $1 \leqs k \leqs n-1$. This implies that $L''$ is obtained from $L'$ by setting $L_{k+1}'' = L_k'$, $L_{k}'' = L_{k+1}'$, and by inverting all crossings between these two components, while $L_{i}'' = L_i'$ for every $i \neq k,k+1$. Then, according to Proposition~\ref{labeled-kirby-hopf/thm} and to the definition of the functor $\hat\Forget_k$ in Proposition~\ref{ttheta2/thm}, $\barPhi_{L'}(T) = \Forget(G_{L',\alpha}) \circ (\id_s \otimes \unitH^{\otimes n})$, while 
$\barPhi_{L''}(T) = \hat\Forget_k(G_{L',\alpha}) \circ (\id_s \otimes \unitH^{\otimes n})$.
Hence, the statement will follow if we show that
\[
\Forget(G_{L'}) \circ (\id_s \otimes \unitH^{\otimes n}) = \hat\Forget_k(G_{L'}) \circ (\id_s \otimes  \unitH^{\otimes n}).
\]
This is done in Figure~\ref{indep-numbering/fig}, where we have assumed that the endpoints $a_1, a_2, \dots, a_n$ of $\alpha_1, \alpha_2, \dots, \alpha_n$ have been positioned in the lower right corner of the diagram in increasing order from the left to the right, as it is allowed by Proposition~\ref{indep-bands/thm}. In the first step in Figure~\ref{indep-numbering/fig}, we insert $\counH \circ \unitH$ between the $k$th and the $(k+1)$st strand, and pull the counit $\counH$ through the $n-k$ vertical strands to its right using the naturality of the two braided structures of $\AlgD$. Since $\id_t \otimes \counH = \hat\Theta^\rmL_{(0,\dots,0),k}$, in the second step we use the naturality property \hrel{t3} of $\hat{\Theta}^\rmL_{k}$ (see Proposition~\ref{ttheta2/thm}) to intertwine it with $\Forget(G_{L',\alpha})$, thus obtaining $\hat\Forget_k(G_{L',\alpha})$.
\end{proof}

\begin{figure}[htb]
 \centering
 \includegraphics{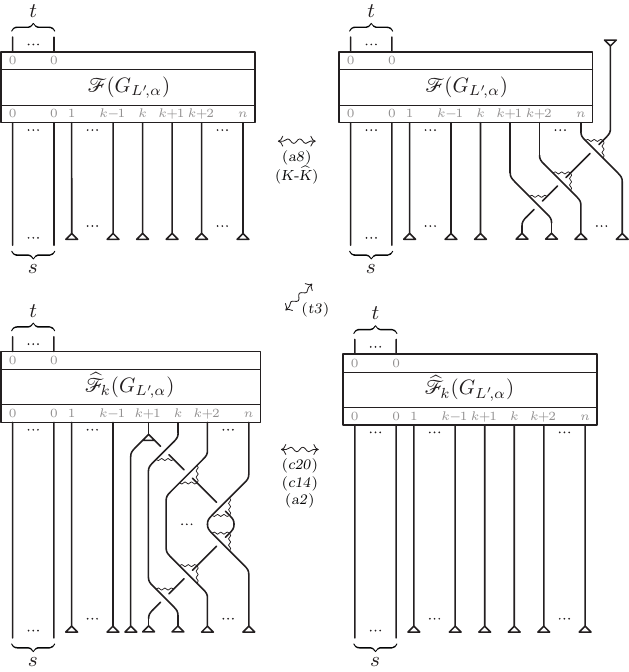}
 \caption{Proof of the independence of the choice of numbering of the components of $L$.}
 \label{indep-numbering/fig}
\end{figure}

\FloatBarrier

In order to prove that $\barPhi_{L'}(T)$ is independent of the choice of the bi-ascending state of the single components of $L'$, we need the following lemma.

\begin{lemma}\label{alg-cutting/thm}
Let $T = T_2 \circ T_1$ be a tangle of the form represented on the left-hand side of Figure~\ref{cutting-lemma1/fig}, where two adjacent strands belonging to the same component $L_n$ of the undotted link $L = L_1 \cup \dots \cup L_n$ of $T$ are joined by a flat band $\delta$. 
Assume that surgering $L_n$ along $\delta$ yields two different components $\hat L_n$ and $\hat L_{n+1}$ of the undotted link $\hat L = L_1 \cup \dots \cup L_{n-1} \cup \hat L_n \cup \hat L_{n+1}$ of a new tangle $\hat T$, where an extra dotted component is added to encircle $\delta$, as shown on the right-hand side of Figure~\ref{cutting-lemma1/fig}. Assume also that $L' = L'_1 \cup \dots \cup L'_n$ and $\hat L' = L'_1 \cup \dots \cup L'_{n-1} \cup \hat L'_n \cup \hat L'_{n+1}$ are bi-ascending states of $L$ and $\hat L$, respectively, whose components are vertically ordered according to the numbering, and such that surgering $L'_n$ along $\delta$ gives the two components $\hat L'_n$ and $\hat L'_{n+1}$. In other words, $L'$ and $\hat L'$ are obtained by inverting the same crossings in $L$ and $\hat L$, respectively. Then $\barPhi_{L'}(T) = \barPhi_{\hat L'}(\hat T)$.
\end{lemma}

\begin{figure}[htb]
 \centering
 \includegraphics{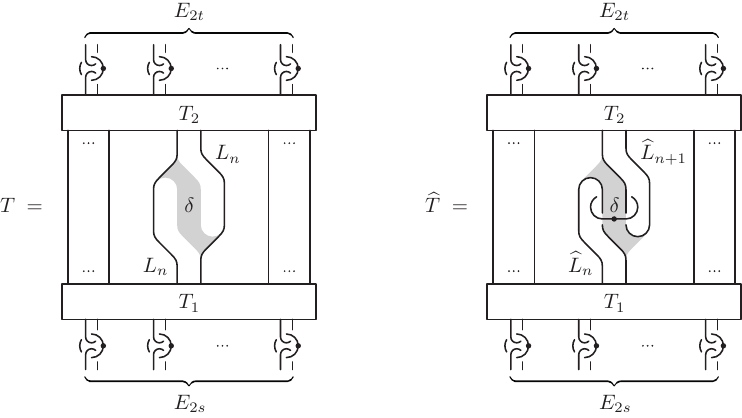}
 \caption{Cutting the component $L_n$.}
 \label{cutting-lemma1/fig}
\end{figure}

\begin{proof}
Choose a set of arcs $\alpha_1, \dots, \alpha_{n+1}$ for $\hat L$ that is consistent with the requirements in Step~\(2) of Subsection~\ref{FK/sec} except for \(b) and \(c), as allowed by Proposition~\ref{indep-bands/thm}). This yields the diagrams $T_\alpha$ and $\hat T_\alpha$ shown in Figure~\ref{cutting-lemma2/fig}, where $\alpha_n$ and $\alpha_{n+1}$ are two parallel arcs that cross over the vertical strands belonging to $\hat L_{n+1}$ and under all the others. In particular, they form the same sequence of crossing states along the pair of gray boxes.

\begin{figure}[htb]
 \centering
 \includegraphics{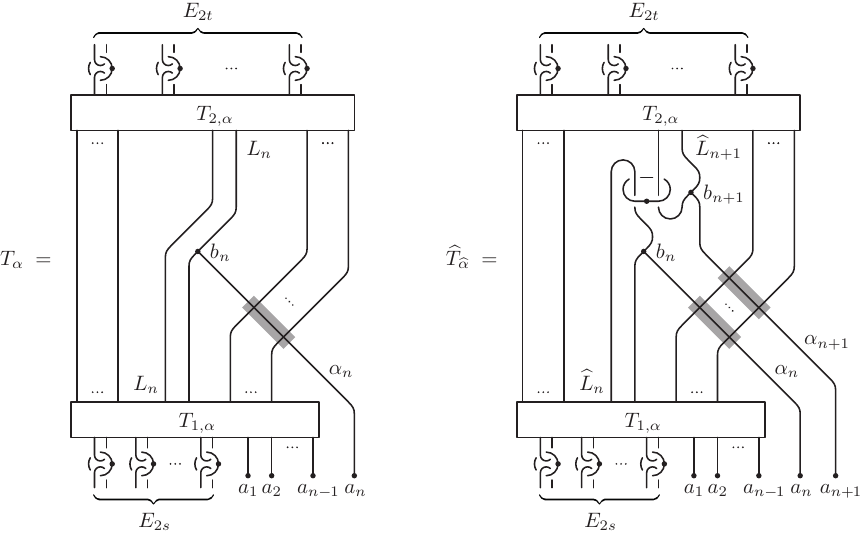}
 \caption{Choice of the arcs $\alpha_n$ and $\alpha_{n+1}$ in the proof of Lemma \ref{alg-cutting/thm}: $\alpha_n$ and $\alpha_{n+1}$ cross over the vertical strings which belong to $\hat L_{n+1}$ and under all the others.}
 \label{cutting-lemma2/fig}
\end{figure}

Now, the equality $\barPhi_{\hat L'}(\hat T) = \barPhi_{L'}(T)$ is proved in Figure~\ref{cutting-lemma3/fig}. In the first step of the figure, after canceling together the two ribbon morphisms lying on the zig-zag strand, we slide down the morphism $U_2$ and the two copairings until they separate the two unit morphisms at the bottom. In the second step, we use \hrel{a2'} to absorb two of the units, and then again \hrel[c18]{c18-24} to slide $\Delta \circ \Lambda$ back up. In the last step, by using \hrel[r5]{r5-5'}, we slide the ribbon weights $\hat f_n-1$ and $\hat f_{n+1}$ along the image of $L_n$ to collect them together, and then use the fact that $\hat f_{n} + \hat f_{n+1} = 2 - \wr(\hat L_{n}) - \wr(\hat L_{n+1}) = 2 - \wr(L_{n}) = 1 + f_n$. Finally, using \hrel{a7}, we collect together the two units and slide the resulting coproduct morphism up.
\end{proof}

\begin{figure}[htb]
 \centering
 \includegraphics{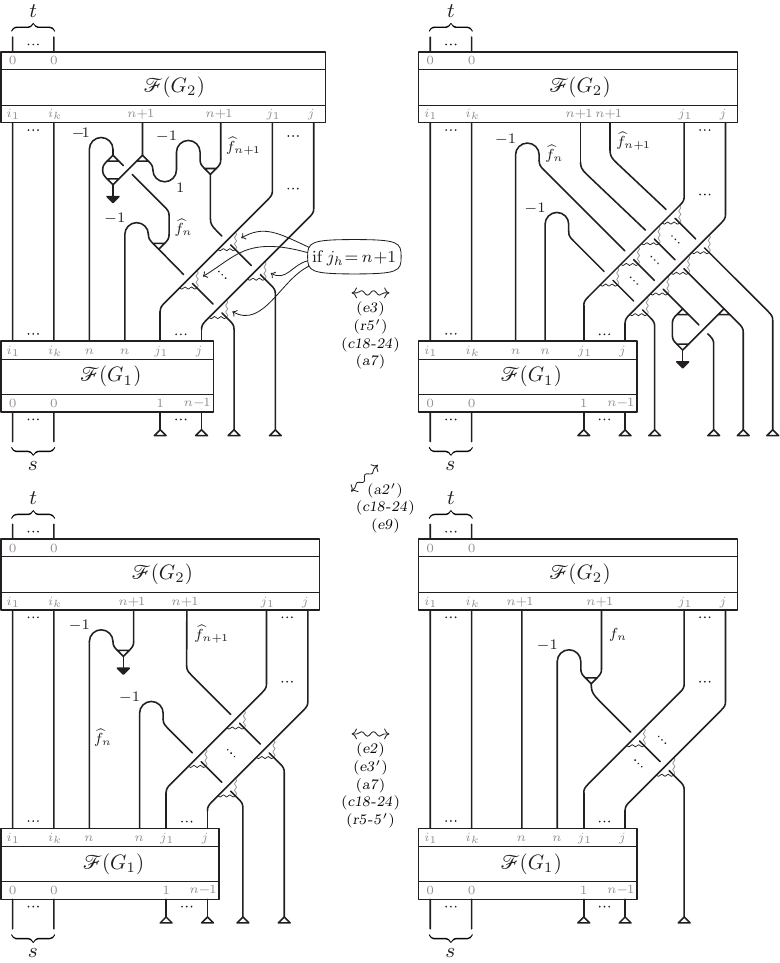}
 \caption{Proof of Lemma~\ref{alg-cutting/thm}.}
 \label{cutting-lemma3/fig}
\end{figure}

\FloatBarrier

\begin{remark}\label{cutting/rmk}
We observe that replacing $T=T_2\circ T_1$ by $\hat T$, as described in Lemma~\ref{alg-cutting/thm}, is a \dfrmtn{2}. Indeed, one can go back by sliding $\hat L_n$ over $\hat L_{n+1}$, and then by canceling the extra \hndl{1} with $\hat L_{n+1}$.
\end{remark}

\begin{proposition}\label{vert-st-comp/thm}
The morphism $\barPhi_{L'}(T)$ does not depend on the choice of the bi-ascending state of the single components of $L'$. 
\end{proposition}

\begin{proof}
Using Proposition~\ref{vert-exchange/thm}, we can assume that the component whose bi-ascending state we want to change is the last one, $L_n$. According to Proposition~\ref{ba-states/thm}, it is enough to prove that $\barPhi_{L'}(T) = \barPhi_{L''}(T)$ whenever the bi-ascending state $L''$ is obtained from $L'$ by a single crossing change in $L_n'$.

Notice that $\barPhi_{L'}(T)$ is invariant under the planar isotopy moves in Figure~\ref{rotating-crossing/fig}, which rotate crossings and hold in any braided rigid monoidal category, as a consequence of the planar isotopy moves in Table~\ref{table-rigid/fig} and of the naturality of the braiding. Indeed, by definition, the images under $\barPhi$ of all crossings of $T_\alpha$ are in the subcategory $\AlgD$ which, according to Theorem~\ref{decorated-moves/thm}, is a braided rigid monoidal category whose rigid structure is induced by the morphisms $\ev$ and $\coev$. Therefore, we can assume that the changing crossing is oriented in one of the two ways shown in the top line in Figure~\ref{proof-biasc1/fig}.

\begin{figure}[htb]
 \centering
 \includegraphics{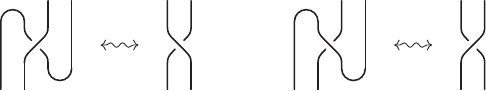}
 \caption{Rotating a crossing in a braided rigid monoidal category.}
 \label{rotating-crossing/fig}
\end{figure}

\begin{figure}[htb]
 \centering
 \includegraphics{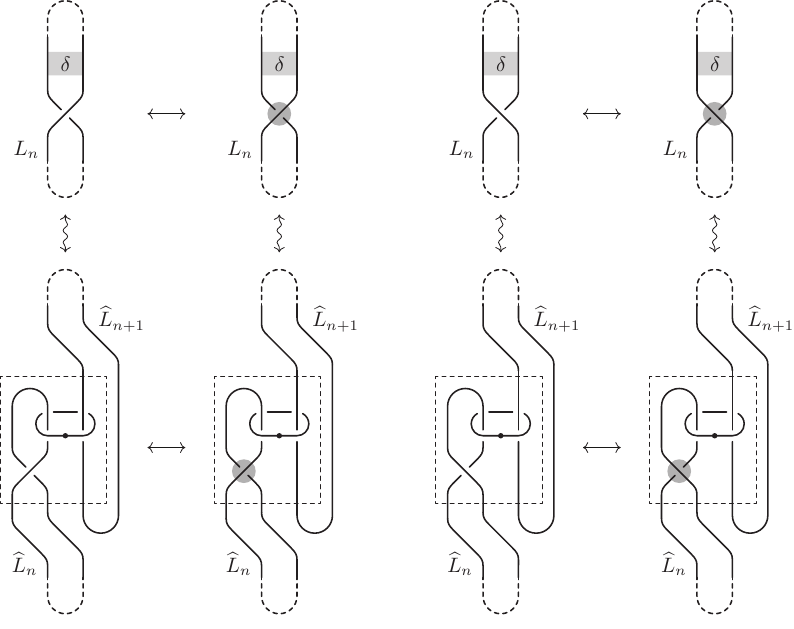}
 \caption{Orienting the changing crossing and cutting the component $L_n$.}
 \label{proof-biasc1/fig}
\end{figure}

In both cases, we cut the component $L_n$ by performing a surgery on it, as described in Lemma~\ref{alg-cutting/thm} and in Figure~\ref{cutting-lemma1/fig}, along a band $\delta$ which is located immediately above the changing crossing, as shown in Figure~\ref{proof-biasc1/fig}. We obtain this way the two new components $\hat L_n$ and $\hat L_{n+1}$. According to the second part of Proposition~\ref{ba-states/thm}, we can assume that $\hat L_n$ and $\hat L_{n+1}$ are vertically separated unknots. Actually, Proposition~\ref{ba-states/thm} tells us that this is true for the two components obtained by cutting $L_n$ at the changing crossing, but since there is no other crossing inside the dashed boxes in Figure~\ref{proof-biasc1/fig}, we are free to vertically isotope $\hat L_n$ and $\hat L_{n+1}$ inside those boxes in such a way that the same holds for them.

\begin{figure}[b]
 \centering
 \includegraphics{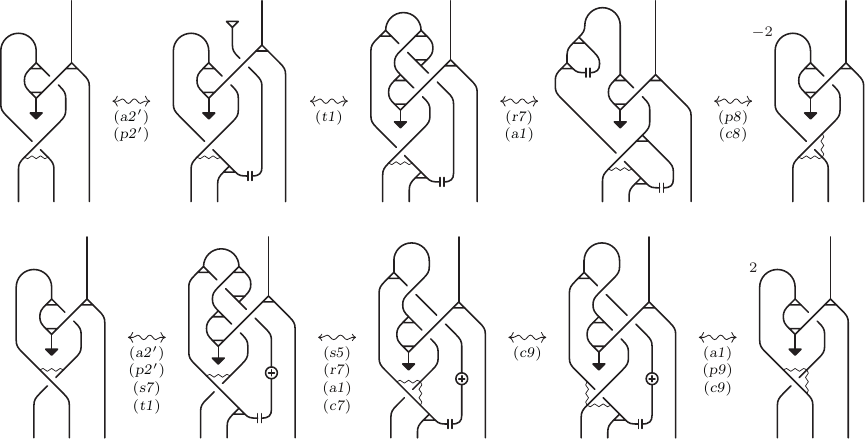}
 \caption{Proof of the invariance of $\barPhi$ under change of crossing.
 }
 \label{proof-biasc2/fig}
\end{figure}

Then, $L'$ and $\hat L'$ satisfy the hypotheses of Lemma~\ref{alg-cutting/thm}, and hence we have $\barPhi_{L'}(T) = \barPhi_{\hat L'}(\hat T)$ and $\barPhi_{L''}(T) = \barPhi_{\hat L''}(\hat T)$. So, we are left to prove that $\barPhi_{\hat L'}(\hat T) = \barPhi_{\hat L''}(\hat T)$. This is done in Figure~\ref{proof-biasc2/fig}, where only the parts of the images corresponding to the parts of the diagrams inside the dashed rectangles in Figure~\ref{proof-biasc1/fig} are compared, since the rest is fixed. Notice that, in the first case (first line in Figure~\ref{proof-biasc2/fig}), the move  increases $\wr(\hat L_n')$ by $2$, so the extra ribbon weight $-2$ ensures that the total ribbon weight $r_n$ of the image of $\hat L_n$ remains equal to $-\wr(\hat L_n')$. Analogously, in the second case (second line in Figure~\ref{proof-biasc2/fig}), the move reduces $\wr(\hat L_n')$ by $2$, so the extra ribbon weight $2$ ensures once again that $r_n$ remains equal to $-\wr(\hat L_n')$.
\end{proof}

It is left to show that $\barPhi$ is invariant under the \qvlnc{2} moves in Table~\ref{table-Ktangles/fig}.

\begin{proposition}\label{sliding/thm}
The morphism $\barPhi(T)$ depends only on the \qvlnc{2} class of $T$ in $4\KT$.
\end{proposition}

\begin{proof} 
In order to see that $\barPhi(T)$ is invariant under the isotopy moves presented in Table~\ref{table-Ktangles/fig}, we observe that, using Propositions~\ref{vert-exchange/thm} and \ref{vert-st-comp/thm}, the bi-ascending state $L'$ can be chosen so that the image under $\barPhi$ of the isotopy move we are interested in reduces to one of the identities in Tables~\ref{table-decorated-moves/fig} and \ref{table-decorated-moves-3/fig}. On the other hand, the invariance under the pushing-through move in Table~\ref{table-Ktangles/fig} reduces to \hrel{t1'} in Figure~\ref{theta-moves/fig}.

The proof of the invariance under \hndl{1/2} cancellation of an undotted component $L_i$ with a dotted meridian is illustrated in Figure~\ref{proof-cancel/fig}. We start by using the integral axiom \hrel{i2} to express multiplication by $\inteH$ as the composition $\inteH \circ \counH$. Then, since $\counH$ and $\inteH$ belong to the subcategory $\AlgD$, and since they are dual to each other with respect to $\ev$ and $\coev$, we use moves \hrel{c18}, \hrel{c19}, \hrel{c24}, \hrel{c25}, and \hrel{u2} in Tables~\ref{table-decorated-moves/fig}  and \ref{table-decorated-moves-3/fig} to slide them along the image of the undotted component $L_i$ until it is transformed in the composition $\counH \circ \unitH$, which is removed through relation~\hrel{a8}.

\begin{figure}[htb]
 \centering
 \includegraphics{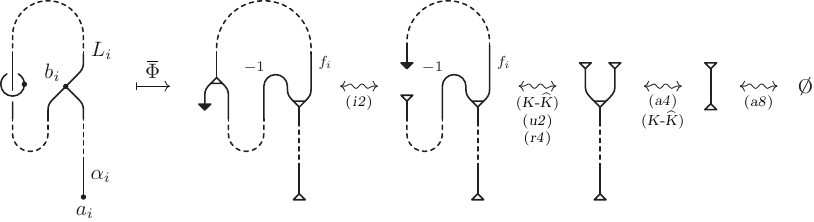}
 \caption{Invariance of $\barPhi(T)$ under \hndl{1/2} cancellation.}
 \label{proof-cancel/fig}
\end{figure}

It remains to prove that $\barPhi(T)$ is invariant under \hndl{2} slide of a component $L_j$ over another component $L_i$, that is, under the replacement of $L_j$ by the band connected sum of $L_j$ and a parallel copy $L_i^{\hbox{\tiny$\parallel$}}$ of $L_i$. Since we have already proved the invariance of $\barPhi(T)$ under isotopy and \hndl{1/2} cancellations, we can assume, thanks to Proposition~\ref{2-equiv-special/thm}, that $L_i$ has at most one self-crossing, and that the components $L_i$ and $L_j$ and the sliding band $\beta$ have one of the forms outlined on the left-hand sides of Figures~\ref{proof-sliding0-1/fig}, \ref{proof-slidingpos-1/fig}, and \ref{proof-slidingneg-1/fig} below. In these pictures, using the independence of $\barPhi(T)$ of the choice of the bi-ascending state, we have assumed that, in the first two cases, it is $L_j = L_2$ that slides over $L_i = L_1$, while, in the third case, it is $L_j = L_1$ that slides over $L_i = L_2$. We are also assuming that the visible part of the diagram has been pulled down outside the box $T_\alpha$ (see Figure~\ref{kirby-hopf01/fig}), except for the dashed lines which interact with the rest of the diagram inside $T_\alpha$. In particular, the dashed part of $L_i$ cannot form self-crossings, but it can cross the dashed part of $L_j$, which can also form self-crossings. 

We consider the three cases separately, starting from the one where $L_i$ has no self-crossings. In this case, Figure~\ref{proof-sliding0-1/fig} shows the two tangles before and after the slide, together with suitable choices for arcs $\alpha_i = \alpha_1$ and $\alpha_j = \alpha_2$ and for the data determining bi-ascending 
states of $L_i = L_1$ and $L_j = L_2$.  Notice that, before the slide, both $L_1$ and $L_2$ are actually in ascending states, while, after the slide, $L_2$ is in bi-ascending state, with positive arc $A_2^+$ given by $L_1^\parallel$.

\begin{figure}[htb]
 \centering
 \includegraphics{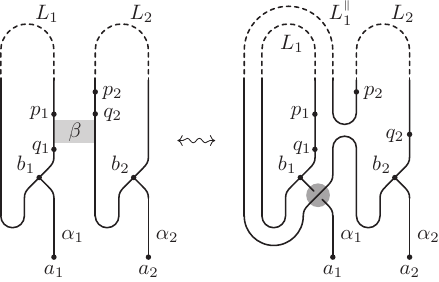}
 \caption{Sliding $L_2$ over $L_1$, for $L_1$ with no self-crossings. 
 }
 \label{proof-sliding0-1/fig}
\end{figure}

The fact that the images under $\barPhi$ of the two tangles in Figure~\ref{proof-sliding0-1/fig} are  the same is shown in Figure~\ref{proof-sliding0-2/fig}, where $r_2$ is the total ribbon weight of the image of $L_2$, while the total ribbon weight of the image of $L_1$ is $r_1=0$ (see Remark~\ref{rib-weight/rmk}). In the first step in Figure~\ref{proof-sliding0-2/fig}, we apply \hrel{a7} to split the image of $\alpha_2$ into two units, and then attach the left one to the image of $L_{\alpha,1}$, thus obtaining the morphism $\tilde{\prodH}$ highlighted inside the dashed box in the second diagram of the figure. We recall that $\tilde{\prodH}$ is in $\AlgD$, therefore we can use its properties in Table~\ref{table-mu/fig}, as well as the naturality of the two braided structures of $\AlgD$ (relations~\hrel[c18]{c18-19} and \hrel[c24]{c24-25} in Table~\ref{table-decorated-moves/fig}) and \hrel{u2} in Table~\ref{table-decorated-moves-3/fig}, to slide $\tilde{\prodH}$ all around the dashed arc, thus creating a parallel copy of the arc. Notice that, as indicated by the arrows in relations~\hrel{q2} and \hrel{q3} in Table~\ref{table-mu/fig}, when the product $\tilde{\prodH}$ passes through $\ev$, it turns into the coproduct $\coprH$, and when $\coprH$ passes through $\coev$, it turns back into $\tilde{\prodH}$. This means that, in the process of sliding, the morphism $\tilde{\prodH}$ always moves upwards, while the morphism $\coprH$ always moves downwards. In particular, when we arrive to the left end of the dashed arc, we get the copy of $\coprH$ highlighted inside the dashed box in the third diagram of the figure. Moreover, when $\tilde{\prodH}$ or $\coprH$ slide through a crossing, two crossings of the same type are created, while, when they slide through a morphism $U_k$, they turn it into $U_{k+1}$. Therefore, we have indeed doubled the dashed arc.

\begin{figure}[htb]
 \centering
 \includegraphics{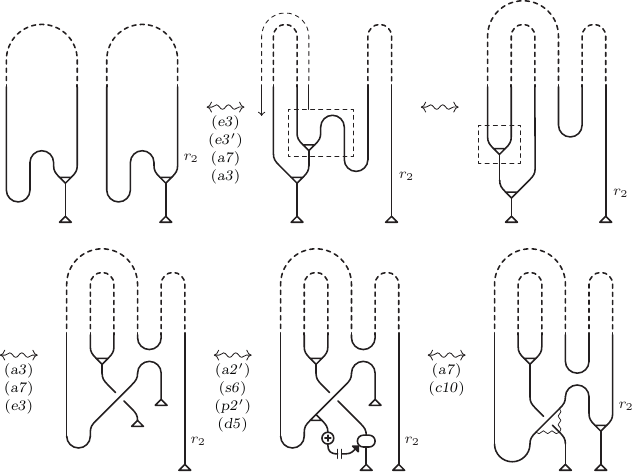}
 \caption{Proof of the invariance of $\barPhi(T)$ under the slide of $L_2$ over $L_1$, for $L_1$ with no self-crossings.}
 \label{proof-sliding0-2/fig}
\end{figure}

Now, we pass to the second case, when $L_i$ forms a positive self-crossing. The two tangles before and after the slide, together with suitable choices for the bi-ascending states and the arcs $\alpha$, are given by the first two diagrams in Figure~\ref{proof-slidingpos-1/fig}, while the third is an equivalent form of the second, up to isotopy. The proof that the images of the first and third diagrams under $\barPhi$ are the same in $\Algf$ is presented in Figure~\ref{proof-slidingpos-2/fig}, where the first diagram has been obtained by applying relation~\hrel{c22} to replace the decorated kink in the image of $L_i = L_1$ by the identity morphism, then observing that the total ribbon weight of the image of $L_1$ is  $r_1=-{\rm wr}(L_1')=1$, and finally repeating the first step in Figure~\ref{proof-sliding0-2/fig}.

\begin{figure}[htb]
 \centering
 \includegraphics{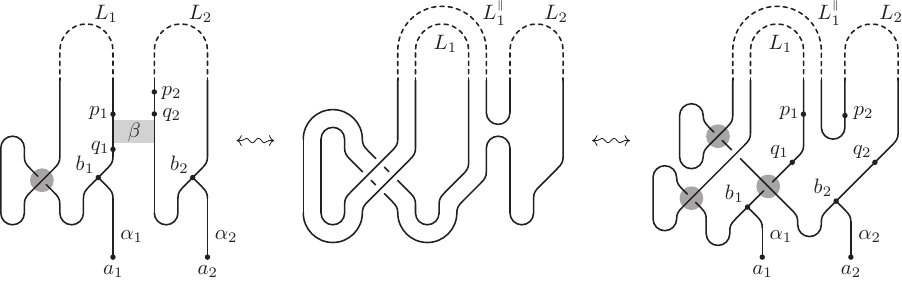}
 \caption{Sliding $L_2$ over $L_1$, for $L_1$ with one positive self-crossing.}
 \label{proof-slidingpos-1/fig}
\end{figure}

\begin{figure}[htb]
 \centering
 \includegraphics{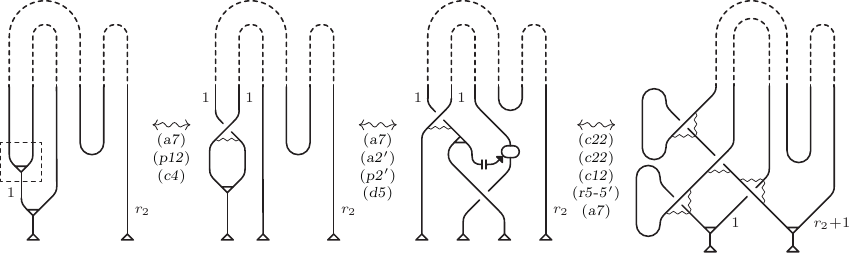}
 \caption{Proof of the invariance of $\barPhi(T)$ under the slide of $L_2$ over $L_1$, for $L_1$ with one positive self-crossing.}
 \label{proof-slidingpos-2/fig}
\end{figure}

In the third case, we slide $L_j = L_1$ over $L_i = L_2$, and assume that $L_2$ forms a single negative self-crossing. The result of the slide is presented in Figure~\ref{proof-slidingneg-1/fig}, where the third diagram is again an equivalent form of the second, up to isotopy.

\begin{figure}[htb]
 \centering
 \includegraphics{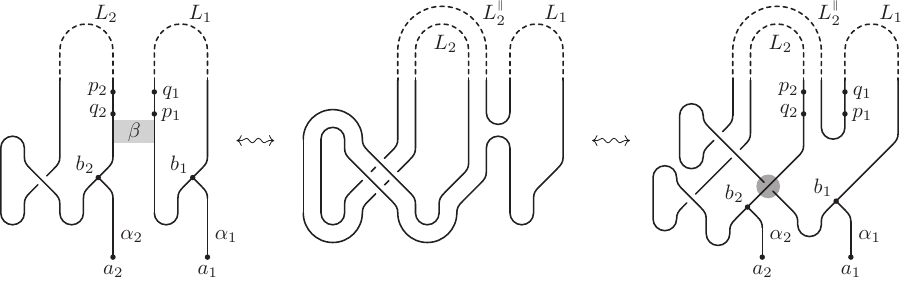}
 \caption{Sliding $L_1$ over $L_2$, for $L_2$ with one negative self-crossing.}
 \label{proof-slidingneg-1/fig}
\end{figure}

\begin{figure}[htb]
 \centering
 \includegraphics{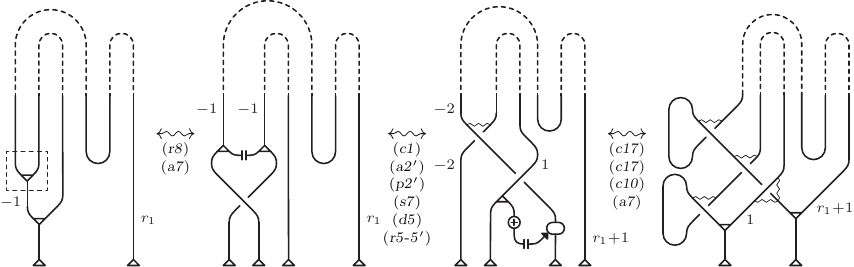}
 \caption{Proof of the invariance of $\barPhi(T)$ under the slide of $L_2$ over $L_1$, for $L_1$ with one negative self-crossing.}
 \label{proof-slidingneg-2/fig}
\end{figure}

In Figure~\ref{proof-slidingneg-2/fig}, we prove that the images of the first and third diagrams under $\barPhi$ are the same in $\Algf$. Notice that, in this case, as a consequence of relations \hrel[c16]{c16-17}, the decorated kink in the image of $L_2$ is equal to $\tau^{-2}$, and hence the total ribbon weight of the component is  $r_2=-{\rm wr}(L_2')=-1$. Then, the first diagram in Figure~\ref{proof-slidingneg-2/fig} has been obtained by repeating the first step in Figure~\ref{proof-sliding0-2/fig}.
\end{proof}

\subsection{Proof of Theorem~\ref{thm:main}}
\label{Psi/sec}

In this subsection, we prove one of the main results of this paper, which can be rephrased as follows.

\begin{theorem}\label{equivalence/thm}
The map $T \to \barPhi(T)$ extends to a braided monoidal functor $\barPhi : 4\KT \to \Algf$ such that $\barPhi \circ \Phi = \id_{\Algf}$ and $\Phi \circ \barPhi = \id_{4\KT}$. In particular, $\barPhi$ and $\Phi$ are equivalences of braided monoidal categories.
\end{theorem}

\begin{proof}
Since we have already proved in Proposition~\ref{sliding/thm} that $\barPhi(T)$ depends only on the \qvlnc{2} class of the Kirby tangle $T$, in order to show that $\barPhi : 4\KT \to \Algf$ is a well-defined monoidal functor, we only need to prove that it preserves identities, compositions, tensor products, and braidings. The proof that it preserves identities is shown in Figure~\ref{identity/fig}. 

\begin{figure}[htb]
 \centering
 \includegraphics{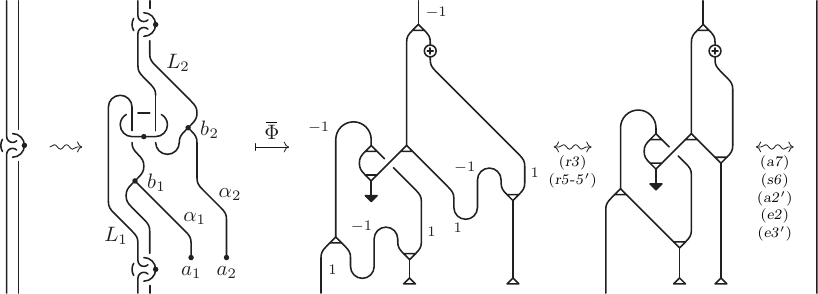}
 \caption{Proof that $\barPhi$ preserves identities.}
 \label{identity/fig}
\end{figure}

Let now $T_1 : E_{2s} \to E_{2t}$ and $T_2 : E_{2t} \to E_{2r}$ be Kirby tangles with $n$ and $m$ undotted components, respectively. Then, the link $L$ of undotted components of their composition $T_2 \circ T_1$ will have exactly $n+m-t$ components. In order to show that $\barPhi(T_2 \circ T_1) = \barPhi(T_2) \circ \barPhi(T_1)$, we make the special choice of bi-ascending state and arcs for $T_2 \circ T_1$ shown in Figure~\ref{composition1/fig}, where the undotted components of $L = L_1 \cup L_2 \cup \dots \cup L_{n+m-t}$ are numbered in such way that:
\begin{itemize}
 \item $L_1, L_2, \dots, L_{n-t}$ are the components of $T_1$ that are not attached to its target;
 \item $L_{n-t+1}, L_{n-t+2}, \dots, L_{n}$ are the components obtained from the gluing of the open components of $T_1$ to the ones of $T_2$ along $E_{2t}$, numbered following the order of the intervals in $E_{2t}$;
 \item $L_{n+1}, L_2, \dots, L_{n+m-t}$ are the components of $T_2$ that are not attached to its source.
\end{itemize}
We observe that the biascending state of each component $L_{n-t+i}$ for $1\leq i\leq t$ is  determined by the choice of points $p_{n-t+i}$ and $q_{n-t+i}$, as indicated in Figure~\ref{composition1/fig}, since the two arcs in which they divide the component never cross each other, while the biascending state of all other components can be chosen arbitrarily. Then, we choose the arcs $\alpha_{n-t+i}$ for $i=1,\dots,t$ in such a way that each $\alpha_{n-t+i}$ forms only positive crossings with the components $L_{n-t+j}$ for $j=i+1,\dots, t$, as represented in Figure~\ref{composition1/fig} 

\begin{figure}[htb]
 \centering
 \includegraphics{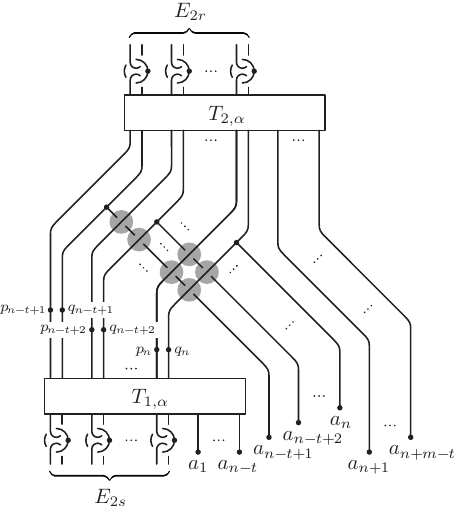}
 \caption{Special choice of bi-ascending state and arcs for $T_2 \circ T_1$.}
 \label{composition1/fig}
\end{figure}

\begin{figure}[htb]
 \centering
 \includegraphics{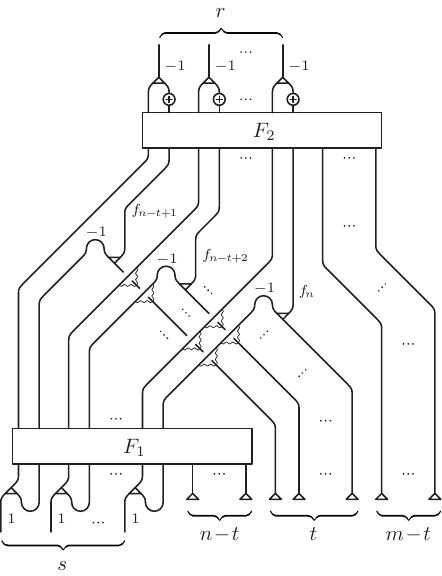}
 \caption{$\barPhi(T_2 \circ T_1)$.}
 \label{composition2/fig}
\end{figure}

\begin{figure}[htb]
 \centering
 \includegraphics{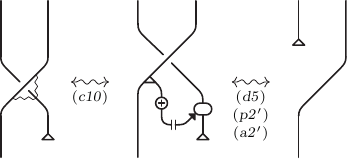}
 \caption{Retracting the images of the arcs $\alpha_{n-t+1}, \dots, \alpha_n$ in $T_2 \circ T_1$.}
 \label{composition3/fig}
\end{figure}

\begin{figure}[b]
 \centering
 \includegraphics{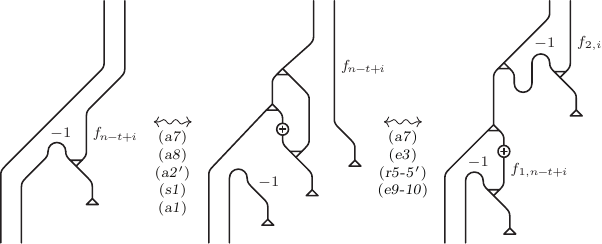}
 \caption{Doubling the retracted images of the arcs $\alpha_{n-t+1}, \dots, \alpha_n$.}
 \label{composition4/fig}
\end{figure}

Then, $\barPhi(T_2 \circ T_1)$ is presented in Figure~\ref{composition2/fig}. In order to see that it is equivalent to $\barPhi(T_2) \circ \barPhi(T_1)$, we first retract the images of the arcs $\alpha_{n-t+i}$ for $i=1,\dots, t$ by passing the identity morphisms through the adjoint morphisms in the decorated crossings of type $\hat X$, as shown in Figure~\ref{composition3/fig} (notice that the move in the figure is actually a special case of the extended version of \hrel{c24}, as discussed in the proof of Proposition~\ref{theta1/thm}). Then, we double the image of each arc through the move shown in Figure~\ref{composition4/fig}, and separate its weight as
\[
f_{n-t+i}-1 = -{\rm wr}(L_{n-t+i}')= -(1-f_{1,n-t+i})-(1-f_{2,i}) = f_{1,n-t+i}+f_{2,i}-2,
\]
where $(1-f_{1,n-t+i})$ is equal to the writhe of the bi-ascending state of the $(n-t+i)$th component of $T_1$, and $(1-f_{2,i})$ is equal to the writhe of the bi-ascending state of the $i$th component of $T_2$ (the numbering and the biascending states of the undotted components of $T_1$ and $T_2$ are induced by the ones of $T_2 \circ T_1$). Finally, by the inverse of the move presented in Figure~\ref{composition3/fig}, we pull down all identity morphisms back to the bottom-right corner of the diagram, thus creating new decorated crossings of type $\hat X$. The resulting diagram is exactly $\barPhi(T_2) \circ \barPhi(T_1)$.

In order to prove the monoidality of the functor $\barPhi$, let $T_1 : E_{2s_1} \to E_{2t_1}$ and $T_2 : E_{2s_2} \to E_{2t_2}$ be morphisms in $\KTf$. Consider $T_1 \sqcup T_2$, and order its undotted components by letting the ones of $T_1$ precede the ones of $T_2$; moreover, choose the arcs $\alpha_i$ by pulling the ones of $T_1$ across the $s_2$ vertical strands connected to the source of $T_2$, forming positive crossings with them. Then, the images under $\barPhi$ of such crossings are decorated crossings of type $\hat X$, and we can transform $\barPhi(T_1\sqcup T_2)$ into $\barPhi(T_1) \otimes \barPhi(T_2)$ by retracting the units at the end of the images of the arcs $\alpha_i$ of $T_1$ through the decorated crossings, by the move presented in Figure~\ref{composition3/fig}.

Finally, we recall that $\Phi \circ \barPhi = \id_{4\KT}$ has been proved in Proposition~\ref{full-phi/thm}, so it remains to prove that $\barPhi \circ \Phi = \id_{\Algf}$. In order to see this, it is enough to show that $\barPhi(\Phi(F)) = F$ when $F$ is a generating morphism of $\Algf$. The proofs for all elementary morphisms, with the exception of $\antipH^{-1}$, $\ribmorH^{-1}$, $\braid$, and $\copairH$ are presented (up to compositions with identity morphisms) in Figures~\ref{proof-inv-prod/fig}--\ref{proof-inv-cointegral/fig}. We observe that, in the second-to-last move of Figure~\ref{proof-inv-coprod/fig} and in the last move of Figure~\ref{proof-inv-antipode/fig}, we have expressed the decorated crossings of type $\hat X$ and $\hat Y$ in terms of the adjoint morphism, and we have intertwined the adjoint and the identity morphisms as we did in Figure~\ref{composition3/fig}. Now, the statements for $\antipH^{-1}$ and $\ribmorH^{-1}$ follow from the ones for $\antipH$ and $\ribmorH$, while the ones for $\braid$ and $\copairH$ follow from relation~\hrel{s8}, axiom~\hrel{r6}, and the fact that $\barPhi$ preserves compositions.
\end{proof}

\begin{figure}[hbt]
 \centering
 \includegraphics{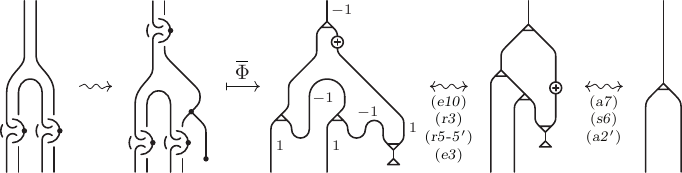}
 \caption{$\barPhi(\Phi(\prodH)) = \prodH$.}
 \label{proof-inv-prod/fig}
\end{figure}

\begin{figure}[hbt]
 \centering
 \includegraphics{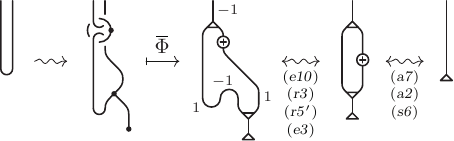}
 \caption{$\barPhi(\Phi(\unitH))=\unitH$.}
 \label{proof-inv-unity/fig}
\end{figure}

\begin{figure}[hbt]
 \centering
 \includegraphics{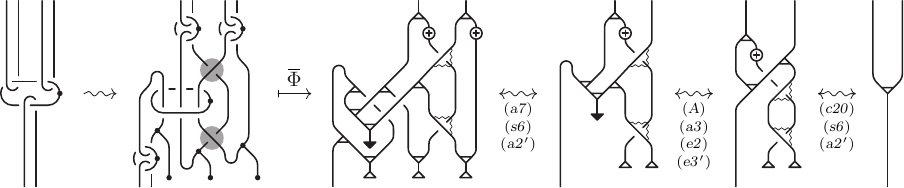}
 \caption{$\barPhi(\Phi(\coprH))=\coprH$.}
 \label{proof-inv-coprod/fig}
\end{figure}

\begin{figure}[hbt]
 \centering
 \includegraphics{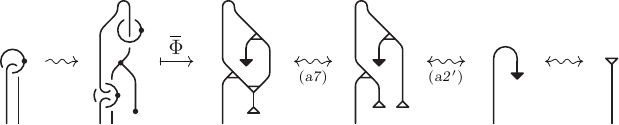}
 \caption{$\barPhi(\Phi(\counH))=\counH$.}
 \label{proof-inv-counit/fig}
\end{figure}

\begin{figure}[hbt]
 \centering
 \includegraphics{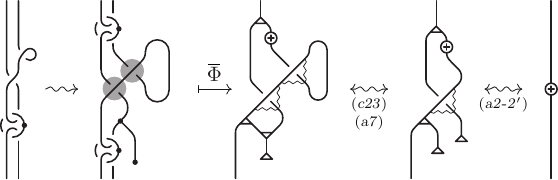}
 \caption{$\barPhi(\Phi(\antipH))=\antipH$.}
 \label{proof-inv-antipode/fig}
\end{figure}

\begin{figure}[hbt]
 \centering
 \includegraphics{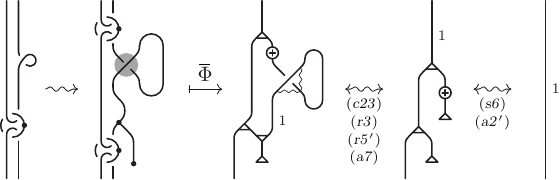}
 \caption{$\barPhi(\Phi(\ribmorH))=\ribmorH$.}
 \label{proof-inv-ribbon/fig}
\end{figure}

\begin{figure}[hbt]
 \centering
 \includegraphics{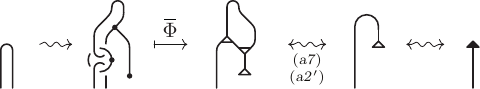}
 \caption{$\barPhi(\Phi(\intfH))=\intfH$.}
 \label{proof-inv-integral/fig}
\end{figure}

\begin{figure}[hbt]
 \centering
 \includegraphics{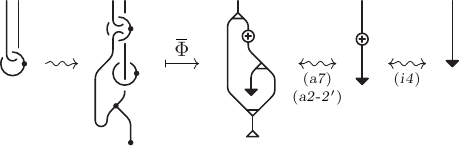}
 \caption{$\barPhi(\Phi(\inteH))=\inteH$.
 }
 \label{proof-inv-cointegral/fig}
\end{figure}

%% file: S6-appendix.tex
\numberwithin{figure}{section}
\numberwithin{table}{section}

\section{Tables.} \label{tables/app}

\begin{table}[hbt]
 \centering
 \includegraphics{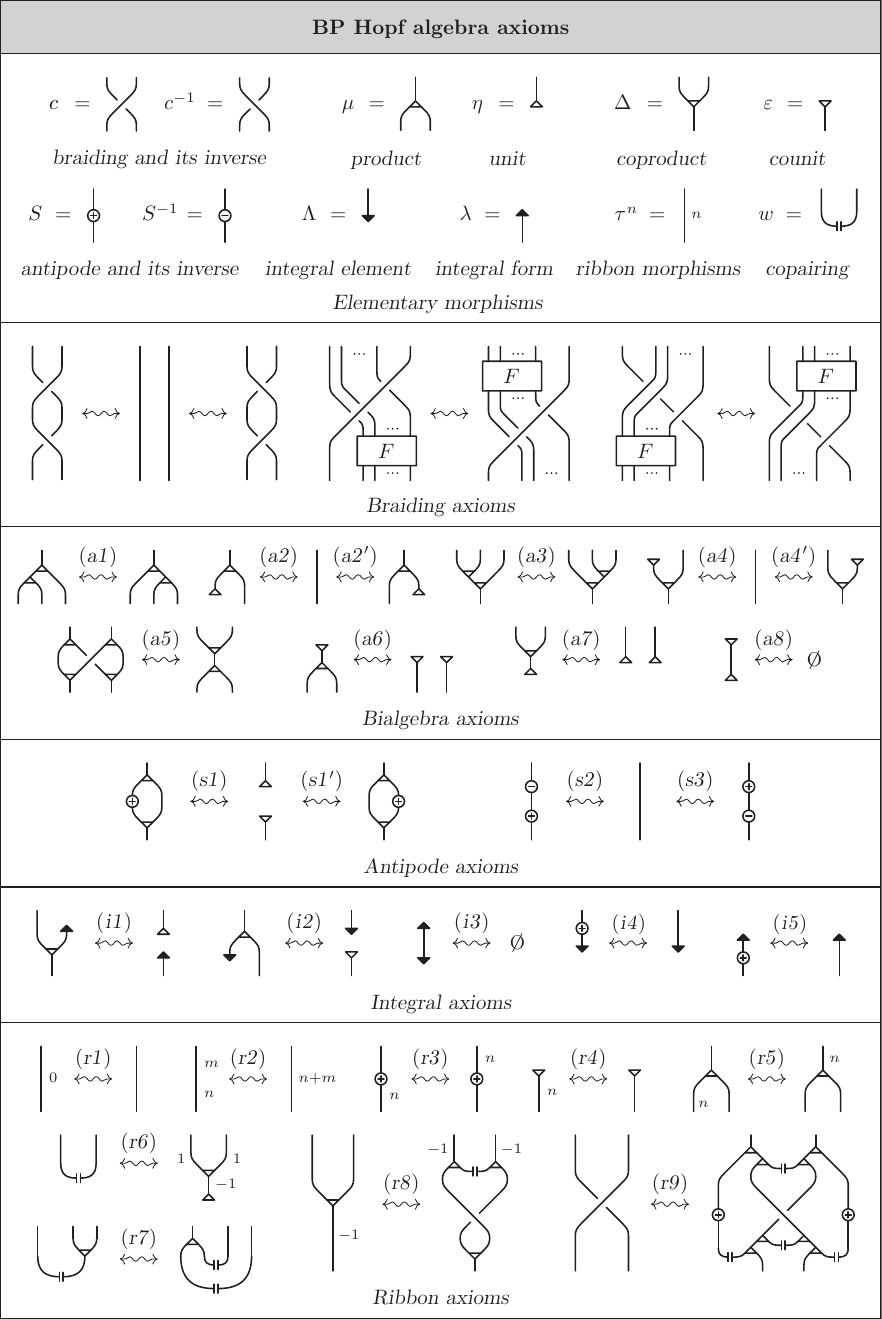}
 \caption{(Compare with Tables~\ref{table-braided/fig}, \ref{table-Hopf/fig}, and \ref{table-BPHopf/fig})}
 \label{table-app-BPHopf/fig}
 \label{E:BPH}
\end{table}

\begin{table}[hbt]
 \centering
 \includegraphics{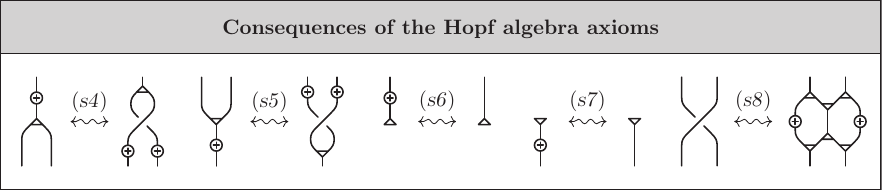}
 \caption{(Compare with Table~\ref{table-Hopf-prop/fig})}
 \label{table-app-antipode/fig}
 \label{E:ant}
\end{table}

\begin{table}[hbt]
 \centering
 \includegraphics{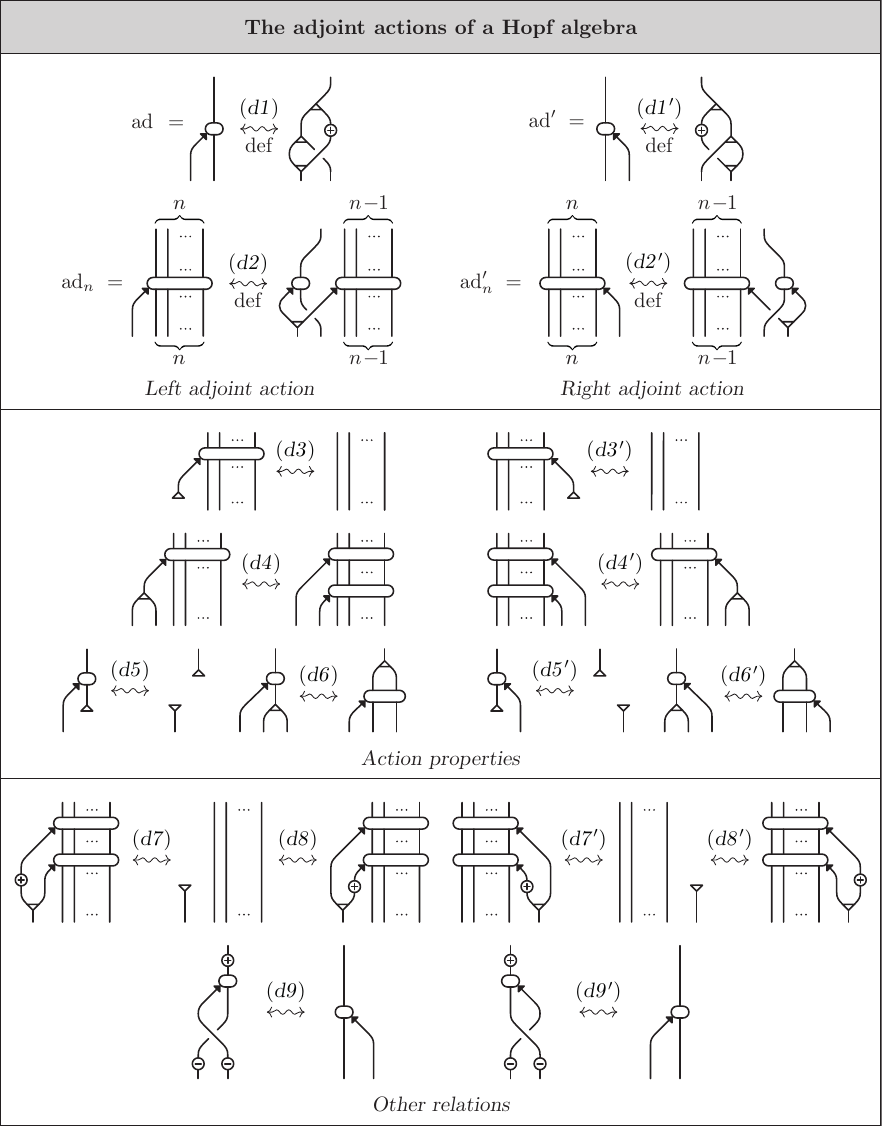}
 \caption{(Compare with Table~\ref{table-adjoint/fig})}
 \label{table-app-adjoint/fig}
 \label{E:adj}
\end{table}

\begin{table}[hbt]
 \centering
 \includegraphics{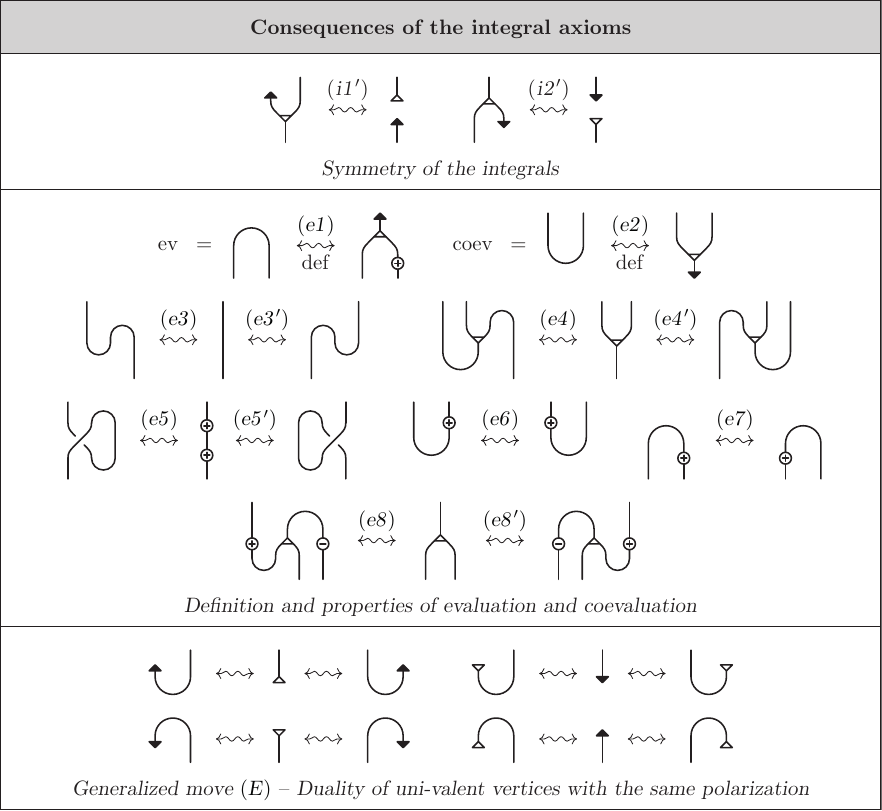}
 \caption{(Compare with Table~\ref{table-BPHopf-prop1/fig})}
 \label{table-app-integral/fig}
 \label{E:int}
\end{table}

\begin{table}[hbt]
 \centering
 \includegraphics{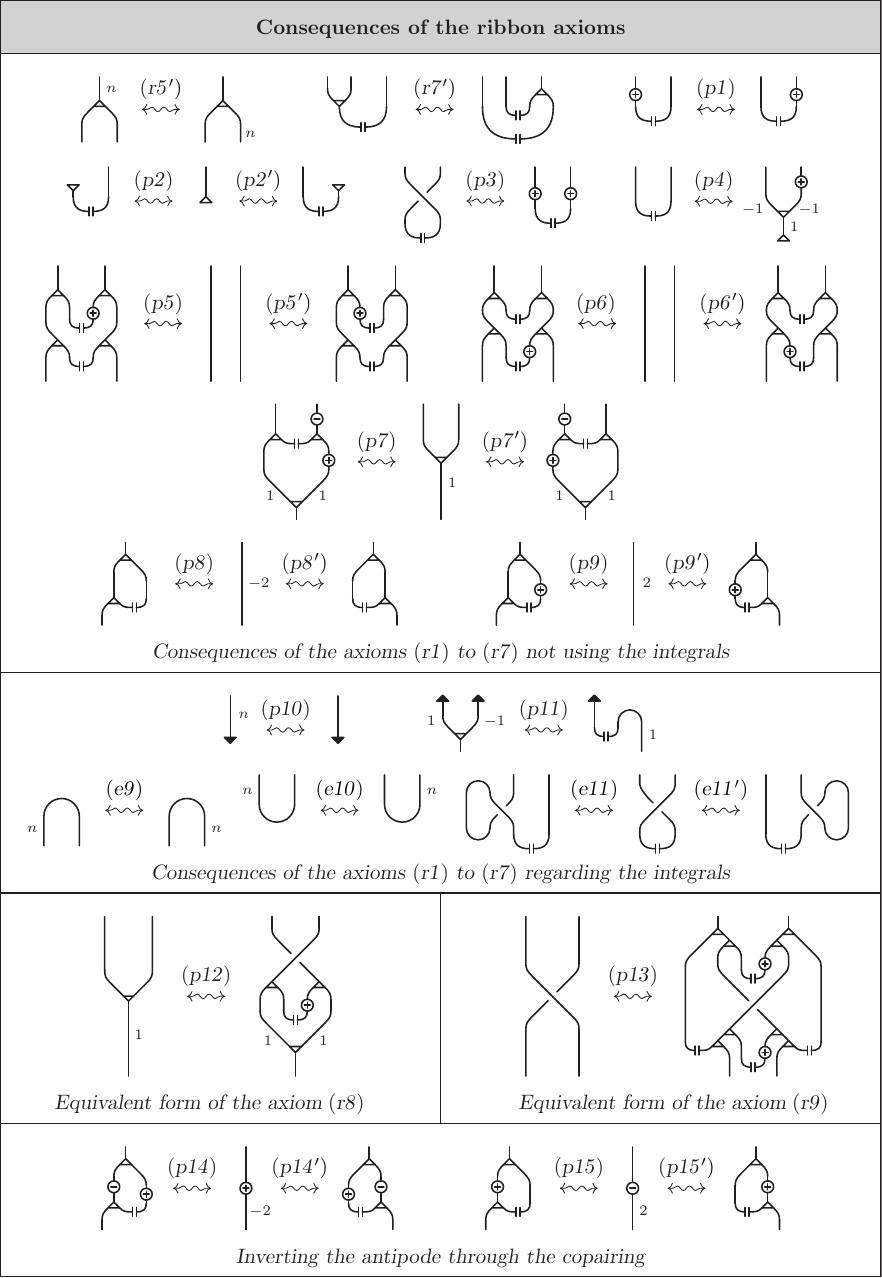}
 \caption{(Compare with Tables~\ref{table-BPHopf-prop21/fig}, \ref{table-BPHopf-prop22/fig}, and \ref{table-BPHopf-prop3/fig})}
 \label{table-app-ribbon/fig}
 \label{E:rib}
\end{table}

\begin{table}[hbt]
 \centering
 \includegraphics{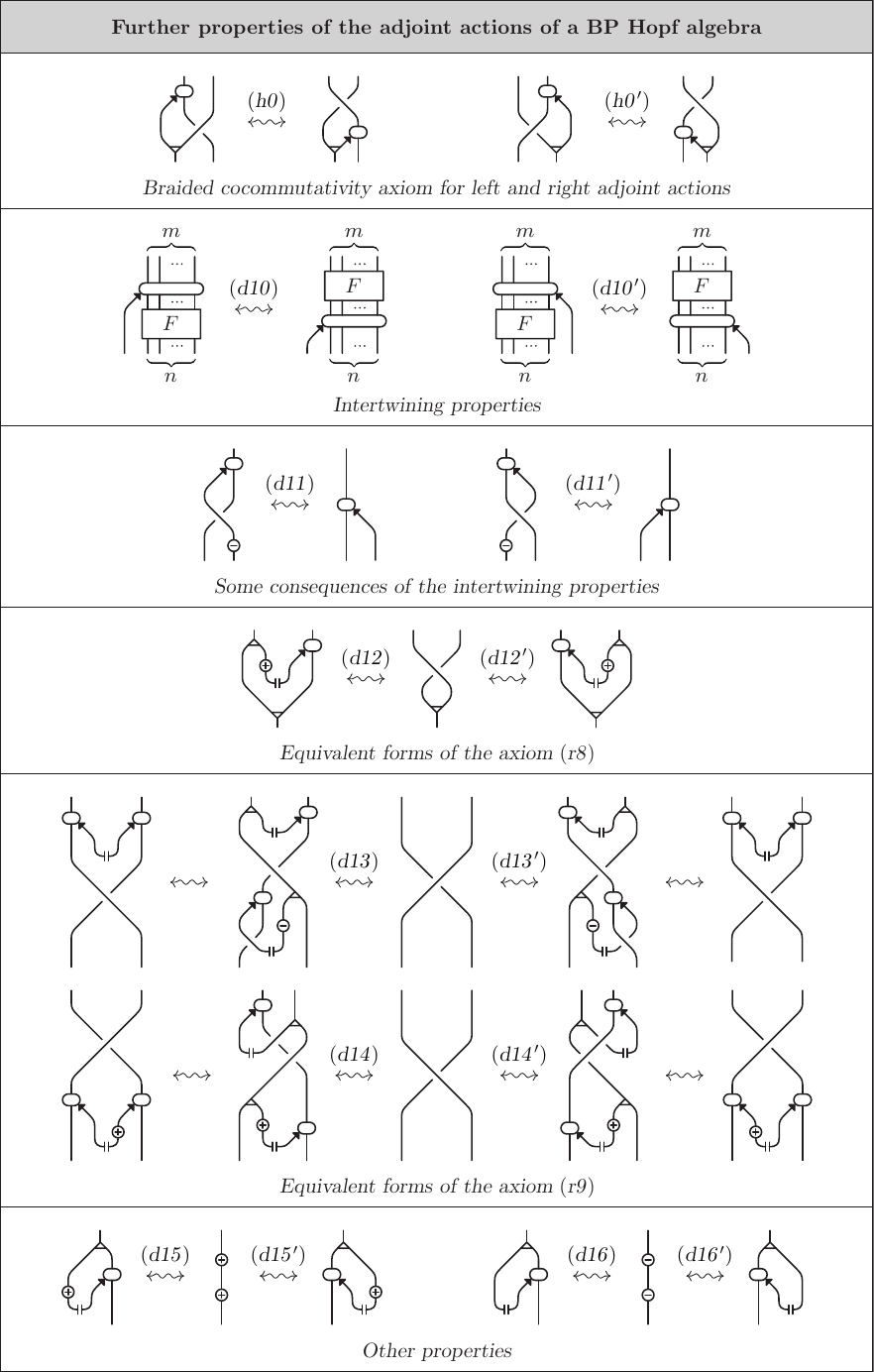}
 \caption{(Compare with Tables~\ref{table-cocommutative/fig} and \ref{table-adjoint-prop/fig})}
 \label{table-app-adjoint-prop/fig}
 \label{E:adjprop}
\end{table}

\begin{table}[htb]
 \centering
 \includegraphics{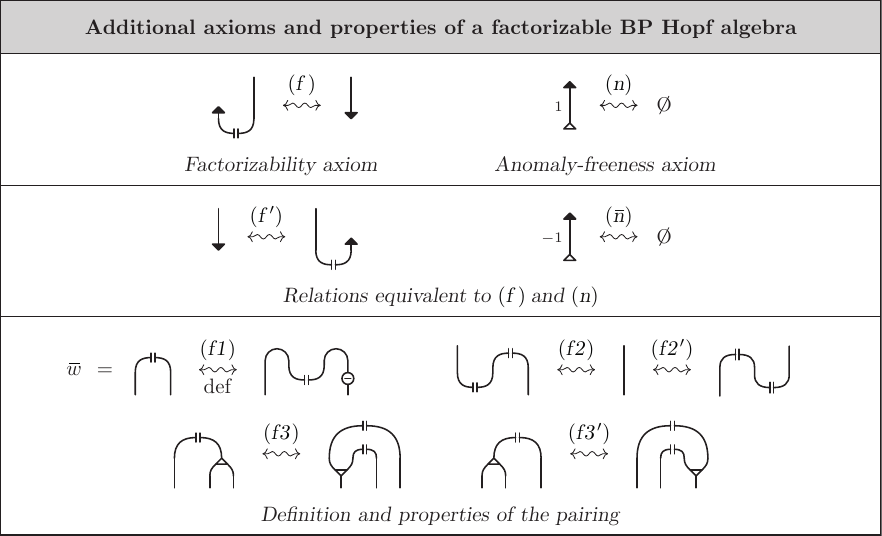}
 \caption{(Compare with Tables~\ref{table-BPHopf3/fig} and \ref{table-BPHopf3-prop/fig})}
 \label{table-app-BPHopf3/fig}
 \label{E:BPH3}
\end{table}

\FloatBarrier

\section{Proofs.} \label{proofs/app}

In this appendix, we give the proofs of the properties of a BP Hopf algebra presented in Tables~\ref{table-app-antipode/fig}, \ref{table-app-integral/fig}, \ref{table-app-ribbon/fig}, and \ref{table-app-BPHopf3/fig}. Some of them, like the ones concerning the antipode in Table~\ref{table-app-antipode/fig} and the symmetry of the integrals in Table~\ref{table-app-integral/fig}, are well-known, and can be found in any basic textbook on Hopf algebras. Also, the rest of the properties have already appeared in the literature. For example, the non-degeneracy of $\ev$ and $\coev$ (relations \hrel[int]{e8-8'}) were proven by Kerler in \cite[Lemma~7]{Ke01}. He also showed\footnote{Kerler's axioms use ribbon elements instead of ribbon morphisms, but the two languages are equivalent.} in \cite[Lemmas~3 \& 4]{Ke01} that relation \hrel[rib]{p4} is equivalent to the ribbon axiom \hrel[BPH]{r6} modulo the Hopf algebra axioms together with the ribbon axioms \hrel[BPH]{r1}--\hrel[BPH]{r5}, and that those ribbon axioms imply \hrel[rib]{p1} and \hrel[rib]{r7'}. The diagrammatic proofs of all relations, with the exception of \hrel[ant]{s8} and \hrel[rib]{p3}, appear in \cite[Propositions/Lemmas~4.1.4, 4.1.5, 4.1.6, 4.1.9, 4.1.10, 4.2.5, 4.2.6, 4.2.7, 4.2.11, 4.2.13]{BP11} in the more general context of a groupoid Hopf algebra. The reason why we present the proofs here is, on the one hand, for the sake of completeness, and, on the other hand, because the equivalence results in Subsection~\ref{HabiroHalgebra/sec} require the precise knowledge of which properties of the algebra follow from which set of axioms.

\numberwithin{figure}{subsection}
\numberwithin{table}{subsection}

\subsection{Consequences of the braided Hopf algebra axioms in Table~\ref{table-app-antipode/fig}}
\label{hopf-proofs/sec}

\begin{proof}[Proof of Proposition~\ref{hopf-prop/thm}]
 We have to show the properties of the antipode in Table~\ref{table-app-antipode/fig} (which coincides with Table~\ref{table-Hopf-prop/fig}). Relations~\hrel[ant]{s4-5}, \hrel[ant]{s6-7}, and \hrel[ant]{s8} are proved in Figures~\ref{BP-proof-s4/fig}, \ref{BP-4-1-3/fig}, \ref{BP-4-1-4a/fig}, and \ref{braiding-proof/fig}\footnote{Note that there is an equivalence functor from the category $\Alg$ to its opposite, which in particular implies that if some relation is satisfied in $\Alg$, then the relation obtained from it by rotating the two diagrams of an angle $\pi$ around an axis perpendicular to the projection plane holds as well. This is how one can obtain \hrel[ant]{s5} from \hrel[ant]{s4} and \hrel[ant]{s7} from \hrel[ant]{s6}.}.
\end{proof}

\begin{figure}[hbt]
 \centering
 \includegraphics{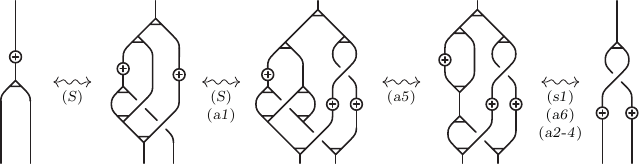}
 \caption{Proof of \hrel[ant]{s4}.}
 \label{BP-proof-s4/fig}
\end{figure}

\begin{figure}[hbt]
 \centering
 \includegraphics{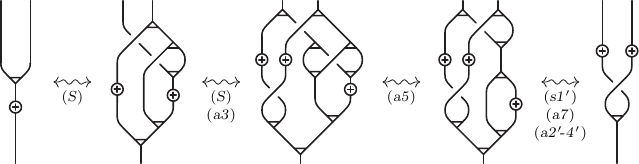}
 \caption{Proof of \hrel[ant]{s5}.}
 \label{BP-4-1-3/fig}
\end{figure}

\begin{figure}[hbt]
 \centering
 \includegraphics{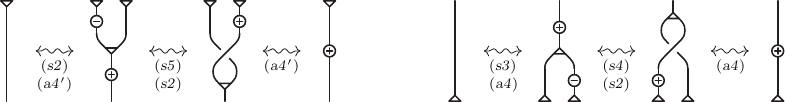}
 \caption{Proof of \hrel[ant]{s6} and \hrel[ant]{s7}.}
 \label{BP-4-1-4a/fig}
\end{figure}

\begin{figure}[htb]
    \centering
    \includegraphics{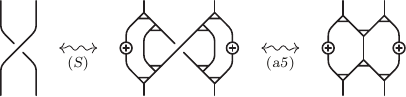}
    \caption{Proof of \hrel[ant]{s8}.}
    \label{braiding-proof/fig}
\end{figure}

\FloatBarrier

\subsection{Consequences of the integral axioms in Table~\ref{table-app-integral/fig}}
\label{integral-proofs/sec}

\begin{lemma}\label{intHopf-sym/thm}
Let $\Alg^\rmU$ be the braided monoidal category freely generated by a braided Hopf algebra with $\antipH$-invariant integral form and element\footnote{Following \cite{BP11}, the superscript $\rmU$ stands for \textit{unimodular}.}. Then, the functor $\sym : \Alg^\rmU \to \Alg^\rmU$ is an involutive braided anti-monoidal equivalence functor that sends $H$ and all elementary morphisms to themselves.
\end{lemma}

\begin{proof}
The statement follows from Propositions~\ref{funt-sym/thm} and \ref{symmetry-alg/thm}, once we observe that, thanks to axioms~\hrel[BPH]{i4-5}, the functor $\sym$ sends all elementary morphisms of $\Alg^\rmU$, including the integral form and element, to themselves.
\end{proof}

\begin{proof}[Proof of Proposition~\ref{BP-prop1/thm}] We will show that the properties in Table~\ref{table-app-integral/fig} (which coincides with Table~\ref{table-BPHopf-prop1/fig}) hold in $\Alg^\rmU$.
Relations~\hrel[int]{i1'} and \hrel[int]{i2'} in Table~\ref{table-app-integral/fig} state the well-known fact that the $\antipH$-invariance of the integral form $\intfH$ and of the integral element $\inteH$ imply that $\intfH$ and $\inteH$ are, respectively, a two-sided integral form and a two-sided integral element. They follow directly from Lemma~\ref{intHopf-sym/thm} by applying the functor $\sym$ to \hrel[int]{i1} and \hrel[int]{i2}, respectively, while in Figures~\ref{BP-4-1-5/fig}, \ref{BP-4-1-5b-new/fig}, \ref{BP-4-1-7/fig}, \ref{BP-4-1-4-new/fig}, \ref{BP-4-1-10/fig}, \ref{BP-4-1-11-new/fig}, and \ref{BP-4-1-11/fig} we prove in order relations~\hrel[int]{e3}, \hrel[int]{e3'}, \hrel[int]{e4}, \hrel[int]{e5'}, \hrel[int]{e6}, \hrel[int]{e7}, and \hrel[int]{e8}. Notice that relations~\hrel[int]{e6-7} imply that $\sym(\ev)=\ev$ and $\sym(\coev)=\coev$.

Then, relations~\hrel[int]{e4'}, \hrel[int]{e5}, and \hrel[int]{e8'} follow from Lemma~\ref{intHopf-sym/thm} by applying $\sym$ to \hrel[int]{e4}, \hrel[int]{e5'}, and \hrel[int]{e8}, respectively. Finally, half of the relations that concern the duality of univalent vertices with the same polarization are proved in Figure~\ref{BP-4-1-12-new/fig}, while the rest follows from them by applying $\sym$.
\end{proof}


\begin{figure}[hbt]
 \centering
 \includegraphics{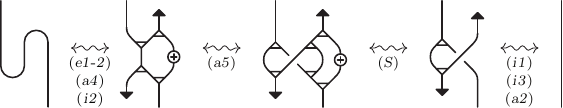}
  \caption{Proof of \hrel[int]{e3}.}
 \label{BP-4-1-5/fig}
\end{figure}

\begin{figure}[hbt]
 \centering
 \includegraphics{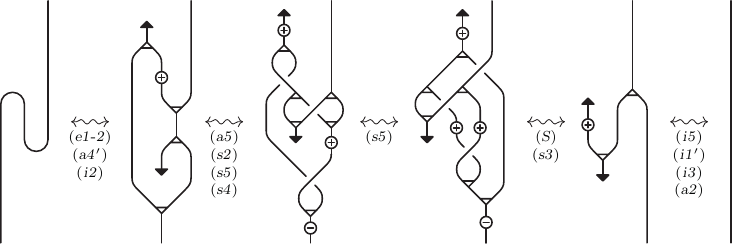}
 \caption{Proof of \hrel[int]{e3'}.}
 \label{BP-4-1-5b-new/fig}
\end{figure}

\begin{figure}[hbt]
 \centering
 \includegraphics{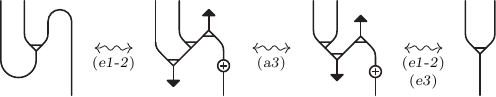}
 \caption{Proof of \hrel[int]{e4}.}
 \label{BP-4-1-7/fig}
\end{figure}

\begin{figure}[hbt]
 \centering
 \includegraphics{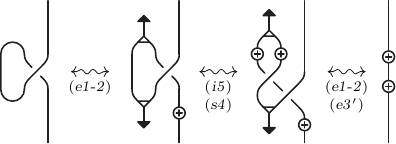}
 \caption{Proof of \hrel[int]{e5'}.}
 \label{BP-4-1-4-new/fig}
\end{figure}


\begin{figure}[hbt]
 \centering
 \includegraphics{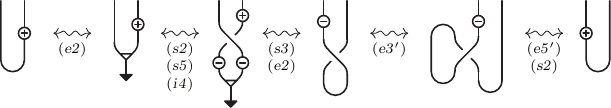}
 \caption{Proof of \hrel[int]{e6}.}
 \label{BP-4-1-10/fig}
\end{figure}

\begin{figure}[hbt]
 \centering
 \includegraphics{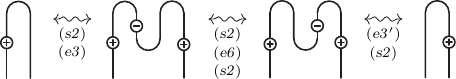}
 \caption{Proof of \hrel[int]{e7}.}
 \label{BP-4-1-11-new/fig}
\end{figure}

\begin{figure}[hbt]
 \centering
 \includegraphics{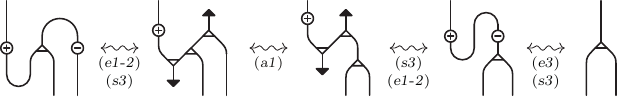}
 \caption{Proof of \hrel[int]{e8}.}
 \label{BP-4-1-11/fig}
\end{figure}

\begin{figure}[hbt]
 \centering
 \includegraphics{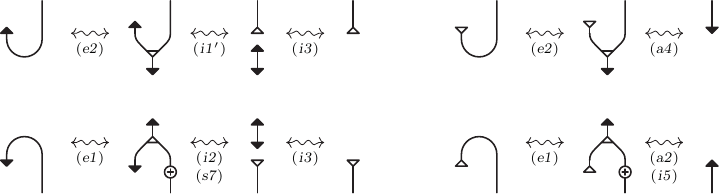}
 \caption{Proof of the generalized move~\rel{E} -- duality of univalent vertices with the same polarization.}
 \label{BP-4-1-12-new/fig}
\end{figure}

\FloatBarrier

It will be useful to introduce a notation for a general class of equivalence moves, which are deduced from \hrel[int]{e5-5'} by applying \hrel[int]{e6-7} and the braiding axioms.





\FloatBarrier

\subsection{Properties of the ribbon structure of a BP Hopf algebra in Table~\ref{table-app-ribbon/fig}}
\label{ribbon-proofs/sec}

\begin{lemma}\label{ribHopf-sym/thm}
Let $\Alg^\rmR$ be the braided monoidal category freely generated by a braided Hopf algebra $H$ with a ribbon morphism $\ribmorH : H \to H$ and a copairing $\copairH : \one \to H \otimes H$ that satisfy axioms \hrel[BPH]{r1}--\hrel[BPH]{r7}. Then, the functor $\sym : \Alg^\rmR \to \Alg^\rmR$ is an involutive braided anti-monoidal equivalence functor that sends $H$ and all elementary morphisms to themselves.
\end{lemma}

\begin{proof}
The statement follows from Propositions~\ref{funt-sym/thm} and \ref{symmetry-alg/thm}, once we show that $\sym$ also sends the ribbon morphism and the copairing to themselves. For the ribbon morphisms, this follows from axiom~\hrel[BPH]{r3}, while, for the copairing, it follows from \hrel[BPH]{r6} (see Figure~\ref{proof-inv-copairing/fig}). 
\end{proof}


\begin{proof}[Proof of Proposition \ref{BP-prop21/thm}]
We will show that the properties in the first section of Table~\ref{table-app-ribbon/fig} (which coincides with Table~\ref{table-BPHopf-prop21/fig}) hold in $\Alg^\rmR$, meaning that they are consequences of the ribbon axioms \hrel[BPH]{r1}--\hrel[BPH]{r7} together with the braided Hopf algebra axioms (without assuming the existence of integrals).

Relations~\hrel[rib]{r5'} and \hrel[rib]{r7'} follow from Lemma~\ref{ribHopf-sym/thm} by applying $\sym$ to \hrel[rib]{r5} and \hrel[rib]{r7}, respectively.

Relation~\hrel[rib]{p3} is equivalent to the invariance of the copairing under $\sym$ (as can be seen by composing the diagrams in Figure~\ref{proof-inv-copairing/fig} with $\braid^{-1}$ on the top).
 
Relation~\hrel[rib]{p2} is obtained by first composing \hrel[BPH]{r6} on the left with the counit, and then applying \hrel[BPH]{r4} and \hrel[BPH]{a4}, while relation~\hrel[rib]{p2'} follows from \hrel[rib]{p2} by applying $\sym$.

Relation~\hrel[rib]{p5} follows from \hrel[BPH]{r7} and \hrel[rib]{p2}, as shown in Figure~\ref{BP-4-2-7/fig}. 

\begin{figure}[htb]
 \centering
 \includegraphics{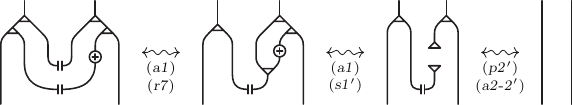}
 \caption{Proof of \hrel[rib]{p5}.}
 \label{BP-4-2-7/fig}
\end{figure}

Relation~\hrel[rib]{p6} is proved in the same way as \hrel[rib]{p5}, using \hrel[BPH]{s1} instead of \hrel[BPH]{s1'}.

Relations~\hrel[rib]{p5'} and \hrel[rib]{p6'} follow from Lemma~\ref{ribHopf-sym/thm} by applying $\sym$ to \hrel[rib]{p5} and \hrel[rib]{p6}, respectively.

Relation~\hrel[rib]{p1} follows from the fact that, according to relations~\hrel[rib]{p5-5'} and \hrel[rib]{p6-6'}, both morphisms
\begin{gather*}
\barmonH = (\prodH \otimes \prodH) \circ (\id \otimes ((\id \otimes \antipH) \circ \copairH) \otimes \id) : H \otimes H \to H \otimes H \\
\barmonH' = (\prodH \otimes \prodH) \circ (\id \otimes ((\antipH \otimes \id) \circ \copairH) \otimes \id) : H \otimes H \to H \otimes H
\end{gather*}
are two-sided inverses of the monodromy morphism $\monH$ (introduced in Definition~\ref{Hr/def}), so they must be the same.

Relation~\hrel[rib]{p4} is proved in Figure~\ref{proof-eq-r10-r6/fig}. 

\begin{figure}[htb]
 \centering
 \includegraphics{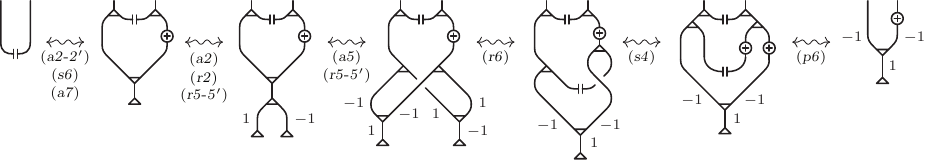}
 \caption{Proof of \hrel[rib]{p4}.}
 \label{proof-eq-r10-r6/fig}
\end{figure}

Relations~\hrel[rib]{p7}, \hrel[rib]{p8}, and \hrel[rib]{p9} are proved in Figures~\ref{BP-4-2-4A/fig}, and \ref{BP-4-2-4B/fig}, while relations~\hrel[rib]{p7'}, \hrel[rib]{p8'}, and \hrel[rib]{p9'} follow from them by applying the functor $\sym$ (Lemma~\ref{ribHopf-sym/thm}).

\begin{figure}[htb]
 \centering
 \includegraphics{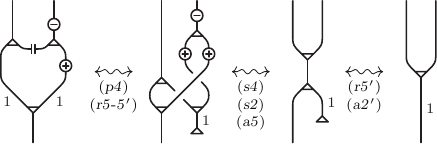}
 \caption{Proof of \hrel[rib]{p7}.}
 \label{BP-4-2-4A/fig}
\end{figure}

\begin{figure}[htb]
 \centering
 \includegraphics{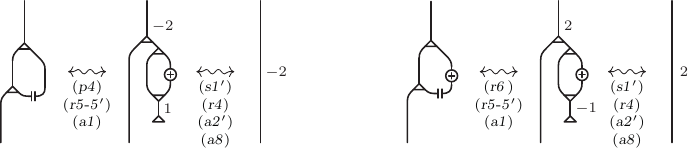}
 \caption{Proofs of \hrel[rib]{p8} and \hrel[rib]{p9}.}
 \label{BP-4-2-4B/fig}
\end{figure}


\end{proof}

\begin{proof}[Proof of Proposition~\ref{BP-prop22/thm}]
We proceed now with the proof of the identities in the second section of Table~\ref{table-app-ribbon/fig} (which coincides with Table~\ref{table-BPHopf-prop22/fig}) concerning the relation between the ribbon structure and the integrals.

Relation~\hrel[rib]{p10} is proved in Figure~\ref{proof-p10-new/fig}. 

\begin{figure}[hbt]
 \centering
 \includegraphics{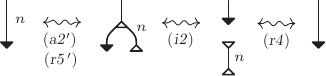}
 \caption{Proof of \hrel[rib]{p10}.}
 \label{proof-p10-new/fig}
\end{figure}

Relation~\hrel[rib]{e9} follows immediately from \hrel[BPH]{r5}, \hrel[rib]{r5'}, and \hrel[rib]{r3}.

Relations~\hrel[rib]{e11} and \hrel{e11'} follow from \hrel[rib]{p3}, \hrel[rib]{p1}, and \hrel[int]{e5-5'}. 

Relations~\hrel[rib]{p11} and \hrel[rib]{e10} are proved in Figures~\ref{BP-4-2-9/fig} and \ref{BP-4-2-5/fig}, respectively. 
\end{proof}

\begin{figure}[hbt]
 \centering
 \includegraphics{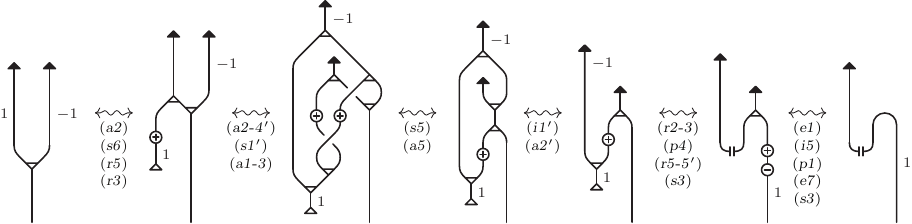}
 \caption{Proof of \hrel[rib]{p11}.}
 \label{BP-4-2-9/fig}
\end{figure}

\begin{figure}[hbt]
 \centering
 \includegraphics{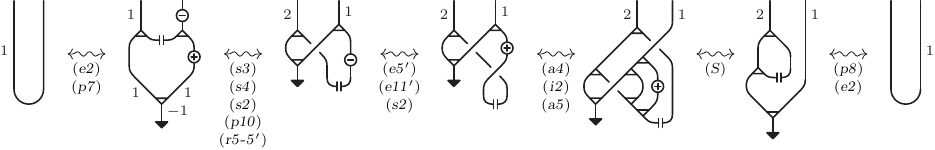}
 \caption{Proof of \hrel[rib]{e10}.}
 \label{BP-4-2-5/fig}
\end{figure}

\begin{proof}[Proof of Proposition \ref{BP-prop3/thm}]
Finally, we show that the properties in the last three sections of Table~\ref{table-app-ribbon/fig} (which coincides with Table~\ref{table-BPHopf-prop3/fig}) hold in $\Algf$. 

The equivalence between relation~\hrel[rib]{p12} and axiom~\hrel[BPH]{r8} follows from the fact that \hrel[rib]{p12} can be obtained by composing both sides of \hrel[BPH]{r8} with the invertible morphisms $\ribmorH$, on the bottom, and $\braid \circ \barmonH \circ (\ribmorH \otimes \ribmorH)$, on the top. Indeed, the right-hand side of such composition is equal to the left-hand side of \hrel[rib]{p12}, while the left-hand side can be transformed into the right-hand side of \hrel[rib]{p12} by applying \hrel[rib]{p5}.

Analogously, modulo the rest of the algebra axioms, relation~\hrel[rib]{p13} is equivalent to axiom~\hrel[BPH]{r9}. Indeed, to see that \hrel[BPH]{r9} implies \hrel[rib]{p13}, it is enough to observe that the diagram on the right-hand side of \hrel[rib]{p13} can be reduced to the single crossing on the left-hand side by applying \hrel[BPH]{r9} at the crossing in the middle, and then  four moves type \hrel[rib]{p5'-5-6} to cancel the corresponding copairings. The opposite argument shows that \hrel[rib]{p13} implies \hrel[BPH]{r9} as well. 

Relations~\hrel[rib]{p14} and \hrel[rib]{p15} are proved in Figures~\ref{BP-4-2-14/fig} and \ref{proof-p15/fig}, while relations~\hrel[rib]{p14'} and \hrel[rib]{p15'} follow by applying $\sym$ to \hrel[rib]{p14} and \hrel[rib]{p15}, respectively (see Proposition~\ref{symmetry/thm}).

\begin{figure}[hbt]
 \centering
 \includegraphics{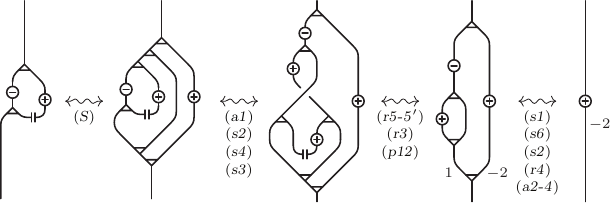}
 \caption{Proof of \hrel[rib]{p14}.}
 \label{BP-4-2-14/fig}
\end{figure}

\begin{figure}[hbt]
 \centering
 \includegraphics{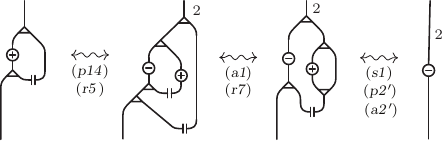}
 \caption{Proof of \hrel[rib]{p15}.}
 \label{proof-p15/fig}
\end{figure}

\end{proof}

\FloatBarrier

\subsection{Properties of a factorizable anomaly free BP Hopf algebra in Table~\ref{table-app-BPHopf3/fig}} \label{BPHopf3-proofs/sec}

We have to show the properties in the last two sections of Table~\ref{table-app-BPHopf3/fig} (which coincides with Table~\ref{table-BPHopf3-prop/fig}). 
Relation~\hrel[BPH3]{f'} is equivalent to \hrel[BPH3]{f} modulo \hrel[BPH]{i5}, \hrel[int]{e5}, and \hrel[rib]{e11} (general isotopy moves). Relation~\hrel[BPH3]{f2} is proved in Figure~\ref{proof-d2-app/fig}, and then relation~\hrel[BPH3]{f2'} follows by symmetry, while, using \hrel[BPH3]{f2-2'}, one can easily derive relations~\hrel[BPH3]{f3-3'} from \hrel[BPH]{r7} and \hrel[rib]{r7'}. Finally, relation~\hrel[BPH3]{\bn} is proved in Figure~\ref{proof-nbar-app/fig}.

\begin{figure}[hbt]
 \centering
 \includegraphics{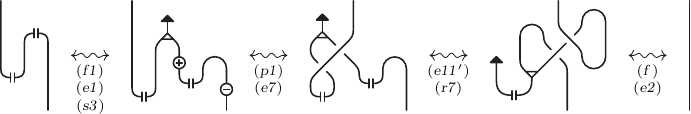}
 \caption{Proof of \hrel[BPH3]{f2}.}
 \label{proof-d2-app/fig}
\end{figure}

\begin{figure}[hbt]
 \centering
 \includegraphics{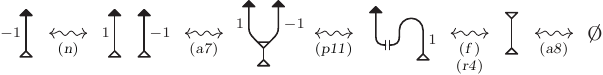}
 \caption{Proof of \hrel[BPH3]{\bn}.}
 \label{proof-nbar-app/fig}
\end{figure}

%% file: algebraic_presentation_of_4-dimensional_2-handlebodies_and_3-dimensional_cobordisms.bbl
\begin{thebibliography}{9}

\bibitem[Ak16]{Ak16}
S.~Akbulut, 
\textit{4-Manifolds}, 
\doi{10.1093/acprof:oso/9780198784869.001.0001}
{Oxford Grad. Texts Math. \textbf{25}, Oxford University Press, Oxford, 2016}.

\bibitem[As11]{As11}
M.~Asaeda,
\textit{Tensor Functors on a Certain Category Constructed From Spherical Categories},
\doi{10.1142/S0218216511008607}
{J. Knot Theory Ramifications \textbf{20} (2011), no.~1, 1--46}. 

\bibitem[BD21]{BD21}
A.~Beliakova, M.~De~Renzi, 
\textit{Kerler--Lyubashenko Functors on $4$-Di\-men\-sion\-al $2$-Han\-dle\-bod\-ies},
\doi{10.1093/imrn/rnac039}
{Int. Math. Res. Not. IMRN \textbf{2024} (2024), no.~13, 10005--10080};
\arxiv{2105.02789}{[math.GT]}.

\bibitem[BD22]{BD22}
A.~Beliakova, M.~De~Renzi,
\textit{Refined Bobtcheva--Messia Invariants of 4-Dimensional 2-Handlebodies},
\doi{10.4171/IRMA/34/20}
{Essays in Geometry, 387--432, IRMA Lect. Math. Theor. Phys. \textbf{34}, Eur. Math. Soc., Zürich, 2023};
\arxiv{2205.11385}{[math.GT]}.

\bibitem[Bo20]{Bo20}
I.~Bobtcheva,
\textit{On the Algebraic Characterization of the Category of $3$-Dimensional Cobordisms};
\arxiv{2008.06706}{[math.GT]}.

\bibitem[Bo23]{Bo23}
I.~Bobtcheva,
\textit{Algebraic Characterisation of the Category of Cobordisms of
$2$-dimensional CW-complexes and the Andrews--Curtis Conjecture};
\arxiv{2309.04830}{[math.GT]}.

\bibitem[BM02]{BM02}
I.~Bobtcheva, M.~Messia,
\textit{HKR-Type Invariants of $4$-Thickenings of $2$-Dimensional CW Complexes},
\doi{10.2140/agt.2003.3.33}
{Algebr. Geom. Topol. \textbf{3} (2003), no.~1, 33--87};
\arxiv{math/0206307}{[math.QA]}.

\bibitem[BP11]{BP11}
I.~Bobtcheva, R.~Piergallini,
\textit{On $4$-Dimensional $2$-Handlebodies and $3$-Manifolds},
\doi{10.1142/S0218216512501106}
{J. Knot Theory Ramifications \textbf{21} (2012), no.~12, 1250110, 230~pp};
\arxiv{1108.2717}{[math.GT]}.

\bibitem[CY94]{CY94}
L.~Crane, D.~Yetter,
\textit{On Algebraic Structures Implicit in Topological Quantum Field Theories},
\doi{10.1142/S0218216599000109}
{J. Knot Theory Ramifications \textbf{8} (1999), no.~2, 125--163};
\href{http://arXiv.org/abs/hep-th/9412025}
{\texttt{arXiv:hep-th/9412025}}.

\bibitem[EGNO15]{EGNO15}
P.~Etingof, S.~Gelaki, D.~Nikshych, V.~Ostrik,
\textit{Tensor Categories},
\doi{10.1090/surv/205}
{Math. Surveys Monogr. \textbf{205}, Amer. Math. Soc., Providence, RI, 2015}.

\bibitem[FS10]{FS10}
J.~Fuchs, C.~Schweigert,
\textit{Hopf Algebras and Finite Tensor Categories in Conformal Field Theory},
\href{https://inmabb.criba.edu.ar/revuma/revuma.php?p=toc/vol51}
{Rev. Un. Mat. Argentina \textbf{51} (2010), no.~2, 43--90};
\arxiv{1004.3405}{[hep-th]}.


\bibitem[Go91]{Go91}
R.~Gompf,
\textit{Killing the Akbulut--Kirby 4-Sphere, With Relevance to the Andrews--Curtis and Schoenflies Problems},
\doi{10.1016/0040-9383(91)90036-4}
{Topology \textbf{30} (1991), no.~1, 97--115}.

\bibitem[GS99]{GS99}
R.~Gompf, A.~Stipsicz,
\textit{4-Manifolds and Kirby Calculus},
\doi{10.1090/gsm/020}
{Grad. Stud. Math. \textbf{20}, American Mathematical Society, Providence, RI, 1999}.

\bibitem[Ha00]{Ha00}
K.~Habiro,
\textit{Claspers and Finite Type Invariants of Links},
\doi{10.2140/gt.2000.4.1}
{Geom. Topol. \textbf{4} (2000), no.~1, 1--83};
\arxiv{math/0001185}{[math.GT]}.

\bibitem[Ha05]{Ha05}
K.~Habiro,
\textit{Bottom Tangles and Universal Invariants},
\doi{10.2140/agt.2006.6.1113}
{Algebr. Geom. Topol. \textbf{6} (2006), no.~3, 1113--1214};
\arxiv{math/0505219}{[math.GT]}.

\bibitem[Ha22]{Ha22}
K.~Habiro,
private communication.

\bibitem[Ju14]{Ju14}
A.~Juhász,
\textit{Defining and Classifying TQFTs via Surgery},
\doi{10.4171/QT/108}
{Quantum Topol. \textbf{9} (2018), no.~2, 229--321};
\arxiv{1408.0668}{[math.GT]}.

\bibitem[Ke98]{Ke98}
T.~Kerler,
\textit{Bridged Links and Tangle Presentations of Cobordism Categories},
\doi{10.1006/aima.1998.1772}
{Adv. Math. \textbf{141} (1999), no.~2, 207--281};
\arxiv{math/9806114}{[math.GT]}.

\bibitem[Ke01]{Ke01}
T.~Kerler,
\textit{Towards an Algebraic Characterization of $3$-Dimensional Cobordisms},
\doi{10.1090/conm/318}
{Diagrammatic Morphisms and Applications (San Francisco, CA, 2000), 141--173,
Contemp. Math. \textbf{318}, Amer. Math. Soc., Providence, RI, 2003};
\arxiv{math/0106253}{[math.GT]}.

\bibitem[KL01]{KL01}
T.~Kerler, V.~Lyubashenko,
\textit{Non-Semisimple Topological Quantum Field Theories for 3-Manifolds with Corners},
\doi{10.1007/b82618}
{Lecture Notes in Math. \textbf{1765}, Springer-Verlag, Berlin, 2001}.

\bibitem[Ki89]{Ki89}
R.~Kirby,
\textit{The Topology of $4$-Manifolds},
\doi{10.1007/BFb0089031}
{Lecture Notes in Math. \textbf{1374}, Springer-Verlag, Berlin, 1989}.

\bibitem[Li97]{Li97}
W.~Lickorish,
\textit{Introduction to Knot Theory},
\doi{10.1007/978-1-4612-0691-0}
{Graduate Texts in Mathematics \textbf{175}, Springer, New York, 1997}.

\bibitem[Ly94]{L94}
V.~Lyubashenko,
\textit{Invariants of $3$-Manifolds and Projective Representations of Mapping Class Groups via Quantum Groups at Roots of Unity},
\doi{10.1007/BF02101805}
{Comm. Math. Phys. \textbf{172} (1995), no.~3, 467--516};
\href{http://arXiv.org/abs/hep-th/9405167}
{\texttt{arXiv:hep-th/9405167}}.

\bibitem[Ma71]{Ma71}
S.~Mac~Lane,
\textit{Categories for the Working Mathematician},
\doi{10.1007/978-1-4612-9839-7}
{Grad. Texts in Math. \textbf{5}, Springer-Verlag, New York-Berlin, 1971}.

\bibitem[Ma93]{Ma93}
S.~Majid,
\textit{Braided Groups and Algebraic Quantum Field Theories},
\doi{10.1007/BF00403542}
{Lett. Math. Phys. \textbf{22} (1991), no.~3, 167--175}.

\bibitem[Ma94]{Ma94}
S.~Majid,
\textit{Algebras and Hopf Algebras in Braided Categories},
\doi{https://www.routledge.com/Advances-in-Hopf-Algebras/Bergen-Montgomery/p/book/9780824790653}
{Lecture Notes in Pure and Appl. Math. \textbf{158}, Marcel Dekker, Inc., New York, NY, 1994, 55--105};
\arxiv{q-alg/9509023}{[math.QA]}.

\bibitem[Mo93]{Mo93}
S.~Montgomery,
\textit{Hopf Algebras and Their Actions on Rings},
\doi{10.1090/cbms/082}
{CBMS Regional Conf. Ser. in Math. \textbf{82}, published for the Conference Board of the Mathematical Sciences, Washington, DC, by the American Mathematical Society, Providence, RI, 1993}.

\bibitem[MP92]{MP92}
S.~Matveev, M.~Polyak,
\textit{A Geometrical Presentation of the Surface Mapping Class Group and Surgery},
\doi{10.1007/BF02173428}
{Comm. Math. Phys. \textbf{160} (1994), no.~3, 537--550}.

\bibitem[Oh02]{Oh02}
T.~Ohtsuki,
\textit{Problems on Invariants of Knots and $3$-Manifolds},
\doi{10.2140/gtm.2002.4.377}
{Geom. Topol. Monogr. \textbf{4}, Geometry \& Topology Publications, Coventry, 2002, 377--572};
\arxiv{math/0406190}{[math.GT]}.

\bibitem[RT90]{RT90}
N.~Reshetikhin, V.~Turaev,
\textit{Ribbon Graphs and Their Invariants Derived From Quantum Groups},
\doi{10.1007/BF02096491}
{Comm. Math. Phys. \textbf{127} (1990), 1--26}.

\bibitem[RT91]{RT91}
N.~Reshetikhin, V.~Turaev,
\textit{Invariants of $3$-Manifolds via Link Polynomials and Quantum Groups},
\doi{10.1007/BF01239527}{Invent. Math. \textbf{103} (1991), no.~1, 547--597}.


\bibitem[Tu94]{Tu94}
V.~Turaev,
\textit{Quantum Invariants of Knots and 3-Manifolds},
\doi{10.1515/9783110435221}
{De Gruyter Stud. Math. \textbf{18}, Walter de Gruyter \& Co., Berlin, 1994}.

\bibitem[TV17]{TV17}
V.~Turaev, A.~Virelizier,
\textit{Monoidal Categories and Topological Field Theory},
\doi{10.1007/978-3-319-49834-8}
{Progr. Math. \textbf{322}, Birkhäuser/Springer, Cham, 2017}.

\end{thebibliography}
